\documentclass[12pt,oneside]{book}
\usepackage{longtable}

\usepackage{graphicx,amsthm,fullpage} 
\usepackage{amssymb}
\usepackage{amsfonts}
\usepackage{amsmath}
\usepackage{enumitem}
\usepackage{tabularx}
\usepackage{multirow}
\usepackage{setspace}
\usepackage{units}
\usepackage[mathscr]{euscript}
\usepackage{arydshln}
\usepackage{array}
\usepackage{bbm}
\usepackage{enumitem}
\usepackage{subcaption}
\usepackage{tikz}
\usepackage{scalefnt}
\usepackage{setspace}
\usepackage{bbm}
\usepackage{multicol}

\usepackage{thmtools}

\usepackage{thm-restate}
\usepackage{comment}
\usepackage{float}
\usepackage{lmodern}

\usepackage{epigraph}

\newcommand{\vnorm}[1]{\left\|#1\right\|}

\newcommand{\PLH}{{\mkern-2mu\times\mkern-2mu}}
\newcommand{\bigslant}[2]{{\left.\raisebox{.2em}{$#1$}\middle/\raisebox{-.2em}{$#2$}\right.}}

\makeatletter
\newenvironment{chapquote}[2][2em]
{\setlength{\@tempdima}{#1}%
	\def\chapquote@author{#2}%
	\parshape 1 \@tempdima \dimexpr\textwidth-2\@tempdima\relax%
	\itshape}
{\par\normalfont\hfill--\ \chapquote@author\hspace*{\@tempdima}\par\bigskip}
\makeatother

\newcommand{\brk}{\vspace*{0.18in}}
\newcommand{\smlbrk}{\vspace*{0.04in}}

\newcommand{\tI}{\tilde{I}}

\newcommand{\cI}{{\mathcal I}}
\newcommand{\cJ}{{\mathcal J}}

\newcommand{\cT}{{\mathcal T}}
\newcommand{\tU}{\tilde{U}}

\newcommand{\tW}{\tilde{W}}

\DeclareMathOperator\diam{diam}

\newcommand{\scR}{\mathscr{R}}

\newcommand{\BMA}{\mathbb{A}}
\newcommand{\BMB}{\mathbb{B}}
\newcommand{\BMC}{\mathbb{C}}
\newcommand{\bbF}{\mathbb{F}}
\newcommand{\bbL}{\mathbb{L}}
\newcommand{\bbR}{\mathbb{R}}
\newcommand{\bbZ}{\mathbb{Z}}

\newcommand{\cR}{\mathcal{R}}

\newcommand{\cC}{\mathcal{C}}

\renewcommand{\vec}[1]{\mathbf{#1}}

\newtheorem*{thm*}{Theorem}
\newtheorem{thm}{Theorem}[chapter]
\newtheorem{lem}[thm]{Lemma}
\newtheorem{prop}[thm]{Proposition}
\newtheorem{cor}[thm]{Corollary}
\newtheorem*{remark}{Remark}
\newtheorem*{conj*}{Conjecture}

\theoremstyle{definition}
\newtheorem{example}{Example}[chapter]
\newtheorem*{definition}{Definition}

\newcommand\abstractname{Abstract}  
\makeatletter
\if@titlepage
\newenvironment{abstract}{%
	\thispagestyle{plain}
	\null\vfil
	\@beginparpenalty\@lowpenalty
	\begin{center}%
		\bfseries \abstractname
		\@endparpenalty\@M
\end{center}}%
{\par\vfil\null}
\else
\newenvironment{abstract}{%
	\if@twocolumn
	\section*{\abstractname}%
	\else
	\small
	\begin{center}%
		{\bfseries \abstractname\vspace{-.5em}\vspace{\z@}}%
	\end{center}%
	\quotation
	\fi}
{\if@twocolumn\else\endquotation\fi}
\fi
\makeatother

\usepackage{hyperref}
\usepackage{cleveref}

\usepackage{makeidx}
\makeindex

\author{Brian G. Kodalen}
\date{}
\begin{document}
	\frontmatter
	\thispagestyle{empty}
	\begin{center}
		
		{\large\textbf{Cometric Association Schemes}}

		by
		
		\smlbrk
		Brian G.\ Kodalen
		
		\smlbrk\smlbrk
		A Thesis

		Submitted to the Faculty

		of the

		WORCESTER POLYTECHNIC INSTITUTE

		In partial fulfillment of the requirements for the

		Degree of Doctor of Philosophy

		in

		Mathematical Sciences

		by
		
		\brk\brk
		\rule{3in}{1.2pt}
		
		\smlbrk
		March 2019
	\end{center}
\vfill
APPROVED:

{\setlength{\parindent}{0cm}
\begin{center}
\vspace{0.4in}
\rule{3in}{0.8pt}

Dr.\ William J.\ Martin\\ Major Thesis Advisor
\end{center}
\vspace{0.4in}
\rule{3in}{0.8pt}\hfill\rule{3in}{0.8pt}

Dr.\ Peter J.\ Cameron\hfill Dr. Padraig \'{O} Cath\'{a}in\\University of St Andrews\hfill Worcester Polytechnic Institute

\vspace{0.4in}
\rule{3in}{0.8pt}\hfill\rule{3in}{0.8pt}

Dr. Peter R. Christopher\hfill Dr. William M. Kantor\\Worcester Polytechnic Institute \hfill University of Oregon

\vspace{0.4in}
\rule{3in}{0.8pt}\hfill\rule{3in}{0.8pt}

Dr. G\'{a}bor N. S\'{a}rk\"{o}zy\hfill Dr. Luca Capogna\\Worcester Polytechnic Institute\hfill Worcester Polytechnic Institute\\\hspace*{\fill} Head of Department

}
\clearpage


\begin{abstract}
	\thispagestyle{plain}
	
	The combinatorial objects known as association schemes arise in group theory, extremal graph theory, coding theory, the design of experiments, and even quantum information theory. One may think of a $d$-class association scheme as a $(d+1)$-dimensional matrix algebra over $\mathbb{R}$ closed under entrywise products. In this context, an imprimitive scheme is one which admits a subalgebra of block matrices, also closed under the entrywise product. Such systems of imprimitivity provide us with \textit{quotient schemes}, smaller association schemes which are often easier to understand, providing useful information about the structure of the larger scheme.
	
	One important property of any association scheme is that we may find a basis of $d+1$ idempotent matrices for our algebra. A \textit{cometric} scheme is one whose idempotent basis may be ordered $E_0,E_1,\dots,E_d$ so that there exists polynomials $f_0,f_1,\dots,f_d$ with $f_i\circ(E_1) = E_i$ and $\text{deg}(f_i) = i$ for each $i$. Imprimitive cometric schemes relate closely to $t$-distance sets, sets of unit vectors with only $t$ distinct angles, such as equiangular lines and mutually unbiased bases. Throughout this thesis we are primarily interested in three distinct goals: building new examples of cometric association schemes, drawing connections between cometric association schemes and other objects either combinatorial or geometric, and finding new realizability conditions on feasible parameter sets --- using these conditions to rule out open parameter sets when possible.
	
	After introducing association schemes with relevant terminology and definitions, this thesis focuses on a few recent results regarding cometric schemes with small $d$. We begin by examining the matrix algebra of any such scheme, first looking for low rank positive semidefinite matrices with few distinct entries and later establishing new conditions on realizable parameter sets. We then focus on certain imprimitive examples of both 3- and 4-class cometric association schemes, generating new examples of the former while building realizability conditions for both. In each case, we examine the related $t$-distance sets, giving conditions which work towards equivalence; in the case of 3-class $Q$-antipodal schemes, an equivalence is established. We conclude by partially extending a result of Brouwer and Koolen concerning the connectivity of graphs arising from metric association schemes.
\end{abstract}
\clearpage

\vfill
\begin{center}
	\textbf{Acknowledgments}
\end{center}

First and foremost, I would like to thank my advisor Dr.\ William Martin. The guidance he gave over the last five years has made me a better instructor, writer, researcher, and student. His patience, understanding, and willingness to read every line of this document and many others was essential to my success at WPI. The sheer enjoyment he finds in this subject, paired with his knowledge of many neighboring areas, drew me to this field originally and spurred me on throughout the more frustrating times --- instilling the same love for association schemes within me.

\brk

I would also like to thank those with whom I have had numerous conversations throughout my time at WPI, helping to teach me a broader picture of discrete math. Specifically Dr.\ William Kantor, Dr.\ Tyler Reese, Dr.\ Padraig \'{O} Cath\'{a}in, Dr.\ Peter Christopher, Dr.\  G\'{a}bor N. S\'{a}rk\"{o}zy, and Mike Yereniuk, all of whom have challenged me to step out of my comfort zone to tackle problems in nearby topics. I am additionally thankful to Dr.\ Peter Cameron for agreeing to serve on my dissertation committee along with many of those mentioned above. I am also grateful to the math department at WPI as a whole for the encouragement and support which has characterized my time at WPI.

\brk

Finally, I would like to thank those who have kept me together the last five years. My Mom and Dad, whose love, encouragement, and support has built me into the man I am today. My wife Cara for laughing when I wasn't funny, nodding along when I wasn't making sense, and pushing me forward whenever I was ready to quit. My pastor Neal who took a personal interest in my well-being and met with me consistently as both a counselor and friend. And finally, my Lord and Savior Jesus Christ, without whom I would be lost.
 
 \vfill
 \clearpage

	\tableofcontents
	\clearpage
	\section*{List of Symbols}
	\begin{longtable}{p{.15\textwidth} p{.75\textwidth}}
		$\mathbbm{1}$ & All ones vector\\
		& \\
		$A_i$ & Adjacency matrix\\
		$\mathbb{A},\mathbb{B}$ & Bose-Mesner algebra\\
		$a$ & Vertex\\
		$a^\perp$ & Set containing $a$ and its neighborhood\\
		$a_i^*,b_i^*,c_i^*$ & Non-zero entries of $L_1^*$ for a $Q$-polynomial association scheme\\
		$\alpha$ & Inner product of two vectors\\
		$\alpha(\Gamma)$ & independence number of a Graph $\Gamma$\\
		$\mathcal{A}$ & Regular simplex\\
		&\\
		$\mathcal{B}$ & Set of blocks of a $2$-design\\
		$B$ & Incidence matrix of a block design\\
		
		&\\
		$\mathcal{C}[g]$ & Conjugacy class of group element $g$\\	
		&\\
		$\Delta_m$ & Diagonal matrix with entries $m_0,\dots,m_d$\\
		$\Delta_k$ & Diagonal matrix with entries $k_0,\dots,k_d$\\
		$\delta_{ij}$ & Delta function (1 if $i=j$, 0 otherwise)\\
		$d_{\Gamma}(a,b)$ & Distance between vertices $a$ and $b$ in $\Gamma$\\
		&\\
		$E_j$ & Idempotent matrix corresponding to eigenspace $V_j$\\
		$E_j^\prime$ & Idempotent of a subscheme\\
		$\tilde{E}_j$ & Idempotent of a quotient scheme\\
		$E\Gamma$ & Edge set of graph $\Gamma$\\
		$\vec{e}_j$ & $j^\text{th}$ standard basis vector\\
		&\\
		$\mathbb{F}$ & Field\\	
		$f\circ(A)$ & Function $f$ applied entrywise to matrix $A$\\
		&\\
		$G$ & Group, Gram matrix\\
		$\Gamma$ & Graph\\
		$\Gamma(a)$ & Neighborhood of vertex $a$ in graph $\Gamma$\\
		$\Gamma\backslash S$ & Resultant graph when vertices in $S$ are deleted from the graph $\Gamma$\\
		$\Gamma_a$ & $\Gamma\backslash a^\perp$\\
		&\\
		$H_i$ & Unweighted distribution diagram of $(X,\mathcal{R})$ via $R_i$\\
		&\\
		$I$ & Identity matrix\\
		$\mathcal{I}$ & Set of relations corresponding to a system of imprimitivity\\
		$\iota^*(X,\mathcal{R})$ & Krein array of a $Q$-polynomial association scheme\\
		&\\
		$J$ & All ones matrix\\	
		$\mathcal{J}$ & Set of idempotents corresponding to a system of imprimitivity\\
		&\\
		$K_{m,n}$ & Complete bipartite graph with fibers of size $m$ and $n$\\
		$K_s$ & Complete graph on $s$ vertices\\
		$k_i$ & valency of relation $R_i$\\
		&\\
		$L_i$ & array of intersection numbers\\
		$\mathbb{L}$ & Algebra spanned by arrays of intersection numbers\\
		$L_i^*$ & array of intersection numbers\\
		$\mathbb{L}^*$ & Algebra spanned by arrays of Krein parameters\\
		$\lambda_{j},\lambda_{j,i}$ & Entry of $E_i$ corresponding to Schur-idempotent $A_i$\\
		$\lambda$ & Dual Eigenvalue (Chapter 3)\\
		& $\vert B_i\cap B_j\vert$ for distinct blocks $B_i$ and $B_j$ (Chapter 4)\\
		& $\vert \Gamma(x)\cap\Gamma(y)\vert$ for adjacent vertices $x$ and $y$ in an SRG (Chapter 5)\\
		&\\
		$\left<M\right>_\circ$ & Algebra generated by matrix $M$ using entrywise products\\
		$\left<M\right>_*$ & Algebra generated by matrix $M$ using standard matrix products\\
		$M_{ij}$, $[M]_{ij}$ & Entry of matrix $M$ in row $i$ and column $j$\\
		$M\circ N$ & Entrywise product of $M$ and $N$\\
		$M\otimes N$ & Tensor product of $M$ with $N$\\
		$m_i$ & multiplicity of eigenspace $V_i$\\
		&\\
		$\mathbb{N}$ & Natural numbers\\
		&\\
		$O_n$ & Orthogonal array on $n$ symbols\\
		&\\
		$P$ & First eigenmatrix of an association scheme\\
		$PD_m(\gamma,\delta)$ & Projective double cover in $\mathbb{R}^m$ with angles $\gamma,\delta$\\
		$p^k_{i,j}$ & Intersection number\\
		$\phi$,$\phi^*$ & Algebra isomorphisms\\\
		$\text{proj}_S(\vec{x})$ & Orthogonal projection of vector $\vec{x}$ onto the subspace $S$\\
		&\\
		$Q$ & Second eigenmatrix of an association scheme\\
		$Q_\ell^m$, $Q_\ell$ & Gegenbauer polynomial of degree $\ell$ (with multiplicity $m$)\\
		$q^k_{ij}$ & Krein parameter\\
		$q_0,\dots,q_{d+1}$ & Orthogonal polynomials of a $Q$-polynomial association scheme\\
		&\\
		$R_i$ & Relation\\
		$R_i^\prime$ & Relation of a subscheme\\
		$\tilde{R}_i$ & Relation of a quotient scheme\\
		$R_i(a)$ & $i^\text{th}$ neighborhood of the vertex $a$\\
		$\mathcal{R}$ & Set of relations\\
		$\mathbb{R}$ & Set of real numbers\\
		$\mathscr{R}$ & Ring of polynomials\\
		$\mathscr{R}(t,n)$ & $t^\text{th}$ order Reed Muller code on $n$ variables\\
		$\bigslant{\mathscr{R}}{(f)}$ & Quotient ring of $R$ given the ideal $(f)$\\
		
		&\\
		$\theta_{ij}$ & Eigenvalue of $A_i$ corresponding to the eigenspace $V_j$\\
		$\theta$ & Eigenvalue, angle\\
		
		&\\
		$V_j$ & Eigenspace\\
		$V\oplus W$ & Direct sum of vector spaces $V$ and $W$.\\
		$V\Gamma$ & Vertex set of graph $\Gamma$\\
		
		&\\
		$X$ & Finite set\\
		$\vert X\vert$ & Cardinality of set $X$\\ 
		$\chi$ & Character\\
		$(X,\mathcal{R})$ & Association scheme with point set $X$ and relations $\mathcal{R}$\\
		$(\hat{X},\hat{\mathcal{R}})$ & Association scheme (formally) dual to $(X,\mathcal{R})$\\
		$\vec{x}$ & Vector\\
		$\vert\vert \vec{x}\vert\vert$ & Norm of the vector $\vec{x}$\\
		
		&\\
		$\mathbb{Z}$ & Set of integers\\
		$\mathbb{Z}_n$ & Cyclic group on $\left\{0,\dots,n-1\right\}$\\

	\end{longtable}
	\clearpage

	\mainmatter
	
	\chapter{Introduction}\label{introduction}
	\begin{chapquote}{Bertrand Russell, 1907}
		``Mathematics, rightly viewed, possesses not only truth, but supreme beauty"
	\end{chapquote}
	Much of the motivation for the theory of association schemes arises from coding theory; for the purpose of illustration, we will use this application as an entry point into our discussion of association schemes. A binary code of length $n$ may simply be viewed as a subset of $\mathbb{Z}_2^n$. First consider the parity check code on two bits: $P = \left\{000,011,101,110\right\}.$ This code has the additional property that it forms a subspace of $\mathbb{Z}_2^3$, not just a subset; any code with this property is called a \emph{linear code}. Given a linear code, $C$, we may represent the code using a \emph{generator matrix} -- a matrix whose rows form a basis for $C$; we say the \emph{dimension} of $C$ is the number of rows in the generator matrix. For instance, the parity check code on two bits may be described as $P = \text{rowspan}\left[\begin{array}{ccc}
	1 & 1 & 0\\
	0 & 1 & 1\\
	\end{array}\right]$ and thus has dimension $2$. We may equivalently define this code via $P = \text{null}\left[\begin{array}{ccc}
	1 & 1 & 1\\
	\end{array}\right]$. The dual code of a linear code $C$, denoted $C^\perp$, consists of the subspace formed by swapping the two previous matrices. Returning to our example, we find $P^\perp = \text{null}\left[\begin{array}{ccc}
	1 & 1 & 0\\
	0 & 1 & 1\\
	\end{array}\right] = \text{rowspan}\left[\begin{array}{ccc}
	1 & 1 & 1\\
	\end{array}\right].$ Given a code $C$ of length $n$, we may form graphs $\Gamma_1,\dots,\Gamma_n$ on $C$ where two codewords are adjacent in $\Gamma_i$ if and only if they differ in exactly $i$ positions. Using $P$ as our code, we find $\Gamma_1$ and $\Gamma_3$ are both empty while $\Gamma_2\simeq K_4$, the complete graph on four vertices. Using the dual code $P^\perp$ instead, we find that $\Gamma_1$ and $\Gamma_2$ are both empty while $\Gamma_3\simeq K_2$. The interaction of these two codes and their corresponding graphs will be discussed later as subobjects of one association scheme called the 3-cube.
	
	We now move to a family of codes known as the Reed Muller codes, denoted $\mathcal{R}(t,m)$ for $t\geq 0$ and $m\geq 1$. For fixed $t$ and $m$, $\mathcal{R}(t,m)$ is a linear code of length $2^m$ with dimension $\sum_{i=0}^t\binom{m}{i}$. While there are many ways to represent the codewords of this family, we will use a construction relying on binary polynomials. Let $P_t\subset\bbZ_2[x_1,\dots,x_m]$ be the space of binary polynomials of degree $t$ or less on $m$ variables. We begin by imposing an ordering on the elements of $\bbZ_2^m$, say $p_1,\dots,p_{2^m}$. Then, for each $f\in P_t$, we build the corresponding codeword, $c_f$, by evaluating $f$ at each point of $\bbZ_2^m$ in order; that is, the $i^\text{th}$ element of $c_f$ is given by $f(e_i)$. As $P_t$ is a vector space, we may create codewords for each polynomial in some basis of $P_t$ and use the resultant codewords as the rows of our generator matrix; that is, if $\left\{f_1,\dots,f_\ell\right\}$ is a basis for $P_t$ then $c_{f_1},\dots,c_{f_\ell}$ is a basis for the binary code. Since we have indexed each entry of each codeword by some element of $\bbZ_2^m$, we find that this code has length $2^m$. Further, consider the basis of $P_t$ given by the set of monomials. This basis has $\sum_{i=0}^t\binom{m}{i}$ polynomials, giving us the dimension of our code. For example, $\mathcal{R}(1,4)$ may be generated using the following generator matrix $M$, where the rows of $M$ are indexed by the basis $1$, $x_1$, $x_2$, $x_3$, and $x_4$, while the columns are indexed by the elements of $\mathbb{Z}_2^4$ ordered lexicographically. Then the element in row $p$ and columns $x$ of $M$ is $p(x)$.
	\[M = \left[\begin{array}{cccccccccccccccc}
	1 & 1 & 1 & 1 & 1 & 1 & 1 & 1 & 1 & 1 & 1 & 1 & 1 & 1 & 1 & 1\\
	0 & 0 & 0 & 0 & 0 & 0 & 0 & 0 & 1 & 1 & 1 & 1 & 1 & 1 & 1 & 1\\
	0 & 0 & 0 & 0 & 1 & 1 & 1 & 1 & 0 & 0 & 0 & 0 & 1 & 1 & 1 & 1\\
	0 & 0 & 1 & 1 & 0 & 0 & 1 & 1 & 0 & 0 & 1 & 1 & 0 & 0 & 1 & 1\\
	0 & 1 & 0 & 1 & 0 & 1 & 0 & 1 & 0 & 1 & 0 & 1 & 0 & 1 & 0 & 1\\
	\end{array}\right].\]
	Here, the coordinates are indexed by $0000,$ $0001,$ $0010,$ $0011,$ etc. This example, with $t=1$, contains 32 distinct codewords, though the size of the code increases rapidly with $t$. In fact, $\mathcal{R}(2,4)$ contains $2048$ distinct codewords. Unfortunately, it is not only the number of codewords we typically care about. Another main parameter we are interested in is the \emph{minimum distance} --- the smallest number of entries in which unequal codewords may differ. It is in this parameter that we pay for the extra codewords in the higher order Reed Muller code; the minimum distance of $\mathcal{R}(1,4)$ is 8, while $\mathcal{R}(2,4)$ has a minimum distance of only half that. Given the large minimum distance of $\mathcal{R}(1,4)$ and the large size of $\mathcal{R}(2,4)$, the question arises: what is the largest subcode of $\mathcal{R}(2,4)$ such that the minimum distance is six? One may show that any generator matrix cannot have more than seven rows and thus we will not find any linear subcode with more than 128 codewords. However, we may do better than this if we do not require linearity. Thus, we will instead define a code explicitly by providing a polynomial for each and every codeword. First, consider the eight quadratic polynomials
	\[\begin{aligned}p_1 &=x_1x_2+x_1x_3+x_1x_4+x_2x_3+x_2x_4+x_3x_4,\\
	p_2 &= x_1x_2+x_2x_3+x_3x_4,\\p_3 &=x_1x_2+x_2x_4+x_4x_3,\\p_4 &=x_1x_3+x_3x_2+x_2x_4,\\p_5 &=x_1x_3+x_3x_4+x_4x_2,\\p_6 &=x_1x_4+x_4x_2+x_2x_3,\\p_7 &=x_1x_4+x_4x_3+x_3x_2,\\p_8 &=0.\end{aligned}\]
	Each of these determines a coset of $\mathcal{R}(1,4)$ inside $\mathcal{R}(2,4)$ by adding the resultant codeword to each of the words in $\mathcal{R}(1,4)$; for example, the coset corresponding to $p_8$ is $\mathcal{R}(1,4)$ itself. The union of these cosets gives us a code with 256 distinct words with minimum distance six. This code is known as the (extended) Nordstrom-Robinson code, the first in an infinite family of non-linear codes which may be defined similarly by taking cosets of the first order Reed Muller codes inside the respective second order Reed-Muller code.
	
	It turns out this code has an interesting history behind it. The code, originally given in \cite{Nordstrom1967}, was found by a high-school student after attending an introductory talk at his school. John Robinson, a professor at the University of Iowa at the time, gave the talk in the mid 1960's in which he discussed both linear and non-linear binary codes. After introducing the best possible linear code of length 15 and minimum distance 5 (the double-error-correcting BCH code), Robinson pointed out that the upper bound on non-linear codes with the same length and minimum distance was a factor of 2 greater --- yet no such code was known.  Alan Nordstrom responded to the challenge and, through trial and error, was able to produce what is now known as the Nordstrom-Robinson code. This code attracted attention quickly and within a few years it was discovered that the extended version (as described above) may be generalized to two infinite families of non-linear codes, first the Preparata codes in 1968 \cite{Preparata1968} and four years later the Kerdock codes \cite{Kerdock1972}.
	
	Perhaps one of the most intriguing questions arising from these families at the time was the notion that they were formally dual --- despite the notion of ``duality" being a property of linear codes. Recall that we define the dual of a linear code as the null-space of the generator matrix --- that is, the dual code consists of all codewords which are orthogonal to every codeword of the original code. Using the MacWilliams identity \cite{MacWilliams1963}, this notion was generalized to a statement about the parameters of codes. For a binary code $C$ of length $n$, MacWilliams defined the \emph{weight distribution} as the sequence of numbers $A_t = \left\vert\left\{c\in C\mid w(c)=t\right\}\right\vert$ where $w(c)$ counts the number of non-zero entries of codeword $c$. Then the \emph{weight enumerator polynomial} is given by
	\[W(C;x,y) = \sum_{t=0}^n A_t x^ty^{n-t}.\]
	MacWilliams showed that any pair of dual codes $C$ and $C^\perp$ must satisfy the identity
	\begin{equation}W(C^\perp;x,y) = \frac{1}{\vert C\vert}W(C;y-x,y+x).\label{macwill}\end{equation}
	Thus MacWilliams defined the notion of formal duality; we say two codes are \emph{formally dual} if they satisfy Equation \eqref{macwill}. For linear codes, this allows us to show certain linear codes do not exist --- that is, given $A_t$, if any coefficient of $\frac{1}{\vert C\vert}W(C;y-x,y+x)$ is not a non-negative integer, we have a quick proof that $C^\perp$ does not exist based solely on its purported weight enumerator. The converse however is not true; in fact, we may find a non-linear code $C$ for which no formally dual code exists, while the right hand side of Equation \eqref{macwill} has only non-negative integer coefficients. Therefore it is important to emphasize that the notion of formal duality is a statement of the parameters of a code, not the code itself. In fact, a linear code may have many codes formally dual to it, despite always having a unique dual. As an example in the non-linear case, the weight enumerator of the Nordstrom-Robinson code is
	\[y^{16}+112x^6y^{10}+30x^8y^{8}+112x^{10}y^6+x^{16}.\]
	One may check that the RHS of Equation \eqref{macwill} results in the same polynomial; we therefore say that the Nordstrom-Robinson code is \emph{formally self-dual}. More generally, one finds that the Kerdock and Preparata codes are formally dual codes. However, since this duality is based solely on the parameters, it does not provide a way to construct one family from the other. It was not until two decades later in 1994 that Hammons, Kumar, Calderbank, Sloane, and Sole \cite{Hammons1994} showed that certain codes with these parameters are the images of submodules of $\bbZ_4^n$ under the Gray map --- that is, they are linear when viewed as codes of length 8 with an alphabet of size 4. This was illuminated further by Calderbank, Cameron, Kantor, and Seidel \cite{Calderbank1997} who gave a geometric path from the binary Kerdock codes to $\bbZ_4$-Kerdock codes. Thus, while there cannot exist linear binary codes with these parameters, one may two construct $\bbZ_4$-submodules of $\bbZ_4^n$ corresponding to each parameter set which are dual in the traditional sense.
	
	Outside the question of duality, the Kerdock codes have many other, quite fascinating, connections. In the early 1970s, Cameron \cite{Cameron1972} introduced a type of multipartite graph called a \emph{linked system of symmetric designs} (``LSSDs", refer to Chapter \ref{3class} for more detail). Around that time, Goethals communicated to Cameron that one may build examples of such objects using the Kerdock codes; these examples where shown to be optimal with respect to the number of fibers \cite{Noda1974}. This family of LSSDs became known as the Cameron-Seidel association scheme (see Section \ref{kerdock}), remaining the archetypal example of LSSDs even to this day. A second (though not completely independent) use of Kerdock codes is in the construction of real mutually unbiased bases. Here, we look for orthonormal bases in $\bbR^m$ where vectors from distinct bases have an inner product of $\pm\frac{1}{\sqrt{m}}$. With connections to quantum cryptography and Euclidean geometry, mutually unbiased bases have been an area of interest for quite some time now. It was shown using quadratic forms \cite{Cameron1973} that the Kerdock sets not only gave examples of real MUBs, but that these examples were optimal with respect to the number of bases \cite{Calderbank1997} --- this is the only known infinite family of real MUBs achieving this upper bound.
	
	A similar problem is that of finding lines in $\bbR^m$ in which any pair of lines intersect in a fixed angle; such sets of lines are called ``equiangular lines". Gerzon showed that the upper bound on the number of lines in $\bbR^m$ is given by $\frac{m(m+1)}{2}$ \cite{Lemmens1973}, yet the known constructions all scaled linearly with the dimension. It was not until nearly 30 years later that de Caen \cite{deCaen2000} used the Cameron-Seidel scheme to build $\frac{2}{9}\left(d+1\right)^2$ real equiangular lines whenever $d = 3\left(2^{2t-1}\right)-1$ for some positive integer $t$, resulting in the first known infinite family of size quadratic in the dimension.
	
	We therefore find the Kerdock codes, the first example of which was discovered by a high school student, have deep connections to many other areas of study including design theory, quantum cryptography, and equiangular lines; objects such as these clearly warrant further study. Central to many of these connections is the fact that the Cameron-Seidel scheme --- the graphs given the distinct distances in any Kerdock code --- forms a 3-class association scheme. A \textit{symmetric $d$-class association scheme} (see Chapter \ref{association} for a more thorough definition) may be thought of as a edge-coloring of the complete graph using $d$ colors such that: given any colors $c_1$, $c_2$, and $c_3$, the number of $c_1,c_2,c_3$ triangles containing a fixed $c_1$-edge depend only on the colors chosen, not on the edge. We also include the ``0-color" as the graph of loops where we define a ``triangle" containing a loop as a loop paired with any incident edge. Using the Kerdock codes as an example, we color the edges by distance between codewords and this tells us that any pair of words at distance $i$ have a constant number of words distance $j$ from one and $k$ from the other, independent of the pair chosen.
	
	Within the study of association schemes, we will often find rich connections to other areas of mathematics. In this thesis we will examine a type of association schemes known as ``cometric" (see Section \ref{poly} for the definition); this class includes many of the objects mentioned already. Within the field of association schemes, we find many parameters which describe the structure of an association scheme (see Chapter \ref{association}) analogous to the weight distribution of a binary code. Using these parameters, we arrive at a notion of formal duality of association schemes; two association schemes are formally dual if the first and second eigenmatrices of one are swapped for the other. Just as formal duality for general codes arose from explicit duality of linear codes, formal duality in association schemes arises from character duality of abelian groups. Given an abelian group $G$, the dual group $G^*$ is given by taking the set of characters of $G$ where for $\chi,\psi\in G^*$ and $x\in G$, $(\chi*\psi) (x) = \chi(x)*\psi(x)$. If there exists an abelian group acting sharply transitively on the points of an association scheme, then the dual scheme is well-defined. However, without this added structure, there is no clear way to build the ``dual" of a general association scheme. Despite this, we may define duality formally, at the parameter level, and find concrete examples of formally dual pairs of association schemes without a clear way of constructing one from the other. We finish this introduction with a brief history of association schemes followed by an outline of this thesis.
	
	First introduced by Bose and Nair \cite{Bose1939} in 1939 with connections to certain block designs, the algebraic structure known as an ``association scheme" was formally defined later in 1952 by Bose and Shimamoto \cite{Bose1952} as a set of relations on a point set satisfying certain strong regularity properties (see Chapter \ref{association}). It was not until seven years later that Bose and Mesner \cite{Bose1959} described the equivalence between association schemes and Schur-closed matrix algebras --- commutative vector spaces of matrices closed under two distinct matrix products. Around the same time Wielandt was expanding on the theory of Schur (\cite{Wielandt1964},\cite{Schur1933}) to understand the commuting algebra, or centralizer ring, of permutation groups. These two concepts were generalized together by Higman in 1967 \cite{Higman1967} who discussed so-called ``coherent configurations". Shortly thereafter Biggs introduced a generalization of distance-transitive graphs known as ``distance regular graphs" \cite{Biggs1971}, showing that the adjacency matrix of any such graph generates the matrix algebra of an association scheme with very particular ``polynomial" properties. Over the next few years, Biggs continued to develop the notion of distance-regular graphs and their relationship with association schemes, culminating in parts of his book Algebraic Graph Theory \cite{Biggs1974}. Arguably one of the most influential works on this topic is the thesis of Delsarte in 1973 \cite{Delsarte1973}, developed seemingly independently of Biggs. In this thesis, he lays out the definitions and parameters central to association schemes, discusses subsets of associations schemes, and defines both $P$-polynomial (metric) and $Q$-polynomial (cometric) association schemes; the former are equivalent to Biggs' notion of distance regular graphs while the latter remained largely unexplored until decades later. These cometric examples will be the main focus of this thesis. Delsarte devoted particular attention to the Johnson and Hamming schemes, defining more clearly the notion of duality within these two schemes and bringing to the forefront the connections between association schemes and coding theory.
	
	The two decades that followed brought with them many new results concerning polynomial association schemes, especially those that are $P$-polynomial. Authors such as Biggs, Damerell, Gardiner, Meredith, and Smith (see \cite{Brouwer1989} for a list of their relevant publications) continued to develop our understanding of distance-regular and distance-transitive graphs. Meanwhile authors such as Terwilliger \cite{Terwilliger1986} and Neumaier \cite{Neumaier1985} focused on specific families, developing parameter characterizations of Johnson and Hamming schemes --- the major examples of metric association schemes. Terwilliger then went on to work towards classifying association schemes which are both $P$- and $Q$-polynomial in papers such as \cite{Terwilliger1987} and \cite{Terwilliger1988}. Much of what is known has been compiled into books, first by Bannai and Ito in \cite{Bannai1984}, then by Brouwer, Cohen, and Neumaier in \cite{Brouwer1989}, and most recently by Bailey in \cite{Bailey2004}.
	
	Despite the great attention devoted to $P$-polynomial schemes, it seems not much progress was made in understanding their $Q$-polynomial analogues until Dickie's thesis \cite{Dickie1995} in 1995 and two papers of Suzuki (\cite{Suzuki1998},\cite{Suzuki1998-2}) three years later. In the latter two papers, Suzuki showed, apart from cycles, any $Q$-polynomial association scheme may have at most two $Q$-polynomial orderings and any imprimitive $Q$-polynomial association scheme must be either $Q$-bipartite or $Q$-antipodal (except possibly for two sporadic cases which were later ruled out). These results were analogues of results for distance regular graphs dating as much as three decades prior, yet the method for proving these results was quite different. These papers triggered a resurgence of interest in cometric association schemes as the following two decades brought many new results. Some results included finding equivalences between certain classes of cometric association schemes and other geometric structures; for instance \cite{VanDam1999} discusses $3$-class $Q$-antipodal schemes while \cite{LeCompte2010} focuses on $4$-class schemes which are both $Q$-antipodal and $Q$-bipartite. New examples were found, including families discovered by Penttila and Williford \cite{Penttila2011}, another family found by Moorhouse and Williford \cite{Moorhouse2016}, and many new sporadic examples found by Gavin King \cite{King2018}. While far from an exhaustive list, this author would be remiss without also mentioning papers of Suda (\cite{Suda2011},\cite{Suda2012}), van Dam, Martin and Muzychuk \cite{vanDam2013}, Martin and Williford \cite{Williford2009}, and Martin, Muzychuk and Williford \cite{Martin2007}.
	
	In this thesis, we begin by defining association schemes and their associated Bose-Mesner algebras. The remainder of Chapter \ref{association} consists of the various definitions which occur within this field such as the parameters, feasibility conditions, substructures, and polynomial structures --- some of which are mentioned above. We then focus on the matrix algebra in Chapter \ref{psdcone}, where we examine the cone of positive semidefinite matrices. Here, we introduce methods to build other line systems, such as equiangular lines, as well as develop a new feasibility condition on association schemes using a theorem of Sch\"{o}nberg --- we explicitly calculate many of these new conditions for cometric schemes, in particular. In Chapter \ref{3class}, we recall the definition of linked systems of symmetric designs (LSSDs) defined by Cameron \cite{Cameron1972} and explore constructions of these as well as connections to other objects. We review past results on such objects including the Cameron-Seidel scheme mentioned previously, results of Noda on the maximum size of any such object, and their equivalence with 3-class $Q$-antipodal association schemes, a result of van Dam. We then introduce a new geometric object called a ``set of linked simplices" and we show that these are equivalent to LSSDs. Using this new equivalence, we investigate when we may build real mutually unbiased bases from these association schemes as well as construct new examples using parameters distinct from those of the Cameron-Seidel scheme. Chapter 5 focuses on $4$-class $Q$-bipartite association schemes, beginning with a motivational discussion indicating how these association schemes naturally occur. While we do not develop an equivalence in this chapter between the association schemes and so-called ``orthogonal projective double covers", we make substantial progress towards that goal. In the final chapter, Chapter 6, we investigate the connectivity of relations in association schemes in general. With the goal of understanding the ``nearest-neighbor graph" of a cometric association scheme, we show that any connected graph (in the absence of twin vertices --- vertices with the same neighborhood) within an association scheme remains connected after deleting the entire neighborhood of a vertex (including the vertex itself). We use this result to show that any graph within an association scheme with connectivity equal to 2 must be a cycle. In the appendix, we finish be surveying the known infinite families of symmetric designs in order to determine which families could yield the parameter sets of LSSDs with more than two fibers.
	
	\chapter{Association schemes}\label{association}
Association schemes arise in group theory, graph theory, design theory, coding theory and more. For example, let $X$ be a finite group with conjugacy classes $\cC[g] = \{hgh^{-1}:h\in X\}$ ($g\in X$). We may form relations on the pairs of group elements via $R_g = \left\{ (a,b) \mid ab^{-1} \in \cC[g]  \right\}$; in this way, each conjugacy class results in a single relation independent of the representative group element chosen. We may then merge the relations $R_g\cup R_{g^{-1}}$ to arrive at symmetric relations. The group, along with the relations, then form a $d$-class association scheme where $d$ is the number of non-trivial relations formed (relations other than $R_e$ for identity element $e$). While we may easily construct such an association scheme from a group, we find that there are many more association schemes which do not require the group structure on the point set. In fact, for any finite set $X$, the orbits on $X\times X$ of any 
permutation group $\mathcal{G}$ acting generously transitively (for any two points in $X$, there exists a group element swapping these elements) also gives an association scheme.
Some of the most well-studied association schemes are distance-regular graphs, including Moore graphs, distance-transitive graphs, strongly regular graphs, generalized polygons, etc.  For an introduction to the 
extensive literature on the subject, the reader may consult \cite{Delsarte1973,Bannai1984,Brouwer1989,Godsil1993}, 
the survey \cite{Martin2009}, or the more recent book of  Bailey \cite{Bailey2004} which focuses on 
connections to the statistical design of experiments.
\begin{definition}
	Let $X$ be a finite set of vertices. A \textit{symmetric d-class association scheme}\index{association scheme!symmetric} (see \cite{Brouwer1989}) on $X$ is a pair $(X,\mathcal{R})$ where $\mathcal{R} =\left\{R_0,R_1,\dots,R_d\right\}$ is a set of $d+1$ relations on $X$ satisfying the following properties:
	\begin{enumerate}[label=$(\roman*)$]
		\item $R_0$ is the identity relation;
		\item $\left\{R_0,R_1,\dots, R_d\right\}$ forms a partition of $X\times X$;
		\item $(x,y)\in R_i$ implies $(y,x)\in R_i$;
		\item for $0\leq i,j,k\leq d$ there exist constants $p_{i,j}^k$ such that for any vertices $x,y\in X$ with $(x,y)\in R_k$, the number of vertices $z$ for which $(x,z)\in R_i$ and $(z,y)\in R_j$ is $p_{i,j}^k$, depending only on $i$, $j$, and $k$.
	\end{enumerate}
\end{definition}
The constants $p_{i,j}^k$ are known as the \emph{intersection numbers}\index{parameters!intersection numbers} of our association scheme and we allow ourselves to suppress the comma whenever $i$ and $j$ are given by single digits, thus $p_{5,2}^7$ and $p_{52}^7$ are synonymous throughout this thesis. Properties \emph{(iii)} and \emph{(iv)} together imply that $p_{ij}^k = p_{ji}^k$ for all $i,j,k$; we call such an association scheme \textit{commutative}. There is a broader definition for a \emph{commutative association scheme}\index{association scheme!commutative} where we replace \emph{(iii)} with the condition that for every $i$, there exists some $i'$ such that $R_{i'} = R_i^T$; that is $(x,y)\in R_i$ if and only if $(y,x)\in R_{i'}$. In this case however, we add the property $p_{ij}^k = p_{ji}^k$. Throughout this thesis, all association schemes will be symmetric, though we will add remarks at times when the theorems apply directly to the non-symmetric case as well.

For each $0\leq i\leq d$ we define the (undirected) graph $\Gamma_i = \Gamma(X,R_i)$ on $X$ with $\Gamma_1,\dots,\Gamma_d$ all simple. Note, throughout this thesis we will use the notation $\Gamma(V,E)$ to denote a graph with vertex set $V$ and edge set $E\subset V\times V$. For each $a\in X$ we define the $i^\text{th}$ \emph{neighborhood}\index{relation!neighborhood} of $a$ $R_i(a) = \left\{b\in X\mid (a,b)\in R_i\right\}$; i.e. $R_i(a)$ is the neighborhood of $a$ in the graph $\Gamma_i$. Then for any $a\in X$, the set $X$ is partitioned into the \emph{subconstituents}\index{relation!subconstituents} $\left\{R_i(a)\mid 0\leq i\leq d\right\}$ with respect to $a$.
\begin{example} The following association scheme is known as the \emph{3-cube}\index{3-cube} with vertex set $X = \left\{0,\dots,7\right\}$ and relations corresponding to the graphs $\Gamma_0,\dots,\Gamma_3$ given below.
	\begin{figure}[H]\begin{center}\scalebox{.7}{$\begin{aligned}
				\begin{tikzpicture}[shorten >=1pt,auto,node distance=2cm,
				thin,main node/.style = {circle,draw, inner sep = 0pt, minimum size = 15pt}]
				
				\node[main node,fill=white] (1) {0};
				\node[main node,fill=red] [right of = 1](2) {1};
				\node[main node,fill=red] [above of = 1](3) {2};
				\node[main node,fill=cyan] [right of = 3](4) {3};
				\node[main node,fill=red] [above right of = 1](5) {4};
				\node[main node,fill=cyan] [right of = 5](6) {5};
				\node[main node,fill=cyan] [above of = 5] (7) {6};
				\node[main node,fill=green] [right of = 7](8) {7};
				\node at (2,4.5) (9) {$\Gamma_0$};
				
				\path (1) edge [in=120,out=145,loop] ();
				\path (2) edge [in=120,out=145,loop] ();
				\path (3) edge [in=120,out=145,loop] ();
				\path (4) edge [in=120,out=145,loop] ();
				\path (5) edge [in=120,out=145,loop] ();
				\path (6) edge [in=120,out=145,loop] ();
				\path (7) edge [in=120,out=145,loop] ();
				\path (8) edge [in=120,out=145,loop] ();
				
				\end{tikzpicture}\qquad&\qquad\begin{tikzpicture}[shorten >=1pt,auto,node distance=2cm,
				thin,main node/.style = {circle,draw, inner sep = 0pt, minimum size = 15pt}]
				
				\node[main node,fill=white] (1) {0};
				\node[main node,fill=red] [right of = 1](2) {1};
				\node[main node,fill=red] [above of = 1](3) {2};
				\node[main node,fill=cyan] [right of = 3](4) {3};
				\node[main node,fill=red] [above right of = 1](5) {4};
				\node[main node,fill=cyan] [right of = 5](6) {5};
				\node[main node,fill=cyan] [above of = 5] (7) {6};
				\node[main node,fill=green] [right of = 7](8) {7};
				\node at (2,4.5) (9) {$\Gamma_1$};
				
				\draw[-] (3)--(1)--(2)--(4)--(3)--(7)--(8)--(6)--(5)--(7);
				\draw[-] (1)--(5)--(6)--(2)--(4)--(8);
				\end{tikzpicture}\qquad\qquad
				\begin{tikzpicture}[shorten >=1pt,auto,node distance=2cm,
				thin,main node/.style = {circle,draw, inner sep = 0pt, minimum size = 15pt}]
				
				\node[main node,fill=white] (1) {0};
				\node[main node,fill=red] [right of = 1](2) {1};
				\node[main node,fill=red] [above of = 1](3) {2};
				\node[main node,fill=cyan] [right of = 3](4) {3};
				\node[main node,fill=red] [above right of = 1](5) {4};
				\node[main node,fill=cyan] [right of = 5](6) {5};
				\node[main node,fill=cyan] [above of = 5] (7) {6};
				\node[main node,fill=green] [right of = 7](8) {7};
				\node at (2,4.5) (9) {$\Gamma_2$};
				
				\path[-]
				(1)edge [bend left=25] node {} (4)
				edge node {} (6)
				edge node {} (7)
				(2)edge node {} (3)
				edge node {} (5)
				edge node {} (8)
				(3)edge node {} (5)
				edge node {} (8)
				(4) edge node {} (6)
				(7) edge node {} (4)
				edge node {} (6)
				(5) edge [bend right = 25] node {} (8);
				\end{tikzpicture}\qquad&\qquad
				\begin{tikzpicture}[shorten >=1pt,auto,node distance=2cm,
				thin,main node/.style = {circle,draw, inner sep = 0pt, minimum size = 15pt}]
				
				\node[main node,fill=white] (1) {0};
				\node[main node,fill=red] [right of = 1](2) {1};
				\node[main node,fill=red] [above of = 1](3) {2};
				\node[main node,fill=cyan] [right of = 3](4) {3};
				\node[main node,fill=red] [above right of = 1](5) {4};
				\node[main node,fill=cyan] [right of = 5](6) {5};
				\node[main node,fill=cyan] [above of = 5] (7) {6};
				\node[main node,fill=green] [right of = 7](8) {7};
				\node at (2,4.5) (9) {$\Gamma_3$};
				
				\path[-]
				(1) edge [bend right] node {} (8)
				(2) edge [bend left=45] node {} (7)
				(3) edge [bend left=45] node {} (6)
				(4) edge node {} (5);
				\end{tikzpicture}
				\end{aligned}$}\end{center}
		\caption[Graphs of the 3-cube.]{The four graphs of the 3-cube. The four subconstituents of the vertex $0$ are colored white, red, blue, and green respectively.}\label{3cube}
	\end{figure}
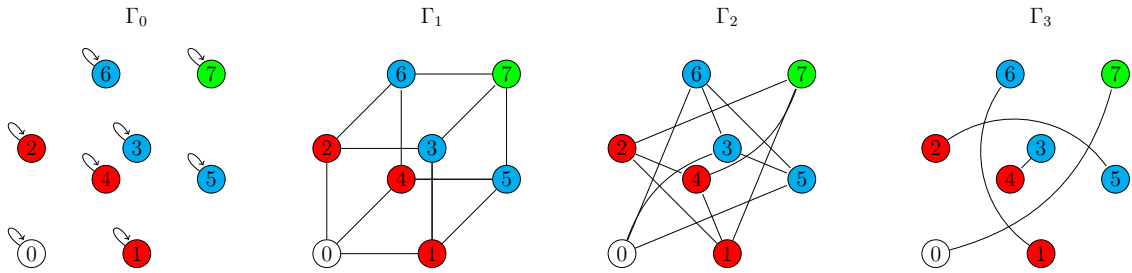
	The following matrices give the intersection numbers of this association scheme where the $i^\text{th}$ matrix contains $p^k_{ij}$ with rows indexed by $k$ and columns indexed by $j$:
	\[\left[\begin{array}{cccc}
	1&0&0&0\\
	0&1&0&0\\
	0&0&1&0\\
	0&0&0&1\\
	\end{array}\right],\quad\left[\begin{array}{cccc}
	0&3&0&0\\
	1&0&2&0\\
	0&2&0&1\\
	0&0&3&0\\
	\end{array}\right],\quad\left[\begin{array}{cccc}
	0&0&3&0\\
	0&2&0&1\\
	1&0&2&0\\
	0&3&0&0\\
	\end{array}\right],\quad\left[\begin{array}{cccc}
	0&0&0&1\\
	0&0&1&0\\
	0&1&0&0\\
	1&0&0&0\\
	\end{array}\right].\]
	Since $p^k_{ij} = p^k_{ji}$, many of the columns listed above are redundant, thus we may instead give a more brief list of the intersection numbers as follows:
	\[\begin{array}{c|cccc|ccc|cc|c}
	k & p^{k}_{0,0}	&p^{k}_{0,1}& p^{k}_{0,2} 	& p^{k}_{0,3}   & p^{k}_{1,1}	&p^{k}_{1,2}& p^{k}_{1,3} 	& p^{k}_{2,2} 	& p^{k}_{2,3} 	& p^{k}_{3,3}\\\hline
	0 & 1			& 0			& 0 			& 0				& 3				& 0			& 0 			& 3				& 0   			& 1\\
	1 & 0			& 1			& 0 			& 0				& 0				& 2			& 0 			& 0   			& 1   			& 0\\
	2 & 0			& 0			& 1 			& 0				& 2				& 0 		& 1				& 2 			& 0 			& 0\\
	3 & 0			& 0			& 0 			& 1				& 0 			& 3			& 0 			& 0   			& 0 			& 0\\
	\end{array}\]
	We often find this brief description useful and will further reduce our description of the parameters when there is no loss of clarity.
	
	We note with this example that for each $i$, $\Gamma_i$ may be formed by taking the distance $i$ graph of $\Gamma_1$. This will not hold in general --- that is, we will typically not find such a graph encoding our relations using distance. When it does however, we say the association scheme is metric (see Section \ref{poly}) and $\Gamma_1$ is a distance-regular graph (drg). Due to the encoding of the relations in such a drg, metric schemes and their paired drgs are referred to synonymously; for instance, this association scheme was refereed to as the ``3-cube".
\end{example}
For any $0\leq i\leq d$ and any vertex $x\in X$,
\[p^{0}_{ii} = \left\vert\left\{y:(y,x)\in R_i\right\}\right\vert = \left\vert R_i(x)\right\vert.\]
Thus we define $k_i:=p^0_{ii}$ as the \emph{valency}\index{valency} of the $i^\text{th}$ relation. Many other restrictions on our intersection numbers follow immediately from our definition, for instance $p^0_{12} = 0$; we will summarize these in a lemma at the end of the next section.
\section{Bose-Mesner algebra}\label{BMA}
Often it becomes useful to order the vertices in $X$ and represent each $R_i$ as a 01-matrix $A_i$ where the $(x,y)$ entry of $A_i$ is 1 if and only if $(x,y)\in R_i$; thus $A_i$ is the adjacency matrix of $\Gamma_i$. Let $\vert X\vert = n$ and denote the $n\times n$ identity as $I_n$ and the $n\times n$ matrix of ones as $J_n$; we suppress the subscript $n$ when it is clear from the context. The defining properties of a symmetric association scheme are then encoded as:
\begin{enumerate}[label=$(\roman*)$]
	\item $A_0 = I$;
	\item $\sum_i A_i = J$;
	\item for all $0\leq i\leq d$, $A_i^T = A_i$;
	\item for all $0\leq i,j\leq d$, $A_iA_j = \sum p_{ij}^k A_k$,
\end{enumerate}
where each $A_i$ has only zeros and ones as entries. The fourth condition tells us that $\BMA = \text{span}\left\{A_0,A_1,\dots A_d\right\}$ forms a matrix algebra under standard matrix multiplication --- we will refer to the matrices $\left\{A_i\right\}$ as our basis of 01-matrices\index{01-basis}. We call this algebra the \emph{Bose-Mesner algebra}\index{Bose-Mesner algebra} and note that the remaining conditions ensure it is a $(d+1)$-dimensional algebra of symmetric matrices containing the identity. Further, as our basis matrices are 01-matrices with pairwise disjoint support, this algebra is also closed under Schur (entrywise) products and contains the Schur identity, $J$. We find that, conversely, any such an algebra determines an association scheme; that is, any $(d+1)$-dimensional vector space of symmetric matrices closed under both standard and Schur matrix products containing the identities for both operations corresponds to the Bose-Mesner algebra of some symmetric association scheme. Throughout, we will use this algebraic definition interchangeably with the combinatorial definition. As our algebra contains only symmetric matrices, it is necessarily commutative (more generally the assumption $p^k_{ij} = p^k_{ji}$ made for commutative association schemes is sufficient to guarantee the Bose-Mesner algebra is commutative). Therefore, we may simultaneously diagonalize the matrices $A_0,\dots,A_d$ resulting in the maximal common orthogonal eigenspaces $V_0,\dots,V_{d'}$ with corresponding projection matrices $E_0,\dots,E_{d'}$. Since, for every $i$, there exist eigenvalues $\theta_{ij}$ such that $A_i = \sum_{j=0}^{d'}\theta_{ij}E_j$ we find that $\BMA \subseteq \text{span}\left\{E_0,E_1,\dots, E_{d'}\right\}$ thus $d\leq d'$. Further since the eigenspaces $V_j$ are maximal for each $0\leq j\leq d$ and pairwise orthogonal,
\[E_j = \frac{1}{c_j}\prod_{i=0}^d\left(\prod_{\theta_{ik}\neq\theta_{ij}}\left(A_i-\theta_{ik}I\right)\right)\]
for some normalization constant $c_j$. Thus each $E_j\in \BMA$ giving $\text{span}\left\{E_0,E_1,\dots, E_{d'}\right\}\subseteq\BMA$ and therefore $d=d'$ --- we will refer to the matrices $\left\{E_j\right\}$ as our basis of idempotents\index{idempotent basis}. This shows that $\BMA$ contains a basis of $d+1$ idempotents $E_0,\dots,E_d$ which together diagonalize every matrix in $\BMA$ and act as projection matrices onto the common eigenspaces. Since the rank 1 matrix $J\in\BMA$, we find that $\frac{1}{\vert X\vert}J$ must belong to this basis; by convention we assume $E_0= \frac{1}{\vert X\vert}J$. For each $0\leq j\leq d$ we define $m_j = \text{rank } E_j$ and note that $m_0 = 1$ and $\sum_{j=0}^dm_j = \vert X\vert$.

While $\BMA = \text{span}\left\{A_0,A_1,\dots A_d\right\} = \text{span}\left\{E_0,E_1,\dots, E_{d}\right\}$, we often find that we may generate $\BMA$ with a single matrix if we allow ourselves to take products and not just linear combinations. For example, the Bose-Mesner algebra of the 3-cube (Example \ref{3cube}) may be generated by taking linear combinations of powers of $A_1$, the adjacency matrix of $\Gamma_1$. For any matrix $M\in\BMA$, we define $\left<M\right>_*$ as the set of matrices which are linear combinations of the powers of $M$; the set $\left<M\right>_*$ is called the \emph{subalgebra}\index{subalgebra} of $\BMA$ generated by $M$. 
\begin{lem} \label{mstardim}
	Let $\BMA$ be the Bose-Mesner algebra of a symmetric association scheme. For any $A\in \BMA$, the dimension of $\left<A\right>_*$ equals the number of distinct eigenvalues.
\end{lem}
\begin{proof}
	Let $A\in \BMA$ be given with $n$ distinct eigenvalues. Then there exists $\alpha_1,\dots,\alpha_n$ so that $A = \sum_{i=1}^n\alpha_i E_i$ for some orthogonal projection matrices $E_1,\dots,E_n$ --- implying the dimension of $\left<A\right>_*$ is no larger than $n$. However, since $I = A^0$, we may build each $E_j$ via
	\[E_j = \frac{1}{c_j}\prod_{i\neq j} (A - \alpha_i I)\]
	showing that the dimension of $\left<A\right>_*$ is at least $n$.
\end{proof}
\begin{cor}
	For a Bose-Mesner algebra $\BMA$ and matrix $A\in \BMA$, $\BMA = \left<A\right>_*$ if and only if $A$ has $d+1$ distinct eigenvalues.
\end{cor}
Similarly, we define the \emph{Schur subalgebra}\index{subalgebra!Schur} of $M$, denoted $\left<M\right>_\circ$, as the set of matrices which are linear combinations of all Schur powers of $M$. We immediately see analogous results.
\begin{lem} \label{mcircdim}
	Let $\BMA$ be the Bose-Mesner algebra of a symmetric association scheme. For any $E\in \BMA$, the dimension of $\left<E\right>_\circ$ equals the number of distinct entries in $E$.
\end{lem}
\begin{cor}
	For a Bose-Mesner algebra $\BMA$ and matrix $E\in \BMA$, $\BMA = \left<E\right>_\circ$ if and only if $E$ has $d+1$ distinct entries.
\end{cor}
As a simple example, it is clear that the subalgebra generated by any idempotent basis matrix (with two distinct eigenvalues) contains only multiples of that idempotent and the identity matrix. Similarly, the Schur subalgebra generated by any $01$-matrix contains only multiples of that $01$-matrix and any constant matrix ($cJ$ for some constant $c$). These specific subalgebras are typically not useful, thus we will typically only consider subalgebras of $01$-matrices and Schur subalgebras of idempotent matrices. For a more interesting example, consider again the 3-cube in Example 2.1. Let $A_i$ be the adjacency matrix of $\Gamma_i$ and note that $\left<A_1\right>_* = \BMA$ while $\left<A_2\right>_* = \left<A_0,A_2\right>$; that is, $A_1$ generates the entire Bose-Mesner algebra while $A_2$ generates a proper subset of $\BMA$. While not listed in the example, we find that this association scheme contains a single minimal idempotent whose Schur subalgebra equals $\BMA$ --- all other minimal idempotents generate a proper subset of $\BMA$. Throughout this thesis we will consider cases where the (Schur) subalgebra of a single matrix is the entire Bose-Mesner algebra (for instance, polynomial schemes) as well as others where the (Schur) subalgebra of some matrix is a proper subset of the Bose-Mesner algebra (for instance, imprimitive schemes) --- both situations may give rise to useful structure.

We take a moment here to remark on the notion of duality in our matrix algebra. We have already mentioned that $\BMA$ is closed under two distinct products: standard matrix multiplication and Schur multiplication. While it is clear these are distinct products, they are indistinguishable at the formal level. Consider an abstract inner product space which admits a set of orthogonal basis vectors $\left\{b_i\right\}$ under the product $\star$. For any pair of vectors $v = \sum_i v_ib_i$ and $w = \sum_i w_ib_i$, we may define the \emph{product with respect to basis $\left\{b_i\right\}$} as $v\star w = \sum_i (v_iw_i)b_i$. In this light, the two products used in association schemes become very similar. Returning to our Bose-Mesner algebra, let $F,F'\in\BMA$ with $F = \sum_if_iE_i = \sum_i g_iA_i$ and $F' = \sum_if_i'E_i = \sum_i g_i'A_i$. We then find
\[\begin{aligned}
FF' & =\sum_{i,j}f_if_j'E_iE_j = \sum_i f_if_i'E_i;\qquad
F\circ F' &=\sum_{i,j}g_ig_i'A_i\circ A_j= \sum_i g_ig_i'A_i.
\end{aligned}\]
Thus our two products may be described as the product with respect to the basis $\left\{E_i\right\}$ (standard matrix multiplication) and the product with respect to the basis $\left\{A_i\right\}$ (entrywise multiplication). We consider these two products dual operations on our algebra and their basis matrices dual bases. Throughout this thesis, we will often investigate this duality and point out when there are differences and/or gaps in our understanding of the landscape.

While the algebraic structure of either basis with respect to its corresponding product is trivial, the manner in which matrix products and entrywise products interact is more interesting. Our first tool for understanding this interaction is the change of basis matrix from one set of idempotents to the other. For any Bose-Mesner algebra, the \emph{first and second eigenmatrices}\index{eigenmatrices} are given by $P$ and $Q$ respectively so that
\begin{equation}
\label{PQmat}
A_i = \sum_{j} P_{ji} E_j,\qquad E_j = \frac{1}{\vert X\vert} \sum_{i} Q_{ij}A_i.
\end{equation}
The name of these matrices arises from the fact that column $i$ of $P$ consists of the eigenvalues of $A_i$ while column $j$ of $Q$ gives the ``dual eigenvalues" of $\vert X\vert E_j$ --- eigenvalues with respect to the Schur product. Let $\Delta_m=\text{diag}(m_0,m_1,\dots,m_d)$ and $\Delta_k=\text{diag}(k_0,k_1,\dots,k_d)$ and the following two relations hold for our eigenmatrices:
\begin{lem}[\cite{Brouwer1989}, First and second orthogonality relations] \label{orthorels}\index{eigenmatrices!orthogonality relations}The eigenmatrices of an association scheme satisfy
	\begin{equation}
	PQ = \vert X\vert I, \qquad \Delta_mP = Q^T\Delta_k.
	\end{equation}
\end{lem}
These relations generalize the orthogonality relations of a group.

A second consideration for our dual bases is to compare the structure constants for each product with respect to each basis. While we used the existence of structure constants $p^k_{ij}$ to show that $\BMA$ is closed under matrix multiplication, closure under Schur products is seen from the fact that the $A_i$ are pairwise orthogonal idempotents. This implies the existence of structure constants for our second basis. Thus, for $0\leq i,j,k\leq d$ there exist constants $q^k_{ij}$ so that
\begin{equation}E_i\circ E_j = \frac{1}{\vert X\vert}\sum_k q_{ij}^k E_k.\label{Emult}\end{equation}
We call these constants the \emph{Krein parameters}\index{parameters!Krein} of the association scheme. We note here that, while there are many distinct parameters of any association scheme, there are many relations which reduce our parameter space. For instance, any strongly regular graph (2-class association scheme) contains two $3\times 3$ matrices $P$ and $Q$ along with 27 intersection numbers and 27 Krein parameters, resulting in 72 total parameters --- yet the parameters of a connected strongly regular graph may be uniquely determined by exactly three parameters. In fact, just within the intersection numbers of the form $p^0_{ij}$, only $d+1$ of the $(d+1)^2$ intersection numbers are non-zero. With this in mind, we list out the basic properties of the adjacency matrices and orthogonal idempotents as well as the relations on the parameters which help reduce our parameter space. While we will use the lemmas that follow throughout the thesis, we emphasize here the duality at play between the two bases as well as their structure constants. See Lemmas 2.1.1, 2.2.1, 2.3.1, and Theorem 2.3.2 in the book of Brouwer, Cohen, and Neumaier \cite{Brouwer1989} for proofs of Lemma \ref{kitchensink}.

\newpage
\begin{lem}\label{AElem} The adjacency matrices $A_0,\dots,A_d$ and minimal idempotents $E_0,\dots,E_d$ satisfy:
	\begin{multicols}{2}
		\begin{enumerate}
			\item[(i)] $\displaystyle{A_0 = I}$,
			\item[(ii)] $\displaystyle{A_i\circ A_j = \delta_{ij}A_i}$,
			\item[(iii)] $\displaystyle{\sum_i A_i = J}$,
			\item[(iv)] $\displaystyle{A_iA_j = \sum_k p^k_{ij} A_k}$,
			\item[(v)] $\displaystyle{A_iE_j = P_{ji}E_j}$,\vfill\null
			\item[$(i')$] $\displaystyle{E_0 = J}$,
			\item[$(ii')$] $\displaystyle{E_iE_j = \delta_{ij}E_j}$,
			\item[$(iii')$] $\displaystyle{\sum_j E_j = I}$,
			\item[$(iv')$] $\displaystyle{\vert X\vert E_i\circ E_j = \sum_k q^k_{ij}E_k}$,
			\item[$(v')$] $\displaystyle{\vert X\vert A_i\circ E_j = Q_{ij} A_i}$.\vfill\null
		\end{enumerate}
	\end{multicols}
\end{lem}

\begin{lem}[{\cite{Brouwer1989}}]\label{kitchensink} The parameters $p^\ell_{ij}$, $q^\ell_{ij}$, $k_i = p^0_{ii}$, $m_j = q^0_{jj}$, and the eigenmatrices $P$ and $Q$ satisfy:
	\begin{multicols}{2}
		\begin{enumerate}
			\item[(i)] $\displaystyle{p_{0j}^\ell = \delta_{j\ell}}$,
			\item[(ii)] $\displaystyle{p^0_{ij} = \delta_{ij}k_i}$,
			\item[(iii)] $\displaystyle{p^\ell_{ij} = p^\ell_{ji}}$,
			\item[(iv)] $\displaystyle{p^\ell_{ij}k_\ell= p^j_{i\ell}k_j}$,
			\item[(v)] $\displaystyle{\sum_jp^\ell_{ij} = k_i}$,
			\item[(vi)] $\displaystyle{\sum_\ell p^\ell_{ij}p^m_{\ell h} = \sum_\ell p^m_{i\ell}p^\ell_{jh}}$,
			\item[(vii)] $\displaystyle{P_{ij}P_{ih} = \sum_\ell p^\ell_{jh}P_{i\ell}}$,
			\item[(viii)] $\displaystyle{P_{ji}Q_{hj} = \sum_\ell p_{i\ell}^hQ_{\ell j}}$,
			\item[(ix)] $\displaystyle{\sum_{j}P_{ji} = \sum_{h}p^h_{hi}}$,
			\item[(x)] $\displaystyle{P_{j0} = 1}$,
			\item[(xi)] $\displaystyle{P_{0i} = k_i}$,
			\item[(xii)] $\displaystyle{\sum_j m_jP_{ji}P_{jh} = \vert X\vert k_i\delta_{ih}}$,
			\item[(xiii)] $\displaystyle{p^\ell_{ij} = \frac{1}{\vert X\vert k_\ell}\sum_{h=0}^d m_hP_{hi}P_{h j}P_{h\ell}}$,
			
			\item[($i^\prime$)] $\displaystyle{q^\ell_{0j} = \delta_{j\ell}}$,
			\item[($ii^\prime$)] $\displaystyle{q^0_{ij} = \delta_{ij}m_j}$,
			\item[($iii^\prime$)] $\displaystyle{q^{\ell}_{ij} = q^\ell_{ji}}$,
			\item[($iv^\prime$)] $\displaystyle{q^\ell_{ij}m_\ell = q^j_{i\ell}m_j}$,
			\item[($v^\prime$)] $\displaystyle{\sum_j q^\ell_{ij} = m_i}$,
			\item[($vi^\prime$)] $\displaystyle{\sum_\ell q^\ell_{ij}q^{m}_{\ell h} = \sum_\ell q^m_{i\ell}q^\ell_{jh}}$,
			\item[($vii^\prime$)] $\displaystyle{Q_{ij}Q_{ih} = \sum_{\ell}q^\ell_{jh}Q_{i\ell}}$,
			\item[($viii^\prime$)] $\displaystyle{P_{ij}Q_{jh} = \sum_\ell q^i_{h\ell}P_{\ell j}}$,
			\item[($ix^\prime$)] $\displaystyle{\sum_{j}Q_{ji} = \sum_{h}q^h_{hi}}$,
			\item[($x^\prime$)] $\displaystyle{Q_{i0} = 1}$,
			\item[($xi^\prime$)] $\displaystyle{Q_{0j} = m_j}$,
			\item[($xii^\prime$)] $\displaystyle{\sum_i k_iQ_{ij}Q_{ih} = \vert X\vert m_j\delta_{jh}}$,
			\item[($xiii^\prime$)] $\displaystyle{q^\ell_{ij} = \frac{1}{\vert X\vert m_\ell}\sum_{h=0}^d k_hQ_{hi}Q_{h j}Q_{h\ell}}$.
			
		\end{enumerate}
	\end{multicols}
\end{lem}
We note that using this lemma, we find that any of the sets $\left\{p^k_{ij}\right\}_{i,j,k}$, $\left\{q^k_{ij}\right\}_{i,j,k}$, $\left\{P_{ij}\right\}_{i,j}$, or $\left\{Q_{ij}\right\}_{i,j}$ may be used to determine all the others. We define the collection of all four of these sets to be the \emph{parameter set}\index{parameter set} of an association scheme, noting immediately that it suffices to give any one of the four.
\section{Parameter arrays}
For a matrix $A$, we denote the entry in row $i$ and column $j$ as $\left[A\right]_{ij}$. We define the \emph{arrays of intersection numbers}\index{parameters!arrays of intersection numbers} $L_0,\dots,L_d$ as $(d+1)\PLH(d+1)$ matrices with $\left[L_i\right]_{kj} = p^k_{ij}$. We then define the vector space $\bbL = \text{span}\left\{L_0,\dots,L_d\right\}$ and note that Lemma \ref{kitchensink} \emph{(vi)} gives us
\[
\left[L_iL_j\right]_{mk} = \sum_\ell p^m_{i\ell }p^\ell_{jk} = \sum_{\ell }p^\ell_{ij}p^m_{\ell k} = \sum_\ell p^\ell_{ij}\left[L_\ell\right]_{mk}
\]
for $0\leq m,k\leq d$. Therefore we find that this vector space forms a matrix algebra under matrix multiplication with
\begin{equation}\label{Lprod}
L_iL_j = \sum_\ell p^\ell_{ij}L_\ell.
\end{equation}
Likewise, we define the \emph{arrays of Krein parameters}\index{parameters!arrays of Krein parameters} as $L_0^*,\dots,L_d^*$ with $\left[L_i^*\right]_{kj} = q^k_{ij}$. These span a dual matrix algebra $\bbL^* = \text{span}\left\{L_0^*,\dots,L_d^*\right\}$ as Lemma \ref{kitchensink} \emph{($\mathit{vi'}$)} gives
\begin{equation}\label{Lstarprod}
L_i^*L_j^* = \sum_{\ell}q^\ell_{ij}L_\ell^*.
\end{equation}
We now define algebra isomorphisms $\phi:\BMA\rightarrow\bbL$ and $\phi^*:\BMA\rightarrow\bbL^*$ via linear extension of the mappings
\begin{equation}\label{algiso}
\phi(A_i) = L_i,\qquad \phi^*(E_i) = \frac{1}{\vert X\vert}L_i^*.
\end{equation}
The first isomorphism preserves standard matrix multiplication while the second respects the Schur product in the sense that $\phi^*(M\circ N) = \phi^*(M)\phi^*(N)$. This provides us with the following lemma.
\begin{lem}\label{arrayeigenvec}
	Let $P$ and $Q$ be the eigenmatrices of an association scheme with arrays of intersection numbers $L_0,\dots,L_d$ and arrays of Krein parameters $L_0^* ,\dots, L_d^*$. For each $0\leq i,j\leq d$, column $j$ of $Q$ is an eigenvector of $L_i$ with eigenvalue $P_{ji}$. Likewise column $i$ of $P$ is an eigenvector of $L_j^*$ with eigenvalue $Q_{ij}$.
\end{lem}
\begin{proof}
	Let $A_0,\dots,A_d$ be the 01-basis matrices and $E_0,\dots,E_d$ be the basis idempotents of our association scheme. Using Lemma \ref{AElem}, we recall that $A_i = \sum_{k} P_{ki}E_k$ and $E_iE_k = \delta_{ik}E_i$, thus $A_iE_j = P_{ji}E_j$. Applying $\phi$ to both sides of the equality $E_j = \frac{1}{\vert X\vert}\sum_{j}Q_{ji} A_i$ results in $\phi(\vert X\vert E_j) = \sum_{i}Q_{ij} L_i$ giving
	\[L_i\left(\sum_iQ_{ij}L_i\right) = \phi(A_i)\phi(\vert X\vert E_j)= \phi(\vert X\vert A_iE_j) = P_{ji}\left(\sum_iQ_{ij}L_i\right)  .\]
	Similarly, we have $E_j = \frac{1}{\vert X\vert}\sum_{k} Q_{kj}A_k$ and $A_k\circ A_i = \delta_{ik}A_i$, thus $E_j\circ A_i = \frac{Q_{ji}}{\vert X\vert}A_i$. Now applying $\phi^*$ to both sides of $A_i = \sum_{j}P_{ji} E_j$ results in $\phi^*(\vert X\vert A_i) = \sum_{j}P_{ji} L_i^*$ giving
	\[L_j^*\left(\frac{1}{\vert X\vert}\sum_jP_{ji}L_j^*\right)=\phi^*(\vert X\vert E_j)\phi^*(A_j) = \phi^*(\vert X\vert E_j\circ A_i) =Q_{ji}\left(\frac{1}{\vert X\vert}\sum_jP_{ji}L_j^*\right).\]
	In each case, we note that $\left[L_i\right]_{j0} = \left[L_j^*\right]_{i0} = \delta_{ij}$ and thus column ``zero" of each matrix product gives our result.
\end{proof}
\begin{cor}
	Let $L_i$ ($L_j^*$) be an array of intersection numbers (Krein parameters) of a $d$-class association scheme. If $L_i$ ($L_j^*$) has $d+1$ distinct eigenvalues then it uniquely determines all remaining parameters of the scheme.
\end{cor}
\begin{proof}
	We will prove this for an array of intersection numbers, noting that the proof for an array of Krein parameters follows analogously. Let $L_i$ be an array of intersection numbers with $d+1$ distinct eigenvalues. First, $k_i = p^0_{ii} = [L_i]_{0i}$. We then use Lemma \ref{kitchensink} (iv) to solve for the valencies $k_0,\dots,k_d$. Now, since $L_i$ has $d+1$ distinct eigenvalues, we find $d+1$ distinct eigenvectors $v_0,\dots,v_d$ --- normalize these so that $(v_i)_0 = 1$ and define the matrix $M$ so that the $i^\text{th}$ column of $M$ is $v_i$. From Lemma $\label{arrayeigenvec}$, we know that the columns of $Q$ are exactly $m_iv_i$ where $m_i$ is the multiplicity of the given eigenvalue, thus $Q = M\Delta_m$. Thus, we must determine these multiplicities to finish our proof. However, Lemma \ref{orthorels} tells us that
	\[\begin{aligned}
	\vert X\vert \Delta_m = Q^T\Delta_k Q  = \Delta_m M^T\Delta_k M\Delta_m\\
	\end{aligned}\]
	Noting that all matrices given are invertible since each $k_i$ and $m_i$ are strictly positive, this gives $m_i = \bigslant{\vert X\vert}{v_i^T\Delta_kv_i}$.
\end{proof}
When relation $R_1$ is distinguished or projection matrix $E_1$ is distinguished, we refer to $L_1$ as the \emph{intersection matrix}\index{parameters!intersection matrix} and $L_1^*$ as the \emph{Krein matrix}\index{parameters!Krein matrix}; these two matrices will become particularly important for polynomial schemes (see Section \ref{poly}), for which they are tridiagonal. We end this section by noting that the matrices given in Example \ref{3cube} are exactly the arrays of intersection numbers for that association scheme. A final note before moving on is that the parameters of an association scheme need not define the scheme uniquely. In fact, there exist non-isomorphic $2$-class association schemes with exactly the same parameter set; consider the $4\times 4$ rook graph and the Shrikhande graph.
\section{Formal duality of association schemes}\label{dualpair}
In this section, we give an explicit example of the duality previously alluded to. Consider first the strongly regular graph $K_{3,3}$. This association scheme is a 2-class bipartite scheme with nontrivial relations given by adjacency in $K_{3,3}$ and non-adjacency in $K_{3,3}$ respectively. Thus the three graphs for this scheme are
\begin{figure}[H]\begin{center}\scalebox{.7}{$\begin{aligned}
			\begin{tikzpicture}[shorten >=1pt,auto,node distance=2cm,
			thin,main node/.style = {circle,draw, inner sep = 0pt, minimum size = 15pt}]
			
			\node[main node,fill=white] (1) {0};
			\node[main node,fill=cyan] [right of = 1](2) {2};
			\node[main node,fill=cyan] [right of = 2](3) {4};
			\node[main node,fill=red] [below of = 1](4) {1};
			\node[main node,fill=red] [right of = 4](5) {3};
			\node[main node,fill=red] [right of = 5](6) {5};
			\node [above of =2] (9) {$\Gamma_0$};
			
			\path (1) edge [in=120,out=145,loop] ();
			\path (2) edge [in=120,out=145,loop] ();
			\path (3) edge [in=120,out=145,loop] ();
			\path (4) edge [in=120,out=145,loop] ();
			\path (5) edge [in=120,out=145,loop] ();
			\path (6) edge [in=120,out=145,loop] ();
			
			\end{tikzpicture}\qquad&\qquad\begin{tikzpicture}[shorten >=1pt,auto,node distance=2cm,
			thin,main node/.style = {circle,draw, inner sep = 0pt, minimum size = 15pt}]
			
			\node[main node,fill=white] (1) {0};
			\node[main node,fill=cyan] [right of = 1](2) {2};
			\node[main node,fill=cyan] [right of = 2](3) {4};
			\node[main node,fill=red] [below of = 1](4) {1};
			\node[main node,fill=red] [right of = 4](5) {3};
			\node[main node,fill=red] [right of = 5](6) {5};
			\node [above of =2] {$\Gamma_1$};
			
			\draw[-] (1)--(4)--(2)--(5)--(3)--(6)--(1)--(5)--(2)--(6)--(3)--(4);
			\end{tikzpicture}\qquad\qquad
			\begin{tikzpicture}[shorten >=1pt,auto,node distance=2cm,
			thin,main node/.style = {circle,draw, inner sep = 0pt, minimum size = 15pt}]	
			\node[main node,fill=white] (1) {0};
			\node[main node,fill=cyan] [right of = 1](2) {2};
			\node[main node,fill=cyan] [right of = 2](3) {4};
			\node[main node,fill=red] [below of = 1](4) {1};
			\node[main node,fill=red] [right of = 4](5) {3};
			\node[main node,fill=red] [right of = 5](6) {5};
			\node [above of =2] {$\Gamma_2$};
			
			\path[-]
			(1)edge node {} (2)
			edge [bend right=35] node {} (3)
			(2)edge node {} (3)
			(4)edge node {} (5)
			edge [bend left=35] node {} (6)
			(5) edge node {} (6);
			\end{tikzpicture}
			\end{aligned}$}\end{center}
	\caption[Graphs of $K_{3,3}$.]{The association scheme of $K_{3,3}$. The subconstituents of vertex $0$ are colored white, red,  and blue respectively.}\label{k33}
\end{figure}
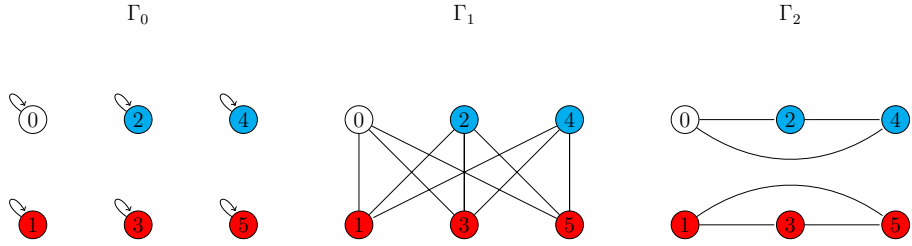
The intersection numbers, Krein parameters, and eigenmatrices are
\[\begin{aligned}L_0 &= \left[\begin{array}{cccc}
1&0&0\\
0&1&0\\
0&0&1\\
\end{array}\right],\quad L_1 &= \left[\begin{array}{cccc}
0&3&0\\
1&0&2\\
0&3&0\\
\end{array}\right],\quad L_2 &= \left[\begin{array}{cccc}
0&0&2\\
0&2&0\\
1&0&1\\
\end{array}\right];\\
L_0^* &= \left[\begin{array}{cccc}
1&0&0\\
0&1&0\\
0&0&1\\
\end{array}\right],\quad L_1^* &= \left[\begin{array}{cccc}
0&4&0\\
1&2&1\\
0&4&0\\
\end{array}\right],\quad L_2^* &= \left[\begin{array}{cccc}
0&0&1\\
0&1&0\\
1&0&0\\
\end{array}\right];\end{aligned}\]
\[P = \left[\begin{array}{rrr}
1&3&2\\
1&0&-1\\
1&-3&2\\
\end{array}\right],\qquad Q = \left[\begin{array}{rrr}
1&4&1\\
1&0&-1\\
1&-2&1\\
\end{array}\right].\]
Now consider the $2$-class association scheme determined by the octahedron. The three graphs for this scheme are
\begin{figure}[H]\begin{center}\scalebox{.7}{$\begin{aligned}
			\begin{tikzpicture}[shorten >=1pt,auto,node distance=2cm,
			thin,main node/.style = {circle,draw, inner sep = 0pt, minimum size = 15pt}]
			
			\node[main node,fill=red] (2) {1};
			\node[main node,fill=white] [below right of = 2](1) {0};
			\node[main node,fill=cyan] [right of = 1](6) {3};
			\node[main node,fill=red] [above right of = 6](3) {2};
			\node[main node,fill=red] [below right of = 6](4) {4};
			\node[main node,fill=red] [below left of  = 1](5) {5};
			\node at (2.5,1) (9) {$\Gamma_0$};
			
			\path (1) edge [in=120,out=145,loop] ();
			\path (2) edge [in=120,out=145,loop] ();
			\path (3) edge [in=120,out=145,loop] ();
			\path (4) edge [in=120,out=145,loop] ();
			\path (5) edge [in=120,out=145,loop] ();
			\path (6) edge [in=120,out=145,loop] ();
			
			\end{tikzpicture}\qquad&\qquad\begin{tikzpicture}[shorten >=1pt,auto,node distance=2cm,
			thin,main node/.style = {circle,draw, inner sep = 0pt, minimum size = 15pt}]
			
			\node[main node,fill=red] (2) {1};
			\node[main node,fill=white] [below right of = 2](1) {0};
			\node[main node,fill=cyan] [right of = 1](6) {3};
			\node[main node,fill=red] [above right of = 6](3) {2};
			\node[main node,fill=red] [below right of = 6](4) {4};
			\node[main node,fill=red] [below left of  = 1](5) {5};
			\node at (2.5,1) (9) {$\Gamma_1$};
			
			\draw[-] (1)--(2)--(3)--(4)--(5)--(2)--(6)--(5)--(1)--(4)--(6)--(3)--(1);
			\end{tikzpicture}\qquad\qquad
			\begin{tikzpicture}[shorten >=1pt,auto,node distance=2cm,
			thin,main node/.style = {circle,draw, inner sep = 0pt, minimum size = 15pt}]	
			
			\node[main node,fill=red] (2) {1};
			\node[main node,fill=white] [below right of = 2](1) {0};
			\node[main node,fill=cyan] [right of = 1](6) {3};
			\node[main node,fill=red] [above right of = 6](3) {2};
			\node[main node,fill=red] [below right of = 6](4) {4};
			\node[main node,fill=red] [below left of  = 1](5) {5};
			\node at (2.5,1) (9) {$\Gamma_2$};
			
			\path[-]
			(1)edge node {} (6)
			(2)edge [bend right=45] node {} (4)
			(3)edge [bend left=45] node {} (5);
			\end{tikzpicture}
			\end{aligned}$}\end{center}
	\caption[Graphs of the octahedron.]{The association scheme of the octahedron. The subconstituents of vertex $0$ are colored white, red, and blue respectively.}\label{octahedron}
\end{figure}
The intersection numbers, Krein parameters, and eigenmatrices of this association scheme are
\[\begin{aligned}
L_0 &= \left[\begin{array}{cccc}
1&0&0\\
0&1&0\\
0&0&1\\
\end{array}\right],\quad L_1 &= \left[\begin{array}{cccc}
0&4&0\\
1&2&1\\
0&4&0\\
\end{array}\right],\quad L_2 &= \left[\begin{array}{cccc}
0&0&1\\
0&1&0\\
1&0&0\\
\end{array}\right];\\
L_0^* &= \left[\begin{array}{cccc}
1&0&0\\
0&1&0\\
0&0&1\\
\end{array}\right],\quad L_1^* &= \left[\begin{array}{cccc}
0&3&0\\
1&0&2\\
0&3&0\\
\end{array}\right],\quad L_2^* &= \left[\begin{array}{cccc}
0&0&2\\
0&2&0\\
1&0&1\\
\end{array}\right];
\end{aligned}\]
\[P = \left[\begin{array}{rrr}
1&4&1\\
1&0&-1\\
1&-2&1\\
\end{array}\right],\qquad Q = \left[\begin{array}{rrr}
1&3&2\\
1&0&-1\\
1&-3&2\\
\end{array}\right].\]
One may observe that the intersection numbers and the Krein parameters of the two schemes are interchanged, as are the first and second eigenmatrices. This specific example is due to a duality arising from the context of translation schemes and regular group actions. Given a point set $X$ and a group $G$, we say $G$ acts regularly on $X$ if, for any pair of points $x,y\in X$, there exists a unique $g$ mapping $x$ to $y$; clearly this may only occur if $\vert X\vert = \vert G\vert$. We then say an association scheme is a \emph{translation schemes}\index{translation schemes} if there is an abelian group acting regularly on the vertices such that for each $g\in \mathcal{G}$ and any pair of vertices $x,y\in X$, $(g(x),g(y))\in R_i$ if and only if $(x,y)\in R_i$. When this occurs, we may fix a vertex $x\in X$ and identify the remaining vertices with non-identity elements of $G$ corresponding to which element of $G$ maps $x$ to each vertex; thus we may identify $X$ with $G$ and consider our association scheme to be on the group $G$. The dual group $G^*$ is defined as the set of characters (homomorphisms mapping $G\rightarrow \mathbb{C}^*$) equipped with entrywise products as the group action. We then define the dual association scheme $(G^*,\mathcal{R}^*)$ where $(\chi,\chi')\in R_i^*$ if and only if $E_i\left(\bigslant{\chi}{\chi'}\right) = \left(\bigslant{\chi}{\chi'}\right)$ where this division is done entrywise. Finally, we find that the new association scheme is also a translation scheme using the group $\mathcal{G}^*$. In the case of the two association schemes given above, the desired group is $\mathcal{G}\simeq \bbZ_6$ and we find each scheme results as the dual of the other; in general we expect the dual operation to be an involution if the dual association scheme exists. The pair of schemes listed is a specific case of the more general duality between $\overline{rK_s}$ and $\overline{sK_r}$; that is, the complement of $r$ copies of $K_s$ and the complement of $s$ copies of $K_r$ respectively. In this more general setting, we find that the cyclic group $\bbZ_{rs}$ acts regularly on each set of vertices.

In general, the automorphism group of an association scheme could be trivial. In this case, no group acts transitively on the vertices, and a dual scheme as described above is not defined. However, just as with non-linear codes in coding theory, we abstract the notion of duality to a formal definition on the parameters. We say two association schemes are \emph{formally dual}\index{formal duality} if the first and second eigenmatrices of one are the second and first eigenmatrices, respectively, of the other. This necessarily swaps the intersection numbers and Krein parameters so that, for example, $p^k_{ij}$ of a scheme will equal $q^k_{ij}$ of it's dual. In many cases, no formally dual scheme may exist; consider any association scheme for which the Krein parameters are not integral. However, there are non-trivial examples of formally dual pairs of association schemes, for instance consider the two infinite families generated by de Caen and van Dam \cite{deCaen1999}. This notion of formal duality still plays a major role in our understanding of the field of association schemes as a whole, motivating many of the questions we will focus on in this thesis.

We finish this discussion with one final note: any definition based solely on the parameters of an association scheme gives rise to an analogous ``dual" definition. For instance, we find that $K_{3,3}$ is \emph{bipartite}\index{bipartite}, meaning $p^k_{ij}= 0$ whenever $i+j+k\notin2\bbZ$. The dual graph, the octahedron, must then have $q^k_{ij}=0$ whenever $i+j+k\notin2\bbZ$; we call this \emph{dual-bipartite}\index{bipartite!dual-} or $Q$-bipartite (see Section \ref{poly}). Similarly, we find that both $K_{3,3}$ and the octahedron are \emph{antipodal graphs}\index{antipodal}: $p^d_{di} = 0$ whenever $i\notin\left\{0,d\right\}$. Both graphs are then \emph{dual-antipodal}\index{antipodal!dual-} or $Q$-antipodal (see Section \ref{poly}): $q^d_{di}=0$ whenever $i\notin\left\{0,d\right\}$. While there has been much research into bipartite and antipodal graphs, in this thesis we are interested in the implications of these dual properties, seeking when such objects may exist and what combinatorial or geometric structure the properties impose.
\section{Feasibility and realizability}
One main point of interest is whether or not an association scheme exists, given a (possibly partial) parameter set. While existence often cannot be proven without explicitly constructing the scheme, we often may rule out the existence of a scheme due to the values its intersection numbers or Krein parameters must take. In this section we examine three main conditions which we will use throughout this thesis in addition to what we already stated in Lemma \ref{kitchensink}. We begin with an immediate restriction on the intersection numbers:
\begin{lem}[\cite{Brouwer1989}]\label{intfeas}
	The intersection numbers of an association scheme must be non-negative integers.\qed
\end{lem}
This condition is easy to verify since, by definition, each $p^k_{ij}$ is the cardinality of a set. While this property is immediate, it can be a powerful tool to eliminate examples with very little information about the association scheme. Next consider the Krein parameters of our association scheme.
\begin{lem}[\cite{Scott1973},Krein conditions]\label{kreinfeas}\index{Krein conditions}
	The Krein parameters of an association scheme must be non-negative real numbers.
\end{lem}
\begin{proof}
	Given a matrix $M$, let $\sigma(M)$ be the set of eigenvalues of $M$. From Equation \eqref{Emult}, we find that $\sigma(E_i\circ E_j) = \left\{\bigslant{q_{ij}^0}{\vert X\vert},\dots,\bigslant{q_{ij}^d}{\vert X\vert}\right\}$. However $E_i\circ E_j\in \BMA$ and therefore it must be symmetric, implying all of its eigenvalues are real. Further, note that $\sigma(E) = \left\{0,1\right\}$ for any projection matrix $E$. We also note that $\sigma(A\otimes B) = \left\{\alpha\beta,\alpha\in\sigma(A),\beta\in\sigma(B)\right\}$ for any pair of matrices $A$ and $B$. Thus, $\sigma(E_i\otimes E_j) = \left\{0,1\right\}$. Finally, since $E_i\circ E_j$ is a principal submatrix of $E_i\otimes E_j$, the eigenvalues of $E_i\circ E_j$ must be contained in the interval $0\leq \lambda\leq 1$.
\end{proof}
The final feasibility condition we will list here is known as the \emph{absolute bound}\index{absolute bound}.
\begin{lem}[\cite{Neumaier1981},Absolute bound]\label{absolute}
	The multiplicities $m_i$ ($0\leq i\leq d$) of a $d$-class association scheme satisfy:
	\[\sum_{q_{ij}^k\neq 0} m_k\leq\begin{cases}
	m_im_j & \text{ if }i\neq j\\
	\binom{m_i+1}{2} & \text{ if }i= j.
	\end{cases}\]
\end{lem}
\begin{proof}
	The sum on the left is the rank of $E_i\circ E_j$, a principal submatrix of the rank $m_im_j$ matrix $E_i\otimes E_j$. Further, if $i=j$, $E_i\circ E_j$ is the entrywise square of $E_i$. Assuming $\text{col}(E_i) = \text{span}(v_1,\dots,v_{m_i})$, the columns of $E_i\circ E_i$ must be linear combinations of the vectors $v_j\circ v_k$ for $1\leq j\leq k\leq d$, a total of $\binom{m_i+1}{2}$ vectors.
\end{proof}
There are many other feasibility conditions we may list here including some arising from design theory and others as simple as the handshaking lemma. In this thesis however, we consider the conditions already stated to be a baseline. Thus, we do not claim that any parameter set fulfilling these conditions is guaranteed to correspond to the parameters of some association scheme, rather we instead simply ignore parameter sets which do not fulfill these basic parameter restrictions. We therefore define two separate terms which will be used throughout this thesis, \emph{feasible parameter sets}\index{feasible parameter set} and \emph{realizable parameter sets}\index{realizable parameter set}. \label{FCpage}
\begin{definition}
	A \emph{feasible parameter set} is a set of Krein parameters, intersection numbers, and eigenmatrices such that:
	\begin{itemize}
		\item[FC1:] The Krein parameters satisfy Lemmas \ref{kreinfeas} and \ref{kitchensink} \emph{($\mathit{i'}$) -- ($\mathit{xiii'}$)},
		\item[FC2:] The intersection numbers satisfy Lemmas \ref{intfeas} and \ref{kitchensink} \emph{(i) -- (xiii)},
		\item[FC3:] The integers $m_j = q^0_{jj}$ satisfy Lemma \ref{absolute}.
	\end{itemize}
\end{definition}
\begin{definition}
	A feasible parameter set is \emph{realizable} if there exists an association scheme $(X,\cR)$ with the given parameter set.
\end{definition}
\section{Imprimitivity}\label{imprimitivity}
In group theory, one is often interested in learning about the subgroups of any given group. That is, given a finite group $G$, are there subsets of the vertices $S\subset G$ for which the subset is closed under the group product. When such a subset occurs, we say $S$ is a subgroup of $G$, denoted $S\leq G$. In some cases, namely when our subgroup $N\leq G$ is normal, we may define another subgroup called the quotient group: $\nicefrac{G}{N}$. In this context, we often think of $G$ being built out of the normal subgroup $N$ and the quotient group $\nicefrac{G}{N}$, even though we cannot necessarily recover $G$ from these two ingredients. This line of investigation arises naturally in association schemes through the notion of subschemes. Within any Bose-Mesner algebra $\BMA$, we often find smaller subspaces $\BMB\subset\BMA$ which are closed under both entrywise products and standard matrix products. While these smaller algebras are typically not Bose-Mesner algebras (they need not include both identities $I$ and $J$), we may find a mapping from $\BMB$ to a vector space on smaller matrices which does result in a new Bose-Mesner algebra. Given an association scheme $(X,\mathcal{R})$, a subscheme $(X',\left\{R'_i\right\})$ is a subset of the points $X'\subset X$ paired with the non-empty relations $\left\{R'_i\right\}$ where $R'_i = R_i\cap(X'\times X')$ which itself forms an association scheme. The subalgebra mapping in this case corresponds to taking the principal submatrix of each matrix in $\BMA$ obtained by restricting to the rows and columns corresponding to the points $X'$. Just as in the group case, where every group $G$ contains two trivial subgroups $\left\{e_G\right\}$ and $G$ itself, every Bose-Mesner algebra $\BMA$ contains trivial subspaces closed under both products such as $\left<I\right>$ and $\BMA$; these will be ignored for all that follows.

In this section we will examine the analogue of a non-simple group, called an imprimitive scheme which admits not only a non-trivial subscheme, but a quotient scheme as well. Analogous to a non-simple group, we often think of the imprimitive scheme as being built out of these two smaller schemes. However, knowing both the quotient scheme and the subscheme is not sufficient to entirely determine the original scheme. This is, in part, due to the many different types of products which may be used to piece smaller schemes together. For instance, Song \cite{Song2002} indicates the many distinct association schemes on 12 points using the standard direct and wreath products. Additionally, Cameron and Bailey \cite{Bailey2005} introduce a third product, called the ``crested product", which is distinct, in general, from the other two.

An association scheme $(X,\mathcal{R})$ is \emph{imprimitive}\index{imprimitive} if there exists a non-trivial union of relations $\cup_{i\in \mathcal{I}} R_i$ which forms an equivalence relation on $X\times X$; here, the only trivial unions are $\mathcal{I} = \left\{0\right\}$ and $\mathcal{I} = \left\{0,\dots,d\right\}$. Given an imprimitive association scheme, we call the set of equivalence classes a \emph{system of imprimitivity}\index{imprimitivity!system of -} --- noting that a system of imprimitivity is determined uniquely by the index set $\mathcal{I}$. An equivalent definition for imprimitivity is as follows: an association scheme is \emph{imprimitive} if there exists a non-trivial relation whose graph is disconnected. To see the equivalence of these two definitions, assume that $R_i$ is disconnected for $i\neq 0$. Then $\cI = \left\{j:p^j_{ii}\right\}$ yields a system of imprimitivity.

In some cases, we may have multiple systems of imprimitivity in our association scheme. Consider Example \ref{3cube} and note that $\cI_1 = \left\{0, 2\right\}$ and $\cI_2 = \left\{0, 3\right\}$ both yield systems of imprimitivity yet $\cI_1$ results in two equivalence classes while $\cI_2$ gives four. Thus we must be careful to distinguish between distinct systems of imprimitivity for any given association scheme. In each case, we find that the size of any equivalence class is equal to the sum of the valencies $k_i$ $(i\in\cI)$; hence all equivalence classes have the same size for a given system of imprimitivity. In the example of the 3-cube, the size of each equivalence class for $\cI_1$ is $v_0+v_2 = 4$; likewise the size of each equivalence class for $\cI_2$ is $v_0+v_3 = 2$. For all that follows we denote the size of any given equivalence class by $r$ and then the number of equivalence classes by $w$ noting that $\vert X\vert = wr$. We now define the subscheme and quotient scheme of an imprimitive scheme. Note that there is no distinction made between a general subscheme and one which corresponds to a single fiber of a system of imprimitivity. To avoid confusion, we will henceforth only refer to subschemes which arise from systems of imprimitivity throughout this text.

\begin{definition}
	Let $(X,\mathcal{R})$ be an imprimitive association scheme with system of imprimitivity $X_1,\dots, X_w$ with index set $\mathcal{I}$.
	\begin{itemize}
		\item[(i)] The ($j^\text{th}$) \emph{subscheme} of $(X,\mathcal{R})$ is $\left(X_j,\left\{R'_i\right\}_{i\in\mathcal{I}}\right)$ where $R'_i = R_i\cap (X_j\times X_j)$.\index{imprimitivity!subscheme}
		\item[(ii)] The \emph{quotient scheme} of $(X,\mathcal{R})$ is the association scheme $\left(\left\{X_1,\dots,X_w\right\},\left\{\tilde{R}_i\right\}\right)$ where $(X_i,X_j)\in \tilde{R}_k$ if and only if there exists $x\in X_i$, $y\in X_j$ with $(x,y)\in R_k$.\index{imprimitivity!quotient scheme}
	\end{itemize}
\end{definition}
Note that the subscheme and quotient scheme depend explicitly on the system of imprimitivity chosen. Thus, in cases where multiple systems of imprimitivity exist, we must be careful to distinguish which we are using. Thus, if multiple systems of imprimitivity occur, say with index sets $\mathcal{I}_1$ and $\mathcal{I}_2$, we will say ``the quotient scheme (or subscheme) corresponding to $\mathcal{I}_1$". Before moving to examples, we will provide a derivation of each using their Bose-Mesner algebras. We do this to illuminate the duality at play, referring the reader to numerous other sources for a combinatorial derivation (\cite{Brouwer1989},\cite{Rao1984},\cite{Cameron1978},\cite{Martin2007}). 

We begin with the subscheme. Let $(X,\cR)$ be given with the system of imprimitivity $X_1,\dots,X_w$ with index set $\cI = \left\{0,i_1,\dots,i_s\right\}$ and fiber size $r$. Then $\left\{A_i\right\}_{i\in\cI}$ forms a basis for a second matrix algebra $\mathbb{B}\subset\BMA$ which is also closed under both matrix and Schur multiplication. We may order the vertices by equivalence classes so that every matrix in $\mathbb{B}$ is block diagonal with $w$ blocks of size $r\times r$. Further, \[\left[\sum_{i\in\cI}A_i\right]_{x,y} = \begin{cases}
1 \text{ if there exists a }0\leq k\leq w\text{ with }x,y\in X_k,\\
0 \text{ otherwise.}
\end{cases}\]
Each of these matrices will therefore be block diagonal matrices as the $(x,y)$ entry of each matrix will be 0 for $x\in X_i$ and $y\in X_j$ $(i\neq j)$. However, this tells us that for $i,j\in \cI$,

\[A_iA_j = \left[\begin{array}{cccc}
\phi_1(A_i)\phi_1(A_j) & 0 & \dots & 0\\
0 & \phi_2(A_i)\phi_2(A_j) & \dots & 0\\
\vdots & \vdots & \ddots & \vdots \\
0 & 0 & \dots & \phi_w(A_i)\phi_w(A_j)
\end{array}\right]\]
where $\phi_\ell$ maps any matrix $A$ to the $\ell^\text{th}$ $r\times r$ diagonal block. Thus we have $w$ vector spaces which are all isomorphic to $\BMB$, in particular, each vector space is closed under matrix multiplication and entrywise products. In addition however, each of these vector spaces contain both $I$ and $J$ since $\phi_\ell(A_0) = I$ and $\phi_\ell(\sum_{i\in\cI}A_i) = J$ for any choice of $1\leq \ell\leq w$. Thus each of these vector spaces is a subscheme of $\BMA$.

\begin{lem}\label{sub}
	Let $\BMA$ be the Bose-Mesner algebra of an imprimitive scheme with system of imprimitivity with $w$ fibers of size $r$ and index set $\cI$. Then (under an appropriate ordering of the vertices) $\sum_{i\in\cI} A_i = I_w\otimes J_r$ and the Bose-Mesner algebra of the subscheme, $\BMA'$, is isomorphic to the vector space of matrices $\left<A_i\right>_{i\in\cI}$.
\end{lem}

Before moving to the quotient scheme, we investigate the subscheme further. As before, we note that $\BMB$ is commutative, and thus we may simultaneously diagonalize all matrices in $\BMB$ giving the basis of minimal idempotents $\left\{E^\prime_{0'},E'_{1'},\dots,E'_{s'}\right\}$ (here, we use $j'$ to denote an index of the subscheme noting that, in general, $E'_{j'}$ will not necessarily be related to $E_j$). Since every maximal common eigenspace of $\BMB$ must be the direct sum of maximal eigenspaces in $\BMA'$, the dimension of any maximal common eigenspace in $\BMB$ must be a multiple of $w$. Then, since we have shown the rank $w$ matrix $I_w\otimes J_r\in\BMB$, we know that $\frac{1}{r}I_w\otimes J_r$ must be one minimal idempotent as it is has the minimum possible non-zero rank of any idempotent matrix in $\BMB$; without loss of generality we assume this is $E^\prime_{0'}$. Now, since each maximal common eigenspace of $\BMA$ must be common eigenspaces of $\BMB$, we must have that for each $0\leq i\leq d$, there exists a unique $j'\in\left\{0',\dots,s'\right\}$ such that $E_j'E_i = E_i$ --- that is, the maximal common eigenspaces of $\BMB$ partition those of $\BMA$. We therefore define for each $j'\in\left\{0',\dots,s'\right\}$ the set $\left\{i:E'_{j'}E_i = E_i \right\}$ giving $E'_{j'} = \sum_{i\in\textbf{j}'}E_i$; for ease of notation, we identify an index $j'$ with both an index of the subscheme and the corresponding set of indices of idempotents from the original scheme contributing to $E_{j'}$. This is to emphasize that the idempotents of $\BMA'$ come from a partition of the idempotents of $\BMA$ --- we will see a similar notation for quotient schemes where relations of the quotient scheme arise from a partition of the relations of $\BMA$.	

Since matrix multiplication is preserved by our mapping, we find that for $i,j,k\in\cI$, $p^{\prime k}_{ij} = p^k_{ij}$ and thus the intersection numbers match our original scheme for indices within $\cI$. While Schur products are also preserved by $\psi_\ell$, recall that the idempotents of $\BMB$ were not the same as the idempotents of $\BMA$, thus we do not expect our Krein parameters to be preserved. Instead we must determine the parameters $q^{\prime k'}_{i'j'}$ so that 
\[E_{i'}^\prime\circ E_{j'}^\prime = \frac{1}{\vert X_\ell\vert} \sum_{k'=0'}^{s'}q^{\prime k'}_{i'j'}E_{k'}^\prime.\]
First, note that such constants must exist since $\mathbb{B}$ is closed under entrywise multiplication. Now, recalling that $E_i^\prime = \sum_{a\in\hat{i}} E_a$, we may multiply each side of the above equation by $E_h$ for some $0\leq h\leq d$ and find
\[\frac{1}{\vert X\vert}\sum_{a\in\hat{i},b\in\hat{j}} q^h_{ab}E_h = \left(E_{i'}^\prime\circ E_{j'}^\prime\right)E_h = \frac{1}{\vert X_\ell\vert}q^{\prime k'}_{i'j'} E_h \]
where $h\in k'$. Thus we find
\begin{equation}q^{\prime k'}_{i'j'} = \frac{1}{w}\sum_{a\in i',b\in j'}q^h_{ab}\end{equation}
for any $h\in k'$.

We note that, while there exist algebra isomorphisms between any pair of subschemes, the $w$ distinct subschemes need not be combinatorially isomorphic; that is, while $\psi_i\circ\psi_j^{-1}$ maps $\BMA_j'\rightarrow \BMA_i'$, there need not be a bijection between $X_j\rightarrow X_i$ for which $\left(\gamma(x),\gamma(y)\right)\in R_i^\prime$ if and only if $(x,y)\in R_i^\prime$ (see \cite{Jaeger1998} for further discussion on the distinction between ``isomorphic" and ``combinatorially isomorphic"). We summarize this construction as follows.

We now consider the dual notion, the \emph{quotient scheme}, by swapping the roles of our adjacency matrices and idempotents in the above derivation. As the quotient scheme and subscheme arise in the same context, it should not be surprising that they are highly related. Thus, we will hold on to some of the notation we have already established: using $j'\in\left\{0',\dots,s'\right\}$ to denote an index of the subscheme as well as the set $\left\{i:E^\prime_{j'} E_i = E_i\right\}$. Now, observe that Lemma \ref{kitchensink} \emph{($\mathit{i'}$)} applied to any subscheme tells us that $q^{\prime k'}_{0'0'} = 0$ for any $k'\neq 0'$. Thus, our original scheme must have $q^k_{ab} = 0$ whenever $a,b\in0'$ and $k\notin0'$. We denote $\cJ = 0'$ and find that the set of matrices $\left\{E_j\right\}_{j\in\cJ}$ is then closed under entrywise multiplication. Therefore we have a third vector space of matrices $\BMC = \text{span}_{j\in\cJ}\left(E_j\right)$ closed under both matrix and Schur multiplication. This vector space will be isomorphic to the Bose-Mesner algebra of our quotient scheme just as $\BMB$ was isomorphic to each $\BMA'$, however we require more work to show this.	

Without loss of generality assume that $\vert \cJ\vert = t+1$. Since Schur products are trivially commutative, we may guarantee the existence of Schur idempotents (01-matrices) which span our matrix algebra. Let $\left\{\tilde{A}_{\tilde{0}},\dots,\tilde{A}_{\tilde{t}}\right\}$ correspond to the set of minimal Schur idempotents in $\BMC$; that is, the idempotents contained in $\BMC$ which are not sums of other Schur idempotents in $\BMC$. These new idempotents correspond to the maximal common Schur-eigenspaces of $\BMC$. However, since $\BMC\subset\BMA$, each schur idempotent must be a sum of 01-matrices from $\BMA$. As before, for each index $\tilde{i}\in\left\{\tilde{0},\dots,\tilde{t}\right\}$ we define the set $\tilde{i} = \left\{j:\tilde{A}_{\tilde{i}}\circ A_j = A_j\right\}$ giving $\tilde{A}_{\tilde{i}} = \sum_{j\in\tilde{i}}A_j$. This immediately implies that $\tilde{A}_{\tilde{i}}E_k = 0$ for any $k\notin \cJ$. Then each $\tilde{A}_{\tilde{i}}$ may be diagonalized using only the $\left\{E_j\right\}_{j\in\cJ}$ giving,
\[\tilde{A}_{\tilde{i}}(I_w\otimes J_r) = \tilde{A}_{\tilde{i}}\left(r\sum_{j\in \cJ} E_j\right) = r\tilde{A}_{\tilde{i}}.\] 

Thus, each $A^\prime_i$ is a block matrix, constant on each block. This block form tells us that any Schur idempotent of $\BMC$ must have a multiple of $wr^2$ ones since there must be at least $r^2$ ones in each of the $w$ rows of blocks. Then $I_w\otimes J_r$ must be one of our minimal idempotents since $I_w\otimes J_r\in \BMC$ and it has exactly $wr^2$ ones; without loss of generality, we assume $\tilde{A}_{\tilde{0}} = I_w\otimes J_r$. Further, the block form of our 01-matrices tells us there exists matrices $\tilde{M}_{\tilde{i}}$ for $\tilde{i}\in\left\{\tilde{0},\dots,\tilde{t}\right\}$ such that $\tilde{A}_{\tilde{i}} = J_r\otimes \tilde{M}_{\tilde{i}}$ and we may define an algebra homomorphism $\tilde{\psi}$ mapping $\tilde{A}_{\tilde{i}}\mapsto \tilde{M}_{\tilde{i}}$ which preserves entrywise multiplication. We further find that
\begin{equation}\tilde{\psi}\left(AB\right) = r\tilde{\psi}(A)\tilde{\psi}(B).\label{tildepsiprod}\end{equation}
Then $\tilde{\psi}(A^\prime_0) = I$, $\tilde{\psi}\left(\sum_{i\in\cJ}A'_i\right) =J_w$, and Equation \eqref{tildepsiprod} gives us,
\[\tilde{p}^{\tilde{k}}_{\tilde{i}\tilde{j}} = \frac{1}{r}\sum_{i\in\tilde{i},j\in\tilde{j}}p^k_{ij}\]
for any choice of $k\in\tilde{k}$; this smaller algebra satisfies all the conditions of a Bose-Mesner algebra. Finally, since $\tilde{\psi}$ preserves entrywise products, we find that $\tilde{q}^k_{ij} = q^k_{ij}$ for $i,j,k\in\cJ$.
\begin{lem}\label{quotient}
	Let $\BMA$ be the Bose-Mesner algebra of an imprimitive scheme with system of imprimitivity with $w$ fibers of size $r$ and index set $\cI$. Then there exists an index set $\cJ$ such that $\sum_{j\in\cJ} E_j = \frac{1}{r}\left(I_w\otimes J_r\right)$ and the Bose-Mesner algebra of the quotient scheme, $\tilde{\BMA}$, is isomorphic to the vector space of matrices $\left<E_j\right>_{j\in\cJ}$.
\end{lem}	

From the above two derivations we find that both the subscheme and the quotient scheme may be obtained by finding a subset of idempotents under one product which are closed under the second product. When this occurs we find a smaller matrix algebra which is isomorphic to a Bose-Mesner algebra on fewer vertices. Thus the only algebraic difference between a subscheme and the quotient scheme is simply which set of idempotents we begin with (or rather, with respect to which product we take idempotents for). Throughout this thesis we will use $\tilde{i}$ to describe an index of the quotient scheme and $i^\prime$ an index of the subscheme. Further $\cI$ and $\cJ$ will continue to denote the subset of indices for which
\[\sum_{i\in\cI} A_i = I_w\otimes J_r = r\sum_{j\in\cJ} E_j.\]
In particular, we also allow ourselves to designate a system of imprimitivity via the index set $\cJ$, ensuring from context that the meaning of the index set is clear.	Finally, we have a lemma concerning the eigenmatrices of the subscheme and quotient scheme.
\begin{lem}[\cite{Brouwer2003},\cite{Martin2007}]\label{repeatedcols}
	Let $P$ and $Q$ be the first and second eigenmatrix of an imprimitive association scheme with a system of imprimitivity given by index sets $\cI$ and $\cJ$. Let $P^\prime$ be the first eigenmatrix of the subscheme and $\tilde{Q}$ be the second eigenmatrix of the quotient scheme. Then for any $i\in \cI$ and $j\in \cJ$ we have
	\[P^\prime_{k'i} = P_{ki},\qquad \tilde{Q}_{\tilde{h}j} = Q_{hj}\]
	where $k\in k'$ and $h\in\tilde{h}$.
\end{lem}
\begin{proof}
	Recall that each subscheme of $\BMA$ is found by mapping each matrix in $\left<A_i\right>_{i\in\cI}$ to the principal submatrix corresponding to one of the equivalence classes. However, since $\phi:\BMA\rightarrow\BMA'$ via this mapping preserves matrix products, we see that 
	\[ P'_{k'i}\phi(E'_{k'}) = \phi(A_i)\phi(E'_{k'})= \phi(A_iE'_{k'})  =P_{ki}\phi(E'_{k'})\]
	for any choice of $i\in\cI$ and $k\in k'$. Similarly, the mapping $\tilde{\psi}:\BMA\rightarrow\tilde{\BMA}$ preserves entrywise products. Thus choosing any $j\in\cJ$ and $h\in\tilde{h}$ results in
	\[Q_{\tilde{h}j}\phi(\tilde{A}_{\tilde{h}}) = \tilde{\psi}(E_{j}\circ \tilde{A}_{\tilde{h}}) = \tilde{\psi}(E_{j})\circ\tilde{\psi}(\tilde{A}_{\tilde{h}}) = \tilde{Q}_{hj}\tilde{\psi}(\tilde{A}_{\tilde{h}})\]
\end{proof}
We finish this section by providing an illustration of each type of system of imprimitivity, returning to the example of the $3$-cube. Recall that this association scheme had two systems of imprimitivity with index sets $\cI_1 = \left\{0, 2\right\}$ and $\cI_2 = \left\{0, 3\right\}$. We begin with $\cI_1 = \left\{0,2\right\}$ and display the components of $\Gamma_2$ below.
\begin{center}\scalebox{.7}{$\begin{aligned}
		\begin{tikzpicture}[shorten >=1pt,auto,node distance=2cm,
		thin,main node/.style = {circle,draw, inner sep = 0pt, minimum size = 15pt}]
		
		\node[main node,fill=white] (1) {0};
		\node [right of = 1](2) {};
		\node [above of = 1](3) {};
		\node[main node,fill=cyan] [right of = 3](4) {3};
		\node [above right of = 1](5) {};
		\node[main node,fill=cyan] [right of = 5](6) {5};
		\node[main node,fill=cyan] [above of = 5] (7) {6};
		\node [right of = 7](8) {};
		
		\node [below right of =6] (11) {};
		\node[main node,fill=red] [right of = 11](12) {1};
		\node[main node,fill=red] [above of = 11](13) {2};
		\node [right of = 13](14) {};
		\node[main node,fill=red] [above right of = 11](15) {4};
		\node [right of = 15](16) {};
		\node [above of = 15] (17) {};
		\node[main node,fill=green] [right of = 17](18) {7};
		\node at (4,3.5) (9) {$\Gamma_2$};
		
		\path[-]
		(1)edge node {} (4)
		edge node {} (6)
		edge node {} (7)
		(12)edge node {} (13)
		edge node {} (15)
		edge node {} (18)
		(13)edge node {} (15)
		edge node {} (18)
		(4) edge node {} (6)
		(7) edge node {} (4)
		edge node {} (6)
		(15) edge node {} (18);
		\end{tikzpicture}
		\end{aligned}$}\end{center}
For each of the two components we find a subscheme isomorphic to $K_4$:
\begin{center}\scalebox{.7}{$\begin{aligned}
		\begin{tikzpicture}[shorten >=1pt,auto,node distance=2cm,
		thin,main node/.style = {circle,draw, inner sep = 0pt, minimum size = 15pt}]
		
		\node[main node,fill=white] (1) {0};
		\node[main node,fill=cyan] [right of = 3](4) {1};
		\node[main node,fill=cyan] [right of = 5](6) {2};
		\node[main node,fill=cyan] [above of = 5] (7) {3};
		\node at (2,4.5) (9) {$\Gamma_0^\prime$};
		
		\path (1) edge [in=120,out=145,loop] ();
		\path (4) edge [in=120,out=145,loop] ();
		\path (6) edge [in=120,out=145,loop] ();
		\path (7) edge [in=120,out=145,loop] ();
		
		\end{tikzpicture}\qquad\qquad
		\begin{tikzpicture}[shorten >=1pt,auto,node distance=2cm,
		thin,main node/.style = {circle,draw, inner sep = 0pt, minimum size = 15pt}]
		
		\node[main node,fill=white] (1) {0};
		\node[main node,fill=cyan] [right of = 3](4) {1};
		\node[main node,fill=cyan] [right of = 5](6) {2};
		\node[main node,fill=cyan] [above of = 5] (7) {3};
		\node at (2,4.5) (9) {$\Gamma_1^\prime$};
		
		\path[-]
		(1)edge node {} (4)
		edge node {} (6)
		edge node {} (7)
		(4) edge node {} (6)
		(7) edge node {} (4)
		edge node {} (6);
		\end{tikzpicture}
		\end{aligned}$}\end{center}
The quotient scheme is found by collapsing each component to a single point, giving an association scheme isomorphic to $K_2$:
\begin{center}\scalebox{.7}{$\begin{aligned}
		\begin{tikzpicture}[shorten >=1pt,auto,node distance=2cm,
		thin,main node/.style = {circle,draw, inner sep = 0pt, minimum size = 15pt}]
		
		\node[main node,fill=white] (1) {0};
		\node[main node,fill=red] [right of = 1](2) {1};
		\node at (1,1.5) (9) {$\Gamma_{\tilde{0}}$};
		
		\path (1) edge [in=120,out=145,loop] ();
		\path (2) edge [in=120,out=145,loop] ();
		
		\end{tikzpicture}\qquad\qquad
		\begin{tikzpicture}[shorten >=1pt,auto,node distance=2cm,
		thin,main node/.style = {circle,draw, inner sep = 0pt, minimum size = 15pt}]
		
		\node[main node,fill=white] (1) {0};
		\node[main node,fill=red] [right of = 1](2) {1};
		\node at (1,1.5) (9) {$\Gamma_{\tilde{1}}$};
		
		\path[-]
		(1)edge node {} (2);
		\end{tikzpicture}
		\end{aligned}$}\end{center}
Similarly, the system of imprimitivity corresponding to $\cI_2 = \left\{0,3\right\}$ results in the following components of $\Gamma_3$:
\begin{center}\scalebox{.7}{$\begin{aligned}
		\begin{tikzpicture}[shorten >=1pt,auto,node distance=2cm,
		thin,main node/.style = {circle,draw, inner sep = 0pt, minimum size = 15pt}]
		
		\node[main node,fill=white] (1) {0};
		\node[main node,fill=red] [right of = 1](2) {1};
		\node[main node,fill=red] [above of = 1](3) {2};
		\node[main node,fill=cyan] [right of = 3](4) {3};
		\node[main node,fill=red] [above right of = 1](5) {4};
		\node[main node,fill=cyan] [right of = 5](6) {5};
		\node[main node,fill=cyan] [above of = 5] (7) {6};
		\node[main node,fill=green] [right of = 7](8) {7};
		\node at (2,4.5) (9) {$\Gamma_3$};
		
		\path[-]
		(1) edge [bend right] node {} (8)
		(2) edge [bend left=45] node {} (7)
		(3) edge [bend left=45] node {} (6)
		(4) edge node {} (5);
		\end{tikzpicture}
		\end{aligned}$}\end{center}
Each component here results in a subscheme isomorphic to $K_2$:
\begin{center}\scalebox{.7}{$\begin{aligned}
		\begin{tikzpicture}[shorten >=1pt,auto,node distance=2cm,
		thin,main node/.style = {circle,draw, inner sep = 0pt, minimum size = 15pt}]
		
		\node[main node,fill=white] (1) {0};
		\node[main node,fill=red] [right of = 1](2) {1};
		\node at (1,1.5) (9) {$\Gamma_0^\prime$};
		
		\path (1) edge [in=120,out=145,loop] ();
		\path (2) edge [in=120,out=145,loop] ();
		
		\end{tikzpicture}\qquad\qquad
		\begin{tikzpicture}[shorten >=1pt,auto,node distance=2cm,
		thin,main node/.style = {circle,draw, inner sep = 0pt, minimum size = 15pt}]
		
		\node[main node,fill=white] (1) {0};
		\node[main node,fill=red] [right of = 1](2) {1};
		\node at (1,1.5) (9) {$\Gamma_1^\prime$};
		
		\path[-]
		(1)edge node {} (2);
		\end{tikzpicture}
		\end{aligned}$}\end{center}
while, the quotient scheme is isomorphic to $K_4$:

\begin{center}\scalebox{.7}{$\begin{aligned}
		\begin{tikzpicture}[shorten >=1pt,auto,node distance=2cm,
		thin,main node/.style = {circle,draw, inner sep = 0pt, minimum size = 15pt}]
		
		\node[main node,fill=white] (1) {0};
		\node[main node,fill=red] [right of = 1](2) {1};
		\node[main node,fill=red] [above of = 1](3) {2};
		\node[main node,fill=red] [above of =2](4) {3};
		\node at (1,2.5) (9) {$\Gamma_{\tilde{0}}$};
		
		\path (1) edge [in=120,out=145,loop] ();
		\path (2) edge [in=120,out=145,loop] ();
		\path (3) edge [in=120,out=145,loop] ();
		\path (4) edge [in=120,out=145,loop] ();
		
		\end{tikzpicture}\qquad&\qquad\begin{tikzpicture}[shorten >=1pt,auto,node distance=2cm,
		thin,main node/.style = {circle,draw, inner sep = 0pt, minimum size = 15pt}]
		
		\node[main node,fill=white] (1) {0};
		\node[main node,fill=red] [right of = 1](2) {1};
		\node[main node,fill=red] [above of = 1](3) {2};
		\node[main node,fill=red] [above of =2](4) {3};
		\node at (1,2.5) (9) {$\Gamma_{\tilde{1}}$};
		
		\path[-]
		(1) edge node {} (2)
		edge node {} (3)
		edge node {} (4)
		(2) edge node {} (3)
		edge node {} (4)
		(3) edge node {} (4);
		\end{tikzpicture}
		\end{aligned}$}\end{center}
It is not typical that the two examples --- both in the same Bose-Mesner algebra --- have subscheme and quotient scheme swapped. This is an artifact of the self-duality of the 3-cube. These schemes correspond to the dual pair of linear codes $\text{rowsp}\left[\begin{array}{ccc}
1 & 1 & 1
\end{array}\right]$ and $\text{nullsp}\left[\begin{array}{ccc}
1 & 1 & 1
\end{array}\right].$
\section{Polynomial schemes}\label{poly}
In this section we define and briefly develop the notion of polynomial association schemes. We again present two dual concepts, $P$-polynomial and $Q$-polynomial, though we will focus primarily on the $Q$-polynomial case outside of this section. Let $(X,\cR)$ be a $d$-class association scheme. We say $(X,\cR)$ is \emph{$Q$-polynomial}\index{$Q$-polynomial}, or \emph{cometric}\index{$Q$-polynomial!cometric}, if there exists an ordering of the idempotents, say $E_0,E_1,\dots,E_d$, such that the Krein parameters satisfy the following conditions:
\begin{enumerate}
	\item $q^k_{ij} = 0$ whenever $k>i+j$, and
	\item $q^k_{ij} > 0$ whenever $k = i+j$.
\end{enumerate}
Additionally, we note that it is sufficient to check only the above conditions with $i=1$ (see \cite[Prop.\ 2.7.1]{Brouwer1989}). Thus we may characterize $Q$-polynomial association schemes as exactly those for which there exists an eigenspace ordering for which the Krein matrix $L_1^*$ is irreducible tridiagonal. Noting that each row of $L_1^*$ sums to $q^0_{11}$, the parameters of $(X,\cR)$ are then entirely determined by its \emph{Krein array}\index{parameters!Krein array} $\iota^*(X,\cR) = \left\{b_0^*,\dots,b_{d-1}^*;c_1^*,\dots,c_d^*\right\}$ where $b_i^* = q^i_{1,i+1}$ and $c_i^* = q^{i+1}_{1i}$. When this occurs, we find that $\BMA = \left<E_1\right>_\circ$; that is, $E_1$ generates our entire Bose-Mesner algebra using Schur products. Additionally, for each $0\leq k\leq d$, we may define a single-variable polynomial $q_k(t)$ of degree $k$ so that $Q_{ik} = q_k\left(Q_{i1}\right)$ for $0\leq i\leq d$. This is equivalent to $E_k = \frac{1}{\vert X\vert}q_k\circ\left(\vert X\vert E_1\right)$; that is $q_k$ applied entrywise to $\vert X\vert E_1$ results in $\vert X\vert E_k$ (again see \cite[Prop.\ 2.7.1]{Brouwer1989}). Then we may define one final polynomial $q_{d+1}(t)$ with degree $d+1$ and no repeated roots such that $q_{d+1}\circ(\vert X\vert E_1) = 0$. This immediately implies that $E_1$ has $d+1$ distinct entries and we find it convenient to order the relations according to these values so that $Q_{01}>Q_{11}>\dots>Q_{d1}$; we call this the \emph{natural ordering of relations}\index{natural ordering} with respect to the $Q$-polynomial ordering $E_0,E_1,\dots,E_d$. As is suggested by this definition, it is possible to find multiple $Q$-polynomial orderings for the same association scheme. However, Suzuki \cite{Suzuki1998-2} showed that, with the exception of cycles, any $Q$-polynomial association scheme has at most two $Q$-polynomial orderings. We say a $Q$-polynomial association scheme is \emph{$Q$-bipartite}\index{$Q$-polynomial!$Q$-bipartite} if the Krein parameters satisfy $q^k_{ij} = 0$ whenever $i+j+k\notin 2\bbZ$. We find in this case that $\left\{E_i\right\}_{i\in 2\bbZ}$ serves as a Schur-closed subalgebra. In contrast, we say a $Q$-polynomial scheme is \emph{$Q$-antipodal}\index{$Q$-polynomial!$Q$-antipodal} if $q^k_{dd} = 0$ whenever $k\notin\left\{0,d\right\}$ and thus $\left\{E_0,E_d\right\}$ is a Schur-closed subalgebra. Each case coincides with a system of imprimitivity; the following is Suzuki's theorem concerning these systems.
\begin{thm}[\cite{Suzuki1998},\cite{Cerzo2009},\cite{Tanaka2011}] \label{suzukiimprim} Suppose $(X,\mathcal{R})$ is an imprimitive cometric association scheme with $Q$-polynomial ordering $E_0,\dots,E_d$ and natural ordering $A_0,\dots,A_d$. Then one of the following holds:
	\begin{enumerate}[label=$(\roman*)$]
		\item $(X,\mathcal{R})$ is $Q$-bipartite and $\mathcal{J} = \left\{0,2,4,\dots\right\}$, $\mathcal{I} = \left\{0,d\right\}$;
		\item $(X,\mathcal{R})$ is $Q$-antipodal and $\mathcal{J} = \left\{0,d\right\}$, $\mathcal{I} = \left\{0,2,4,\dots\right\}$.
	\end{enumerate}
\end{thm}
The original theorem in \cite{Suzuki1998} allowed for two exceptional cases, one with $d=4$ and another with $d=6$. These two cases were later ruled out in \cite{Cerzo2009} and \cite{Tanaka2011} respectively.

We now consider the more familiar dual notion: $P$-polynomial association schemes. Again let $(X,\cR)$ be a $d$-class association scheme. We say $(X,\cR)$ is \emph{$P$-polynomial}\index{$P$-polynomial}, or \emph{metric}\index{$P$-polynomial!metric}, if there exists an ordering of the relations, say $A_0,A_1,\dots,A_d$, such that the intersection numbers satisfy the following conditions:
\begin{enumerate}
	\item $p^k_{ij} = 0$ whenever $k>i+j$, and
	\item $p^k_{ij} > 0$ whenever $k = i+j$.
\end{enumerate}
Just as with cometric association schemes, we find that it suffices to check the above conditions only when $i=1$ (\cite[Prop.\ 2.7.1]{Brouwer1989}) and thus an association scheme is $P$-polynomial if and only if there exists an ordering of the relations for which the intersection matrix $(L_1)$ is irreducible tridiagonal. In this case we find that $\BMA= \left<A_1\right>_*$ and it is therefore common to consider a $P$-polynomial scheme synonymous with $\Gamma_1$ --- a \emph{distance-regular graph}; that is $(x,y)\in R_i$ if and only if the distance between $x$ and $y$ in $\Gamma_1$ is $i$. We find analogous results as we saw in the $Q$-polynomial case with \cite[Theorem 4.2.12]{Brouwer1989} and \cite{Taylor1978} giving that any $P$-polynomial association scheme which is not a cycle has at most two $P$-polynomial orderings. Further, we may define \emph{$P$-bipartite}\index{$P$-polynomial!bipartite} (or more commonly \emph{bipartite}) and \emph{$P$-antipodal}\index{$P$-polynomial!antipodal} (\emph{antipodal}) as those schemes for which $p^k_{ij}= 0$ if $i+j+k\not\in 2\bbZ$ and $p^k_{dd}= 0$ whenever $k\not\in\left\{0,d\right\}$ respectively. For these schemes we again find systems of imprimitivity. This time, however, $\mathcal{I} = \left\{0,2,4,\dots\right\}$ and $\mathcal{J} = \left\{0,d\right\}$ corresponding to the bipartition of a bipartite graph, while $\mathcal{I} = \left\{0,d\right\}$ and $\mathcal{J} = \left\{0,2,4,\dots\right\}$ give the antipodal classes as our system. As before, with the exception of cycles, we find that these systems of imprimitivity are all that can occur for $P$-polynomial schemes and note that both can occur within the same association scheme, for example the $3$-cube (more generally the $n$-cube) is both bipartite and antipodal. For details, see Theorem 4.2.1 in the book \cite{Brouwer1989} of Brouwer, et al.\ where credit is given to Smith \cite{Smith1971} and Gardiner \cite{Gardiner1980} where this is reference ``313" in the book \cite{Brouwer1989}.

Despite the close connection between $P$-polynomial and $Q$-polynomial association schemes, we note that many of the theorems mentioned here in the $P$-polynomial case predate their $Q$-polynomial analogues by as much as 30 years. Further, there are many other theorems which are known to be true for metric schemes whose cometric analogues have yet to be proven. For instance Taylor and Levingston showed in 1978 (\cite{Taylor1978}) that the sequence $k_0,k_1,\dots,k_d$ is unimodal, however the $Q$-analogue --- the property that the sequence $m_0,m_1,\dots,m_d$ is unimodal for any $Q$-polynomial ordering --- remains a conjecture to this day. Chapter 6 of \cite{Brouwer1989} details many of the known examples of distance-regular graphs --- all tables mentioned here appear in that chapter. Brouwer, et al.\ list 21 classical parameter sets (Tables 6.1 \& 6.2), 15 of which correspond to infinite families, three folded classical graphs (Table 6.3), nine near regular polygons including the generalized polygons (Tables 6.5 \& 6.6), as well as 20 more primitive distance regular graphs (Tables 6.8 \& 6.9). Further they give the known bipartite and antipodal examples (Tables 6.9 \& 6.10 resp.). On the cometric side however, much less is known. Out of the infinite families of $P$-polynomial schemes listed in \cite{Brouwer1989}, five families are also $Q$-polynomial: the Johnson schemes, Hamming schemes, Grassmann schemes, dual polar spaces, and sesquilinear/quadratic forms. More recently \cite{VanDam2005}, Van Dam and Koolen discovered the twisted Grassmann graphs which are also both metric and cometric. The remaining known families $Q$-polynomial association schemes are: linked systems of symmetric designs, mutually unbiased bases, bipartite doubles of some polar spaces, duals of metric translation schemes, two families found by Penttila and Williford \cite{Penttila2011} --- one 4-class $Q$-bipartite and one 3-class primitive --- and one more family found by Moorhouse and Williford \cite{Moorhouse2016}. In addition to these we also know of sporadic examples such as the 22 listed in \cite{Martin2007} and new examples found by King \cite{King2018}. Out of these examples, we note that only the families which are both metric and cometric allow for unbounded $d$. In fact, Bannai and Ito made the following conjecture concerning primitive association schemes.
\begin{conj*}[Bannai \& Ito] 
	For $d$ sufficiently large, a primitive association scheme with $d$ classes is metric if and only if it is cometric.
\end{conj*}
We note that every primitive $2$-class association scheme (connected strongly regular graph) is vacuously both metric and cometric. Thus, in view of the above conjecture, it should not be surprising that the most fruitful place to search for new examples of $Q$-polynomial schemes which are not also metric is the range $3\leq d\leq 6$.
	\chapter{Positive semidefinite cone of an association scheme}\label{psdcone}
	We begin this chapter with a simple problem: How many lines may one draw on a piece of paper which intersect in a single angle? It is clear that we may draw two lines which intersect with any angle $0< \theta< \frac{\pi}{2}$. However, if we fix $\theta = \frac{\pi}{3}$, we may draw a third line bisecting the obtuse angle between our pair of lines which results in a angle of $\frac{\pi}{3}$ which both original lines. The same question may be asked in a higher dimension: how many lines may one construct in $\mathbb{R}^3$ so that any pair of lines intersect in the same angle? For this case, we construct a regular pentagon centered at the origin of the $xy$-plane and draw fives lines connecting the origin to each vertex of the pentagon. The resulting five lines (identified by each vertex) admit two possible angles between them: $\frac{2\pi}{5}$ between adjacent vertices and $\frac{\pi}{5}$ between non-adjacent vertices.
	
	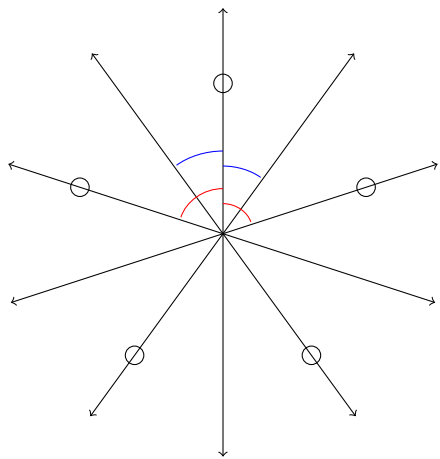
\begin{figure}[h]\centering\begin{tikzpicture}[shorten >=1pt,auto,node distance=2cm,
		thin,main node/.style = {circle,draw, inner sep = 0pt, minimum size = 7pt}]
		\foreach \a in {90,162,...,378} { 
			\node[main node,fill=white] at (\a:2cm) (\a) {};
			\node at (\a+180:2cm) (\a2) {};
			\draw[<->] (\a:3cm) -- (\a+180:3cm);
		}
		\draw[blue] (90:0.9cm) arc (90:54:0.9cm);
		\draw[red] (90:0.4cm) arc (90:18:0.4cm);
		\draw[blue] (90:1.1cm) arc (90:126:1.1cm);
		\draw[red] (90:0.6cm) arc (90:162:0.6cm);
		\end{tikzpicture}\caption[Five lines with two angles in $\mathbb{R}^2$]{Five lines resulting from a regular pentagon in $\mathbb{R}^2$. The blue angle is $\frac{\pi}{5}$ while the red angle is $\frac{2\pi}{5}$}.\end{figure}

	Now shift the pentagon along the $z$-axis and track the five lines joining the origin of $\mathbb{R}^3$ to its vertices. By shifting the pentagon far enough, we eventually find that each pair of lines with adjacent vertices admit a smaller angle than those pairs which are not adjacent.
	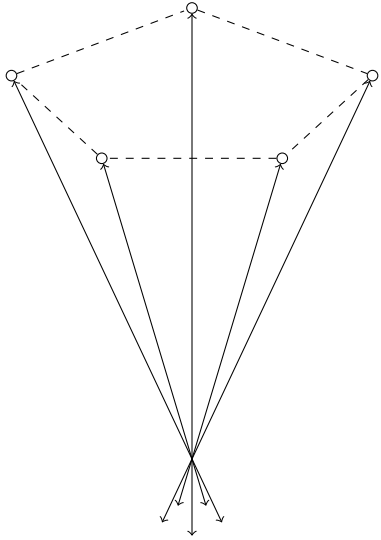
\begin{figure}[h]\centering\begin{tikzpicture}[shorten >=1pt,auto,node distance=2cm,
		thin,main node/.style = {circle,draw, inner sep = 0pt, minimum size = 4pt}]
		\def\z{6}
		\def\x{1.2}
		\def\y{.9}
		\def\xx{2.4}
		\def\yy{2}
		\node[main node, fill=white] at (0,\z) (1) {};
		\node[main node, fill=white] at (\xx,\z-\y) (2) {};
		\node[main node, fill=white] at (-\xx,\z-\y) (3) {};
		\node[main node, fill=white] at (\x,\z-\yy) (4) {};
		\node[main node, fill=white] at (-\x,\z-\yy) (5) {};
		\begin{scope}[yscale=-0.2,xscale=-0.2]
		\node at (0,\z) (1b) {};
		\node at (\xx,\z-\y) (2b) {};
		\node at (-\xx,\z-\y) (3b) {};
		\node at (\x,\z-\yy) (4b) {};
		\node at (-\x,\z-\yy) (5b) {};
		\end{scope}
		\node at (0,0) (0) {};
		\draw[<->] (1) -- (1b);
		\draw[<->] (2) -- (2b);
		\draw[<->] (3) -- (3b);
		\draw[<->] (4) -- (4b);
		\draw[<->] (5) -- (5b);
		\draw[-,dashed] (1) -- (2) -- (4) -- (5) -- (3) -- (1);

		\end{tikzpicture}\caption[Pentagon lifted out of the plane]{Lines joining the origin to the points on a regular pentagon after lifting it out of the plane. In this case, the smaller angle occurs between consecutive lines, rather than those coming from non-adjacent points.}\end{figure}
	
	Thus there exists some height for which the angles are all the same --- this occurs when the cosine of the angle between pairs is $\frac{1}{\sqrt{5}}$. Further, one finds that this configuration allows for a sixth line to be added, given by the $z$-axis, which is ``equiangular" with each of the other five lines. More commonly, this configuration of lines is given by connecting the center of an icosahedron to the six pairs of antipodal points. In both cases listed above, we find that these configurations are optimal for that dimension --- that is, one cannot find a set of four lines in $\mathbb{R}^2$ nor seven lines in $\mathbb{R}^3$ for which any pair of lines intersect in a given angle.
	
	This simple problem generalizes to a classic unsolved problem in discrete geometry (\cite{Haantjes1948},\cite{Lemmens1973},\cite{vanLint1966}). Given a fixed positive integer $n$ and angle $0<\theta<\frac{\pi}{2}$, a set of \emph{equiangular lines}\index{equiangular lines} is a set of lines through the origin of $\mathbb{R}^n$ such that the angle between any pair of distinct lines is $\theta$. We then ask, what is the maximum number of equiangular lines one may find for each choice of $n$ and $\theta$. Over the past 70 years, researchers in both math and physics have developed the theory of equiangular lines, finding upper bounds on the number of equiangular lines in any given dimension using tools such as linear programming \cite{Delsarte1975}, number theory \cite{Lemmens1973}, graph spectra \cite{Brouwer2012}, and most recently semidefinite programming \cite{Barg2014}. A related question is as follows: How many orthonormal bases may we find in $\mathbb{R}^n$ such that the absolute value of the inner product between vectors in distinct bases is fixed. Such a set is called a set of \emph{real mutually unbiased bases}\index{mutually unbiased bases} (see \cite{Boykin2005},\cite{Calderbank1997},\cite{Delsarte1975}). Recently, the complex analogues have attracted interest in quantum computing (\cite{Appleby2005},\cite{Grassl2008},\cite{Goyeneche2015},\cite{Ivanovic1981},\cite{Waldron2018},\cite{Wootters1989}), though we will restrict ourselves to the questions in real space as the adjacency matrices of symmetric association schemes are always diagonalizable over the reals.
	
	Both of the problems mentioned above are restricted versions of the more general question of finding spherical $t$-distance sets. A spherical $t$-distance set is a set of unit vectors such that the inner product between any pair of distinct vectors may take only $t$ unique values in the interval $[-1,1)$. For instance, a set of equiangular lines is a spherical 2-distance set; even though we allow only one angle between lines, this results in two possible inner products between unit vectors. Likewise, a set of mutually unbiased bases is a spherical 3-distance set. In each case the question of finding equiangular lines or mutually unbiased bases may be posed as finding large positive semidefinite matrices with a low rank, constant diagonal, and few distinct entries off the diagonal. To wit, let $G$ be an $n\times n$ positive semidefinite matrix with rank $r$. Then there exists a $r\times n$ matrix $U$ such that $G = U^TU$; that is $G$ is the Gram matrix of the columns of $U$. If $G$ has a constant diagonal $\alpha$, then the columns of $\frac{1}{\sqrt{\alpha}}U$ are unit vectors with the same number of distinct inner products between distinct pairs of vectors. The columns of $U$ therefore give a spherical $t$-distance set in $\mathbb{R}^r$ where $t$ is the number of distinct entries off the diagonal of $G$. Now recall that a $d$-class association scheme results in a Bose-Mesner algebra with $d+1$ basis idempotents. Fix some Bose-Mesner algebra $\BMA$ on $n$ points. We are interested in positive semidefinite matrices and thus we will focus on the cone of positive semidefinite matrices within $\BMA$. In any real vector space $V$, a \emph{cone} is a subset $C\subset V$ closed under addition and multiplication by non-negative scalars. It often becomes useful to define the cone via the extremal vectors; an extremal vector $v\in C$ is a vector for which $x+y = v$ with $x,y\in C$ implies either $x=cv$ or $y=cv$ for some constant $c$. Given these points, our cone is simply the non-negative linear combinations of these extremal vectors. Returning to $\BMA$, the basis of idempotents are the extremal points. Thus, the \emph{cone of positive semidefinite matrices}\index{positive semidefinite cone} in $\BMA$ is exactly the set $\left\{\sum_{i=0}^d\alpha_iE_i\right\}$ for $\alpha_i\geq 0$.
	
	Noting that every matrix in $\BMA$ has a constant diagonal and at most $d$ entries off the main diagonal, if $\BMA$ contains a matrix with non-zero diagonal and rank $r$ then there exists a spherical $d$-distance set of size $n$ in $\mathbb{R}^r$. Conversely, any restrictions we find on the existence of such $t$-distance sets allows us to add feasibility conditions on the parameters of association schemes. For instance, Delsarte, Goethals, and Seidel \cite{Delsarte1975} bounded the maximum size of a $t$-distance set based solely on the dimension of the ambient space and the inner products allowed between vectors. An example of one of the bounds found is as follows: Let $X$ be a spherical $4$-distance set in $\mathbb{R}^n$ with inner products $\left\{\pm\alpha,\pm\beta \right\}$ for $\alpha,\beta\in \mathbb{R}$. Then as long as $\alpha^2+\beta^2\leq \frac{6}{n+4}$ we must have
	\[\vert X\vert \leq\frac{n(n+2)(1-\alpha^2)(1-\beta^2)}{2-(n+2)(\alpha^2+\beta^2)+n(n+2)\alpha\beta}\]
	provided the denominator is positive. Using these same techniques they also confirm the \emph{relative bound}\index{relative bound}: given any $2$-distance set $X$ in $\mathbb{R}^n$ with inner products $\pm\alpha\in \mathbb{R}$, 
	\[\vert X\vert \leq\frac{n(1-\alpha^2)}{1-n\alpha^2}.\]
	Such an observation allows us to rule out certain parameter sets if the rank of any matrix in the PSD cone falls too low compared to the number and values of the distinct entries off the diagonal. In a later paper \cite{Delsarte1977}, the same authors proved that a $s$-distance set is a spherical $t$-design if and only if each of the first $t$ Gegenbauer polynomials, summed over the inner products $\left<x,x'\right>$ ($x,x'\in X$) are all zero. Further, they showed that if a $t$-design has fewer than $\frac{t+1}{2}$ distinct inner products, then it must correspond to the point set of an association scheme. Later, Bannai and Bannai \cite{Bannai2008} showed that whenever the number of distinct inner products is less than $\frac{t+3}{2}$, the $t$-design still admits an association scheme as long as the set of vectors are antipodal ($x\in X$ if and only if $-x\in X$). This close relation between ``near-tight" $t$-designs and association schemes hints at the types of vector systems we should expect to find within Bose-Mesner algebras. For instance, Suda \cite{Suda2011} used these results to characterize when the first idempotent of a $Q$-polynomial ordering results in a spherical $t$-design. A central tool in many of these results is a family of single-variable polynomials known as the Gegenbauer polynomials. A key result of Sch\"{o}nberg \cite{Schoenberg1942} concerning this family tells us that each Gegenbauer polynomial, when applied to a Gram matrix entrywise, produces a positive semidefinite matrix.
	
	In this chapter we will further examine the cone of positive semidefinite matrices within Bose-Mesner algebras. We first display how one may build large spherical $t$-distance sets by taking non-negative linear combinations of the idempotents. Using this, we will build examples of equiangular lines meeting a known upper bound for the number of lines in given dimensions. We will then introduce the Gegenbauer polynomials as well as Sch\"{o}nberg's Theorem. Using this theorem we will derive new constraints on the Krein parameters of any association scheme. We further examine these restrictions in the case of $Q$-polynomial association schemes where we give seven feasibility conditions which are not implied by the previously mentioned feasibility conditions; we believe these results are new. These results will be applied directly to the case of $4$-class $Q$-bipartite association schemes in Chapter \ref{4classbip}. The material in this chapter grew out of discussions with W-H.\ Yu and contains joint work with W.\ J.\ Martin. Below we list the main contributions in this chapter.
	
	In both Theorem \ref{schoen-as} and the final line of Theorem \ref{degreebound}, $Q_k^m(t)$ denotes the Gegenbauer polynomial of degree $k$ in dimension $m$. The superscript $m$ should help the reader to distinguish this from entries of the second eigenmatrix $Q$.
	\begin{restatable*}{lem}{buildingequi}\label{equilines}
		Let $\BMA$ be a Bose-Mesner algebra with second eigenmatrix $Q$ and multiplicities $m_0,\dots,m_d$. Define $Q^\prime$ to be the submatrix of $Q$ given by deleting row $0$. If there exists a non-negative vector $x$ so that $Q^\prime x\in\left\{-1,1\right\}^d$, then $\BMA$ contains the Gram matrix of $\vert X\vert$ equiangular lines in dimension $n=\sum_{x_j\neq 0} m_j$ with angle $\alpha =\left(\sum x_jm_j\right)^{-1}$.
	\end{restatable*}
	\begin{restatable*}{thm}{schAS}\label{schoen-as}
		Let $(X,\mathcal{R})$ be an association scheme with minimal idempotents $E_0,\dots,E_d$ and matrices of Krein parameters $\displaystyle{L_0^*,\dots,L_d^*}$. Fix some $E_i$, $0\leq i\leq d$, and let $m_i:=\text{rank}\left(E_i\right)$. Then for any choice of $\ell>0$, there exist non-negative constants $\theta_{\ell j}$, $0\leq j\leq d$, such that
		\begin{equation}\label{ELGeg}
		Q_\ell^{m_i}\circ\left(\frac{\vert X\vert}{m_i}E_i\right) = \sum_j \theta_{\ell j} E_j;\qquad	Q_\ell^{m_i}\left(\frac{1}{m_i}L_i^*\right) = \frac{1}{\vert X\vert}\sum_j \theta_{\ell j} L_j^*.
		\end{equation}
		The eigenvalues of $Q_\ell^{m_i}\circ\left(\frac{\vert X\vert}{m_i}E_i\right)$ are $\theta_{\ell 0},\dots,\theta_{\ell d}$ where $\theta_{\ell j}$ is non-zero only if $E_j$ is contained in the Schur subalgebra generated by $E_i$.
	\end{restatable*}
	
	\begin{restatable*}{thm}{degbnd}\label{degreebound}
		Suppose we have a feasible parameter set for an association scheme with first and second eigenmatrices $P$ and $Q$. Fix $0\leq i\leq d$ and set $m_i:=Q_{0i}$. For $\ell>0$ define $\mu_\ell = \frac{\ell-1}{\ell+m_i-3}$ and $\lambda_j = \nicefrac{Q_{ji}}{Q_{0i}}$. Let $\gamma_{\ell,j}$ equal $\lambda_j^2$ if $(1+\mu_\ell)^2\lambda_j^2\geq 4\mu_\ell$ and $\mu_\ell$ otherwise. Then all entries of $Q_{\ell^*}^{m_i}\left(\frac{1}{m_i}L_i^*\right)$ are non-negative whenever
		\[\sum_{j=1}^d \left(\prod_{\ell=2}^{\ell^*+1}\gamma_{x,j}\right)P_{0j}\left(1+\lambda_j^2\right)\leq\frac{1}{\vert X\vert}.\]
	\end{restatable*}
	\begin{restatable*}{thm}{cometricbounds}\label{cometricbnds}
		Suppose we have a feasible parameter set for a cometric association scheme with Krein array $\iota^*(X,\cR) = \left\{m,b^*_1,\dots,b^*_{d-1};1,c_2^*\dots,c^*_{d}\right\}$ where $m>2$. Define $b_{j-1}^* = c_j^*=a_j^*=0$ for $j>d$. Then the scheme is realizable only if
		\begin{enumerate}[label=$(\roman*)$]
			\item $\left(a_1^*\right)^2 + b_1^*c_2^* \geq\frac{2m(m-1)}{m+2},$
			\item $\left(a_1^*\right)^2+2a_1^*a_2^*+c_2^*q^2_{22}\geq \frac{4m(m-2)}{m+4},$
			\item $\frac{6m(m-1)(m-4)}{(m+4)(m+6)}+\frac{\left(3a_1^*\left(a_1^*+a_2^*\right)+c_2^*q_{22}^2\right)b_1^*c_2^*+\left(a_1^*\right)^4}{m}\geq \frac{(7m-18)\left(\left(a_1^*\right)^2+b_1^*c_2^*\right)}{m+6},$
			\item $\sum_{i=1}^3\left(b_i^*c_{i+1}^* + a_i^*\sum_{j=i}^3 a_j^*\right)\leq \frac{3(3m-2)}{m+6}.$
			\item $\frac{16m(m-1)}{(m+4)(m+8)} + \frac{\left(a_1^*\right)^4 + \left(3a_1^*\left(a_1^*+a_2^*\right)+c_2^*q_{22}^2\right)b_1^*c_2^*}{(m-2)m}\geq \frac{12\left(\left(a_1^*\right)^2+b_1^*c_2^*\right)}{m+8},$
		\end{enumerate}
		Additionally, if $a_1^*>0$, then
		\begin{enumerate}[label=(\roman*)]
			\addtocounter{enumi}{5}
			\item $\left(a_1^*\right)^2 + b_1^*c_2^*\left(2 + \frac{a_2^*}{a_1^*}\right)\geq \frac{4m(2m-3)}{m+6},$
			\item $\left(a_1^*\right)^2+2a_1^*a_2^*-\left(a_2^*\right)^2+2c_2^*q_{22}^2+\frac{b_2^*c_3^*\left(a_3^*-a_1^*\right)-ma_2^*}{a_1^*+a_2^*}\geq\frac{6m(m-4)}{m+6}.$
		\end{enumerate}
	\end{restatable*}

	\section{Lines with few angles}\label{equianglines}
	Let $(X,\cR)$ be an association scheme with basis relations $A_0,\dots,A_d$, orthogonal idempotents $E_0,\dots,E_d$, and Bose-Mesner algebra $\BMA$. Since each $E_i$ is an idempotent matrix, it has spectrum $1^{m_i}$, $0^{v-m_i}$ and thus is positive semidefinite, denoted $E_i\succeq 0$. Any $G\in\BMA$ may be uniquely expressed
	\begin{equation}G = \sum_i\alpha_i E_i,\label{lincomb}\end{equation}
	and we have $\text{spec}\left(G\right) = \left\{\alpha_0,\dots,\alpha_d\right\}$. Therefore $G\succeq 0$ if and only if $\alpha_i\geq 0$ for all $0\leq i\leq d$. Thus the \emph{positive semidefinite cone}\index{positive semidefinite cone} of $(X,\cR)$ is the set of non-negative linear combinations of the idempotents $E_0,\dots,E_d$. Further, Equation \eqref{lincomb} gives us that \begin{equation}\text{rank}\left(G\right) = \sum_{\alpha_i\neq 0} m_i.\label{rank}\end{equation}
	Finally, for any $G$ in the positive semidefinite cone, we must have $G = \sum_i \beta_i A_i$ with $\beta_0>0$ if $G\neq 0$ and thus there are at most $d$ distinct values off the main diagonal. Recall that a set of \emph{equiangular lines}\index{equiangular lines} is a set of vectors for which the angle between any pair is given by some fixed $0<\theta<\frac{\pi}{2}$. By abuse of terminology, we say the (common) ``angle" for this set is $\alpha = \cos(\theta)$ and thus we identify a set of equiangular lines in $\mathbb{R}^n$ with angle $0<\alpha<1$ with a rank $n$ positive semidefinite matrix for which each entry on the main diagonal is $1$ and every entry off the diagonal is one of $\pm \alpha$; we call such a matrix the \emph{Gram matrix of a set of equiangular lines with angle $\alpha$}. While we have mentioned equiangular lines and mutually unbiased bases multiple times, we will also find other $t$-distance sets useful to consider, fixing certain angles to fit our application. For instance, in Chapter $\ref{3class}$ we examine one type of cometric association scheme which may be characterized by the 3-distance set in $\mathbb{R}^r$ with $-\frac{1}{r}$ as a possible inner product. We then show that certain linear combinations of $E_0$, $E_1$, and $E_3$ will also produce equiangular lines in this way. We show later in the chapter that a certain family of these association schemes also contains the Gram matrix of real mutually unbiased bases. In the optimal case, we find that the number of vectors in each set scales as $\frac{1}{2}v^2$ for dimension $v$. In this section, we restrict ourselves to building equiangular lines and consider 3-class primitive $Q$-polynomial association schemes. Before moving to examples, we note the following bound on the maximum number of equiangular lines in a given dimension known as the relative bound.
	\begin{thm}[\cite{vanLint1966}]\label{relbound}\index{relative bound}
		For $0<\alpha<1$, let $v_\alpha(n)$ be the maximum number of equiangular lines with angle $\alpha$ in $\mathbb{R}^n$. If $n<\alpha^{-2}$ then
		\[v_\alpha(n)\leq\frac{n(1-\alpha^2)}{1-n\alpha^2}.\]
	\end{thm}
	In this section, we refer to a system of equiangular lines as \emph{optimal}\index{equiangular lines!optimal} if it achieves the relative bound for the given dimension and angle; likewise a system of equiangular lines is \emph{near optimal}\index{equiangular lines!near optimal} if it is within one line of being optimal. We note that the relative bound need not give an integer on the right hand side of the inequality. In the case where the relative bound gives $v_\alpha(n)\leq x+\eta$ for $x\in\mathbb{Z}$ and $0<\eta<1$, we reduce the bound to $v_\alpha(n)\leq x$. In this case, a set of $x$ vectors with inner products $\pm\alpha$ would still be considered optimal, since no larger set could ever be found. We now examine two 3-class primitive association schemes to illustrate how we may build equiangular lines from these schemes. The first comes from the halved 7-cube while the second corresponds to the dual polar space $B_3(2)$.
	\begin{example}
		Consider the 3-class primitive association scheme given by the halved 7-cube \cite{Brouwer1989}. This scheme is both metric and cometric with the following eigenmatrices
		\[P = \left[\begin{array}{rrrr}
		1& 21& 35& 7\\
		1& 9& -5& -5\\
		1& 1& -5& 3\\
		1& -3& 3& -1\\
		\end{array}\right],\qquad  Q= \left[\begin{array}{rrrr}
		1& 7& 21& 35\\
		1& 3& 1& -5\\
		1& -1& -3& 3\\
		1& -5& 9& -5\\
		\end{array}\right].\]
		Recall that $E_j = \frac{1}{\vert X\vert}\sum_{i}Q_{ji}A_i$ and consider
		\[G = \frac{16}{7}\left(E_1 + E_2\right)\]
		where the constant $\frac{16}{7} = \frac{\vert X\vert}{m_1+m_2}$ is chosen to make the main diagonal of $G$ equal to 1. We may then use the entries of $Q$ to replace each idempotent with the corresponding sum of adjacency matrices giving 
		\[G = A_0 + \frac{1}{7}A_1 -\frac{1}{7}A_2+\frac{1}{7}A_3.\]
		Then $G$ is the Gram matrix of $64$ lines in dimension $m_1+m_2 = 28$ with angle $\frac{1}{7}$. Using the relative bound, we find that this system is optimal.
	\end{example}
	\begin{example}
		Consider the 3-class primitive cometric (and metric) association scheme defined on the dual polar space $B_3(2)$. The first and second eigenmatrices are
		\[P = \left[\begin{array}{rrrr}
		1& 14& 56& 64\\
		1& 5& -2& -8\\
		1& -1& -4& 4\\
		1& -7& 14& -8\\
		\end{array}\right],\qquad  Q= \left[\begin{array}{rrrr}
		1& 35& 84& 15\\
		1& \nicefrac{25}{2}& -6& -\nicefrac{15}{2}\\
		1& \nicefrac{5}{4}& -6& \nicefrac{15}{4}\\
		1& -\nicefrac{35}{8}& \nicefrac{21}{4}& -\nicefrac{15}{8}\\
		\end{array}\right].\]
		Consider the matrix
		\[G = 15E_0+24E_1 + 24E_3 = 9A_0 + A_1 +A_2-A_3.\]
		Similar to before, $\frac{1}{9}G$ is the Gram matrix of $135$ lines in dimension $m_0+m_1+m_2 = 51$ with angle $\pm\frac{1}{9}$. Here, the relative bound tells us that optimal number of lines in dimension 51 with these inner products is bounded above by 136, thus this construction is near optimal.
	\end{example}
	We now consider the construction in general.
	\buildingequi
	\begin{proof}
		Let $y = Q^\prime x$. Then we have the $d$ equations
		\[Q_{i0}x_0 + Q_{i1}x_1 + \dots + Q_{id}x_d = y_i\]
		for $1\leq i\leq d$ where each $y_i$ is either $1$ or $-1$.	Consider the matrix
		\[G = \sum x_jE_j = \frac{1}{\vert X\vert}\sum_{i=0}^{d}\left(\sum_{j=0}^dx_jQ_{ij}\right)A_i = \frac{1}{\vert X\vert}\left(\sum_{j=0}^d x_jm_j\right)A_0 + \frac{1}{\vert X\vert}\sum_{i=1}^d y_iA_i.\]
		Since $\vert y_i\vert = 1$ for each $1\leq i\leq d$, each off-diagonal entry of $G$ has the same absolute value. Thus we may scale $G$ by its diagonal entry to obtain the Gram matrix of a set of equiangular lines with angle $\left(\sum x_j Q_{0j}\right)^{-1}$. The rank of $G$ is the sum of the ranks of those $E_j$ with $x_j\neq 0$.
	\end{proof}
	\begin{table}[H]
		We finish this section by applying Lemma \ref{equilines} to the table of 3-class primitive $Q$-polynomial schemes found at \cite{Willifordtable3}.
		\begin{center}
			Optimal constructions\\
			\begin{tabular}{l|c|c|ccl|c|c|c}
				Label & $\vert X\vert$ & $n$ & $\nicefrac{1}{\alpha}$&&Label & $\vert X\vert$ & $n$ & $\nicefrac{1}{\alpha}$ \\\cline{1-4}\cline{6-9}
				$\left<64,7\right>$ & 64 & 28 & 7 & \qquad &$\left<120,14\right>$ & 120 & 35 & 7\\
				$\left<64,9\right>$ & 64 & 36 & 9 & \qquad	&$\left<120,17a\right>$ & 120 & 35 & 7\\
				$\left<64,21\right>$ & 64 & 28 & 7 & \qquad & $\left<1024,31\right>$ & 1024 & 496 & 31 \\
				$\left<120,9\right>$ & 120 & 35 & 7 & \qquad &$\left<1024,66\right>$ & 1024 & 528 & 33\\
			\end{tabular}\\\vspace{5mm}
			Potential optimal constructions\\
			\begin{tabular}{l|c|c|ccl|c|c|c}
				Label & $\vert X\vert$ & $n$ & $\nicefrac{1}{\alpha}$&&Label & $\vert X\vert$ & $n$ & $\nicefrac{1}{\alpha}$ \\\cline{1-4}\cline{6-9}
				$\left<280,27a\right>$ & 280 & 63 & 9 &		\qquad	&$\left<1456,97\right>$ & 1456 & 195 & 15 \\
				$\left<324,19a\right>$ & 324 & 171 & 19 &	\qquad  &$\left<1520,49\right>$ & 1520 & 589 & 31 \\	
				$\left<344,42\right>$ & 344 & 43 & 7 &		 \qquad  &$\left<1520,56\right>$ & 1520 & 589 & 31 \\
				$\left<460,51\right>$ & 460 & 69 & 9 &		 \qquad &$\left<1596,55\right>$ & 1596 & 551 & 29\\
				$\left<540,44\right>$ & 540 & 99 & 11 &		\qquad  &$\left<2016,62a\right>$ & 2016 & 651 & 31 \\	
				$\left<540,49\right>$ & 540 & 99 & 11 &			\qquad   &$\left<2016,65\right>$ & 2016 & 651 & 31\\	
				$\left<936,51\right>$ & 936 & 221 & 17 & 			\qquad  &$\left<2160,119\right>$ & 2160 & 255 & 17\\
				$\left<936,51a\right>$ & 936 & 221 & 17 &			\qquad &$\left<2160,119a\right>$ & 2160 & 255 & 17 \\
				$\left<1024,33\right>$ & 1024 & 528 & 33 &  		\qquad &$\left<2160,119b\right>$ & 2160 & 255 & 17 \\
				$\left<1200,55\right>$ & 1200 & 110 & 11& 			\qquad &$\left<2500,51\right>$ & 2500 & 1225 & 49\\
				$\left<1200,109a\right>$ & 1200 & 110 & 11			&\qquad &$\left<2500,51a\right>$ & 2500 & 1275 & 51 \\
				$\left<1344,79\right>$ & 1344 & 238 & 17			&\qquad &$\left<2500,75\right>$ & 2500 & 1275 & 51 \\
				$\left<1456,90a\right>$ & 1456 & 195 & 15												
			\end{tabular}\\\vspace{5mm}
			Near optimal constructions\\
			\begin{tabular}{l|c|c|ccl|c|c|c}
				Label & $\vert X\vert$ & $n$ &  $\nicefrac{1}{\alpha}$&&Label & $\vert X\vert$ & $n$ & $\nicefrac{1}{\alpha}$\\\cline{1-4}\cline{6-9}
				$\left<35,6\right>^*$ & 35 & 21 & 7 & \qquad &$\left<135,35\right>^*$ & 135 & 51 & 9\\
			\end{tabular}\\\vspace{5mm}
			Potential near optimal constructions\\
			\begin{tabular}{l|c|c|ccl|c|c|c}
				Label & $\vert X\vert$ & $n$ &  $\nicefrac{1}{\alpha}$&&Label & $\vert X\vert$ & $n$ & $\nicefrac{1}{\alpha}$\\\cline{1-4}\cline{6-9}
				$\left<279,30\right>$ & 279 & 63 & 9 &\qquad &$\left<923,70\right>$ & 923 & 143 & 13 \\
				$\left<319,28\right>$ & 319 & 88 & 11 &\qquad &$\left<1035,68\right>$ & 1035 & 185 & 15 \\
				$\left<377,28\right>$ & 377 & 117 & 13 &  \qquad &$\left<1349,70\right>$ & 1349 & 285 & 19 \\
				$\left<527,30\right>$ & 527 & 187 & 17 & \qquad &$\left<1975,78\right>$ & 1975 & 475 & 25 \\
				$\left<527,30a\right>$ & 527 & 187 & 17 & \qquad &$\left<2159,126\right>$ & 2159 & 255 & 17 \\
				$\left<729,56\right>$ & 729 & 337 & 25& \qquad &$\left<2759,88\right>$ & 2759 & 713 & 31 \\
				
			\end{tabular}
			\caption[Optimal and near-optimal constructions for equiangular lines using 3-class primitive cometric association schemes]{In these tables we give the parameters for sets of equiangular lines which arise from the given 3-class primitive cometric scheme --- if the scheme exists. In both cases, we split the tables into those schemes known to exist and those whose existence has not yet been determined. Thus, the sets of equiangular lines in the first and third tables do exist, while the sets in the second and forth tables are not guaranteed to exist as the parameter set might not be realizable. Each parameter set is listed in \cite{Willifordtable3} using the label in the far left column. For each set we list the number of lines $\vert X\vert$, dimension $n$, and the inverse of the inner product $\frac{1}{\alpha}$.}\label{optimalconst}
		\end{center}
	\end{table}
	While it seems surprising that each angle is the inverse of an odd integer, the following result of Neumann guarantees this holds for any large set of equiangular lines.
	\begin{thm}[\cite{Lemmens1973}]
		Let $X$ be a set of $v$ equiangular lines in $\mathbb{R}^n$ with angle $\alpha$. If $v>2n$ then $\frac{1}{\alpha}$ is an odd integer.
	\end{thm}
	\begin{thm}
		For each near-optimal case listed in Table \ref{optimalconst} except possibly for $\left<729,56\right>$, if the parameter set is realizable, we may extend the set of vectors by one to obtain an optimal set.
	\end{thm}
	\begin{proof}
		In each case, $G = x_0 E_0 + \sum_{j=1}^{d}x_jE_j$ for orthogonal idempotents $E_0,\dots,E_d$ and $E_0 = \frac{1}{\vert X\vert}J$. Therefore $G\mathbbm{1} = x_0\mathbbm{1}$ and we must have $\text{rank}(\left[\begin{array}{c|c}
		G & x_0\mathbbm{1}
		\end{array}\right]) = \text{rank}\left(G\right)$ since the last column is the sum of all previous columns. Similarly, we may augment our new matrix with an extra row by adding all previous rows together giving
		\[G' = \left[\begin{array}{c|c}
		G & x_0\mathbbm{1}\\\hline
		x_0\mathbbm{1} & \vert X\vert c_0
		\end{array}\right]\]
		again with $\text{rank}(G') = \text{rank}(G)$. Finally, we may scale the last row and column each by $\sqrt{\vert X\vert x_0}$ without changing the rank, since this equates to multiplying a single vector in the set by a scalar. The resulting matrix is
		\[H = \left[\begin{array}{c|c}
		G & \sqrt{\frac{x_0}{\vert X\vert}}\mathbbm{1}\\\hline
		\sqrt{\frac{x_0}{\vert X\vert}}\mathbbm{1}^T & 1
		\end{array}\right].\]
		If $\sqrt{\frac{x_0}{\vert X\vert}} = \alpha$, then this is the Gram matrix of a set of $\vert X\vert+1$ equiangular vectors in the same dimension. Thus we must check that the coefficient of $E_0$ is $\vert X\vert\alpha^2$. We verify that this holds for every case except $\left<729,56\right>$.
	\end{proof}
	Note that this exceptional case is the only listed case where the relative bound is not an integer. That is, \[\frac{337(1-\frac{1}{25^2})}{1-\frac{337}{25^2}} = \frac{4381}{6}\approx 730.167.\] Since the bound concerns the maximum cardinality of a set, this results in an upper bound of 730. However even a set with 730 vectors in dimension 25 would not make this bound sharp --- this is likely why this case fails.  
	\section{Gegenbauer polynomials}\label{gegdef}
	In all that follows let $m$ be a fixed positive integer and define $\scR:=\bbR[x_1,\dots,x_m]$. A \emph{monomial}\index{polynomial!monomial} is defined as a (possibly empty) product of variables $x_1,\dots,x_m$ and given a monomial $t =\prod_{i=1}^{m}x_i^{d_i}$ ($d_i\in \bbZ^+$), the \emph{degree}\index{polynomial!degree} of $t$ is defined as $\text{deg}(t) = \sum_i d_i$. A polynomial $f\in\scR$ may be represented uniquely as a (finite) linear combination of distinct monomials $f =\sum_i\alpha_it_i$ and $\text{deg}(f) = \max\left\{\text{deg}(t_i)\right\}$. For each variable $x_j$, we define the \emph{derivative with respect to $x_j$}\index{polynomial!derivative} of a monomial as $\frac{\partial}{\partial x_j}t = \frac{d_j}{x_j}t$ and extend the definition linearly for any polynomial in $\scR$. For $f\in \scR$, $f$ is \emph{homogeneous}\index{polynomial!homogeneous} if there exists some constant $d\in\bbZ^+$ such that $\text{deg}(t)=d$ for every monomial $t$ in $f$. Further, $f$ is \emph{harmonic}\index{polynomial!harmonic} if $\Delta f = \sum_i \left(\frac{\partial}{\partial x_i}\circ\frac{\partial}{\partial x_i}\right)\left(f\right) = 0$. For $f\in\scR$ the principal ideal generated by $f$ is $(f) = \left\{gf:g\in\scR\right\}$. We denote the cosets of this ideal by $[g]_f = \left\{h\in\scR: g-h\in(f)\right\}$ where we suppress the subscript if it is clear from the context. The quotient ring $\bigslant{\scR}{(f)}$ is, of course, the set of equivalence classes $\left\{[g]_f:g\in\scR\right\}$ with the obvious operations. Let $S^{m-1}\subset \bbR^m$ be the $(m-1)$-dimensional sphere; we define the set of polynomials on the sphere as
	\[\text{Pol}(S^{m-1}):=\bigslant{\scR}{\displaystyle{\left(1-\sum_ix_i^2\right)}}.\]
	Thus, using $f = 1-\sum_{i}x_i^2$, we say a polynomial $h\in\scR$ is \emph{harmonic on the sphere}\index{polynomial!harmonic!on the sphere} if there exists a harmonic polynomial $g\in[h]_f$; we similarly define \emph{homogeneous on the sphere}\index{polynomial!homogeneous!on the sphere}. Finally, we say a polynomial $h\in\scR$ is \emph{zonal}\index{polynomial!zonal} if there exists a vector $a\in \bbR^m$ and a single-variable polynomial $p(t)\in\bbR[t]$ such that $h(x) = p(\left<a,x\right>)$ for all $x\in S^{m-1}$. Note that since $g\in \left(1-\sum_ix_i^2\right)$ implies $g(x) = 0$ for all $x\in S^{m-1}$, this condition is independent of the representative chosen from the equivalence class.\par
	We now introduce a set of orthogonal polynomials arising from the context of spherical harmonics. The Gegenbauer polynomials in dimension $m$ are defined using the three-term recurrence:
	\begin{equation}\label{recurrence}
	Q_\ell^m(t) = \frac{(2\ell+m-4)tQ^m_{\ell-1}(t) - (\ell-1)Q^m_{\ell-2}(t)}{\ell+m-3} \qquad \ell\geq 2,
	\end{equation}
	\begin{equation*}
	Q^m_0(t) = 1\qquad Q^m_1(t) = t.
	\end{equation*}\index{Gegenbauer polynomials}
	Note that $Q^m_\ell(1) = 1$ for all $k\geq 0$. We will suppress the superscript $m$ if it is clear in the context. Below we list the first six Gegenbauer polynomials and plot $Q_1(t)$ through $Q_5(t)$ along with their roots for $m=10$.
	\[
	Q_0(t)=1,\qquad
	Q_1(t)=t,\qquad
	Q_2(t)=\frac{mt^2 - 1}{m-1},\qquad
	Q_3(t)=\frac{(m+2)t^3 - 3t}{m-1},\qquad\]\[
	Q_4(t)=\frac{(m+4)(m+2)t^4 - 6(m+2)t^2+3}{m^2-1},\]\[
	Q_5(t)=\frac{(m+6)(m+4)t^5-10(m+4)t^3+15t}{m^2-1}.\]
	\begin{figure}[!ht]
		\begin{center}
			\includegraphics[scale=.35]{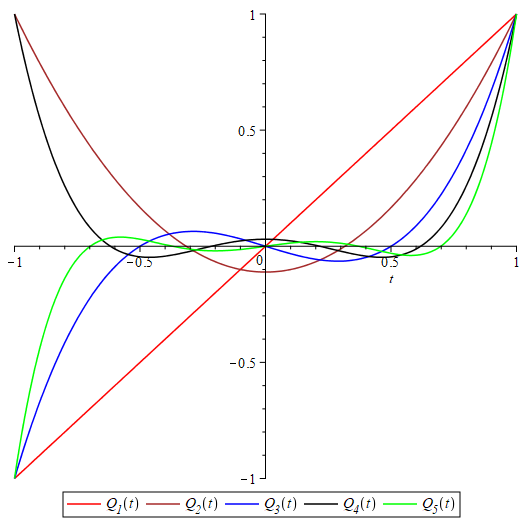}
			\caption[Gegenbauer polynomials]{Gegenbauer polynomials with degree 1 through degree 5 with $m=10$.}\label{gegpic}
		\end{center}
	\end{figure}
	\begin{figure}[!ht]
		\begin{center}
			\includegraphics[scale=.35]{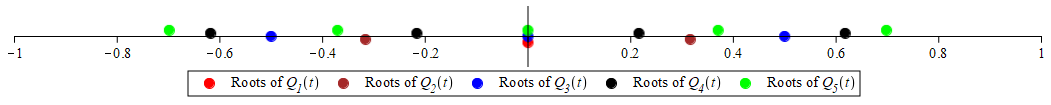}
			\caption[Roots of Gegenbauer polynomials]{Roots of the five Gegenbauer polynomials with degrees $1$ to $5$.}\label{rootpic}
		\end{center}
	\end{figure}
	\begin{thm}\label{gegprop}
		For each $m,\ell\in\bbZ^+$ and $a\in\bbR^m$, the polynomial $Q_\ell^m\left(\left<a,x\right>\right)$ is zonal and both homogeneous and harmonic on the sphere.
	\end{thm}
	\begin{proof}
		The zonal condition is satisfied by construction. The other two conditions follow from considering $F^m_\ell(x)\in\left[Q_\ell^m(\left<a,x\right>)\right]_{1-\sum x_i^2}$ where 
		\[F_\ell^m(x)= \vert\vert x\vert\vert^\ell Q_\ell^m\left(\frac{\left<a,x\right>}{\vert\vert x\vert\vert}\right).\]
		Note that since our ideal is generated by $1-\sum x_i^2$, we find $\vert\vert x\vert\vert^2\equiv 1$.
	\end{proof}
	We note that, up to scaling and rotation of the sphere, $F_\ell^m(x)$ as defined above is the unique degree $\ell$ polynomial which is zonal and both homogeneous and harmonic on the sphere. These polynomials have played an important role in understanding spherical $s$-distance sets and $t$-designs (\cite{Delsarte1977},\cite{Suda2011}). The main results of \cite{Delsarte1977} come from considering a finite set $X\subset\mathbb{S}^m$ and a basis $q_1,\dots, q_M$ of the harmonic polynomial functions on the sphere with fixed degree $k$. We then map $X$ to $\mathbb{R}^M$ by evaluating each basis polynomial at every point. Fixing two points, $\zeta,\xi\in X$, we then use Theorem \ref{gegprop} to show that $\displaystyle{\sum_{i=1}^M\left<q_i(\zeta),q_i(\xi)\right> = c_kQ^m_k\left(\left<\zeta,\xi\right>\right)}$ where the constant $c_k$ does not depend on the points chosen. It should not be surprising that these polynomials were of interest before their use in combinatorics \cite{Delsarte1977}. In fact, 35 years earlier, Sch\"{o}nberg characterized positive definite functions on the sphere using these same polynomials.
	\section{Sch\"{o}nberg's theorem}
	Let $m$ be a fixed positive integer. For any finite set of unit vectors $X\subset S^{m-1}$, let $G_X$ denote the Gram matrix of $X$; then $G_X$ is positive semidefinite ($G_X\succeq 0$). A function $f:[-1,1]\rightarrow\mathbb{R}$ is \textit{positive definite}\index{positive definite function} on $S^{m-1}$ if, for every finite subset $X$, $f$ applied entrywise to $G_X$ results in a positive semidefinite matrix, i.e.\ $f\circ(G_X)\succeq 0$. Here we present Sch\"{o}nberg's Theorem \cite{Schoenberg1942} as it applies to polynomials.
	\begin{thm}[Sch\"{o}nberg \cite{Schoenberg1942}]\label{schoenthm}
		Fix $m\in\mathbb{Z}^+$. A polynomial $f:[-1,1]\rightarrow\mathbb{R}$ with degree $d$ is positive definite on $S^{m-1}$ if and only if $f(t) = \sum_{\ell=0}^d c_\ell Q_\ell^m(t)$ for non-negative constants $c_\ell$.\qed
	\end{thm}
	In particular, $Q_\ell^m(t)$ is a positive definite function for any choice of $m$ and $\ell$. This leads to the following theorem:
	\schAS
	\begin{proof}
		Since $\BMA =\text{span}\left\{E_0,\dots,E_d\right\}$ is closed under entrywise products, we must have $Q_k^{m_i}\circ\left(\frac{\vert X\vert}{m_i}E_i\right)\in \BMA$ and we may write $Q_k^{m_i}\circ\left(\frac{\vert X\vert}{m_i}E_i\right) = \sum_j \theta_{kj} E_j$ where each $\theta_{kj}$ is an eigenvalue of $Q_k^{m_i}\circ\left(\frac{\vert X\vert}{m_i}E_i\right)$. Now recall the algebra isomorphism $\phi^*:\BMA\rightarrow\bbL^*$ in \eqref{algiso} mapping entrywise products to standard matrix products for which $\phi^*\left(E_j\right)=\frac{1}{\vert X\vert}L_j^*$. Applying $\phi^*$ to both sides of the equation on the left in \eqref{ELGeg} gives the equation on the right in \eqref{ELGeg}. Finally, since $E_i$ is an idempotent matrix with constant main diagonal entries given by $\frac{1}{\vert X\vert}Q_{0i} = \frac{m_i}{\vert X\vert}$, we know that $\frac{\vert X\vert}{m_i}E_i$ is the Gram matrix of a set of points in $S^{m_i-1}$. Therefore, Theorem~\ref{schoenthm} tells us that $Q_k^{m_i}\circ\left(\frac{\vert X\vert}{m_i}E_i\right)$ must be positive semidefinite.
	\end{proof}
	While Sch\"{o}nberg's condition $\displaystyle{Q^{\text{rank }G}_\ell\circ(G)\succeq 0}$ is a statement about Gram matrices, Theorem \ref{schoen-as} provides us with an equivalent statement about parameters. So we have a new feasibility condition for parameter sets.
	\begin{cor}\label{schoenL}
		If a parameter set $\left\{q_{ij}^k\right\}_{i,j,k}$ is realizable then, for all $\ell\geq 0$ and all $0\leq i\leq d$, $Q^{m_i}_\ell\left(\frac{1}{m_i}L_i^*\right)\geq 0$.
	\end{cor}
	\begin{proof}
		This follows directly from Theorem \ref{schoen-as}, noting that the Krein conditions imply each $L_j^*$ is non-negative.
	\end{proof}
	We note that, Lemma \ref{kitchensink} \emph{($\mathit{i'}$)} tells us that $q^j_{i0} = \delta_{i,j}$ for $0\leq i\leq d$ and thus Equation $\eqref{ELGeg}$ gives us that the first column of $Q^{m_i}_\ell\left(\frac{1}{m_i}L_i^*\right)$ is \[\frac{1}{\vert X\vert}\sum_j\theta_{\ell j}\vec{e}_j=\left[\begin{array}{c}
	\bigslant{\theta_{\ell0}}{\vert X\vert}\\\vdots\\\bigslant{\theta_{\ell d}}{\vert X\vert}
	\end{array}\right].\]
	Therefore it is sufficient to check that the first column of $Q^{m_i}_\ell\left(\frac{1}{m_i}L_i^*\right)$ is non-negative.
	\subsection{Bound on new feasibility conditions}\label{proof}
	In order to identify where Corollary \ref{schoenL} might have impact, we work in this section to show that, for $\ell$ sufficiently large --- a bound we give entirely in terms of the Krein parameters --- the condition automatically holds. We will do this by first simplifying Equation \eqref{ELGeg} into a vector equation using the observation made after Corollary \ref{schoenL}. From this equation, we derive a three-term recurrence of non-negative vectors and transform our vector space via an invertible transformation which changes our transition matrix from $\frac{1}{m}L_i^*$ to one which is orthogonally diagonalizable. We then map our vector space to one with twice the dimension where our three-term recurrence may be represented as a linear recurrence. Using this linear recurrence we express our initial vector as a sum of steady state vectors and transient vectors in order to find a bound on how quickly the transient vectors must decrease in norm. In all that follows suppose we have a feasible parameter set with first and second eigenmatrix $P$ and $Q$. For fixed $0\leq i\leq d$, let $L_i^*$ be the $i^\text{th}$ matrix of Krein parameters. Let $\ell\geq0$ be an integer and assume $m_i:=q_{0i}^i>2$. Let $\vec{c}_\ell$ be the first column of the matrix $Q_{\ell}^{m_i}\left(\frac{1}{m_i}L_i^*\right)$. Using our recurrence relation \eqref{recurrence} we have,
	\begin{equation}\label{crec}
	\vec{c}_\ell = \frac{(2\ell+m_i-4)\frac{1}{m_i}L_i^*\vec{c}_{\ell-1} - (\ell-1)\vec{c}_{\ell-2}}{\ell+m-3} \qquad \ell\geq 2.
	\end{equation}
	Since $Q_0^m\left(\frac{1}{m_i}L_i^*\right) = I$ and $Q_1^m\left(\frac{1}{m_i}L^*_i\right) = \frac{1}{m_i}L_i$, we have
	\begin{equation*}
	\vec{c}_0 = \vec{e}_0\qquad \vec{c}_1 = \frac{1}{m_i}\vec{e}_i.
	\end{equation*}
	One immediate difficulty is that our transition matrix $L_i^*$ is not symmetric, and thus does not have orthogonal eigenvectors. We will need this property shortly, so we make our first transformation via $\vec{b}_\ell = \sqrt{\Delta_m}\vec{c}_\ell$ where $\sqrt{\Delta_m}$ is the diagonal matrix with $i^\text{th}$ diagonal entry $\sqrt{m_i}$. This transformation turns our recurrence relation into
	\begin{equation}\label{brec}
	\vec{b}_\ell = \frac{(2\ell+m_i-4)M\vec{b}_{\ell-1} - (\ell-1)\vec{b}_{\ell-2}}{\ell+m_i-3} \qquad \ell\geq 2,
	\end{equation}
	where $M = \frac{1}{m_i}\sqrt{\Delta_m}L_i^*\left(\sqrt{\Delta_m}\right)^{-1}$. For this three-term recurrence, our initial vectors are
	\begin{equation*}
	\vec{b}_0 = \vec{e}_0\qquad \vec{b}_1 = \frac{1}{\sqrt{m_i}}\vec{e}_i.
	\end{equation*}
	The new transition matrix $M$ is symmetric and we may easily calculate the eigenvalues and eigenvectors of $M$ via the following lemma.
	\begin{lem}\label{Mmat}
		Let $M = \frac{1}{m_i}\sqrt{\Delta_m}L_i^*\left(\sqrt{\Delta_m}\right)^{-1}$. Then the set $\left\{\vec{p}_0,\vec{p}_1,\dots,\vec{p}_d\right\}$ with
		\[\vec{p}_j = \left[\begin{array}{c}
		P_{0j}\\
		\sqrt{m_1}P_{1j}\\
		\vdots\\
		\sqrt{m_d}P_{dj}\\
		\end{array}\right]\]
		is an orthogonal set of eigenvectors for $M$ with eigenvalues $m_i^{-1}Q_{ji}$ $(0\leq j\leq d)$.
	\end{lem}
	\begin{proof}
		Note that Lemma \ref{kitchensink} \emph{($\mathit{viii'}$)} tells us that the columns of $P$ are eigenvectors of $L_i^*$ with eigenvalues $Q_{ji}$ ($0\leq j\leq d$). Conjugating $L_i^*$ by $\sqrt{\Delta_m}$ results in a matrix with eigenvectors given by the columns of $\sqrt{\Delta_m}P$ with the same eigenvalues. Further, scaling by $\frac{1}{m_i}$ leaves the eigenvectors unchanged but scales the eigenvalues by $\frac{1}{m_i}$. Finally, recall our orthogonality relations (Lemma \ref{orthorels}): $\Delta_m P = Q^T \Delta_k$ and $PQ = \vert X\vert I$. Therefore
		\[\left(\sqrt{\Delta_m}P\right)^T\left(\sqrt{\Delta_m}P\right) = P^T\Delta_mP = P^TQ^T\Delta_k = \vert X\vert\Delta_k.\qedhere\]
	\end{proof}
	Now that our transition matrix has an orthogonal set of eigenvectors, we double the dimension of our vector space in order to make our recurrence relation linear. For $\ell\geq 1$, let 
	\begin{equation}\label{tkdef}\mu_\ell = \frac{\ell-1}{\ell+m_i-3};\qquad \vec{y}_\ell = \left[\begin{array}{c}
	\vec{b}_{\ell}\\\hdashline[2pt/2pt]
	\vec{b}_{\ell-1}
	\end{array}\right];\qquad T_\ell = \left[\begin{array}{>{\centering\arraybackslash}p{2cm};{2pt/2pt}>{\centering\arraybackslash}p{2cm}}
	$(1+\mu_\ell)M$ & $-\mu_\ell I$\\\hdashline[2pt/2pt]
	$I$ & $0$
	\end{array}\right].\end{equation}
	Then we have
	\begin{equation}\label{yrec}
	\vec{y}_\ell = T_\ell\vec{y}_{\ell-1};\qquad \vec{y}_1 = \left[\begin{array}{c}
	\frac{1}{\sqrt{m_i}}\vec{e}_{i}\\\hdashline[2pt/2pt]
	\vec{e}_{0}
	\end{array}\right].
	\end{equation}
	Although $T_\ell$ depends on $\ell$, we may identify some common properties among the eigenvalues and eigenvectors of $T_\ell$ for all values of $\ell>0$. The next few lemmas describe the eigenspaces of $T_\ell$ as well as the action of $T_\ell$ on linear subspaces of our vector space containing our initial vector. Let $\vec{p}_j$ be the $j^\text{th}$ column of $\sqrt{\Delta_m}P$ and define the $d+1$ 2-dim subspaces $B_j = \text{span}\left\{\left[\begin{array}{c}
	\vec{p}\\\hdashline[2pt/2pt]
	0
	\end{array}\right],\left[\begin{array}{c}
	0\\\hdashline[2pt/2pt]
	\vec{p}
	\end{array}\right]\right\}$ which are clearly pairwise orthogonal.
	\begin{lem}\label{Tkeig}
		Let $(\lambda,\vec{p})$ be an eigenpair of $M$. Let $\left\{\eta^+,\eta^-\right\}$ be the two roots of the quadratic polynomial $x^2-(1+\mu_\ell)\lambda x +\mu_\ell$. Then $\left(\eta^+,\left[\begin{array}{c}
		\eta^+\vec{p}\\\hdashline[2pt/2pt]
		\vec{p}
		\end{array}\right]\right)$ is an eigenpair of $T_\ell$. If $\eta^+\neq \eta^-$ then $ \left(\eta^-,\left[\begin{array}{c}
		\eta^-\vec{p}\\\hdashline[2pt/2pt]
		\vec{p}
		\end{array}\right]\right)$ is also an eigenpair, otherwise $\left[\begin{array}{c}
		\vec{p}\\\hdashline[2pt/2pt]
		\vec{0}
		\end{array}\right]$ is a generalized eigenvector of order two with eigenvalue $\eta^+$.
		
	\end{lem}
	\begin{proof}
		Let $M \vec{p} = \lambda\vec{p}$ and let $\eta$ be given so that $\eta^2-(1+\mu_\ell)\lambda\eta +\mu_\ell=0$. Then
		\[T_\ell\left[\begin{array}{c}
		\eta \vec{p}\\\hdashline[2pt/2pt]
		\vec{p}
		\end{array}\right] = \left[\begin{array}{c}
		(1+\mu_\ell)M(\eta \vec{p}) - \mu_\ell \vec{p}\\\hdashline[2pt/2pt]
		\eta \vec{p}
		\end{array}\right] = \left[\begin{array}{c}
		\left((1+\mu_\ell)\lambda\eta - \mu_\ell\right) \vec{p}\\\hdashline[2pt/2pt]
		\eta \vec{p}
		\end{array}\right] =  \left[\begin{array}{c}
		\eta^2 \vec{p}\\\hdashline[2pt/2pt]
		\eta \vec{p}
		\end{array}\right].\]
		If we have two distinct roots $(\eta^+\neq \eta^-)$ then $\left[\begin{array}{c}
		\eta^+ \vec{p}\\\hdashline[2pt/2pt]
		\vec{p}
		\end{array}\right]$ and $\left[\begin{array}{c}
		\eta^- \vec{p}\\\hdashline[2pt/2pt]
		\vec{p}
		\end{array}\right]$ are linearly independent eigenvectors. If instead we find that $\eta^+=\eta^-=\eta$, then $\eta = \frac{(1+\mu_\ell)\lambda}{2}$ since $x^2-bx+c = (x-\eta)^2$ implies $\eta = \frac{b}{2}$. In this case,
		\[T_\ell\left[\begin{array}{c}
		\vec{p}\\\hdashline[2pt/2pt]
		\vec{0}
		\end{array}\right] = \left[\begin{array}{c}
		(1+\mu_\ell)M\vec{p}\\\hdashline[2pt/2pt]
		\vec{p}
		\end{array}\right]= \left[\begin{array}{c}
		(1+\mu_\ell)\lambda\vec{p}\\\hdashline[2pt/2pt]
		\vec{p}
		\end{array}\right] = \left[\begin{array}{c}
		\eta \vec{p}\\\hdashline[2pt/2pt]
		\vec{0}
		\end{array}\right]+\left[\begin{array}{c}
		\eta \vec{p}\\\hdashline[2pt/2pt]
		\vec{p}
		\end{array}\right].\]
		Thus $(T_\ell-\eta I)\left[\begin{array}{c}
		\vec{p}\\\hdashline[2pt/2pt]
		\vec{0}
		\end{array}\right] = \left[\begin{array}{c}
		\eta \vec{p}\\\hdashline[2pt/2pt]
		\vec{p}
		\end{array}\right]$ giving us that $\left[\begin{array}{c}
		\vec{p}\\\hdashline[2pt/2pt]
		\vec{0}
		\end{array}\right]$ is a generalized eigenvector of order two; that is, $\left[\begin{array}{c}
		\vec{p}\\\hdashline[2pt/2pt]
		\vec{0}
		\end{array}\right]$ is in the kernel of $\left(T_\ell-\eta I\right)^2$ but not $\left(T_\ell-\eta I\right)$.
	\end{proof}
	Lemma \ref{Tkeig} gives a complete description of the eigenspaces of our transition matrix $T_\ell$ since we know $M$ is diagonalizable. Further, since the eigenvectors of $M$ are orthogonal, the pairwise orthogonal $B_j$ are either full generalized eigenspaces or the sum of two one-dimensional eigenspaces. In either case, each $B_j$ is $T_\ell$-invariant for all $\ell\geq 0$. This alone is not sufficient to find our bound as we seek to show that the successive actions of $T_\ell$ on our initial vector maps it close enough to some steady state vector to guarantee all entries are non-negative. Thus we must both find the steady state vector as well as determine the action of $T_\ell$ on the projections of our initial vector on each $B_j$. We note however that whenever our eigenvalues are real, the singular values of $T_\ell$ may grow quite large. Thus, instead of bounding the action of $T_\ell$ on each $B_j$ as a whole, we will bound a region of $B_j$ containing our projections and show that $T_\ell$ acts on that region in a consistent manner.
	\begin{lem}\label{realroot}
		Let $\left(\lambda,\vec{p}\right)$ be an eigenpair of $M$ and assume that $x^2-(1+\mu_\ell)\lambda x + \mu_\ell$ has only real roots, $\vert\eta^-\vert\leq\vert\eta^+\vert$. Let $\vec{v} = \left[\begin{array}{c}
		a\vec{p}\\\hdashline[2pt/2pt]
		b\vec{p}
		\end{array}\right]$ with $b\neq 0$. There exists a constant $\beta\in\mathbb{R}$ so that $T_\ell \vec{v} = \left[\begin{array}{c}
		\beta\vec{p}\\\hdashline[2pt/2pt]
		a\vec{p}
		\end{array}\right]$. Further, $ \left(\bigslant{\lambda a}{b}\right)\geq0$ and $\vert\eta^+\vert\leq \vert\frac{a}{b}\vert\leq\vert\lambda\vert$ imply $\left(\bigslant{\lambda\beta}{a}\right)\geq 0$ and $\vert\eta^+\vert\leq \vert\frac{\beta}{a}\vert\leq\vert\lambda\vert$ giving $\vert\vert T_\ell \vec{v}\vert\vert_2\leq \vert \lambda\vert\cdot\vert\vert \vec{v}\vert\vert_2$.
	\end{lem}
	\begin{proof}
		We begin by noting that since our operator $T_\ell$ is linear, we may assume without loss of generality that $b=1$ and $\vert\vert\vec{p}\vert\vert_2 = 1$. This, paired with our constraint that $\lambda^2\geq\bigslant{4\mu_\ell}{(1+\mu_\ell)^2}$ tells us that $a$ and $\lambda$ must share signs. Since $\eta^+$ and $\eta^-$ are the two (not necessarily distinct) roots of the polynomial $x^2-(1+\mu_\ell)\lambda x + \mu_k$, we know that $(x-\eta^-)(x-\eta^+) = x^2-(1+\mu_\ell)\lambda x + \mu_k$ and thus $\eta^++\eta^-=(1+\mu_\ell)\lambda$ and $\eta^+\eta^- = \mu_\ell$. This further implies our two roots $\eta^-$ and $\eta^+$ also share signs with $\lambda$. We first consider the case when $\eta^-\neq \eta^+$. Here, Lemma \ref{Tkeig} tells us that $\eta^+$ and $\eta^-$ are eigenvalues of $T_\ell$ with eigenvectors
		\[\vec{v}^+ = \left[\begin{array}{c}
		\eta^+ \vec{p}\\\hdashline[2pt/2pt]
		\vec{p}
		\end{array}\right];\qquad \vec{v}^-=\left[\begin{array}{c}
		\eta^- \vec{p}\\\hdashline[2pt/2pt]
		\vec{p}
		\end{array}\right].\]
		Noting that $\vec{v} = \left(\frac{\eta^--a}{\eta^--\eta^+}\right)\vec{v}^+ + \left(\frac{a-\eta^+}{\eta^--\eta^+}\right)\vec{v}^-$, we may calculate $T_\ell \vec{v}$ explicitly giving
		\[T_\ell \vec{v} = \left(\frac{\eta^--a}{\eta^--\eta^+}\right)\eta^+\vec{v}^+ + \left(\frac{a-\eta^+}{\eta^--\eta^+}\right)\eta^-\vec{v}^- = \left[\begin{array}{c}
		\beta \vec{p}\\\hdashline[2pt/2pt]
		a\vec{p}
		\end{array}\right]\]
		where $\beta = \left(\frac{\eta^--a}{\eta^--\eta^+}\right)\left(\eta^+\right)^2 + \left(\frac{a-\eta^+}{\eta^--\eta^+}\right)\left(\eta^-\right)^2 = a\left(\eta^++\eta^-\right)-\eta^-\eta^+$. Thus $\frac{\beta}{a} = \eta^+ + \frac{\eta^-}{a}\left(a-\eta^+\right).$
		Since $\eta^-$, $\eta^+,$ and $a$ all share the same sign and $\vert\eta^+\vert\leq\vert a\vert$, $\frac{\eta^-}{a}\left(a-\eta^+\right)$ must also share the same sign as $\eta^+$ providing
		\[\left\vert\eta^+\right\vert\leq  \left\vert\eta^+ + \frac{\eta^-}{a}\left(a-\eta^+\right) \right\vert\leq\left\vert\eta^+ + \left(a-\eta^+\right) \right\vert = \vert a\vert\leq \vert\lambda\vert.\]
		Similarly, consider the case when $\eta^- = \eta^+ = \eta = \frac{(1+\mu_\ell)\lambda}{2}$. Lemma \ref{Tkeig} gives the eigenvector and generalized eigenvector
		\[\vec{w} = \left[\begin{array}{c}
		\eta \vec{p}\\\hdashline[2pt/2pt]
		\vec{p}
		\end{array}\right];\qquad \vec{w}^*=\left[\begin{array}{c}
		\vec{p}\\\hdashline[2pt/2pt]
		\vec{0}
		\end{array}\right]\]
		and we have $\vec{v} = \vec{w} + \left(a-\eta\right)\vec{w}^*$. Then $T_\ell\vec{v} =  a\vec{w} + \eta\left(a-\eta\right)\vec{w}^* = \left[\begin{array}{c}
		(2a-\eta)\eta\vec{p}\\\hdashline[2pt/2pt]
		a\vec{p}
		\end{array}\right].$
		Again taking $\beta$ to be the top coefficient, we find that $\frac{\beta}{a} = \eta+ \frac{\eta}{a}(a-\eta)$ and as before we find that $\frac{\eta}{a}(a-\eta)$ and $\eta$ share the same sign giving
		\[\vert \eta\vert\leq \left\vert \eta+ \frac{\eta}{a}(a-\eta)\right\vert\leq \left\vert \eta+ (a-\eta)\right\vert=\left\vert a\right\vert\leq\left\vert\lambda\right\vert.\]
		Thus in both cases we find that $\vert\eta\vert\leq \left\vert\frac{\beta}{a}\right\vert\leq\vert \lambda\vert$, giving
		\[\vert\vert T_\ell \vec{v}\vert\vert^2_2 =\beta^2 + a^2 = \frac{\beta^2}{a^2}a^2 + a^2\leq \lambda^2a^2+\lambda^2 = \lambda^2\vert\vert \vec{v}\vert\vert_2^2.\] 
		Finally, since $\eta^+$ and $\frac{\eta^-}{a}\left(a-\eta^+\right)$ had the same sign in both cases, we must also have $\frac{\beta}{a}$ share the same sign, forcing $\left(\bigslant{\lambda\beta}{a}\right)\geq 0$.
	\end{proof}
	Controlling the action of $T_\ell$ on any subspace $B_j$ corresponding to complex eigenvalues is much easier.
	\begin{lem}\label{complexroot}
		Let $\left(\lambda,\vec{p}\right)$ be an eigenpair of $M$ with corresponding $T_\ell$-invariant subspace $B_j$. Assume that $(1+\mu_\ell)^2\lambda^2<4\mu_\ell$ and thus $x^2-(1+\mu_\ell)\lambda x + \mu_\ell$ has no real roots. For any vector $\vec{v} \in B_j$, we must have $\vert\vert T_\ell \vec{v}\vert\vert_2=\sqrt{\mu_\ell}\cdot\vert\vert \vec{v}\vert\vert_2$.
	\end{lem}
	\begin{proof}
		Let $\eta^+$ and $\eta^-$ be the two roots of $x^2-(1+\mu_\ell)\lambda x + \mu_\ell$. Since $\eta^+$ and $\eta^-$ are not real, the action of $T_\ell$ on $B_j$ has no real eigenvalues and must act as the product of a rotation matrix and a scalar on the two dimensional real subspace. Thus, every vector in $B_j$ is scaled by the norm of the two (conjugate) eigenvalues. Since $\vert\eta^+\vert = \vert\eta^-\vert = \sqrt{\eta^+\eta^-} = \sqrt{\mu_\ell}$, our result follows.
	\end{proof}
	
	One can imagine applying these two lemmas iteratively in order to bound the norm of our vector after many applications of $T_\ell$ --- we will do exactly this. However, applying Lemma \ref{realroot} requires that the vector $y_\ell$ projected onto $B_j$ is contained within the small subset of vectors described in the lemma; a subset which depends on $\ell$. For the purpose of illustration, fix one such subspace $B_j$ and for each $\ell>0$ let $C_\ell\subset B_j$ be the subset of vectors $\left\{\left[\begin{array}{c}
	a\vec{p}\\\hdashline[2pt/2pt]
	b\vec{p}
	\end{array}\right]\right\}$ such that $\left(\bigslant{\lambda a}{b}\right)\geq0$ and $\vert\eta^+_\ell\vert\leq \vert\frac{a}{b}\vert\leq\vert\lambda\vert$ (we decorate $\eta^+$ with the subscript $\ell$ here only to emphasize the root depends on our choice of $\ell$). We therefore must show that for any $\ell$, $T_\ell C_\ell\subset C_{\ell+1}$. The two complications that arise are, first, that $C_\ell$ changes with $\ell$ and, second, that Lemma \ref{complexroot} gives us no control over where $T_\ell$ maps vectors in $C_\ell$. The following lemma resolves both of these issues, first showing that the roots change in a predictable way, allowing for the conditions at the end of Lemma \ref{realroot} to be sufficient to guarantee $T_\ell C_\ell\subset C_{\ell+1}$ whenever $\eta^+_{\ell+1}$ is real and then showing that if we ever need to apply Lemma \ref{complexroot}, we will not have to apply Lemma \ref{realroot} to that subspace again.
	\begin{lem}\label{rootshift}
		Let $\ell\in\bbZ^+$ and $-1<\lambda< 1$ be given. Assume that $x^2-(1+\mu_{\ell+1})\lambda x+\mu_{\ell+1}$ has real roots $x_{\ell+1}^-\leq x_{\ell+1}^+$. Then $x^2-(1+\mu_{\ell})\lambda x+\mu_{\ell}$ also has real roots $x_{\ell}^-$ and $x_{\ell}^+$ with $x_{\ell}^-<x_{\ell+1}^-\leq x_{\ell+1}^+< x_{\ell}^+$.
	\end{lem}
	\begin{proof}
		Fix $-1<\lambda<1$ and define the multivariate polynomial $p(x,\mu) = x^2-(1+\mu)\lambda x+\mu$ with domain $-1\leq x,\mu\leq 1$. For fixed $\mu'$, $p(x,\mu')=0$ has real solutions if and only if $\mu'\leq 2-\lambda\pm2\sqrt{1-\lambda}$ since $\mu'\leq1$. Fix $m>2$, $\ell\geq 1$ and define $\mu_{\ell} = \frac{\ell -1}{\ell+m-3}$. Assume $p(x,\mu_{\ell+1})$ has real solutions, then $p(x,\mu_{\ell})$ must also have real solutions since $\mu_\ell<\mu_{\ell+1}\leq 2-\lambda\pm2\sqrt{1-\lambda}$. Let $x_\ell^-\leq x_{\ell}^+$ be the real roots of $p(x,\mu_{\ell})$. Since the leading term of $p(x,\mu_\ell)$ is positive, $p(x,\mu_\ell)<0$ only on the interval $x\in\left(x_\ell^-,x_\ell^+\right)$. Now consider that $\frac{\partial p(x,\mu)}{\partial \mu} = 1-\lambda x>0$ for all values in our domain. Since $\mu_{\ell+1}>\mu_{\ell}$, we must have $p(x,\mu_{\ell+1})>p(x,\mu_{\ell})$ and thus the real solutions of $p(x,\mu_{\ell+1})=0$ must lie strictly between $x_\ell^-$ and $x_\ell^+$.
	\end{proof}
	With this lemma, we know that the largest absolute value of the roots of $x^2-(1+\mu_\ell)\lambda_j + \mu_\ell$ decreases as $\mu_\ell$ increases so long as the roots are real. Further, this lemma also implies that if the roots of $x^2-(1+\mu_\ell)\lambda_j + \mu_\ell$ are not real, then $x^2-(1+\mu_{\ell+1})\lambda_j + \mu_{\ell+1}$ cannot have real roots either. We will use these two facts as well as the proceeding lemmas to prove our next lemma. We will use this lemma again in a later section, so we will make it self-contained.
	\begin{lem}\label{normbound}
		Let $P$ and $Q$ be the first and second eigenmatrices of a feasible parameter set for an association scheme. Fix $0\leq i\leq d$. For $0\leq j\leq d$ define $\lambda_j = \bigslant{Q_{j,i}}{m_i}$ and $B_j =\left\{\left[\begin{array}{c}
		a\vec{p}_j\\\hdashline[2pt/2pt]
		b\vec{p}_j
		\end{array}\right]\mid a,b\in\bbR\right\}$ where $\vec{p}_j$ is the $j^\text{th}$ column of $\sqrt{\Delta_m}P$. For $\ell\geq 1$ let $\vec{c}_\ell$ be the first column of $Q^{m_i}_\ell\left(\frac{1}{m_i}L_i^*\right)$ and define
		\[\mu_\ell = \frac{\ell-1}{\ell+m_i-3};\quad \vec{y}_\ell = \left[\begin{array}{c}
		\sqrt{\Delta_m}\vec{c}_{\ell}\\\hdashline[2pt/2pt]
		\sqrt{\Delta_m}\vec{c}_{\ell-1}
		\end{array}\right];\quad \gamma_{\ell,j} = \begin{cases}
		\lambda_j^2 &\text{ if } (1+\mu_\ell)^2\lambda_j^2\geq 4\mu_\ell,\\
		\mu_\ell &\text{ if } (1+\mu_\ell)^2\lambda_j^2< 4\mu_\ell.
		\end{cases}\]
		Then for any $0\leq j\leq d$ and $\ell^*>0$, \[\text{proj}_{B_j}(\vec{y}_1) =\frac{1}{\vert X\vert}\left[\begin{array}{c}
		\lambda_j\vec{p}_j\\\hdashline[2pt/2pt]
		\vec{p}_j
		\end{array}\right];\qquad\left\vert\left\vert \text{proj}_{B_j}(\vec{y}_{\ell^*+1})\right\vert\right\vert^2_2\leq(1+\lambda_j^2)\frac{k_j}{\vert X\vert}\left(\prod_{\ell=2}^{\ell^*}\gamma_{\ell,j}\right).\]
	\end{lem}
	\begin{proof}
		We have already seen that $\vec{y}_1= \left[\begin{array}{c}
		\frac{1}{\sqrt{m_i}}\vec{e}_{i}\\\hdashline[2pt/2pt]
		\vec{e}_{0}
		\end{array}\right]$. We may use $\left\{\left[\begin{array}{c}
		\vec{p}_j\\\hdashline[2pt/2pt]
		\vec{0}
		\end{array}\right],\left[\begin{array}{c}
		\vec{0}\\\hdashline[2pt/2pt]
		\vec{p}_j
		\end{array}\right]\right\}$ as an orthogonal basis for $B_j$, noting that $\left(\sqrt{\Delta_m}P\right)^T\left(\sqrt{\Delta_m}P\right) = \vert X\vert\Delta_k$ and thus $\vert\vert \vec{p}_j\vert\vert_2^2 = \vert X\vert k_j$. This gives 
		\[\text{proj}_{B_j}(\vec{y}_1) = \frac{1}{\vert X\vert k_j}\left[\begin{array}{c}
		P_{i,j}\vec{p}_j\\\hdashline[2pt/2pt]
		P_{0,j}\vec{p}_j
		\end{array}\right] =\frac{1}{\vert X\vert}\left[\begin{array}{c}
		\frac{P_{i,j}}{k_j}\vec{p}_j\\\hdashline[2pt/2pt]
		\vec{p}_j
		\end{array}\right] = \frac{1}{\vert X\vert}\left[\begin{array}{c}
		\lambda_j\vec{p}_j\\\hdashline[2pt/2pt]
		\vec{p}_j
		\end{array}\right].\]
		We split the remaining part of the proof into two sections, first dealing with the integers $\ell$ for which $x^2-(1+\mu_\ell)\lambda_j+\mu_\ell$ has real roots, and then those integers for which no real roots exist. In both cases we will use induction, though the induction steps are slightly different between the real and non-real cases. We begin by noting that since each $B_j$ is $T_\ell$-invariant, $\text{proj}_{B_j}(\vec{y}_{\ell+1}) = \left(\prod_{x=2}^{\ell} T_x\right)\text{proj}_{B_j}(\vec{y}_1)$, thus we will use these two expressions interchangeably. Now, Lemma \ref{rootshift} guarantees that the polynomial $x^2-(1+\mu_\ell)\lambda_j+\mu_\ell$ has no real roots only when every polynomial $x^2-(1+\mu_{\ell'})\lambda_j+\mu_{\ell'}$ for each $\ell'>\ell$ also has no real roots. Therefore let $h$ be the largest integer less than or equal to $\ell^*$ for which $x^2-(1+\mu_h)\lambda_j+\mu_h$ has real roots. We seek to prove that $\text{proj}_{B_j}(y_\ell)$ is within the region for which Lemma \ref{realroot} applies for each $2\leq \ell\leq h$. For the sake of induction, fix $2\leq \ell\leq h$, let $\eta_\ell$ be the largest root (in absolute value) of $x^2-(1+\mu_\ell)\lambda_j+\mu_\ell$, and assume that there exist constants $a_\ell$ and $b_\ell$ such that $\text{proj}_{B_j}\left(\vec{y}_{\ell-1}\right) = \left[\begin{array}{c}
		a_\ell\vec{p}_j\\\hdashline[2pt/2pt]
		b_\ell\vec{p}_j
		\end{array}\right]$ where $\left(\bigslant{\lambda_ja_\ell}{b_\ell}\right) \geq 0$ and $\vert \eta_\ell\vert\leq\vert \frac{a_\ell}{b_\ell}\vert\leq\vert\lambda_j\vert$. Then Lemma \ref{realroot} tells us that there exists $\beta\in\mathbb{R}$ such that $\text{proj}_{B_j}\left(\vec{y}_\ell\right) = T_\ell\left(\text{proj}_{B_j}\left(\vec{y}_{\ell-1}\right)\right) =\left[\begin{array}{c}
		\beta\vec{p}_j\\\hdashline[2pt/2pt]
		a_\ell\vec{p}_j
		\end{array}\right]$ where
		$\left(\bigslant{\lambda_j\beta}{a}\right)\geq 0$, $\vert\eta_\ell\vert\leq \vert\frac{\beta}{a}\vert\leq\vert\lambda\vert$. Using Lemma $\ref{rootshift}$, we know that as long as $\ell<h$, $\vert\eta_{\ell+1}\vert<\vert\eta_{\ell}\vert$ where $\eta_{\ell+1}$ is the largest root in absolute value of the polynomial $x^2-(1+\mu_{\ell+1})\lambda_j+\mu_{\ell+1}$. Thus, either $\ell=h$ or we may use $a_{\ell+1} = \beta$ and $b_{\ell+1} = a_\ell$ to complete our induction step. Our base case is covered by our earlier observation that 
		\[\text{proj}_{B_j}(\vec{y}_1) = \frac{1}{\vert X\vert}\left[\begin{array}{c}
		\lambda_j\vec{p}_j\\\hdashline[2pt/2pt]
		\vec{p}_j
		\end{array}\right].\]
		Further, Lemma \ref{realroot} also tells us that in each case
		\[\left\vert\left\vert \text{proj}_{B_j}\left(\vec{y}_\ell\right)\right\vert\right\vert_2^2\leq \vert \lambda\vert^2\cdot\left\vert\left\vert \text{proj}_{B_j}\left(\vec{y}_{\ell-1}\right)\right\vert\right\vert_2^2.\]
		We therefore find that
		\[\left\vert\left\vert \text{proj}_{B_j}(\vec{y}_h)\right\vert\right\vert^2_2\leq \left\vert\left\vert \text{proj}_{B_j}\left(\vec{y}_1\right)\right\vert\right\vert^2_2\left(\prod_{\ell=2}^{h}\lambda_j^2\right)= (1+\lambda_j^2)\frac{k_j}{\vert X\vert}\left(\prod_{\ell=2}^{h}\lambda_j^2\right).\]
		Now let $h<\ell\leq \ell^*$, and define $\vec{v} = \text{proj}_{B_j}(\vec{y}_h)$. Then Lemma \ref{complexroot} gives us
		\[\left\vert\left\vert \left(\prod_{\ell=h+1}^{\ell^*} T_\ell\right) \vec{v}\right\vert\right\vert^2_2\leq \left\vert\left\vert \vec{v}\right\vert\right\vert^2_2\prod_{\ell=h+1}^{\ell^*}\mu_\ell\leq (1+\lambda_j^2)\frac{k_j}{\vert X\vert}\left(\prod_{\ell=2}^{h}\lambda_j^2\right)\left(\prod_{\ell=h+1}^{\ell^*}\mu_\ell\right).\qedhere\]
	\end{proof}
	We are now ready to prove the main theorem of this section, using Lemma \ref{normbound} to control the norm of the transient vectors at each iteration.
	\degbnd
	\begin{proof}
		Let $\vec{y}_\ell$ be given by the recurrence relation \eqref{yrec} for $\ell>1$ with $\vec{y}_1= \left[\begin{array}{c}
		\frac{1}{\sqrt{m_i}}\vec{e}_{i}\\\hdashline[2pt/2pt]
		\vec{e}_{0}
		\end{array}\right]$. Observe that the sign of $\left[\vec{y}_\ell\right]_i$ is negative if and only if there is a negative entry in either $\vec{c}_\ell$ or $\vec{c}_{\ell-1}$ as defined in Equation \eqref{crec} and, thus, if and only if one of $Q_{\ell}^{m_i}\left(\frac{1}{m_i}L_i^*\right)$ or $Q_{\ell-1}^{m_i}\left(\frac{1}{m_i}L_i^*\right)$ contains a negative value. As before, define subspaces $B_j = \text{span}\left(\left[\begin{array}{c}
		\vec{p}_j\\\hdashline[2pt/2pt]
		\vec{0}
		\end{array}\right],\left[\begin{array}{c}
		\vec{0}\\\hdashline[2pt/2pt]
		\vec{p}_j
		\end{array}\right]\right)$ where $\vec{p}_j$ is the $j^\text{th}$ column of $\sqrt{\Delta_m}P$. Lemma \ref{normbound} then tells us that
		\[\text{proj}_{B_j}(\vec{y}_1) = \frac{1}{\vert X\vert}\left[\begin{array}{c}
		\lambda_j\vec{p}_i\\\hdashline[2pt/2pt]
		\vec{p}_i
		\end{array}\right] \]
		allowing us to split $\vec{y}_1$ into
		\[\vec{y}_1 = \frac{1}{\vert X\vert}\left[\begin{array}{c}
		\sqrt{\Delta_m}\vec{1}\\\hdashline[2pt/2pt]
		\sqrt{\Delta_m}\vec{1}
		\end{array}\right] + \sum_{j=1}^d\text{proj}_{B_j}(\vec{y}_1).\]
		Noting that the first term is an eigenvector for $T_\ell$ with eigenvalue $1$, we then have
		\[\vec{y}_{\ell^*+1} = \left(\prod_{\ell=2}^{\ell^*+1}T_\ell\right)\vec{y}_1 = \frac{1}{\vert X\vert}\left[\begin{array}{c}
		\sqrt{\Delta_m}\vec{1}\\\hdashline[2pt/2pt]
		\sqrt{\Delta_m}\vec{1}
		\end{array}\right] + \sum_{j=1}^d\text{proj}_{B_j}(\vec{y}_{\ell^*+1}).\]
		Note that the smallest entry of $\frac{1}{\vert X\vert}\left[\begin{array}{c}
		\sqrt{\Delta_m}\vec{1}\\\hdashline[2pt/2pt]
		\sqrt{\Delta_m}\vec{1}
		\end{array}\right]$ is $\frac{1}{\vert X\vert}$ and thus in order for $\vec{y}_{\ell^*+1}$ to have a negative entry, we must have 
		\[\sum_{j=1}^d\left\vert\left\vert \text{proj}_{B_j}(\vec{y}_{\ell^*+1})\right\vert\right\vert^2_2>\frac{1}{\vert X\vert^2}.\]
		Again using Lemma \ref{normbound}, we find
		\[\begin{aligned}\sum_{j=1}^d\left\vert\left\vert \text{proj}_{B_j}(\vec{y}_{\ell^*+1})\right\vert\right\vert^2_2 \leq \frac{1}{\vert X\vert}\sum_{j=1}^d v_j(1+\lambda_j^2)\left(\prod_{i=2}^{\ell^*+1}\gamma_{\ell,j}\right).
		\end{aligned}\]
		Therefore, as long as $\sum_{j=1}^d v_j\lambda_j^{2\ell^*}(1+\lambda_j^2)\leq\frac{1}{\vert X\vert}$, we guarantee $\vec{y}_{\ell^*+1}$ has no negative entry. Additionally, since the bottom half of $\vec{y}_{\ell^*+1}$ equals the top half of $\vec{y}_{\ell^*}$, we may assume that any negative entry of $\vec{y}_{\ell^*}$ appears in the bottom half, implying that $\vec{y}_{\ell^*-1}$ also has a negative entry. Thus checking the first $\ell^*-1$ Gegenbauer polynomials is sufficient.
	\end{proof}
	\begin{cor}\label{maxlambda}
		Suppose we have a feasible parameter set for an association scheme with first and second eigenmatrices $P$ and $Q$. Let $0\leq i\leq d$ be given and define $\lambda_j:=\frac{Q_{j,i}}{m_i}$. Further assume $1=\lambda_0>\vert\lambda^*\vert\geq\vert\lambda_j\vert$ for $0<j\leq d$. Then define
		\[\ell^*=\left\lceil\frac{\ln\left[(1+(\lambda^*)^2)\vert X\vert(\vert X\vert-1)\right]}{-2\ln(\lambda^*)}\right\rceil.\]
		If $\vert\lambda^*\vert^2\geq \frac{\ell^*}{\ell^*+m_i-2}$ then $Q^{m_i}_{\ell}\left(\frac{1}{m_i}L_i^*\right)\geq 0$ for any $\ell\geq \ell^*$.
	\end{cor}
	\begin{proof}
		As long as $\vert\lambda^*\vert^2$ is greater than both $\vert\lambda_j\vert^2$ and $\mu_{\ell}$ for $2\leq\ell\leq x$,
		\[\sum_{j=1}^d k_j\left(1+\lambda_j^2\right)\left(\prod_{\ell=2}^{x} \gamma_{x,j}\right)\leq(\lambda^*)^{2x-2}\left(1+(\lambda^*)^2\right)\sum_{j=1}^d k_j =(\lambda^*)^{2x-2}\left(1+(\lambda^*)^2\right)(\vert X\vert -1). \]
		Solving 
		\begin{equation}\label{ineq}(\lambda^*)^{2x-2}\left(1+(\lambda^*)^2\right)(\vert X\vert -1)\leq\frac{1}{\vert X\vert}\end{equation}
		for $x$ gives
		\[x\geq\frac{\ln\left[(1+(\lambda^*)^2)\vert X\vert(\vert X\vert-1)\right]}{-2\ln(\lambda^*)}+1.\]
		Thus for $\ell^*=\left\lceil\frac{\ln\left[(1+(\lambda^*)^2)\vert X\vert(\vert X\vert-1)\right]}{-2\ln(\lambda^*)}\right\rceil$, $\ell^*+1$ is the smallest integer for which Inequality \eqref{ineq} holds. Using Theorem \ref{degreebound}, as long as $\vert\lambda^*\vert^2\geq\mu_{\ell^*+1}$, we have our result.
	\end{proof}
	\begin{example}
		We now give an example of a feasible parameter set which fails to be realizable due to Theorem \ref{schoenL}. We will also use Theorem \ref{degreebound} to specify which Gegenbauer polynomials we will apply. Consider the feasible parameter set with first and second eigenmatrices
		\[P = \left[\begin{array}{rrrr}
		1 & 100 & 240 & 100\\
		1 & 37 & -12 & -26\\
		1 & 2 & -12 & 9\\
		1 & -5 & 9 & -5
		\end{array}\right],\qquad Q = \left[\begin{array}{rrrr}
		1 & 20 & 180 & 240\\
		1 & \nicefrac{37}{5} & \nicefrac{18}{5} & -12\\
		1 & -1 & -9 & 9\\
		1 & \nicefrac{-26}{5} & \nicefrac{91}{5} & -12
		\end{array}\right].\]
		If realizable, an association scheme with this parameter set would be a 3-class primitive $Q$-polynomial association scheme. We will apply these two theorems to the parameters of the first idempotent in the $Q$-polynomial ordering --- this idempotent, if realizable, would correspond to a spherical 3-distance set in $\mathbb{R}^20$. Using column $1$ of $Q$, we find $\lambda_0 = 1$, $\lambda_1 = \frac{37}{100}$, $\lambda_2 = -\frac{1}{20}$, and $\lambda_3 = -\frac{26}{100}$ where $\lambda_j = \nicefrac{Q_{j1}}{Q_{01}}$. Let $\mu_\ell = \frac{\ell-1}{\ell+m-3}$ and, for $1\leq j\leq d$, define $\ell_j$ as the largest integer for which $\frac{\ell_j-1}{\ell_j+m-3}\leq2-\vert\lambda_j\vert-2\sqrt{1-\vert\lambda_j\vert}$. Then we find $\ell_1=\ell_2=\ell_3=1$ giving
		\[\scalebox{.95}{$\begin{aligned}\sum_{j=1}^3k_j\left(1+\lambda_j^2\right)\left(\prod_{\ell=2}^{x+1}\gamma_{\ell,j}\right) &= \left(\prod_{\ell=2}^{x+1}\mu_\ell\right)\left(440+100\left(\frac{37}{100}\right)^2+240\left(\frac{1}{20}\right)^2+100\left(\frac{26}{100}\right)^2\right)\\
			&=\left(\prod_{\ell=2}^{x+1}\mu_\ell\right)\frac{18843}{40}.\end{aligned}$}\]
		We then apply Theorem \ref{degreebound} noting that if we choose $x=7$ then we find
		\[\left(\prod_{\ell=2}^8\mu_\ell\right)\frac{18843}{40}\approx 0.00098<\frac{1}{441} = \frac{1}{\vert X\vert}.\]
		Thus, we have that the conditions $\theta_{\ell j}\geq 0$ are vacuous for this parameter set whenever $\ell\geq 7$. Below we list $\theta_{\ell j}$ for $0\leq \ell\leq 6$ and $0\leq j\leq 3$, listing only the first two decimal places for readability:
		\[\left[\begin{array}{cccccrc}
		441 & 0 & 0 & 4.95 & 0.43 & -0.11 & 0.93 \\
		0 & 22.05 & 4.5 & 0.38 & 0.67 & 0.84 & 0.97 \\
		0 & 0 & 1.95 & 1.09 & 0.81 & 1 & 1.04 \\
		0 & 0 & 0 & 0.97 & 1.17 & 1.02 & 0.98 
		\end{array}\right].\]
		Note that $\theta_{50} = -0.11<0$ and therefore this parameter set is not realizable due to Theorem \ref{schoenL}.
	\end{example}
	
	\section{Cometric association schemes}\label{cometricGeg}
	In this section, we restrict to the case of cometric association schemes and explicitly compute the entries in the first column of $Q_\ell^{m_i}\left(\frac{1}{m_i}L_i^*\right)$. If the parameter set is realizable, these values correspond to the eigenvalues of the same Gegenbauer polynomial applied entrywise to $\frac{\vert X\vert}{m_i}E_i$, thus we denote the $j^\text{th}$ entry of the first column of $Q_\ell^{m_i}\left(\frac{1}{m_i}L_i^*\right)$ by $\theta_{\ell,j}$. We will compute $\theta_{\ell,j}$ for $2\leq \ell\leq 5$ and $0\leq j\leq d$ as well as $\theta_{6,0}$. Theorem \ref{schoen-as} tells us that each eigenvalue must be non-negative, and thus we derive more feasibility conditions on the parameters of a cometric association scheme. Many of these feasibility conditions will be implied by our previous conditions FC1, FC2, and FC3, however there are some which are independent of these three conditions. We will examine one such example closely in the section that follows. Suppose we have a feasible parameter set for a $Q$-polynomial association scheme with Krein array $\left\{m,b^*_1,\dots,b^*_{d-1};1,c_2^*\dots,c^*_{d}\right\}$. From our initial two Gegenbauer polynomials, we know
	\[\theta_{0i}= \vert X\vert\delta_{0i};\qquad \theta_{1i} =\frac{\vert X\vert}{m}\delta_{1i}\]
	where $\delta_{ij}$ equals $q$ if $i=j$ and zero otherwise. Using Equation \eqref{gegprop} and our Krein array, we have for $\ell\geq 2$ and $0\leq i\leq d$,
	\begin{equation}\label{thetarec}
	\theta_{\ell i}= \frac{(2\ell+m-4)(c_{i}^*\theta_{\ell-1,i-1} +a_i^*\theta_{\ell-1,i}+b_{i}^*\theta_{\ell-1,i+1}) - (\ell-1)m\theta_{\ell-2,i}}{m(\ell+m-3)}
	\end{equation}
	where $\theta_{\ell,-1} = \theta_{\ell,d+1} = 0$ for all choices of $\ell$. Before listing the eigenvalues, we note that many of the conditions $\theta_{\ell i}\geq 0$ will be implied by either the Krein conditions FC1 or the cometric property.
	\begin{lem}
		Suppose we have a feasible parameter set for a $Q$-polynomial association scheme with Krein array $\left\{m,b^*_1,\dots,b^*_{d-1};1,c_2^*\dots,c^*_{d}\right\}$. For $\ell\geq 0$, define $\theta_{\ell j}$ via Equation \eqref{thetarec} with $\theta_{0i}=\vert X\vert\delta_{0i}$ and $\theta_{1i}= \frac{\vert X\vert}{m}\delta_{1i}$. Then $\theta_{\ell i} = 0$ for $i>\ell$ and FC1 implies $\theta_{\ell i}\geq 0$ for $i\in\left\{\ell-1,\ell\right\}$.
	\end{lem}
	\begin{proof}
		We prove this by induction, showing first that the cometric property implies $\theta_{\ell,i}=0$ for $i>\ell$, $\ell\geq 0$ and then the conditions FC1 imply $\theta_{\ell,\ell}$ and $\theta_{\ell,\ell-1}$ are both non-negative. First note from our initial conditions that $\theta_{0i} = 0$ for $i>0$. Now, let $\ell\geq 1$ be given and assume that $\theta_{\ell,i}= 0$ for $i>\ell$. Then choose $i>\ell+1$ and from Equation \eqref{thetarec},
		\[\theta_{\ell+1,i}= \frac{(2\ell+m)(c_{i}^*\theta_{\ell,i-1} +a_i^*\theta_{\ell,i}+b_{i}^*\theta_{\ell,i+1}) - \ell m\theta_{\ell-1,i}}{m(\ell+m-2)}.\]
		However if $i>\ell+1$ then by our induction hypothesis, $\theta_{\ell,i-1} = \theta_{\ell,i} = \theta_{\ell,i+1} = \theta_{\ell-1,i} = 0$, thus $\theta_{\ell+1,i} = 0$. For the remaining two conditions, note that $\theta_{10}\geq 0$ and $\theta_{11}\geq 0$ are both vacuously true. Now let $\ell\geq 1$ be given and assume $\theta_{\ell,\ell-1}\geq 0$ and $\theta_{\ell,\ell}\geq 0$. Then
		\[\theta_{\ell+1,\ell+1}= \frac{(2\ell+m)(c_{\ell+1}^*\theta_{\ell,\ell} +a_{\ell+1}^*\theta_{\ell,\ell+1}+b_{\ell+1}^*\theta_{\ell,\ell+2}) - \ell m\theta_{\ell-1,\ell+1}}{m(\ell+m-2)} = \frac{(2\ell+m)c_{\ell+1}^*\theta_{\ell,\ell}}{m(\ell+m-2)}\]\[
		\theta_{\ell+1,\ell}= \frac{(2\ell+m)(c_{\ell}^*\theta_{\ell,\ell-1} +a_{\ell}^*\theta_{\ell,\ell}+b_{\ell-1}^*\theta_{\ell,\ell+1}) - \ell m\theta_{\ell-1,\ell}}{m(\ell+m-2)} = \frac{(2\ell+m)\left(c_\ell^*\theta_{\ell,\ell-1}+a_\ell^*\theta_{\ell,\ell}\right)}{m(\ell+m-2)}.\]
		Therefore as long as $a_\ell^*,c_\ell^*,c_{\ell+1}^*\geq 0$ (FC1), then $\theta_{\ell+1,\ell+1}\geq 0$ and $\theta_{\ell+1,\ell}\geq 0$.
	\end{proof}
	In view of this lemma, we know that any feasible parameter set for a cometric association scheme and any choice of $\ell\geq 0$, the condition $\theta_{\ell,i}\geq 0$ will be implied by FC1 or the cometric property for $i\geq \ell-1$. Therefore we will omit $\theta_{\ell,i}$ for $i>\ell$ in the discussion that follows, listing $\theta_{\ell,\ell}$ and $\theta_{\ell,\ell-1}$ only to assist in calculating eigenvalues from higher degree polynomials. In each case, we will use our Krein parameters (primarily those from the Krein array) to calculate the eigenvalues and note when our feasibility condition FC1 is sufficient to imply $\theta_{\ell,i}\geq 0$. At the end of the section, we will summarize any of the conditions which are not implied by FC1, noting that these are our (potentially) new constraints on the parameters of cometric schemes. For convenience we extend our Krein array to include $a_j^*,c_j^*,$ and $b_j^*$ for all $j\in\mathbb{Z}^+$ noting that $b_{j-1}^* = c_j^*=a_j^*=0$ for $j>d$. We will also use Lemma \eqref{kitchensink} \emph{(vi')} when convenient, so we assume that our Krein parameters satisfy the condition
	\begin{equation}\label{kreinid}\sum_{l=0}^d q_{ij}^lq_{lk}^m = \sum_{l=0}^dq_{il}^mq_{jk}^l.\end{equation}
	\subsection*{Degree 2 constraint}
	\[\frac{\theta_{20}}{\vert X\vert}  = 0;\qquad\frac{\theta_{21}}{\vert X\vert}  = \frac{a_1^*}{m(m-1)};\qquad\frac{\theta_{22}}{\vert X\vert} = \frac{c_2^*}{m(m-1)}.\]
	Each condition $\theta_{2i}\geq 0$ is implied by FC1 and thus we have no new restrictions from this case.
	\subsection*{Degree 3 constraint}
	\[\frac{\theta_{30}}{\vert X\vert} = \frac{\left(m+2\right)a_1^*}{m^2\left(m-1\right)};\qquad
	\frac{\theta_{31}}{\vert X\vert} = \frac{-2m\left(m-1\right)+\left(m+2\right)\left(\left(a_1^*\right)^2+b_1^*c_2^*\right)}{m^3\left(m-1\right)};\]
	\[\frac{\theta_{32}}{\vert X\vert} = \frac{\left(m+2\right)c_2^*\left(a_1^*+a_2^*\right)}{m^3\left(m-1\right)};\qquad
	\frac{\theta_{33}}{\vert X\vert}= \frac{\left(m+2\right)c_2^*c_3^*}{m^3\left(m-1\right)}.\]
	Here, $\theta_{31}\geq0$ is not implied by FC1. Thus we have the new feasibility condition
	\begin{equation}\label{Co31}\left(a_1^*\right)^2 + b_1^*c_2^* \geq\frac{2m(m-1)}{m+2}.\end{equation}
	\subsection*{Degree 4 constraint}
	\[\frac{\theta_{40}}{\vert X\vert} = \frac{\left(-2m\left(m-1\right)+\left(m+2\right)\left(\left(a_1^*\right)^2+b_1^*c_2^*\right)\right)\left(m+4\right)}{\left(m^2-1\right)m^3}~;\]
	\[\frac{\theta_{41}}{\vert X\vert} = \frac{\left(-4a_1^*m + \left(\left(a_1^*\right)^3+\left(2a_1^*+a_2^*\right)b_1^*c_2^*\right)\left(m+4\right)\right)\left(m+2\right)}{\left(m^2-1\right)m^4}~;\]
	\[\frac{\theta_{42}}{\vert X\vert} = \frac{\left(-4m\left(m-2\right)+\left(\left(a_1^*\right)^2+2a_1^*a_2^*+c_2^*q_{22}^2\right)\left(m+4\right)\right)c_2^*\left(m+2\right)}{\left(m^2-1\right)m^4}~;\]
	\[\frac{\theta_{43}}{\vert X\vert} = \frac{\left(m+4\right)\left(m+2\right)c_3^*c_2^*\left(a_1^*+a_2^*+a_3^*\right)}{\left(m^2-1\right)m^4}~;\]
	\[\frac{\theta_{44}}{\vert X\vert} = \frac{\left(m+4\right)\left(m+2\right)c_4^*c_3^*c_2^*}{\left(m^2-1\right)m^4}~.\]
	First, note that we used Equation \eqref{kreinid} with $j=k=1$ and $i=m=2$ to reduce $\theta_{42}$. Now, since $\theta_{20} = 0$, $\theta_{40} = \frac{(m+4)\theta_{31}}{m+1}$ and thus $\theta_{40}\geq 0$ is equivalent to $\theta_{31}\geq 0$ and we will omit this constraint. We therefore find only two new conditions: $\theta_{41}\geq 0$ and $\theta_{42}\geq 0$, neither of which is implied by FC1.
	\begin{equation}\label{Co41}
	\left(a_1^*\right)^2 + b_1^*c_2^*\left(2+\frac{a_2^*}{a_1^*}\right)\geq \frac{4m}{m+4}\qquad \text{whenever }a_1^*>0;
	\end{equation}
	\begin{equation}\label{Co42}
	\left(a_1^*\right)^2+2a_1^*a_2^*+c_2^*q^2_{22}\geq \frac{4m(m-2)}{m+4}.
	\end{equation}
	\subsection*{Degree 5 constraint}
	\[\frac{\theta_{50}}{\vert X\vert} = \frac{\left(-4a_1^*m\left(2m-3\right) + \left(\left(a_1^*\right)^3+\left(2a_1^*+a_2^*\right)b_1^*c_2^*\right)\left(m+6\right)\right)\left(m+4\right)}{m^4\left(m^2-1\right)}~;\]
	\[\begin{aligned}\frac{\theta_{51}}{\vert X\vert} = &\frac{6m^2(m-1)(m-4)+(m+4)(m+6)\left(\left(3a_1^*\left(a_1^*+a_2^*\right)+c_2^*q_{22}^2\right)b_1^*c_2^*+\left(a_1^*\right)^4\right)}{m^5\left(m^2-1\right)}\\
	&-\frac{m(m+4)(7m-18)\left(\left(a_1^*\right)^2+b_1^*c_2^*\right)}{m^5\left(m^2-1\right)}~;\end{aligned}\]
	\[\begin{aligned}\frac{\theta_{52}}{\vert X\vert\left(m+4\right)c_2^*} = &\frac{(m+6)\left((a_1^*+a_2^*)^3+2\left(c_2^*q_{22}^2-\left(a_2^*\right)^2\right)\left(a_1^*+a_2^*\right)+b_2^*c_3^*\left(a_3^*-a_1^*\right)-ma_2^*\right)}{m^5\left(m^2-1\right)}\\
	&-\frac{6\left(a_1^*+a_2^*\right)m(m-4)}{m^5\left(m^2-1\right)}~;\end{aligned}\]
	\[\frac{\theta_{53}}{\vert X\vert} = \frac{\left(-3\left(3m-2\right)+\left(\sum_{i=1}^3\left(b_i^*c_{i+1}^*+a_i^*\sum_{j=i}^3a_j^*\right)\right)\left(m+6\right)\right)\left(m+4\right)c_2^*c_3^*}{m^5\left(m^2-1\right)}~;\]
	\[\frac{\theta_{54}}{\vert X\vert} = \frac{\left(a_1^*+a_2^*+a_3^*+a_4^*\right)c_2^*c_3^*c_4^*\left(m+4\right)\left(m+6\right)}{m^5\left(m^2-1\right)}~;\]
	\[\frac{\theta_{55}}{\vert X\vert} = \frac{c_2^*c_3^*c_4^*c_5^*\left(m+4\right)\left(m+6\right)}{m^5\left(m^2-1\right)}~.\]
	In this case we find four new conditions arising from $\theta_{5i}\geq 0$ for $0\leq i\leq 3$. They are:
	\begin{equation}\label{Co50}
	\left(a_1^*\right)^2 + b_1^*c_2^*\left(2 + \frac{a_2^*}{a_1^*}\right)\geq \frac{4m(2m-3)}{m+6}
	\end{equation}
	\begin{equation}\label{Co51}
	\frac{6m(m-1)(m-4)}{(m+4)(m+6)}+\frac{\left(3a_1^*\left(a_1^*+a_2^*\right)+c_2q_{22}^2\right)b_1^*c_2^*+\left(a_1^*\right)^4}{m}\geq \frac{(7m-18)\left(\left(a_1^*\right)^2+b_1^*c_2^*\right)}{m+6}
	\end{equation}
	\begin{equation}\label{Co52}
	\left(a_1^*\right)^2+2a_1^*a_2^*-\left(a_2^*\right)^2+2c_2^*q_{22}^2+\frac{b_2^*c_3^*\left(a_3^*-a_1^*\right)-ma_2^*}{a_1^*+a_2^*}\geq\frac{6m(m-4)}{m+6}
	\end{equation}
	\begin{equation}\label{Co53}
	\sum_{i=1}^3\left(b_i^*c_{i+1}^* + a_i^*\sum_{j=i}^3 a_j^*\right)\geq \frac{3(3m-2)}{m+6}~.
	\end{equation}
	Note that Inequality $\eqref{Co41}$ is implied by Inequality $\eqref{Co50}$ for $m>2$.
	\subsection*{Degree 6 constraint}
	Here, we list only $\theta_{60}$ as we will use this bound in Section \ref{fourclassassoc}.
	\[\begin{aligned}\frac{\theta_{60}}{\vert X\vert( m+6)} = &\frac{16(m-2)}{m^3(m+1)(m+3)} + \frac{\left(\left(a_1^*\right)^4 + \left(3a_1^*\left(a_1^*+a_2^*\right)+c_2^*q_{22}^2\right)b_1^*c_2^*\right)(m+8)(m+4)}{m^5(m^2-1)(m+3)}\\
	&-\frac{12\left(\left(a_1^*\right)^2+b_1^*c_2^*\right)(m-2)(m+4)}{m^4(m^2-1)(m+3)}~.
	\end{aligned}\]
	This results in the final condition that
	\begin{equation}\label{Co60}
	\frac{16m(m-1)}{(m+4)(m+8)} + \frac{\left(a_1^*\right)^4 + \left(3a_1^*\left(a_1^*+a_2^*\right)+c_2^*q_{22}^2\right)b_1^*c_2^*}{(m-2)m}\geq \frac{12\left(\left(a_1^*\right)^2+b_1^*c_2^*\right)}{m+8}.
	\end{equation}
	In summary, we have the following theorems where, in each case, we assume the Krein parameters fulfill the condition FC1. We omit the inequality coming from $\theta_{41}$ as we may assume $m>2$ for all cases we are interested in.
	\cometricbounds
	\begin{proof}
		The table below gives the eigenvalue inequality from which each part of the Theorem is derived.
		\[	\begin{tabular}{c|c|c|c|c|c|c|c}
		Part & \emph{(i)} & \emph{(ii)} & \emph{(iii)} & \emph{(iv)} & \emph{(v)} & \emph{(vi)} & \emph{(vii)}\\\hline
		Inequality & $\theta_{31}\geq 0$ & $\theta_{42}\geq 0$ & $\theta_{51}\geq 0$&$\theta_{53}\geq 0$&$\theta_{60}\geq 0$	&$\theta_{50}\geq 0$ &$\theta_{52}\geq 0$\\	
		\end{tabular}\qedhere\]
	\end{proof}
	Jason Williford maintains a list of small ($\vert X\vert<10000$) feasible parameter sets for cometric schemes including primitive cometric schemes with 3 classes \cite{Willifordtable3} and $Q$-bipartite schemes with 4 classes \cite{Willifordtable4}. Using these lists, we find nine $3$-class primitive cometric schemes which are ruled out by Theorem \ref{cometricbnds} \emph{(vi)} and 11 $4$-class $Q$-bipartite schemes which are ruled out by Theorem \ref{cometricbnds} \emph{(v)}. They are as follows, listed as tuples of the form $\left(\vert X\vert,m_1\right)$:
	\begin{itemize}
		\item $3$-class primitive schemes ruled out by Theorem \ref{cometricbnds} \emph{(vi)}
		\[\left\{(441,20),(576,23),(729,26),(1015,28),(1240,30),\right.\]\[\left.(1548,35),(1836,35),(1944,29),(1976,25)\right\}.\]
		\item $4$-class $Q$-bipartite schemes
		\begin{itemize}
			\item ruled out by Theorem \ref{cometricbnds} \emph{(v)},
			\[\left\{(4464,24),(4968,27),(5280,30),(5436,27),(6148,29)\right\}\]
			\item ruled out by Theorem \ref{cometricbnds} \emph{(v)} and \emph{(iii)},
			\[\left\{(8432,31),(9984,32)\right\}\]
			\item ruled out by Theorem \ref{cometricbnds} \emph{(v)}, \emph{(iii)}, and \emph{(ii)}
			\[\left\{(594,9),(7776,27),(8478,27),(9984,24)\right\}\]
			
		\end{itemize}
	\end{itemize}
	\subsection{\texorpdfstring{$Q$-bipartite}{Q-bipartite} association schemes}
	We conclude this section by proving $Q$-bipartite analogues of Theorem \ref{degreebound}, Corollary \ref{maxlambda}, and Theorem \ref{cometricbnds}. We begin by proving an important theorem concerning the eigenvalues $\theta_{\ell i}$ as described in Theorem \ref{schoen-as}.
	\begin{thm}\label{Qbipeig}
		Let $(X,\cR)$ be a $Q$-bipartite association scheme with cometric ordering $E_0,E_1,\dots,E_d$. Define $\theta_{\ell,j}$ for $0\leq j\leq d$ and $\ell\geq 0$ so that
		\[G_\ell^m\circ\left(\frac{\vert X\vert}{m}E_1\right) = \sum_{j=0}^d\theta_{\ell,j}E_j.\]
		Then $\theta_{\ell,j} = 0$ whenever $\ell+j\notin2\bbZ$.
	\end{thm}
	\begin{proof}
		We prove this by induction. First, note that $\theta_{0,j} = \vert X\vert\delta_{0j}$ and thus $\theta_{0,j} = 0$ whenever $j\notin2\bbZ$. Now let $\ell\in\mathbb{Z}$ be given so that $\theta_{\ell,j} = 0$ whenever $\ell+j\notin 2\bbZ$ and consider
		\[\theta_{\ell+1,j}= \frac{(2\ell+m)(c_{j}^*\theta_{\ell,j-1} +a_i^*\theta_{\ell,j}+b_{j}^*\theta_{\ell,j+1}) - \ell m\theta_{\ell-2,j}}{m(\ell+m-2)}.\]
		If $(\ell+1)+j\notin2\bbZ$ then $\ell+(j-1),\ell+(j+1),(\ell-1)+j\notin2\bbZ$ and thus by our induction hypothesis,
		\[\theta_{\ell+1,j}= \frac{(2\ell+m)(a_i^*\theta_{\ell,j})}{m(\ell+m-2)}.\]
		However since $a_1^*=0$ for $Q$-bipartite schemes, we have $\theta_{\ell+1,j} = 0$.
	\end{proof}
	This theorem tells us that approximately half of the eigenvalues $\theta_{\ell,j}$ will be zero for a $Q$-bipartite scheme. This information allows us to recast many of the previous theorems seen in this section for this case.  We first consider Theorem $\ref{degreebound}$, following much of the same approach to prove the following theorem. We note that in this theorem we only consider the case $i=1$; while the same arguments may be made for any odd $i$, we restrict ourselves to $i=1$ here as we will primarily be interested in the first idempotent of $Q$-polynomial association schemes.
	
	\begin{thm}\label{Qbipdegbnd}
		Suppose we have a feasible parameter set for a $Q$-bipartite association scheme with first and second eigenmatrices $P$ and $Q$. Assume the relations are ordered naturally and set $m_1 = Q_{0,1}$. For $0\leq j\leq d$, define $\lambda_j:=\frac{Q_{j,1}}{m_1}$ and let $\gamma_{\ell,j}$ equal $\lambda_j^2$ if $(1+\mu_\ell)^2\lambda_j^2\geq 4\mu_\ell$ and $\mu_\ell$ otherwise. Then all entries of $Q_{\ell^*}^{m_1}\left(\frac{1}{m_1}L_1^*\right)$ are non-negative whenever
		\[\sum_{j=1}^{d-1} \left(\prod_{\ell=2}^{\ell^*+1}\gamma_{\ell,j}\right)P_{0j}\left(1+\lambda_j^2\right)\leq\frac{4}{\vert X\vert}.\]
	\end{thm}
	
	\begin{proof}
		Let $L_1^*$ be the Krein matrix of our parameter set. Define $M = \frac{1}{m_1}\sqrt{\Delta_m}L_1^*\sqrt{\Delta_m}^{-1}$. As in Section \ref{proof}, let
		\begin{equation*}
		\vec{y}_\ell = T_\ell\vec{y}_{\ell-1} \qquad \ell\geq 2
		\end{equation*}
		with
		\begin{equation*}\mu_\ell = \frac{\ell-1}{\ell+m-3};\qquad \vec{y}_1 = \left[\begin{array}{c}
		\frac{1}{\sqrt{m}}\vec{e}_1\\\hdashline[2pt/2pt]
		\vec{e}_0
		\end{array}\right];\qquad T_\ell = \left[\begin{array}{>{\centering\arraybackslash}p{2cm};{2pt/2pt}>{\centering\arraybackslash}p{2cm}}
		$(1+\mu_\ell)M$ & $-\mu_\ell I$\\\hdashline[2pt/2pt]
		$I$ & $0$
		\end{array}\right].\end{equation*}
		Further define $B_j = \text{span}\left(\left[\begin{array}{c}
		\vec{p}_j\\\hdashline[2pt/2pt]
		\vec{0}
		\end{array}\right],\left[\begin{array}{c}
		\vec{0}\\\hdashline[2pt/2pt]
		\vec{p}_j
		\end{array}\right]\right)$ for $0\leq j\leq d$ where $\vec{p}_j$ is the $j^\text{th}$ column of $\sqrt{\Delta_m}P$.
		Then Lemma \ref{normbound} tells us that 
		\[\text{proj}_{B_j}(\vec{y}_1) = \frac{1}{\vert X\vert}\left[\begin{array}{c}
		\lambda_j\vec{p}_j\\\hdashline[2pt/2pt]
		\vec{p}_j
		\end{array}\right];\quad\left\vert\left\vert \left(\prod_{\ell=2}^{\ell^*} T_\ell\right) \text{proj}_{B_j}(\vec{y}_1)\right\vert\right\vert_2^2\leq \left(\prod_{\ell=2}^{\ell^*}\gamma_{\ell,j}\right)\left\vert\left\vert\text{proj}_{B_j}(\vec{y}_1)\right\vert\right\vert_2^2.\]
		For a $Q$-bipartite scheme, the first and last column of $\sqrt{\Delta}P$ are given by 
		\[\vec{p}_0 = \left[1,\sqrt{m},\dots,\sqrt{m_{d-1}},1\right]^T,\qquad \vec{p}_d = \left[1,-\sqrt{m},\dots,(-1)^{d-1}\sqrt{m_{d-1}},(-1)^d\right]^T\] with $\lambda_0 = -\lambda_d = 1$. Noting that the roots of the polynomial $x^2-(1+\mu_\ell)(\pm 1)x + \mu_\ell$ are $\pm1$ and $\mu_\ell$, we find that both $\text{proj}_{B_0}(\vec{y}_1)$ and $\text{proj}_{B_d}(\vec{y}_1)$ are eigenvectors of $T_\ell$ for all $\ell\geq 2$. Therefore
		\[ \text{proj}_{B_0}(\vec{y}_{\ell^*}) = \text{proj}_{B_0}(\vec{y}_1);\qquad \text{proj}_{B_d}(\vec{y}_{\ell^*}) = \left(-1\right)^{\ell^*-1}\text{proj}_{B_0}(\vec{y}_1).\]
		For a vector $\vec{v}$ of length $d+1$, we refer to the entries $v_0,v_2,v_4,\dots,v_{2\left\lfloor\frac{d}{2}\right\rfloor}$ as the \textit{even part of} $\vec{v}$ and likewise the entries $v_1,v_3,\dots,v_{2\left\lceil\frac{d}{2}\right\rceil-1}$ as the \textit{odd part of} $\vec{v}$. Further, for a vector $\vec{y}$ of length $2d+2$, define the \textit{top half of} $\vec{y}$ to be the first $d+1$ entries and the \textit{bottom half of} $\vec{y}$ to be the last $d+1$ entries. Then we have the following four statements
		\begin{itemize}
			\item For even $\ell$,
			\begin{itemize}
				\item all entries in the odd part of the top half of $\text{proj}_{B_0\oplus B_d}(\vec{y}_\ell)$ are 0;
				\item all entries in the even part of the bottom half of $\text{proj}_{B_0\oplus B_d}(\vec{y}_\ell)$ are 0.
			\end{itemize}
			\item For odd $\ell$,
			\begin{itemize}
				\item all entries in the even part of the top half of $\text{proj}_{B_0\oplus B_d}(\vec{y}_\ell)$ are 0;
				\item all entries in the odd part of the bottom half of $\text{proj}_{B_0\oplus B_d}(\vec{y}_\ell)$ are 0.
			\end{itemize}
		\end{itemize}
		Similarly Theorem \ref{Qbipeig} shows that the same is true for $\vec{y}_\ell$, that is
		\begin{itemize}
			\item For even $\ell$,
			\begin{itemize}
				\item all entries in the odd part of the top half of $\vec{y}_\ell$ are 0;
				\item all entries in the even part of the bottom half of $\vec{y}_\ell$ are 0.
			\end{itemize}
			\item For odd $\ell$,
			\begin{itemize}
				\item all entries in the even part of the top half of $\vec{y}_\ell$ are 0;
				\item all entries in the odd part of the bottom half of $\vec{y}_\ell$ are 0.
			\end{itemize}
		\end{itemize}
		We find also that $\text{proj}_{B_0}(\vec{y}_\ell) + \text{proj}_{B_d}(\vec{y}_\ell)$ is non-zero except in the entries listed above. This implies the sum of the remaining projections must follow the same pattern, permitting non-zero values only when $\text{proj}_{B_0}\left(\vec{y}_\ell\right)+\text{proj}_{B_d}\left(\vec{y}_\ell\right)$ is also non-zero. Since the smallest remaining entry of $\text{proj}_{B_0}(\vec{y}_\ell) + \text{proj}_{B_d}(\vec{y}_\ell)$ is $\frac{2}{\vert X\vert}$, we have the implication: If
		\[\left\vert\left\vert \sum_{j=1}^{d-1}\left(\prod_{\ell=2}^{\ell^*+1}T_\ell\right)\text{proj}_{B_j}(\vec{y}_1)\right\vert\right\vert^2_2 \leq\frac{4}{\vert X\vert^2}\]
		then $\vec{y}_{\ell^*+1}$ has no negative entries. Following the reasoning from before, we note that if $\vec{y}_{\ell^*}$ has a negative value, it must appear in the bottom half, thus $\vec{y}_{\ell^*-1}$ must have a negative value and it is sufficient to check the Gegenbauer polynomials up to degree $\ell^*-1$.	
	\end{proof}
	\begin{cor}\label{Qbipuppbnd}
		Suppose we have a feasible parameter set for a $Q$-bipartite association scheme with relations ordered naturally. Define $\lambda_1:=\nicefrac{Q_{1,1}}{m_1}$ and let $\lambda^*\geq\vert\lambda_1\vert$ be given. Then for
		\[\ell^*=\left\lceil\frac{\ln\left[\left(1+\lambda_1^2\right)\vert X\vert\left(\vert X\vert-2\right)\right]-\ln(4)}{-2\ln\left(\lambda^*\right)}\right\rceil,\]
		if $\left\vert\lambda^*\right\vert^2\geq \frac{\ell^*}{\ell^*+m_1-2}$ then $\theta_{\ell,j}\geq 0$ is implied by FC1 for $\ell\geq \ell^*$ and $0\leq j\leq d$.
	\end{cor}
	\begin{proof}
		Let $\lambda_j:=\frac{Q_{j,1}}{m_1}$ and note that the $Q$-bipartite property, along with the natural ordering of relations, tells us $\vert\lambda_1\vert\geq \vert\lambda_j\vert$ for $1\leq j\leq d-1$. Therefore, as long as $\left\vert \lambda^*\right\vert^2$ is greater than both $\vert\lambda_1\vert^2$ and $\mu_{\ell}$ for $2\leq\ell\leq x$,
		\[\sum_{j=1}^{d-1} P_{0j}(1+\lambda_j^2)\left(\prod_{\ell=2}^x \gamma_{\ell,j}\right)\leq(\lambda^*)^{2x-2}(1+\lambda_1^2)\sum_{j=1}^{d-1} P_{0j} =(\lambda^*)^{2x-2}(1+\lambda_1^2)(\vert X\vert -2). \]
		Solving 
		\begin{equation}\label{ineq2}(\lambda^*)^{2x-2}(1+\lambda_1^2)(\vert X\vert -2)\leq\frac{4}{\vert X\vert}\end{equation}
		for $x$ gives
		\[x\geq\frac{\ln\left[(1+\lambda_1^2)\vert X\vert(\vert X\vert-2)\right]-\ln(4)}{-2\ln(\lambda^*)}+1.\]
		Thus defining $\ell^* = \left\lceil\frac{\ln\left[(1+(\lambda_1)^2)\vert X\vert(\vert X\vert-2)\right]-\ln(4)}{-2\ln(\lambda^*)}\right\rceil$ gives that $\ell^*+1$ is the smallest integer for which Inequality \eqref{ineq2} holds. Then, as long as $\vert\lambda^*\vert^2\geq \mu_{\ell^*+1}$, we may use Theorem \ref{Qbipdegbnd} to give our result.
	\end{proof}
	Finally, we use our $Q$-bipartite property to simplify the expressions in Theorem \ref{cometricbnds}.
	\begin{cor}\label{Qbipbnds}
		Suppose we have a feasible parameter set for a $Q$-bipartite association scheme with Krein array $\left\{m,b^*_1,\dots,b^*_{d-1};1,c_2^*\dots,c^*_{d}\right\}$. Then the scheme is realizable only if each of the following hold:
		\begin{enumerate}[label=(\roman*)]
			\item $\displaystyle{b_1^*c_2^* \geq\frac{2m(m-1)}{m+2},}$
			\item $\displaystyle{c_2^*q^2_{22}\geq \frac{4m(m-2)}{m+4},}$
			\item $\displaystyle{\frac{6m(m-1)(m-4)}{(m+4)(m+6)} + \frac{b_1^*c_2^*c_2^*q_{22}^2}{m}\geq \frac{b_1^*c_2^*(7m-18)}{m+6},}$
			\item $\displaystyle{\sum_{i=1}^3b_i^*c_{i+1}^*\leq \frac{3(3m-2)}{m+6},}$
			\item $\displaystyle{\frac{16m(m-1)}{(m+4)b_1^*c_2^*} + \frac{c_2^*q_{22}^2(m+8)}{(m-2)m}\geq 12.}$
		\end{enumerate}
	\end{cor}
	\begin{proof}
		Consider Theorem \ref{cometricbnds} with the added constraint that $a_i^* = 0$ for $0\leq i\leq d$.
	\end{proof}
	We finish this chapter with two examples, one with a feasible parameter set which satisfies all above conditions and another which is ruled out by Corollary \ref{Qbipbnds}.
	\begin{example}
		In this example we will apply Corollary \ref{Qbipuppbnd} to find the maximum degree of Gegenbauer polynomials which we need to check in order to verify Theorem \ref{Qbipeig}. Consider the feasible parameter set with first and second eigenmatrices 
		\[P = \left[\begin{array}{rrrrr}
		1 & 1116 & 7750 &1116 &1\\
		1 &186& 0 &-186& -1\\
		1 &24& -50& 24& 1\\
		1 &-6& 0 &6 &-1\\
		1 &-36& 70 &-36& 1\\
		\end{array}\right],\qquad Q = \left[\begin{array}{rrrrr}
		1 &156& 2976& 4836 &2015\\
		1 &26& 64 &-26& -65\\
		1 &0 &\nicefrac{-96}{5} &0 &\nicefrac{91}{5}\\
		1 &-26& 64 &26& -65\\
		1 &-156& 2976& -4836& 2015
		\end{array}\right].\]
		If realizable, the association scheme with this parameter set would be a 4-class $Q$-bipartite scheme. Defining $\lambda^*:=\sqrt{\mu_6} = \sqrt{\frac{5}{159}}$ and noting that $\lambda_1 = \frac{Q_{11}}{m_1} = \frac{1}{6}$, we have
		\[\left\lceil\frac{\ln\left[\left(1+\left(\lambda_1\right)^2\right)\vert X\vert\left(\vert X\vert-2\right)\right]-\ln(4)}{-2\ln\left(\lambda^*\right)}\right\rceil = 5.\]
		Since $\left\vert\lambda^*\right\vert^2 = \mu_6 = \frac{4}{158}>\frac{1}{36} = \left\vert \lambda_1\right\vert^2$, we may use Corollary \ref{Qbipuppbnd} to show that $\theta_{\ell,j}\geq 0$ is vacuous for this parameter set for $\ell\geq 5$. We list the remaining eigenvalues below
		\[\left[\begin{array}{ccccc}
		9984 & 0 & 0 & 0 & 2.61\\
		0 & 64 & 0 & 2.56 & 0\\
		0 & 0 & 3.36 & 0 & 1.99\\
		0 & 0 & 0 & 1.98& 0\\
		0 & 0 & 0 & 0 & 2.02
		\end{array}\right]\]
		Since all eigenvalues are non-negative, this parameter set is not ruled out by Theorem \ref{Qbipeig}. We will see many more examples like this in Chapter \ref{4classbip}, many of which will be ruled out.
		
	\end{example}
	\begin{example}
		In contrast to the example above, we now give a feasible parameter set for a 4-class $Q$-bipartite association scheme which is ruled out by the methods discussed in this chapter. Consider the feasible parameter set with first and second eigenmatrices 
		\[P = \left[\begin{array}{rrrrr}
		1 & 128 & 336 &128 &1\\
		1 &64& 0 &-64& -1\\
		1 &20& -42& 20& 1\\
		1 &-2& 0 &2 &-1\\
		1 &-4& 6 &-4 & 1\\
		\end{array}\right],\qquad Q = \left[\begin{array}{rrrrr}
		1 &9& 44& 288 &252\\
		1 &\nicefrac{9}{2}& \nicefrac{55}{8} &\nicefrac{-9}{2}& \nicefrac{-63}{8}\\
		1 &0 &\nicefrac{-11}{2} &0 &\nicefrac{9}{2}\\
		1 &\nicefrac{-9}{2}& \nicefrac{55}{8} &\nicefrac{9}{2}& \nicefrac{-63}{8}\\
		1 &-9& 44& -288& 252
		\end{array}\right].\]
		
		Using Lemma \ref{kitchensink}, we may calculate the various Krein parameters appearing in Corollary \ref{Qbipbnds} as follows
		\[b_1^* = 8;\qquad c_2^* = \frac{18}{11};\qquad q^2_{22} = \frac{121}{16};\qquad m = 9.\]
		Plugging these values into the left hand side of Corollary \ref{Qbipbnds} $(v)$ results in
		\[\frac{16m(m-1)}{(m+4)b_1^*c_2^*} + \frac{c_2^*q_{22}^2(m+8)}{(m-2)m} = \frac{7359}{728}\approx 10.1,\]
		thus, Corollary \ref{Qbipbnds} $(v)$ is violated and this parameter set is not realizable.
		
	\end{example}
\chapter{3-class \texorpdfstring{$Q$-antipodal}{Q-antipodal}: LSSDs}\label{3class}
In \cite{Cameron1972}, Cameron investigated groups with inequivalent doubly-transitive permutation representations having the same permutation character and introduced the notion of a linked system of symmetric designs (LSSD). One such example arising from Kerdock codes was communicated by Goethals \cite{Goethals1976} and studied in depth by Cameron and Seidel \cite{Cameron1973}. These structures, in the homogeneous case, were then further studied by Noda \cite{Noda1974} who bounded the number of fibers in a LSSD in terms of the design parameters induced between any two of the fibers. Focusing on the $(16,6,2)$ designs, Mathon \cite{Mathon1981} classified all inequivalent LSSDs using these design parameters via a computer search, finding that there were multiple inequivalent LSSDs with two or three fibers but only the scheme described by Geothals worked with four or more fibers. Later, Van Dam proved in \cite{VanDam1999} the equivalence between these objects and 3-class Q-antipodal association schemes. Martin, Muzychuk, and Williford found a connection to mutually unbiased bases in certain dimensions \cite{Martin2007}. Finally Davis, Martin, Polhill \cite{Davis2014} and Jedwab, Li, Simon \cite{Jedwab2017} built more non-trivial examples using difference sets in 2-groups.\par
We begin this chapter with a survey of known results focusing on the connection to association schemes. We introduce ``linked simplices", natural geometric objects which are of interest in their own right, as collections of full-dimensional simplices with only two possible angles between vectors in distinct simplices. We establish the equivalence of sets of linked simplices and LSSDs. We compare three known bounds on the number of fibers and explore connections to structures in Euclidean space. We show how to construct equiangular lines from arbitrary LSSDs and explore cases where LSSDs lead to real mutually unbiased bases (MUBs). After reviewing known examples, we focus on the case of Menon design parameters and, employing an equivalence with sets of mutually unbiased Hadamard matrices, we construct new families of LSSDs for many values of $v$. In particular, we show that one may fix the largest power of two dividing $v$ without bounding the number of fibers in an LSSD, a result which was not previously known. In Appendix \ref{families}, we survey the design parameters of known infinite families of symmetric designs and determine which of these cannot be the design parameters of LSSDs with more than two fibers, noting differences vis-\'{a}-vis a recent discovery by Jedwab et.\ al.\ \cite{Jedwab2017} in restricted cases.

Throughout this chapter we will heavily discuss both the parameters of the association schemes given by LSSDs as well as the design parameters of the given symmetric designs. Thus, for clarity sake, we will always refer to the triple $(v,k,\lambda)$ as the \emph{design parameters}\index{symmetric design parameters} of the LSSD so as to distinguish these from the intersection numbers, Krein parameters, and eigenmatrices of an association scheme. This chapter is based on research published in \cite{Kodalen2019}. Below we list the main theorems of this chapter.
\begin{restatable*}{thm}{lssdsimpequiv}
	A $LSSD(v,k,\lambda;w)$ is equivalent to a set of $w$ linked simplices in $\mathbb{R}^{v-1}$ whose angles depend on the parameters $v$, $k$, and $\lambda$.
\end{restatable*}
\begin{restatable*}{thm}{lssdhadequiv}
	\label{equiv}
	An optimistic $LSSD(v,k,\lambda;w)$ with $\vert v-2k\vert = 2\sqrt{k-\lambda}$ exists if and only if there exists a set of $w-1$ regular unbiased Hadamard matrices, $H_i$, with order $v$ and $H_iJ = 2\sqrt{k-\lambda}J$.
\end{restatable*}
\begin{restatable*}{thm}{buildinghad}[cf.\ Thm~13 in \cite{Kharaghani2010}]
	\label{newLSSD}
	Given a regular Hadamard matrix of order $n$ and an orthogonal array of size $n^2\times N$,
	\begin{itemize}
		\item There exist $N-1$ regular unbiased Hadamard matrices of order $n^2$.
		\item There exists a $LSSD$ with $v=n^2$ and $w=N$.
	\end{itemize}
\end{restatable*}
\begin{restatable*}{cor}{asymplssd}
	\label{asymptotics}
	For sufficiently large $n$, if there exists a regular Hadamard matrix of order $n$, then there exists a $LSSD(n^2,k,\lambda; w)$ with $w \geq n^\frac{1}{14.8}$.
\end{restatable*}
\begin{restatable*}{cor}{oddlssd}
	For any $n\geq 1$ and $w>2$, there exists an odd $t$ for which there exists an $LSSD(16^nt,k,\lambda;w)$.
\end{restatable*}
\begin{restatable*}{cor}{lssdthreesix}
	There exists an $LSSD(v,k,\lambda;w)$ with $v=36^{2n}$ and $w = 4^n+1$ for all $n\geq 1$.
\end{restatable*}
\section{Homogeneous linked systems of symmetric designs}
We begin by reviewing symmetric designs as these will play a central role in all that follows. A \textit{symmetric 2-design}\index{block design!2-design} with design parameters $(v,k,\lambda)$ is a set of blocks $\mathcal{B}$ on point set $X$ written as $(X,\mathcal{B})$ satisfying the following three conditions:
\begin{itemize}
	\item There are $v$ blocks and $v$ points ($\vert\mathcal{B}\vert = \vert X\vert = v$);
	\item Every block contains $k$ points and every point is contained in $k$ blocks;
	\item Every pair of points is contained in $\lambda$ blocks and the intersection of any pair of blocks contains $\lambda$ points.
\end{itemize}
We will ignore the case $v=k=\lambda$ and consider the case $k=1,\lambda=0$ ``degenerate".

We form an \textit{incidence}\index{block design!incidence matrix} matrix $B$ for the block design, indexing rows by blocks and columns by points, setting $B_{ij} = 1$ if point $j$ is in block $i$ and $B_{ij} = 0$ otherwise. Finally, we note the following two equivalent equations which hold for any symmetric 2-design:
\begin{align}
k(k-1) &= \lambda(v-1)\label{sym:1},\\
k(v-k) &= (k-\lambda)(v-1).\label{sym:2}
\end{align}
The \emph{incidence graph}\index{symmetric design!incidence graph} of $(X,\mathcal{B})$ is the graph whose adjacency matrix is $\left[\begin{array}{c:c}
0 & B\\\hdashline[2pt/2pt]
B^T & 0
\end{array}\right].$

We now move to a description of a homogeneous\footnote{Here, ``homogeneous" refers to the designs between fibers all having the same design parameters. For the duration of this chapter, we will only concern ourselves with this case, though we drop this clarification later and refer to the structures simply as linked systems of symmetric designs.} linked system of symmetric designs as described by Cameron in \cite{Cameron1972} and Noda in \cite{Noda1974}. 
Consider a multipartite graph $\Gamma$ on $wv$ vertices with vertex set partitioned into $w$ sets of $v$ vertices called ``fibers":
\[X = X_1\dot{\cup} X_2\dot{\cup}\cdots\dot{\cup}X_w.\]
We say $\Gamma$ is a \emph{linked system of symmetric designs}\index{linked system of symmetric designs}, $LSSD(v,k,\lambda;w)$ ($w\geq 2$), if it satisfies the following three properties:
\begin{enumerate}[label=$(\roman*)$]
	\item no edge of $\Gamma$ has both ends in the same fiber $X_i$;
	\item for all $1\leq i,j\leq w$ with $i\neq j$, the induced subgraph of $\Gamma$ between $X_i$ and $X_j$ is the incidence graph of some $(v,k,\lambda)$-design;
	\item there exist constants $\mu$ and $\nu$ such that for distinct $h,i,j$ ($1\leq h,i,j\leq w$), 
	\begin{equation}\label{munudef}
	a\in X_i, b\in X_j \Rightarrow\vert\Gamma(a)\cap\Gamma(b)\cap X_h\vert=\begin{cases}
	\mu & a\sim b\\
	\nu & a\not\sim b
	\end{cases}
	\end{equation}
\end{enumerate}
where $\sim$ denotes adjacency in $\Gamma$ and $\Gamma(x)$ denotes the neighborhood of vertex $x$. Observe that $\Gamma$ is regular with valency $k(w-1)$. A specific type of LSSD introduced in \cite{Davis2014}, constructed from a \text{linking system of difference sets}\index{linking system of difference sets}, is a LSSD where the symmetric design induced between any two fibers comes from a difference set with the further restriction that the difference sets induced between any pair of three fibers interact in a consistent way. Recently Jedwab, Li, and Simon \cite{Jedwab2017} examined these in more detail, building new examples and proving non-existence results for certain design parameters. However, Mathon \cite{Mathon1981} showed there exist three-fiber LSSDs which do not correspond to linking systems of difference sets, thus we expect there to be many LSSDs which we cannot build via linking difference sets. We will consider the general case, and thus do not assume this added structure on our symmetric designs.\par
In \cite[Proposition 0]{Noda1974}, Noda shows that $\mu$ and $\nu$ must take one of two pairs of values given by:
\begin{equation}
\nu = \frac{k(k\pm \sqrt{k-\lambda})}{v}, \qquad \mu = \nu\mp \sqrt{k-\lambda}. \label{mu-nu}
\end{equation}
With these two possibilities for $\mu$ and $\nu$, it becomes useful to distinguish between the two types of LSSDs. We will refer to the LSSD as \emph{$\mu$-heavy}\index{$\mu$-heavy} (resp., \emph{$\nu$-heavy}\index{$\nu$-heavy}) when $\mu>\nu$ (resp., $\nu>\mu$). Note that since both $\mu$ and $\nu$ are integers, we must also have $\sqrt{k-\lambda}\in\bbZ$; for the remainder of this chapter, define $s = \sqrt{k-\lambda}$ Note, this parameter is known as the ``order" of the symmetric design and many authors use $n$ to denote this value --- we will always use $s$. Further, since $k\neq \lambda$, we must have $s>0$ and thus $\mu\neq \nu$. We now review a proposition of Noda which notes that, given an LSSD $\Gamma$, swapping adjacency between fibers produces another LSSD; we call this graph the \emph{multipartite complement}\index{multipartite complement} of $\Gamma$.
\begin{prop}[Noda]\label{complement}
	Let $\Gamma$ be a  $LSSD(v,k,\lambda;w)$ with $w>2$. If $\Gamma$ is $\mu$-heavy (resp., $\nu$-heavy), the multipartite complement $\Gamma'$ is a $\nu$-heavy (resp., $\mu$-heavy) $LSSD(v,v-k,v-2k+\lambda;w)$.
\end{prop}
We further find that, given $v$, $k$, and $\lambda$, only one of the outcomes in \eqref{mu-nu} is possible for $v\geq 3$.
\begin{lem}
	\label{gcd}
	Let $\Gamma$ be a  $LSSD(v,k,\lambda;w)$ with $w>2$ and $1<k<v-1$. Then the following hold:
	\begin{enumerate}[label=$(\roman*)$]
		\item exactly one of $\frac{k(k+ s)}{v}$ and $\frac{k(k-s)}{v}$ is an integer;
		\item $\gcd(k,v)>1$;
		\item $\gcd(s,v)>1$.
	\end{enumerate}
\end{lem}
\begin{proof}
	By Proposition $\ref{complement}$, we may assume $k\leq \frac{v}{2}$ as $\frac{k(k\pm s)}{v}$ is integral if and only if $\frac{(v-k)((v-k)\pm s)}{v}$ is integral. Now, Equation \eqref{sym:1} gives $k^2-s^2 = \lambda v$ and thus we have the equations
	\[\begin{aligned}
	\frac{k(k+ s)}{v} -\frac{s(k+ s)}{v} =\frac{k(k- s)}{v} +\frac{s(k- s)}{v} &= \lambda.
	\end{aligned}\]
	Now assume $\frac{k(k+s)}{v}$ and $\frac{k(k-s)}{v}$ are both integral. This implies that $\frac{s(k+s)}{v}$ and $\frac{s(k-s)}{v}$ are both integral, and thus their difference $\frac{2s^2}{v}$ must also be integral. However this contradicts $s<k\leq\frac{v}{2}$, thus at most one of $\frac{k(k+s)}{v}$ and $\frac{k(k-s)}{v}$ may be integral. The assumption that our parameters were feasible (in fact, realizable) is sufficient to guarantee at least one is an integer, giving $(i)$. For $(ii)$, we again assume $k\leq \frac{v}{2}$ (noting that $\gcd(k,v) = \gcd(v-k,v)$ since $k<v$). Now, since $s<k$ we must have $k+s<v$. Thus if $\frac{k(k+s)}{v}$ is an integer, we must have $\gcd(k,v)>1$. Similarly $\frac{k(k-s)}{v}\in\mathbb{Z}$ implies $\gcd(k,v)>1$. The same argument applies to $(iii)$ noting again that one of $\frac{k(k\pm s)}{v}$ is integral if and only if one of $\frac{s(k\pm s)}{v}$ is integral.
\end{proof}
\begin{remark}
	\label{rmk1}
	Theorem \ref{Qpoly} and the discussion that follows shows that the parameters of an LSSD uniquely determine the parameters of an association scheme. It is important to note here that the first line of Lemma \ref{gcd} may be replaced with the line ``Suppose the parameter set of a $LSSD(v,k,\lambda;w)$ with $w>2$ and $1<k<v-1$ is feasible." In fact, all we need to prove Lemma \ref{gcd} is that either $\frac{k(k+s)}{v}$ or $\frac{k(k-s)}{v}$ is integral. This will become useful as we apply this lemma to rule out feasible parameter sets where, clearly, we cannot assume the LSSD exists.
\end{remark}
The case where $k=1$ or $k=v-1$ produces LSSDs which are not of interest to us and for the remainder of the chapter, we will refer to these designs as ``degenerate"\index{linked system of symmetric designs!degenerate}. For a further description of why these designs are degenerate, see Section \ref{degenerate}. The observations that $\gcd(k,v)>1$ and $\gcd(s,v)>1$ are two tools which help us determine more easily which design parameters might be feasible for a LSSD with $w>2$. We may find many other statements similar to these, but these two will be sufficient for now. Using these, we can immediately rule out many design parameters. For instance:
\begin{cor}
	\label{comp}
	Assume $w>2$. If there exists a non-degenerate $LSSD(v,k,\lambda;w)$, then $v$ is composite.\qed
\end{cor}
For a further use of these tools to rule out certain families of symmetric designs, see Appendix \ref{families}.

We now provide a theorem of Van Dam concerning the equivalence between linked systems of symmetric designs and 3-class $Q$-antipodal association schemes. We do not give a self-contained proof here, instead referring to results of Van Dam to outline the proof. While this result appears in \cite{VanDam1999}, we include a later result appearing in \cite{vanDam2013} to assist with our outline.
\begin{thm}[\cite{VanDam1999}]
	\label{Qpoly}
	Let $\Gamma$ be a non-degenerate LSSD with adjacency matrix $A$. Then the algebra $\langle A\rangle_*$ is the Bose-Mesner algebra of a 3-class Q-antipodal association scheme on $X$. Conversely, every $Q$-antipodal 3-class association scheme arises in this way. More specifically, the natural ordering of the relations of any $Q$-antipodal 3-class association scheme is as follows:
	\begin{itemize}
		\item $R_0$ is the identity relation on $X$;
		\item $R_1$ is given by adjacency in the $\mu$-heavy LSSD;
		\item $R_2$ is the union of complete graphs on the fibers induced by $R_1$;
		\item $R_3$ is given by adjacency in the $\nu$-heavy LSSD.
	\end{itemize}
\end{thm}
\noindent \textit{Proof outline.} Let $\Gamma$ be the graph of a linked system of symmetric designs with $0<\lambda<k-1$ and assume without loss of generality that $k\leq \frac{v}{2}$. Using the definition of an LSSD, one may verify quickly that all intersection numbers are well-defined, and thus the adjacency algebra of $\Gamma$ is the Bose-Mesner algebra of a 3-class association scheme. We may then build $L_1$ to find that the eigenvalues of this graph are $k(w-1)>\sqrt{k-\lambda}(w-1)>-\sqrt{k-\lambda}>-k$ with multiplicities $m_0,m_1,m_2,m_3$ respectively. Since $\Gamma$ is connected, we find that $m_0=1$. Further, since $\Gamma$ has four distinct eigenvalues, Theorem 5.8 of \cite{VanDam1999} tells us $m_1 = v-1$. We then use the equations $\sum_{i=0}^3\theta_im_i = 0$ and $\sum_{i=0}^3 m_i = vw$ to find that $m_3 = w-1$ and $m_2 = (w-1)(v-1)$. We then apply Proposition 6.1 of \cite{vanDam2013} to see that this association scheme is $Q$-antipodal. Conversely, let $(X,\mathcal{R})$ be a 3-class $Q$-antipodal association scheme with relations ordered naturally. Then Theorem \ref{suzukiimprim} tells us the system of imprimitivity is given by $\mathcal{J} = \left\{0,3\right\}$ and $\mathcal{I} = \left\{0,2\right\}$ and Theorem 5.1 of \cite{vanDam2013} tells us $(X,\mathcal{R})$ is uniform. We may then refer again to Proposition 6.1 of \cite{vanDam2013} to see that $m_2 = (w-1)m_1$. This, along with the antipodal property, forces $k_1 = v-1$ where $k_1$ is the valency of the nearest neighbor graph. Thus Theorem 5.8 of \cite{VanDam1999} tells us this nearest neighbor graph is the incidence graph of a linked system of symmetric designs with $0<\lambda<k-1$.\qed

In addition to the above proof, we list the intersection numbers, the first and second eigenmatrices, and a some of the Krein parameters for later use. Let $\Gamma_1 = (X,R_1)$ be a $\mu$-heavy $LSSD(v,k,\lambda)$ with complement design given by $\Gamma_2 = (X,R_3)$. The following are the intersection numbers, listed via the four matrices $L_0,L_1,L_2,L_3$ where $L_i = [p_{i,j}^k]_{k,j}$;
\[\begin{aligned}
L_0 &= \left[\begin{array}{cccc}
1 & 0 & 0 & 0\\
0 & 1 & 0 & 0\\
0 & 0 & 1 & 0\\
0 & 0 & 0 & 1\\
\end{array}\right], \qquad L_1 = \left[\begin{array}{cccc}
0 & k(w-1)       & 0   & 0               \\
1 & \mu(w-2)     & k-1 & (k-\mu)(w-2)    \\
0 & \lambda(w-1) & 0   & (k-\lambda)(w-1)\\
0 & \nu(w-2)     & k   & (k-\nu)(w-2)    \\
\end{array}\right],\\
L_2 &= \left[\begin{array}{cccc}
0 & 0   & v-1 & 0    \\
0 & k-1 & 0   & v-k  \\
1 & 0   & v-2 & 0    \\
0 & k   & 0   & v-k-1\\
\end{array}\right],
\\ L_3 &= \left[\begin{array}{cccc}
0 & 0             & 0     & (v-k)(w-1)\\
0 & (k-\mu)(w-2)  & v-k   & (v+\mu-2k)(w-2)\\
0 & (k-\lambda)(w-1) & 0     &(v+\lambda-2k)(w-1)\\
1 & (k-\nu)(w-2)  & v-k-1 & (v+\nu-2k)(w-2)\\
\end{array}\right].\\
\end{aligned}\]
We note here that, while $\mu$ and $\nu$ alone are not intersection numbers of our association scheme, we will often allow for the slight abuse of terminology and include both of these as parameters of an LSSD. The first and second eigenmatrices are given as:
\begin{equation}\label{3classP}
P = \left[\begin{array}{cccc}
1&k(w-1)&v-1&(v-k)(w-1)\\
1&\sqrt{k-\lambda}(w-1)&-1&-\sqrt{k-\lambda}(w-1)\\
1&-\sqrt{k-\lambda}&-1&\sqrt{k-\lambda}\\
1&-k&v-1&k-v\\
\end{array}\right]
\end{equation}
\begin{equation}\label{3classQ}Q = \left[\begin{array}{cccc}
1 & v-1 & (w-1)(v-1) & w-1\\
1 & \frac{v-k}{\sqrt{k-\lambda}} & -\frac{v-k}{\sqrt{k-\lambda}} & -1\\
1 & -1 & 1-w & w-1\\
1 & \frac{-k}{\sqrt{k-\lambda}} & \frac{k}{\sqrt{k-\lambda}} & -1\\
\end{array}\right].
\end{equation}
Finally, we use this matrix $Q$ to calculate our Krein parameters using standard techniques (see \cite{Brouwer1989}). Defining $L_i^* = [q_{ij}^k]_{k,j}$ similar to before, we find
\[\scalebox{.9}{$L_1^* = \left[\begin{array}{cccc}
	0 & v-1 & 0 & 0\\
	1 & \frac{(1-w)(2k-v)+(v-2)s}{ws} &\frac{(w-1)\left(s(v-2)+(2k-v)\right)}{ws}& 0\\
	0&\frac{s(v-2)+2k-v}{ws} & \frac{s(w-1)(v-2)-(2k-v)}{ws}& 1\\
	0& 0 & v-1& 0\\
	\end{array}\right],\quad L_3^*=\left[\begin{array}{cccc}
	0&0 &0 & w-1\\
	0& 0& w-1& 0\\
	0& 1& w-2& 0\\
	1& 0& 0& w-2\\
	\end{array}\right].$}\]
Note that our $Q$-polynomial property means that $L_1^*$ must be irreducible tridiagonal (\cite[Prop.~2.7.1(i')]{Brouwer1989}). While the tridiagonal property is clear from the above matrix, the irreducible property requires $s(v-2)>v-2k$, for which $k>1$ is necessary (and sufficient). Thus, while we may fulfill the LSSD conditions using \emph{degenerate design parameters}\index{linked system of symmetric designs!degenerate} $(v,1,0)$, these will not satisfy our $Q$-polynomial property. This is one reason why we will ignore this case for much of our discussion.

Finally, the final column of $L_3^*$ tells us that $E_3\circ E_3\in\left<E_0,E_3\right>$; that is, our scheme is $Q$-antipodal. We henceforth use the term \emph{linked system of symmetric designs}\index{linked system of symmetric designs} to refer to either the graph $\Gamma$ or to the association scheme it generates as in Theorem \ref{Qpoly}.
\subsection{Bounds on number of fibers}
\index{Noda bound}
One central question in the study of linked systems of symmetric designs is determining the maximum number of fibers one may use to build LSSDs. Theorem 2 of \cite{Noda1974} provides us with the main non-trivial bound known to date. In this paper, Noda proves
\begin{equation*}\resizebox{\textwidth}{!}{
	$(w-1)\left[(k-2)\lambda\binom{k}{3}-(v-2)\left[(v-k)\binom{\nu}{3} + k\binom{\mu}{3}\right]\right]\leq(v-2)\left[(v-1)\binom{\lambda}{3}+\binom{k}{3}-\left[(v-k)\binom{\nu}{3} + k\binom{\mu}{3}\right]\right]$
}\end{equation*}
with equality if and only if a pair $(X_1,X_2\cup X_3\cup\dots\cup X_w)$ forms a 3-design. If we restrict ourselves to the case of $\mu$-heavy LSSDs, this results in the following theorem
\begin{thm}
	Suppose there exists a $LSSD(v,k,\lambda;w)$ with $\nu = \frac{k(k-s)}{v}$ and $\mu = \nu+s$. Then if $k>\frac{v}{2}$,
	\begin{equation}\label{Noda}w \leq \frac{(v-2)\sqrt{k-\lambda}}{2k-v}+1.\end{equation}
\end{thm}
Note the condition becomes vacuous when $2k<v$, thus the bound only applies when $(2k-v)(\mu-\nu)>0$.

Examining the Krein parameters which arise (in particular those listed in the discussion following Theorem \ref{Qpoly}), we find that the only non-trivial condition from the Krein conditions is $q^1_{11}\geq 0$. While constructing the various different possible configurations of LSSDs using design parameters $(16,6,2)$ \cite{Mathon1981}, Mathon also shows that the condition $q^1_{11}\geq 0$ is equivalent to the Noda bound (seen above). In fact, we have
\begin{equation}\label{q111}
q_{11}^1 =\frac{(1-w)(2k-v)+(v-2)s}{ws}.
\end{equation}
Thus, as long as $k<\frac{v}{2}$ we may rearrange the terms to get
\[w\leq \frac{(v-2)s}{2k-v}+1.\]
This is, of course, equivalent to the previous bound since $s = \sqrt{k-\lambda}$. Note that, as before, we only arrive at this bound if $k<\frac{v}{2}$ for the $\mu$-heavy design.

The final bound we will consider is one which does not require $k<\frac{v}{2}$ for the $\mu$-heavy design. In \cite{Martin2007}, Martin, Muzychuk, and Williford use the absolute bound to bound the number of fibers, relying on the $Q$-polynomial structure (particularly that $q_{11}^3 = 0$ and $q_{11}^2 >0$). Let $m_i = \text{rank}(E_i)$. We know from our $Q$ matrix that $m_2 = (w-1)m_1$. Further,
\begin{equation}\label{abslssd}E_1\circ E_1 = \frac{1}{\vert X\vert}\left(q_{11}^0E_0 + q_{11}^1E_1+q_{11}^2E_2\right).\end{equation}
The bound itself will depend on whether or not $q^1_{11}$ is non-zero. Thus we must split our derivation into two cases. In what follows, we derive the bound for the case $q^1_{11}>0$ as well as examine further the case of $q^1_{11}=0$, deriving the absolute bound for this case and pairing this with tightness in the Noda bound to further restrict our parameters.

First consider the case when $q^1_{11}>0$. Here, we find the rank of the right hand side of \eqref{abslssd} is $m_2+m_1+1$ while the rank on the left is no larger than $\frac{1}{2}m_1\left(m_1+1\right)$. Thus we must have
\[\begin{aligned}
m_2+m_1+1&\leq \frac{1}{2}m_1(m_1+1)\\
(w-1)m_1 &\leq \frac{m_1^2-m_1}{2}-1,\\
w&\leq \frac{m_1-1}{2}-\frac{1}{m_1}.
\end{aligned}\]
Since $m_1 = v-1$, this gives $w \leq \frac{v}{2}-\frac{1}{v-1}$. Further, since $v>2$ this results in $w\leq\frac{v-1}{2}$. 

Now consider the case where $q_{11}^1=0$. The only change is that the right hand side of \eqref{abslssd} now has rank $m_2+1$. Thus our bound gives instead $w\leq \frac{v+1}{2}$. However, \eqref{q111} gives us an expression for $q^1_{11}$. Thus, if $q^1_{11}=0$, we must have $k>\frac{v}{2}$ and $w = \frac{(v-2)s}{2k-v}+1$. Using this value for $w$ in our bound $w\leq \frac{v+1}{2}$ results in the inequality $\frac{(v-2)s}{2k-v}+1\leq \frac{v+1}{2}$ or equivalently, 
\[2s\leq (2k-v) + \frac{2s}{v-1}.\]
Since $0<2s<k<v-1$, we must then have $2s\leq (2k-v)$. Squaring both sides then gives
$4(k-\lambda)\leq 4k^2-4kv+v^2$. Finally, using Equation \eqref{sym:2} we have
\[\begin{aligned}
4(k-\lambda)\leq v.
\end{aligned}\]

In summary, if $q^1_{11}>0$, the absolute bound tells us $w\leq \frac{v-1}{2}$ for both $\mu$-heavy and $\nu$-heavy designs independent of the size of $k$. Additionally, if $q^1_{11}=0$ (i.e.\ the Noda bound is tight), we must have $4(k-\lambda)\leq v$. Recall this is exactly the case where Noda showed $(X_1,X_2\cup X_3\cup\dots\cup X_w)$ forms a 3-design. There is only one known family of constructions which achieve the Noda bound. For this construction (see Section \ref{kerdock}) $v = 4(k-\lambda)$. Further, these designs have parameters belonging to the Menon family --- the only possible family for which this $v=4(k-\lambda)$.

We finish this section by noting that, in either case ($q^1_{11}>0$ and $q^1_{11}=0$), the exact bound arising from the absolute bound is non-integral assuming $m_1>2$. Thus, while we reduce these bounds to $w\leq\frac{v-1}{2}$ and $w\leq\frac{v+1}{2}$, the actual absolute bound will never be tight.

\section{Linked simplices}
\label{linked simplices}
In this section, we will write $\left\{b_j\right\}$ for the set $\left\{b_1,\dots,b_v\right\}$ for both sets of points and sets of blocks. For our purposes a \emph{regular simplex}\index{regular simplex} will be taken to be a set of $v$ unit vectors spanning $\mathbb{R}^{v-1}$ with the property that the inner product of any pair of distinct vectors is $-\frac{1}{v-1}$. Let $\mathcal{A} = \left\{a_i\right\}$ and $\mathcal{B} = \left\{b_j\right\}$ be two regular simplices in $\mathbb{R}^{v-1}$. We say $\mathcal{A}$ and $\mathcal{B}$ are \emph{linked}\index{linked simplices} if there exist two real numbers $\gamma$ and $\delta$ such that for all $1\leq i,j\leq v$, $\left<a_i,b_j\right>\in \left\{\gamma,\delta\right\}$. (Note that, here, ``linked" always implies regular.) Following the same abuse of terminology as we did with equiangular lines in Chapter \ref{psdcone}, we will refer to the real numbers $\gamma$ and $\delta$ as the ``angles" of our linked simplices. Extending this, given regular simplices $\mathcal{A}_1,\dots,\mathcal{A}_w$ in $\mathbb{R}^{v-1}$ we say $\left\{A_1,\dots,A_w\right\}$ is a \emph{set of $w$ linked simplices} if any two of them are linked with the same angles $\gamma$ and $\delta$. The next few theorems establish an equivalence between these objects and LSSDs. 
\begin{thm}
	\label{Association2Simplex}
	Consider a $LSSD(v,k,\lambda;w)$ with Bose-Mesner algebra $\mathbb{A}$. The first idempotent $E_1$ in a $Q$-polynomial ordering of $\mathbb{A}$, appropriately scaled, is the Gram matrix of a set of $w$ linked simplices. In the case $w=2$, $E_2$ scaled similarly is also the Gram matrix of a second set of two linked simplices.
\end{thm}
\begin{proof}
	Let $(X,\mathcal{R})$ be a $LSSD(v,k,\lambda;w)$ with Bose-Mesner algebra $\mathbb{A}$. Let $\left\{A_i\right\}$ and $\left\{E_j\right\}$ be the bases of Schur and matrix idempotents respectively. We have from \eqref{PQmat}
	\[E_j = \frac{1}{\vert X\vert}\sum Q_{ij}A_i.\]
	As $E_j$ is an idempotent, $E_j$ is a positive semidefinite matrix with rank $Q_{0j}$. Therefore using \eqref{3classQ},
	\[G = \frac{vw}{v-1}E_1 = A_0 + \frac{v-k}{(v-1)\sqrt{k-\lambda}}A_1 -\frac{1}{v-1}A_2 -\frac{k}{(v-1)\sqrt{k-\lambda}}A_3\]
	is positive semidefinite with 1 on the main diagonal. Given that $Q_{01} = v-1$, $G$ is the Gram matrix of a set $Y$ of $vw$ vectors in $\mathbb{R}^{v-1}$. Further there are only three possible inner products among distinct vectors of $Y$ given by
	\[\alpha_1 = \frac{v-k}{(v-1)\sqrt{k-\lambda}},\qquad\alpha_2 =-\frac{1}{v-1},\qquad\alpha_3 = -\frac{k}{(v-1)\sqrt{k-\lambda}}.\]
	Since $A_2$ has the form $I_w\otimes J_v$, encoding adjacency in the complete graphs within fibers, our vectors form a set of $w$ linked simplices in $\mathbb{R}^{v-1}$ with $\gamma = \alpha_1$ and $\delta = \alpha_3$ as inner products between vectors in distinct simplices.\par
	Similarly we have
	\[\begin{aligned}
	G' = \frac{vw}{v-1}E_2 &= \left(A_0 + \frac{-k}{(v-1)\sqrt{k-\lambda}}A_1 -\frac{1}{v-1}A_2 +\frac{v-k}{(v-1)\sqrt{k-\lambda}}A_3\right)\\
	\end{aligned}\]
	forcing $G'$ also to be the Gram matrix of a set of vectors coming from $w$ distinct simplices. However, the rank of $E_2$ is $(w-1)(v-1)$ and thus these simplices are full-dimensional, hence linked, only when $w=2$. This means any pair of fibers from our LSSD will give us another set of linked simplices with inner products $-\alpha_1$ and $-\alpha_3$. This corresponds to choosing one of the two simplices and replacing each $x$ in that simplex by $-x$.
\end{proof}
This tells us that every LSSD gives rise to a set of linked simplices. Before proving the converse, we first prove a lemma arising from the observation that a regular simplex is an equiangular tight frame\index{equiangular tight frame}, that is, a set of vectors $\left\{v_j\right\}_{j\in J}\subset V$ for which there exists a constant $A$ so that $v = \frac{1}{A}\sum_{j\in J} \left<v,v_j\right>v_j$ for every vector $v\in V$. We will instead use the equivalent Plancherel definition $\left<v,w\right> = \frac{1}{A}\sum_{j\in J}\left<v,v_j\right>\left<v_j,w\right>$ for all $v,w\in V$ (see \cite{Jasper2014}).
\begin{lem}
	\label{simplextobasis}
	Let $\left\{a_i\right\}$ be a regular simplex in $\mathbb{R}^{v-1}$ and let $x,y\in \mathbb{R}^{v-1}$. Then
	\[\begin{aligned}
	\sum_i \left<a_i,x\right>\left<a_i,y\right> = \frac{v}{v-1}\left<x,y\right>.
	\end{aligned}\]
\end{lem}
\begin{proof}
	For a vector $x$, let $x(i)$ denote the $i^\text{th}$ entry of $x$. For each $1\leq i\leq v$, define $\alpha_i\in \mathbb{R}^{v}$ as the unit vector
	\[\alpha_i = \sqrt{\frac{v-1}{v}}\left[a_i(1),a_i(2),\dots,a_i(v-1),\frac{1}{\sqrt{v-1}}\right].\]
	For $i\neq i^\prime$,
	\[\left<\alpha_i,\alpha_{i^\prime}\right> = \frac{v-1}{v}\left(\left<a_i,a_{i^\prime}\right> + \frac{1}{v-1}\right)=0.\]
	Then $\left\{\alpha_i\right\}$ forms an orthonormal basis for $\mathbb{R}^v$. Now define $\chi,\psi\in\mathbb{R}^v$ as:
	\[\chi = \left[x(1),x(2),\dots,x(v-1),0\right],\qquad\psi = \left[y(1),y(2),\dots,y(v-1),0\right].\]
	Then for each $i$,
	\[\begin{aligned}
	\left<\alpha_i, \chi\right> = \sqrt{\frac{v-1}{v}}\left<a_i,x\right>,\qquad\left<\alpha_i, \psi\right> = \sqrt{\frac{v-1}{v}}\left<a_i,y\right>
	\end{aligned}\]
	giving us
	\[\left<x,y\right> =\left<\chi,\psi\right>=\sum_i\left<\alpha_i, \chi\right>\left<\alpha_i, \psi\right> =\frac{v-1}{v}\sum_i\left<a_i,x\right>\left<a_i,y\right>.\qedhere\]
\end{proof}
Using this lemma, we now seek to build a LSSD with $w$ fibers from a set of $w$ linked simplices. We first provide a construction of the graph $\Gamma$ and then split the verification into two parts: first that $\Gamma$ restricted to a pair of fibers represents a symmetric design, and second that the constants $\mu$ and $\nu$ given by \eqref{munudef} are well-defined. Clearly, we need only consider three fibers in the proofs to follow; the arguments extend to $w$ fibers.
\begin{thm}
	\label{2design}
	Let $\left\{a_i\right\}$ and $\left\{b_j\right\}$ be linked simplices in $\mathbb{R}^{v-1}$ with inner products $\gamma$ and $\delta$. For each $j$, let $B_j = \left\{a_i:\left<a_i,b_j\right> = \gamma\right\}$. Then $(\left\{a_i\right\},\left\{B_j\right\})$ is a symmetric 2-design.
\end{thm}
\begin{proof}
	First we must prove that each block contains a constant number of points. Let $1\leq j\leq v$ be fixed and define $k_j = \vert B_j\vert$. Since the set $\left\{a_i\right\}$ of vectors form a regular simplex, the centroid of those vectors must be the origin. Then, $\sum_i\left<b_j,a_i\right> = \left<b_j,\sum_ia_i\right> = 0$ giving us the equation $k_j\gamma + (v-k_j)\delta = 0$. Solving this for $k_j$ gives $k_j = \frac{\delta v}{\gamma-\delta}$, independent of $j$. Now we will show that any pair of blocks have a constant number of points in common; swapping roles this gives that any pair of points is contained in a constant number of blocks. Fix $1\leq s,t\leq v$ so that $b_s$ and $b_t$ are two distinct vectors in $\left\{b_j\right\}$ with corresponding blocks $B_s$ and $B_t$ respectively. Define $\lambda_{s,t} = \vert B_s\cap B_t\vert$ and
	\[x_s = \left[\left<a_1,b_s\right>,\left<a_2,b_s\right>,\dots,\left<a_v,b_s\right>\right],\qquad x_t = \left[\left<a_1,b_t\right>,\left<a_2,b_t\right>,\dots,\left<a_v,b_t\right>\right].\]
	Recalling that $k\gamma+(v-k)\delta = 0$,
	\[\left<x_s,x_t\right> = \lambda_{s,t} \gamma^2 + 2(k-\lambda_{s,t})\gamma\delta + (v-2k+\lambda_{s,t})\delta^2=\lambda_{s,t}\left(\delta-\gamma\right)^2 - v\delta^2.\]
	We may instead apply Lemma \ref{simplextobasis} to get
	\[\left<x_s,x_t\right> = \sum_i \left<a_i,b_s\right>\left<a_i,b_t\right>=\frac{v}{v-1}\left<b_s,b_t\right>=-\frac{v}{(v-1)^2}.\]
	Equating these two values gives us
	\[\lambda_{s,t} =\frac{v\delta^2}{(\delta-\gamma)^2}-\frac{v}{(v-1)^2(\gamma-\delta)^2}.\]
	The quantity on the right is independent of $s$ and $t$ and therefore $\lambda_{s,t}$ does not depend on $s$ and $t$. This tells us our collection of blocks forms a $2$-design with the above values for $k$ and $\lambda$.
\end{proof}
As both $k$ and $\lambda$ are integers, this gives us restrictions on which inner products are allowed. We solve the system
\[\begin{aligned}
k\gamma + (v-k)\delta &= 0,\\
\lambda(\delta-\gamma)^2 - v\delta^2 &= -\frac{v}{(v-1)^2}\\
\end{aligned}\]
to find that $\delta^2 =\frac{k}{(v-1)(v-k)}$. Using \eqref{sym:2}, this simplifies to
\begin{equation}\label{angles}
\delta = \pm\frac{k}{(v-1)\sqrt{k-\lambda}},\qquad\gamma = \mp\frac{v-k}{(v-1)\sqrt{k-\lambda}}.
\end{equation}
These match the previously determined entries of $E_1$ and $E_2$ corresponding to the first and third relations. Our next theorem concerns the existence of $\mu$ and $\nu$, which arise between triples of fibers.
\begin{thm}
	\label{3fibers}
	Let $\left\{a_i\right\}$, $\left\{b_i\right\}$, and $\left\{c_i\right\}$ be three linked simplices in $\mathbb{R}^{v-1}$ with inner products $\gamma$ and $\delta$ as before. For each $1\leq j,k\leq v$, let $B_j = \left\{a_i:\left<a_i,b_j\right> = \gamma\right\}$ and $C_k = \left\{a_i:\left<a_i,c_k\right> = \gamma\right\}$. Then there exists integers $\mu$ and $\nu$ such that 
	\[\vert B_j\cap C_k\vert = \begin{cases}
	\mu  & \left<b_j,c_k\right> = \gamma\\
	\nu  & \left<b_j,c_k\right> = \delta\\
	\end{cases}\]
	where $\mu$ and $\nu$ are independent of our choice of $j$ and $k$.
\end{thm}
\begin{proof}
	We follow a similar method of calculating an inner product in two ways, then equating the results. Fix $0\leq i,j\leq v$ and let $\eta_{i,j} = \vert B_i\cap C_j\vert$. Define
	\[x_i = \left[\left<a_1,b_i\right>,\left<a_2,b_i\right>,\dots,\left<a_v,b_i\right>\right],\qquad x_j = \left[\left<a_1,c_j\right>,\left<a_2,c_j\right>,\dots,\left<a_v,c_j\right>\right].\]
	Then we have $\left<x_i,x_j\right> =\eta_{i,j}(\gamma-\delta)^2 - v\delta^2$ and using Lemma \ref{simplextobasis},
	\[\left<x_i,x_j\right> = \sum_\ell \left<a_\ell,b_i\right>\left<a_\ell,c_j\right>=\frac{v}{v-1}\left<b_i,c_j\right>.\]
	Equating these two values and solving for $\eta_{i,j}$ gives us
	\[\eta_{i,j} = \frac{1}{(\gamma-\delta)^2}\left(v\delta^2 + \frac{v}{v-1}\left<b_i,c_j\right>\right).\]
	While the right side is not independent of $i$ and $j$ as we saw in the previous theorem, it is only dependent on the value of $\left<b_i,c_j\right>$. Using $\nu$ and $\mu$ for $\eta_{i,j}$ when $\left<b_i,c_j\right>$ is $\delta$ and $\gamma$ respectively, we have
	\[\begin{aligned}
	\nu&=\frac{v}{(\gamma-\delta)^2}\left(\frac{\delta^2(v-1)^2+\delta(v-1)}{(v-1)^2}\right),\\
	\mu&=\frac{v}{(\gamma-\delta)^2}\left(\frac{\delta^2(v-1)^2+\gamma(v-1)}{(v-1)^2}\right)= \nu + \frac{v}{(\gamma-\delta)(v-1)}.
	\end{aligned}\]
	Using equation \ref{angles}, we find
	\[\begin{aligned}
	\gamma-\delta&=\mp\frac{v}{(v-1)\sqrt{k-\lambda}},
	\end{aligned}\]
	giving us that
	\[\begin{aligned}
	\nu&= \frac{k(k\pm \sqrt{k-\lambda})}{v},\\
	\mu&=\nu\mp\sqrt{k-\lambda}.\qedhere
	\end{aligned}\]
\end{proof}
Note that, since $\mu$ and $\nu$ are both cardinalities of sets, any time we find non-integral values for $\mu$ and $\nu$ we can conclude that the hypothesized set of linked simplices does not exist. This brings us to the main theorem of this section (cf.\ \cite[Theorem~2.6(1)]{Suda2011}).
\lssdsimpequiv
\begin{proof}
	Theorem \ref{Association2Simplex} tells us that given any $LSSD(v,k,\lambda;w)$ we can always build a set of $w$ linked simplices using a scaled version of the first idempotent as the Gram matrix. For the converse, let $\left\{X_1,X_2,\dots,X_w\right\}$ be a set of $w$ linked simplices with inner products $\gamma>\delta$. Define a graph $\Gamma$ on vertex set $\bigcup_i X_i$ where $x\in X_j$ and $y\in X_\ell$ ($j\neq \ell$) are adjacent if and only if $\left<x,y\right> = \gamma$. Then $\Gamma$ is a multipartite graph with $w$ fibers. Theorem \ref{2design} tells us that the induced graph between a pair of fibers is a symmetric $2$-design. Theorem $\ref{3fibers}$ shows that given any pair of vertices in distinct fibers $x\in X_i$ and $y\in X_\ell$,
	\[\vert\Gamma(x)\cap\Gamma(y)\cap X_j\vert = \begin{cases}
	\mu & x\sim y\\
	\nu & x\not\sim y
	\end{cases}\]
	where $X_j$ is any third fiber. As we assumed $\gamma>\delta$, this also provides that $\mu>\nu$. Therefore $\Gamma$ is a $\mu$-heavy LSSD and adjacency in $\Gamma$ is the first relation of our proposed association scheme. The third relation (the $\nu$-heavy LSSD) is built using inner product $\delta$ to define adjacency.
\end{proof}
\subsection{A geometric classification}
As every association scheme has relations corresponding to complementary $\mu$-heavy and $\nu$-heavy LSSDs, it becomes useful to differentiate between LSSDs where $P_{01}$ (the valency of the $\mu$-heavy design) is greater than $P_{03}$ (the valency of the $\nu$-heavy design) or vice versa. Noting that the $\mu$-heavy LSSD gives the nearest neighbor graph\index{nearest neighbor graph}\footnote{The ``nearest neighbor graph" here refers to the basis relation in our association scheme corresponding to largest inner product not equal to 1. By definition, this is $R_1$ under the natural ordering of relations for a $Q$-polynomial scheme.} of our association scheme and thus the only positive inner product apart from 1, we classify a LSSD as \emph{optimistic}\index{linked system of symmetric designs!optimistic} if $P_{01}>P_{03}$ (and thus there are more positive inner products than negative). Likewise we classify the opposite case as \emph{pessimistic}\index{linked system of symmetric designs!pessimistic}. While this classification helps designate whether the set of linked simplices has mostly positive or mostly negative inner products between distinct simplices, we also note that every known non-degenerate example of a LSSD is optimistic. At the parameter level, an LSSD is optimistic if $(2k-v)(\mu-\nu)>0$ and pessimistic if $(2k-v)(\mu-\nu)<0$. The following table lists the possibilities.
\[\begin{tabular}{c|c|c}
& $2k>v$ & $2k<v$\\\hline
$\mu$-heavy & optimistic & pessimistic\\\hline
$\nu$-heavy & pessimistic & optimistic\\
\end{tabular}\]
Motivated by the natural ordering of relations, we will adopt the convention of focusing on the $\mu$-heavy LSSD. This forces us to allow for $k>\frac{v}{2}$.
\section{Connections to other Euclidean structures}
In the previous section, we developed the equivalence between LSSDs and linked simplices using the columns of the first idempotent as a spherical code. We now explore similar structures which can be built using combinations of these idempotents, though we will not always be able to reverse these constructions as we did with linked simplices. Recall that
\[E_j = \frac{1}{\vert X\vert}\sum_{i=0}^d Q_{ij}A_i\]
and the rank of $E_j$ is given by $Q_{0j}$. By considering non-negative linear combinations of these idempotents, we construct Gram matrices of systems of vectors with desirable properties. As we are interested in low rank Gram matrices, we will only consider non-negative combinations of two or three of these idempotents, avoiding $E_2$ as this has rank $(w-1)(v-1)$. Before moving to the examples, we note that the matrix $\alpha E_0 + \beta E_1+\gamma E_3$ is expressible as $\sum_i y_i A_i$ with the following values for $y_i$:
\begin{equation}\label{coeffs}
\begin{split}
y_0 &= \frac{1}{vw}(\alpha + (v-1)\beta + (w-1) \gamma),\\
y_1 &= \frac{1}{vw}(\alpha + \frac{v-k}{\sqrt{k-\lambda}}\beta -\gamma),\\
y_2 &= \frac{1}{vw}(\alpha - \beta +(w-1)\gamma),\\
y_3 &= \frac{1}{vw}(\alpha -\frac{k}{\sqrt{k-\lambda}}\beta - \gamma).
\end{split}
\end{equation}

\subsection{Equiangular lines}
Recall that a set of \emph{equiangular lines}\index{equiangular lines} in dimension $n$ with ``angle" $0<\alpha<1$ may be considered a set of unit vectors in $\mathbb{R}^n$ such that the inner product between any distinct vectors has a fixed magnitude $\alpha$. Thus we consider a set of equiangular lines equivalent to a Gram matrix with only one magnitude off the main diagonal. Our task of constructing equiangular lines from an LSSD then reduces to finding a low rank matrix inside the Bose-Mesner algebra for which the off diagonal entries ($y_1$, $y_2$, and $y_3$ in \eqref{coeffs}) all have the same norm.

We now show that we may find a set of equiangular lines from any LSSD whose size and dimension depend on the number of fibers used in the construction. This is a generalization of the construction of de Caen's \cite{deCaen2000} which uses the idempotents Cameron-Seidel scheme to build a family of equiangular lines with $\frac{2}{9}(d+1)^2$ lines in $\mathbb{R}^d$. Let $\BMA$ be the Bose-Mesner algebra of an LSSD with $Q$-polynomial ordering $E_0,\dots,E_3$ and relations ordered naturally. Consider the matrix
\[G = vw(\alpha E_0 + \beta E_1 + \gamma E_3)\]
for $\alpha,\beta,\gamma\geq 0$. This is a $vw\times vw$ matrix with rank at most $v+w-1$. The off-diagonal entries are given by $y_1$, $y_2$, $y_3$ in \eqref{coeffs}. In order to obtain a Gram matrix for a set of equiangular lines, we must have a constant positive value $c$ such that
\[\begin{aligned}
\Big| \alpha + \beta\left(\frac{v-k}{\sqrt{k-\lambda}}\right) - \gamma\Big| =\Big| \alpha - \beta +(w-1)\gamma\Big| =\Big| \alpha - \beta\left(\frac{k}{\sqrt{k-\lambda}}\right) - \gamma\Big| = c.
\end{aligned}\]
First note that, if $\beta = 0$ then the above equations imply $\vert \alpha-\gamma\vert = \vert\alpha+(w-1)\gamma\vert$ which is impossible unless $\gamma = 0$. Since the case $\beta=\gamma=0$ results in the rank 1 matrix $J$, we are not interested in this case. Thus we may assume $\beta>0$. This implies $\alpha+\beta\left(\frac{v-k}{\sqrt{k-\lambda}}\right) - \gamma >\alpha - \beta\left(\frac{k}{\sqrt{k-\lambda}}\right)-\gamma$ and we must have
\[c=\alpha+\beta\left(\frac{v-k}{\sqrt{k-\lambda}}\right) - \gamma = -\left[\alpha-\beta\left(\frac{k}{\sqrt{k-\lambda}}\right) - \gamma\right] .\]
This tells us that
\[\beta = \frac{2c\sqrt{k-\lambda}}{v}\qquad \text{and}\qquad \alpha - \gamma = c\left(\frac{2k-v}{v}\right).\]
Here we have one final choice: the sign of $\alpha-\beta+(w-1)\gamma$. Substituting in our value for $\beta$, we find that 
\[\alpha+(w-1)\gamma = c\left(\frac{2\sqrt{k-\lambda}\pm v}{v}\right).\]
Since we must have $2\sqrt{k-\lambda}<v$, we know that choosing the minus on the right hand side would make the entire side negative. However $\alpha$, $\gamma$, and $(w-1)$ are all positive so this is not possible. Therefore we must use the $+$, giving
\[\gamma=2c\left(\frac{v-k + \sqrt{k-\lambda}}{vw}\right),\qquad
\alpha = c\left(\frac{v+2\sqrt{k-\lambda} - (w-1)\left(v-2k\right)}{vw}\right).\]
Our final constraint is that the main diagonal of $G$ is equal to 1. Setting $y_0\equiv 1$ in \ref{coeffs}, we find
\[\begin{aligned}
c=\frac{1}{2\sqrt{k-\lambda}-1}.
\end{aligned}\]
Scaling by $vw$ for convenience, this gives us the final values:
\[\begin{aligned}vw\alpha&=\frac{v+2\sqrt{k-\lambda} - (w-1)(v-2k)}{(2\sqrt{k-\lambda}-1)},\\
vw\beta&=\frac{2w\sqrt{k-\lambda}}{2\sqrt{k-\lambda}-1},\\
vw\gamma&=\frac{2v-2k+2\sqrt{k-\lambda}}{2\sqrt{k-\lambda}-1},\\
\end{aligned}\]
with inner product $\frac{1}{2\sqrt{k-\lambda} -1}$. It is easy to see that $\beta$ and $\gamma$ will always be positive, thus the rank of our matrix will always be at least $v+w-2$. In the optimistic case, the same holds for $\alpha$ since $v-2k<0$. However, in the pessimistic case, it is possible that $w=\frac{2(k+s)}{v-2k}+2$, resulting in $\alpha = 0$. More broadly, whenever our LSSD is pessimistic, we must have $w\leq\frac{2(k+s)}{v-2k}+2$ where equality implies the rank of our matrix is $v+w-2$. This gives the following generalization of de Caen's construction:
\begin{thm}
	Let $(X,\cR)$ be the association scheme arising from a $LSSD(v,k,\lambda;w)$. If either $(X,\cR)$ is optimistic, or $w\leq2+\frac{2(k+s)}{v-2k}$ then we can build a set of $vt$ equiangular lines in $\mathbb{R}^{v+t-1}$ for any $1\leq t\leq w$. In the pessimistic case with $w> 2+\frac{2(k+s)}{v-2k}$, we can achieve the construction for any $t\leq 2+\frac{2(k+s)}{v-2k}$.\qed
\end{thm}
\subsection{Real mutually unbiased bases}
\label{bases}
Recall that a set of real mutually unbiased bases (MUBs)\index{real mutually unbiased bases} is a set of orthonormal bases of $\mathbb{R}^n$ such that any pair of vectors from distinct bases has inner product equal to one of $\pm\frac{1}{\sqrt{n}}$. We may build structures close to this by taking $x_2 = x_3 = 0$ and $x_0 = x_1 = w$. This gives us the Gram matrix
\[G = A_0 + \frac{v-k+\sqrt{k-\lambda}}{v\sqrt{k-\lambda}}A_1 -\frac{k-\sqrt{k-\lambda}}{\sqrt{k-\lambda}}A_3\]
of a set of $w$ orthonormal bases where two vectors from distinct bases have one of two inner products;
\[\beta_1 = \frac{v-k+\sqrt{k-\lambda}}{v\sqrt{k-\lambda}},\qquad
\beta_2= -\frac{k-\sqrt{k-\lambda}}{v\sqrt{k-\lambda}}.\]
Of particular interest is the case when $\vert \beta_1\vert = \vert\beta_2\vert;$ this is precisely when our construction gives a set of mutually unbiased bases. This will be discussed in greater detail in Section \ref{LSSDMUBs}.
\section{Known infinite families}
This section discusses two families of LSSDs, one trivial and one quite central to the subject. In each case we introduce the family and provide the parameters of the association schemes.
\subsection{Degenerate case}\index{linked system of symmetric designs!degenerate}
\label{degenerate}
We first examine the case when the $Q$-polynomial structure fails as seen in the discussion of Theorem $\ref{Qpoly}$. Arguably the most interesting property of this scheme is that there is no bound on $w$. In fact, for any choice of $v,w>0$, we can build a LSSD with $w$ fibers by building a set of cliques, each of size $w$, where each clique contains a single vertex from each fiber. This gives a $\mu$-heavy LSSD with the complement giving us the $\nu$-heavy LSSD. Below is a representation of the complementary pairs LSSD(4,1,0;3) and LSSD(4,3,2;3).
\[\begin{tikzpicture}[scale = .6,node distance=3cm,
thin,main node/.style={circle,fill=black,scale = .5}]
\node at (0,5) {$\mu$-heavy LSSD};
\def\x{1.5};
\node[main node,color = red] (11) at (-3*\x,1) {};
\node[main node,color = blue] (12) at (-2.5*\x,2) {};
\node[main node,color = green] (13) at (-2*\x,3) {};
\node[main node] (14) at (-1.5*\x,4) {};
\node[main node] (24) at (1.5*\x,4) {};
\node[main node,color = green] (23) at (2*\x,3) {};
\node[main node,color = blue] (22) at (2.5*\x,2) {};
\node[main node,color = red] (21) at (3*\x,1) {};
\node[main node] (34) at (0*\x,-1) {};
\node[main node,color = green] (33) at (0*\x,-2) {};
\node[main node,color = blue] (32) at (0*\x,-3) {};
\node[main node,color = red] (31) at (0*\x,-4) {};

\draw [-] (11) -- (21) -- (31) -- (11);
\draw [-] (12) -- (22) -- (32) -- (12);
\draw [-] (13) -- (23) -- (33) -- (13);
\draw [-] (14) -- (24) -- (34) -- (14);
\end{tikzpicture}\qquad
\begin{tikzpicture}[scale = .6,node distance=3cm,
thin,main node/.style={circle,fill=black,scale = .5}]
\node at (0,5) {$\nu$-heavy LSSD};
\def\x{1.5};
\node[main node,color = red] (11) at (-3*\x,1) {};
\node[main node,color = blue] (12) at (-2.5*\x,2) {};
\node[main node,color = green] (13) at (-2*\x,3) {};
\node[main node] (14) at (-1.5*\x,4) {};
\node[main node] (24) at (1.5*\x,4) {};
\node[main node,color = green] (23) at (2*\x,3) {};
\node[main node,color = blue] (22) at (2.5*\x,2) {};
\node[main node,color = red] (21) at (3*\x,1) {};
\node[main node] (34) at (0*\x,-1) {};
\node[main node,color = green] (33) at (0*\x,-2) {};
\node[main node,color = blue] (32) at (0*\x,-3) {};
\node[main node,color = red] (31) at (0*\x,-4) {};

\draw [-] (11) -- (22);
\draw [-] (11) -- (23);
\draw [-] (11) -- (24);
\draw [-] (11) -- (32);
\draw [-] (11) -- (33);
\draw [-] (11) -- (34);

\draw [-] (12) -- (21);
\draw [-] (12) -- (23);
\draw [-] (12) -- (24);
\draw [-] (12) -- (31);
\draw [-] (12) -- (33);
\draw [-] (12) -- (34);

\draw [-] (13) -- (32);
\draw [-] (13) -- (31);
\draw [-] (13) -- (34);
\draw [-] (13) -- (22);
\draw [-] (13) -- (21);
\draw [-] (13) -- (24);

\draw [-] (14) -- (22);
\draw [-] (14) -- (23);
\draw [-] (14) -- (21);
\draw [-] (14) -- (32);
\draw [-] (14) -- (33);
\draw [-] (14) -- (31);

\draw [-] (21) -- (32);
\draw [-] (21) -- (33);
\draw [-] (21) -- (34);
\draw [-] (22) -- (31);
\draw [-] (22) -- (33);
\draw [-] (22) -- (34);
\draw [-] (23) -- (32);
\draw [-] (23) -- (31);
\draw [-] (23) -- (34);
\draw [-] (24) -- (32);
\draw [-] (24) -- (33);
\draw [-] (24) -- (31);

\end{tikzpicture}\]
In this case the $\mu$-heavy LSSD has design parameters $(v,1,0)$, so we find that $s = \sqrt{k-\lambda} = 1$, $\nu = \frac{k(k-s)}{v} = 0$, and $\mu = \nu+s = 1$. We list the eigenmatrices and use these to further describe the LSSD:
\[\scalebox{.9}{$P = \left[\begin{array}{cccc}
	1&w-1&v-1&(v-1)(w-1)\\
	1&w-1&-1&-(w-1)\\
	1&-1&-1&1\\
	1&-1&v-1&1-v\\
	\end{array}\right],\qquad Q = \left[\begin{array}{cccc}
	1 & v-1 & (w-1)(v-1) & w-1\\
	1 & v-1 & -(v-1) & -1\\
	1 & -1 & 1-w & w-1\\
	1 & -1 & 1 & -1\\
	\end{array}\right]$}.\]
This means that our first idempotent is given by:
\[E_1 = (v-1)I+(v-1)A_1-A_2-A_3 = (vI-J)\otimes J.\]
If we scale this appropriately to obtain a Gram matrix of unit vectors, we find that $E_1$ is the Gram matrix of $w$ copies of the same simplex in $\mathbb{R}^v$. This can be seen as well from the fact that $Q_{01} = Q_{11}$ meaning that any simplex vector has inner product 1 with exactly one vector from each of the ``other" simplices, meaning the simplices are just copies of the same simplex. This explains why $w$ is unbounded as we can always copy the same simplex as many times as we would like. This also indicates why this example is not of interest to us as it is not giving $w$ distinct linked simplices.
\subsection{Cameron-Seidel scheme}\index{Cameron-Seidel Scheme}
\label{kerdock}
This construction is given originally by Goethals \cite{Goethals1976} in terms of Kerdock codes, though it was extensively studied by Cameron and Seidel. A restating of this concept is found in \cite{Bey2008}, where Bey and Kyureghyan frame the properties of Kerdock sets in terms of bent functions rather than quadratic forms. In addition, Kantor \cite{Kantor1982} showed that there are exponentially many non-isomorphic non-linear binary codes having Kerdock parameters. Just as the binary codes are not isomorphic, the resulting LSSDs will not be combinatorially isomorphic --- that is, given two LSSDs $(X,\mathcal{R})$ and $(X',\mathcal{R}')$ built from non-isomorphic Kerdock sets, there cannot be a mapping $\phi:X\rightarrow X'$ such that $(x,y)\in R_i$ if and only if $(\phi(x),\phi(y))\in R_i'$.
\begin{definition}[\cite{Cameron1973}]
	Let $V=V(2m,2)$ be a $2m$-dimensional vector space over the field $\mathbb{F} = GF(2)$, ($m\geq 2$). A quadratic form on $V$ is a function $Q$ from $V$ to $\mathbb{F}$ with the properties
	\begin{itemize}
		\item[(i)] $Q(x_0) = 0$, where $x_0$ is the zero vector;
		\item[(ii)] The function $B = B(Q)$: $V\times V\rightarrow\mathbb{F}$ defined by
		\[B(x,y) = Q(x+y) + Q(x)+Q(y), \forall x,y\in V\]
		is bilinear.
	\end{itemize}
\end{definition}
We note that, over any field $\mathbb{F}$, the bilinear form corresponding to any quadratic form is symmetric $(B(x,y) = B(y,x))$. In the specific case of $\mathbb{F} = GF(2)$ however, this bilinear form is also alternating $(B(x,x) = 0)$. That is, for any $x\in V$, $B(x,x) = Q(x+x) + Q(x) + Q(x) = Q(2x) + 2Q(x) = 0$. Thus, within this context, we associate any quadratic form with an alternating bilinear form given by a square matrix with 0's on the main diagonal.

We now give a derivation of the Kerdock codes coming from quadratic forms. Let $Q_1,Q_2,\dots,Q_w$ be a set of $w$ quadratic forms on $\mathbb{Z}_2^n$ for which $Q_i+Q_j$ is a full rank quadratic form whenever $i\neq j$. Note that, through evaluation, each quadratic form gives us a vector $[Q_i(v)]_{v\in\mathbb{Z}^n_2}$ of length $2^n$. Let $\mathcal{Q}_1,\mathcal{Q}_2,\dots,\mathcal{Q}_w$ be cosets, $\mathcal{Q}_i = [Q_i(v)]_{v\in\mathbb{Z}_2^n}+\mathscr{R}(1,n)$, of the first order Reed Muller code (see \cite[Section 4.5]{vanLint1999}). For each $1\leq i\leq w$, define the shortening of $\mathcal{Q}_i$ as the set of vectors
\[V_i = \big\{\left[q(1),q(2),\dots,q\left(2^n-1\right)\right]\mid q(2^n) = 0\big\}_{q\in\mathcal{Q}_i}\]
where $q(j)$ denotes the $j^\text{th}$ entry of $q$. Since $\mathcal{Q}_i$ is closed under complements, we know that $\vert V_i\vert = \frac{1}{2}\vert\mathcal{Q}_i\vert = 2^{n}$. Further, any pair of vectors $v,w\in\bigcup_iV_i$ come from vectors $q_v,q_w\in\bigcup_i\mathcal{Q}_i$ with last entry $0$, so $wt(v\oplus w) = wt(q_v\oplus q_w)$ where $wt(x)$, the Hamming weight of the binary tuple $x$, is equal to the number of non-zero entries. Finally, for each $i$, construct the set of vectors
\[X_i = \left\{\frac{1}{\sqrt{2^n-1}}\left(2v-\mathbbm{1}\right)\vert v\in V_i\right\}.\]
We claim $\left\{X_i\right\}_{i=1..w}$ is a set of linked simplices. To verify this, fix $1\leq i<j\leq w$ and let $x_i,y_i\in X_i$ and $z_j\in X_j$ with corresponding coset vectors $q_x,q_y,$ and $q_z$ respectively. Then,
\[\begin{aligned}
\left<x_i,x_i\right> &= \frac{1}{2^n-1}\left((wt(x_i)+(-1)^2(2^{n}-1-wt(x_i))\right)=1
\end{aligned}\]
giving that every vector in $\bigcup_iX_i$ is a unit vector. Next,
\[\begin{aligned}
\left<x_i,y_i\right> &= \frac{1}{2^n-1}\left((2^n-1)-2wt(q_x\oplus q_y)\right)=-\frac{1}{2^n-1}
\end{aligned}\]
giving us that $X_i$ forms a regular simplex. Finally,
\[\begin{aligned}
\left<x_i,z_j\right> &= \frac{1}{2^n-1}\left((2^n-1)-2wt(q_x\oplus q_w)\right).
\end{aligned}\]
Since $wt(q_x\oplus q_z)\in \left\{2^{n-1}\pm2^{r-1}\right\}$ we have that
\[\left<x_i,w_j\right> = \begin{cases}
\frac{2^{r}-1}{2^{n}-1} & wt(q_x\oplus q_w) = 2^{n-1}-2^{r-1}\\
-\frac{2^r+1}{2^n-1} & wt(q_x\oplus q_w) = 2^{n-1}+2^{r-1}\\
\end{cases}\]
meaning there are two possible angles between simplices.\par
Therefore we can build a LSSD with $w$ fibers whenever we have $w$ quadratic forms whose pairwise sums are full rank. We represent each quadratic form as the $n\times n$ matrix giving the corresponding alternating bilinear form. Then, any two of these matrices must differ in the first row in order for their difference to be full rank. This means $w\leq2^{n-1}$ as there are only $2^{n-1}$ possible choices for the first row. This upper bound is achievable whenever $n$ is even \cite{Cameron1973}. Below we give an example when $n = 4$ where $Q_i$ is the alternating bilinear form corresponding to the $i^{\text{th}}$ quadratic form.
\setlength{\arraycolsep}{2.5pt}
\[\begin{aligned}
Q_1 &= \left[\begin{array}{cccc}
0 & 0 & 0 & 0\\
0 & 0 & 0 & 0\\
0 & 0 & 0 & 0\\
0 & 0 & 0 & 0\\
\end{array}\right],\quad Q_2 = \left[\begin{array}{cccc}
0 & 1 & 0 & 0\\
1 & 0 & 0 & 0\\
0 & 0 & 0 & 1\\
0 & 0 & 1 & 0\\
\end{array}\right],\quad Q_3 = \left[\begin{array}{cccc}
0 & 0 & 1 & 0\\
0 & 0 & 0 & 1\\
1 & 0 & 0 & 1\\
0 & 1 & 1 & 0\\
\end{array}\right],\quad
Q_4 = \left[\begin{array}{cccc}
0 & 1 & 1 & 0\\
1 & 0 & 1 & 1\\
1 & 1 & 0 & 0\\
0 & 1 & 0 & 0\\
\end{array}\right],
\\
Q_5 &= \left[\begin{array}{cccc}
0 & 0 & 0 & 1\\
0 & 0 & 1 & 1\\
0 & 1 & 0 & 1\\
1 & 1 & 1 & 0\\
\end{array}\right],\quad
Q_6 = \left[\begin{array}{cccc}
0 & 1 & 0 & 1\\
1 & 0 & 1 & 0\\
0 & 1 & 0 & 0\\
1 & 0 & 0 & 0\\
\end{array}\right],\quad Q_7 = \left[\begin{array}{cccc}
0 & 0 & 1 & 1\\
0 & 0 & 1 & 0\\
1 & 1 & 0 & 1\\
1 & 0 & 1 & 0\\
\end{array}\right],\quad Q_8 =\left[\begin{array}{cccc}
0 & 1 & 1 & 1\\
1 & 0 & 0 & 1\\
1 & 0 & 0 & 0\\
1 & 1 & 0 & 0\\
\end{array}\right].\\
\end{aligned}\]
It is straightforward to form the characteristic vectors $[Q_i(v)]_v$. Below we display $[Q_2(v)]_v$ and $[Q_8(v)]_v$:
\[\begin{aligned}
\left[Q_2(v)\right]_v &= \left[\begin{array}{cccccccccccccccc}0 & 0 & 0 & 1 & 0 & 0 & 0 & 1 & 0 & 0 & 0 & 1 & 1 & 1 & 1 & 0\end{array}\right],\\
[Q_8(v)]_v &= \left[\begin{array}{cccccccccccccccc}0 & 0 & 0 & 1 & 0 & 1 & 0 & 0 & 0 & 1 & 1 & 1 & 0 & 0 & 1 & 0\end{array}\right].\\
\end{aligned}\]
Each of these binary vectors of length 16 is a codeword in the second-order Reed-Muller code $\scR(2,4)$, thus each codeword then determines a coset of $\scR(1,4)$ inside $\scR(2,4)$. The coset corresponding to $Q_2(v)$ is given below as the set of rows of the matrix. To improve readability, $+$ denotes a $1$ and an empty space denotes a $0$.
\[[Q_2(v)]_v+\mathscr{R}(1,4) =\setlength{\arraycolsep}{.8pt}{\fontsize{5}{6}\selectfont\left[\begin{array}{cccccccccccccccc}
	&  &  & + &  &  &  & + &  &  &  & + & + & + & + & \\
	+ & + & + &  & + & + & + &  & + & + & + &  &  &  &  & +\\
	& + &  &  &  & + &  &  &  & + &  &  & + &  & + & +\\
	+ &  & + & + & + &  & + & + & + &  & + & + &  & + &  & \\
	+ & + &  & + & + & + &  & + & + & + &  & + &  &  & + & \\
	&  & + &  &  &  & + &  &  &  & + &  & + & + &  & +\\
	+ &  &  &  & + &  &  &  & + &  &  &  &  & + & + & +\\
	& + & + & + &  & + & + & + &  & + & + & + & + &  &  & \\
	+ & + & + &  &  &  &  & + & + & + & + &  & + & + & + & \\
	&  &  & + & + & + & + &  &  &  &  & + &  &  &  & +\\
	+ &  & + & + &  & + &  &  & + &  & + & + & + &  & + & +\\
	& + &  &  & + &  & + & + &  & + &  &  &  & + &  & \\
	&  & + &  & + & + &  & + &  &  & + &  &  &  & + & \\
	+ & + &  & + &  &  & + &  & + & + &  & + & + & + &  & +\\
	& + & + & + & + &  &  &  &  & + & + & + &  & + & + & +\\
	+ &  &  &  &  & + & + & + & + &  &  &  & + &  &  & \\
	+ & + & + &  & + & + & + &  &  &  &  & + & + & + & + & \\
	&  &  & + &  &  &  & + & + & + & + &  &  &  &  & +\\
	+ &  & + & + & + &  & + & + &  & + &  &  & + &  & + & +\\
	& + &  &  &  & + &  &  & + &  & + & + &  & + &  & \\
	&  & + &  &  &  & + &  & + & + &  & + &  &  & + & \\
	+ & + &  & + & + & + &  & + &  &  & + &  & + & + &  & +\\
	& + & + & + &  & + & + & + & + &  &  &  &  & + & + & +\\
	+ &  &  &  & + &  &  &  &  & + & + & + & + &  &  & \\
	&  &  & + & + & + & + &  & + & + & + &  & + & + & + & \\
	+ & + & + &  &  &  &  & + &  &  &  & + &  &  &  & +\\
	& + &  &  & + &  & + & + & + &  & + & + & + &  & + & +\\
	+ &  & + & + &  & + &  &  &  & + &  &  &  & + &  & \\
	+ & + &  & + &  &  & + &  &  &  & + &  &  &  & + & \\
	&  & + &  & + & + &  & + & + & + &  & + & + & + &  & +\\
	+ &  &  &  &  & + & + & + &  & + & + & + &  & + & + & +\\
	& + & + & + & + &  &  &  & + &  &  &  & + &  &  & \\
	\end{array}\right]}.\]
To form our regular simplex we now choose all vectors with 0 in the last coordinate, discard the last entry, replace every 0 with a $-1$ (abbreviated to $-$), and then scale by $\frac{1}{\sqrt{15}}$ giving us the vectors (given by rows):
\[\begin{aligned}
X_2 = \frac{1}{\sqrt{15}}\setlength{\arraycolsep}{.8pt}{\fontsize{5}{6}\selectfont\left[\begin{array}{rrrrrrrrrrrrrrrr}
	- & - & - & + & - & - & - & + & - & - & - & + & + & + & + \\
	+ & - & + & + & + & - & + & + & + & - & + & + & - & + & - \\
	+ & + & - & + & + & + & - & + & + & + & - & + & - & - & + \\
	- & + & + & + & - & + & + & + & - & + & + & + & + & - & - \\
	+ & + & + & - & - & - & - & + & + & + & + & - & + & + & + \\
	- & + & - & - & + & - & + & + & - & + & - & - & - & + & - \\
	- & - & + & - & + & + & - & + & - & - & + & - & - & - & + \\
	+ & - & - & - & - & + & + & + & + & - & - & - & + & - & - \\
	+ & + & + & - & + & + & + & - & - & - & - & + & + & + & + \\
	- & + & - & - & - & + & - & - & + & - & + & + & - & + & - \\
	- & - & + & - & - & - & + & - & + & + & - & + & - & - & + \\
	+ & - & - & - & + & - & - & - & - & + & + & + & + & - & - \\
	- & - & - & + & + & + & + & - & + & + & + & - & + & + & + \\
	+ & - & + & + & - & + & - & - & - & + & - & - & - & + & - \\
	+ & + & - & + & - & - & + & - & - & - & + & - & - & - & + \\
	- & + & + & + & + & - & - & - & + & - & - & - & + & - & - \\
	\end{array}\right]}.
\end{aligned}\]
Likewise each $[Q_j(v)]_v$ gives us a coset of size 32, which in turn is transformed to a regular simplex $X_j$ of sixteen vectors in $\mathbb{R}^{15}$ in this manner.
\subsubsection*{Symmetric design parameters}
The design parameters of this scheme are $\left(2^{2r},2^{r-1}\left(2^{r}+1\right),2^{r-1}\left(2^{r-1}+1\right)\right)$. Using these, we have: $s = 2^{r-1}$, $\nu = 2^{r-2}\left(2^r+1\right)$,  and $\mu=2^{r-2}\left(2^r+3\right)$.
\subsubsection*{Intersection numbers}
Noting that $p_{ij}^k = p_{ji}^k$, we list the unique intersection numbers while omitting the trivial $p_{0i}^j$ parameters. Note that each $p^{j}_{ik}$ is scaled by a constant based on $i$ and $k$ given in the top row of our table.
\[\scalebox{.9}{$\begin{array}{c|ccc|cc|c}
	j & \bigslant{p^{j}_{11}}{2^{r-2}}			& \bigslant{p^{j}_{12}}{2^{r-1}} 			& \bigslant{p^{j}_{13}}{2^{r-2}} & p^{j}_{22} & \bigslant{p^{j}_{23}}{2^{r-1}} & \bigslant{p^{j}_{33}}{2^{r-2}}\\\hline
	0 & \left(2^{r+1}+2\right)(w-1)	& 0 							& 0 					& 2^{2r}-1	& 0   							& (2^{r+1}-2)(w-1)\\
	1 & \left(2^r+3\right)(w-2)		& \left(2^{r}+1\right)-2^{1-r}	& (2^r-1)(w-2) 			& 0   		& \left(2^r-1\right)   			& \left(2^{r-2}-1\right)(w-2)\\
	2 & \left(2^{r}+2\right)(w-1)	& 0 							& 2^{r}(w-1)			& 2^{2r}-2 	& 0 							& (2^{r}-2)(w-1)\\
	3 & \left(2^r+1\right)(w-2) 	& 2^{r}+1						& (2^r+1)(w-2) 			& 0   		&\left(2^{r}-1\right)-1 		& \left(2^{r-2}-3\right)(w-2)\\
	\end{array}$}\]
\subsubsection*{Krein parameters}
As with the intersection numbers, we recall that $q_{ij}^k = q_{ji}^k$ and list each unique Krein parameter omitting the trivial $q_{0i}^j$ parameters. No scaling is done here.
\[\begin{array}{c|ccc|cc|c}
j & q^{j}_{11}			& q^{j}_{12}							& q^{j}_{13}	& q^{j}_{22} 						& q^{j}_{23} 	& q^{j}_{33}\\\hline
0 & 2^{2r}-1			& 0										& 0 			& (w-1)(2^{2r}-1)								& 0  			& w-1\\
1 & \frac{2^{2r}}{w}-2	& 2^{2r}\left(\frac{w-1}{w}\right)		& 0 			& \frac{2^{2r}(w-1)^2}{w}-2(w-1)	& w-1			& 0\\
2 & \frac{2^{2r}}{w}	& 2^{2r}\left(\frac{w-1}{w}\right)-2	& 1				& \frac{2^{2r}(w-1)^2}{w}+2(w-2)	& w-2			& 0\\
3 & 0 					& 2^{2r}-1								& 0 			& (w-2)(2^{2r}-1)								& 0  			& w-2\\
\end{array}\]

\section{New examples with \texorpdfstring{$v\neq 2^m$}{v not equal to a power of 2}}
\label{LSSDMUBs}
Recall from Section $\ref{bases}$ that we found the Gram matrix for a set of bases by adding the first two idempotents of our scheme, giving us
\begin{equation}\label{MubGram}M = w(E_0+E_1) = \left(A_0+\frac{v-k+\sqrt{k-\lambda}}{v\sqrt{k-\lambda}}A_1-\frac{k-\sqrt{k-\lambda}}{v\sqrt{k-\lambda}}A_3\right)\end{equation}
where $k$ is the block size of the $\mu$-heavy LSSD. If $\frac{v-k+\sqrt{k-\lambda}}{v\sqrt{k-\lambda}}$ and $-\frac{k-\sqrt{k-\lambda}}{v\sqrt{k-\lambda}}$ have the same absolute value, then $M$ is the Gram matrix for a set of $w$ mutually unbiased bases in $\mathbb{R}^v$. This will only occur when
\[\begin{aligned}v-2k&=-2\sqrt{k-\lambda}.\\
\end{aligned}\]
This means our LSSD must be optimistic, leading to the following lemma:
\begin{lem}
	\label{twoside}
	Let $\mathbb{A}$ be the Bose-Mesner algebra of an optimistic $LSSD(v,k,\lambda;w)$ with $Q$-polynomial ordering $E_0,E_1,E_2,E_3$ of its primitive idempotents. If $\vert v-2k\vert = 2\sqrt{k-\lambda}$ then $w(E_0+E_1)$ is the Gram matrix of a set of $w$ real MUBs in dimension $v$.\qed
\end{lem}
It is important to note here that there exist pessimistic LSSDs such that $v-2k = 2\sqrt{k-\lambda}$. One such example is a specific case of the degenerate design parameters in Section \ref{degenerate} given by $(v,k,\lambda)=(4,1,0)$ with the graph $\Gamma_1$ displayed below.
\[\begin{tikzpicture}[scale = .5,node distance=2cm,
thin,main node/.style={circle,fill=black,scale = .5}]

\node[main node] (11) at (-3,1) {};
\node[main node] (12) at (-2.5,2) {};
\node[main node] (13) at (-2,3) {};
\node[main node] (14) at (-1.5,4) {};
\node[main node] (24) at (1.5,4) {};
\node[main node] (23) at (2,3) {};
\node[main node] (22) at (2.5,2) {};
\node[main node] (21) at (3,1) {};
\node[main node] (34) at (0,-1) {};
\node[main node] (33) at (0,-2) {};
\node[main node] (32) at (0,-3) {};
\node[main node] (31) at (0,-4) {};

\draw [-] (11) -- (21) -- (31) -- (11);
\draw [-] (12) -- (22) -- (32) -- (12);
\draw [-] (13) -- (23) -- (33) -- (13);
\draw [-] (14) -- (24) -- (34) -- (14);
\end{tikzpicture}\]
We can easily see that $v-2k = 2 = 2\sqrt{k-\lambda}$. However, the sum of the first two eigenspaces gives
\[M = 3(E_0 + E_1) =\frac{1}{3}A_0 -\frac{1}{4}A_1\]
which is not the Gram matrix of a set of MUBs. In fact, any Menon parameter set with $v/4$ odd will satisfy $\vert v-2k\vert = 2\sqrt{k-\lambda}$ yet none of these will produce MUBs. Thus our restriction to optimistic LSSDs is required and we cannot say that any LSSD satisfying $\vert v-2k\vert = 2\sqrt{k-\lambda}$ will give us MUBs using this construction. Conversely, the existence of $w$ MUBs in $\mathbb{R}^v$ does not guarantee the existence of an optimistic $LSSD(v,k,\lambda;w)$; consider 3 MUBs in $\mathbb{R}^4$.
\subsection{Restrictions on the design parameters}
In this section, we show that $\vert v-2k\vert = 2\sqrt{k-\lambda}$ implies $(v,k,\lambda)$ are Menon design parameters. In the case of optimistic LSSDs, we also show $v/4$ is even. We now take a closer look at our restriction $v-2k = -2\sqrt{k-\lambda}$. First note we can square both sides to get
\[\begin{aligned}
4(k-\lambda)&=v^2-4k(v-k).
\end{aligned}\]
Using \eqref{sym:2}, this gives $v=4(k-\lambda)$ where we apply \eqref{sym:1} to get
\[k^2+k+\lambda=4\lambda(k-\lambda).\]
Solving this for $k$ gives $k=\frac{4\lambda+1}{2}\pm\frac{\sqrt{4\lambda+1}}{2}$ requiring $\frac{1}{2}\pm\frac{1}{2}\sqrt{4\lambda+1}$ to be an integer. Therefore $\sqrt{4\lambda+1}$ must be an odd integer. Assume $4\lambda+1 = (2u-1)^2$ for some positive integer $u$. Then $\lambda = u^2-u$ and
\[\begin{aligned}
k&=2u^2-(2\mp 1)u+\left(\frac{1\mp1}{2}\right).\\
\end{aligned}\]
If we re-parameterize the second family to avoid the trivial $(0,0,0)$ design when $u=1$, we get the complementary families:
\[\begin{aligned}
\lambda &= (u-1)u\qquad\qquad &\lambda^\prime &= (u+1)u\\
k&=(2u-1)u\qquad \text{and}\qquad&k^\prime &= (2u+1)u\\
v&= 4u^2\qquad \qquad&v^\prime &= 4u^2.
\end{aligned}\]
If we restrict to an optimistic LSSD, we must use the second family for our $\mu$-heavy LSSD $\Gamma$. From \eqref{mu-nu}, we find
\[\begin{aligned}
s=\sqrt{k-\lambda}=u,\qquad\nu =\frac{k(k-s)}{v}=u^2+\frac{u}{2},\qquad\mu=\nu+s=u^2+\frac{3}{2}u.
\end{aligned}\]
This forces $\nu$ (and $\mu$) to be integral if and only if $u$ is even, resulting in the following theorem
\begin{thm}
	\label{MubLssdrestrictions}
	Let $\Gamma$ be an optimistic $LSSD(v,k,\lambda;w)$. If $v-2k = -2\sqrt{k-\lambda}$ then $(v,k,\lambda)$ are \emph{Menon design parameters}\index{Menon parameters}. That is,
	\begin{enumerate}[label=$(\roman*)$]
		\item $v = 4u^2$
		\item $k = 2u^2+u$
		\item $\lambda = u^2+u$
		\item $w\leq 2u^2$
	\end{enumerate}
	Further, if $u$ is odd, then $w = 2$.
\end{thm}
\begin{proof}
	Statements \emph{(i) -- (iii)} as well as our restriction when $u$ is odd follow directly from above. Conclusion \emph{(iv)} follows from \eqref{Noda}.
\end{proof}
The upper bound for $w$ is achieved whenever $u$ is a power of two using the Cameron-Seidel scheme. In light of this family and the condition that $w=2$ whenever $v/4$ is odd, one might ask if $w$ is bounded as a function of the highest power of $2$ dividing $v$. We will show later that this is not true by constructing examples with $w$ as large as we like and $v/16$ odd.
\subsection{Mutually unbiased Hadamard matrices}
In this section, we establish an equivalence between LSSDs with design parameters as in Theorem \ref{MubLssdrestrictions} with sets of regular unbiased Hadamard matrices (see \cite{Kharaghani2010} for more detailed information on unbiased Hadamard matrices). Van Dam et al.\ \cite[p.~1423]{Davis2014} briefly mention this connection citing constructions of unbiased regular Hadamard matrices by Holzmann, Kharaghani, and Orrick \cite{Kharaghani2010} using mutually orthogonal latin squares. In this section, we describe the connection between unbiased regular Hadamard matrices and LSSDs in full, using Theorem \ref{twoside}. A real Hadamard matrix\index{Hadamard matrix!real} of order $v$ is a $v\times v$ matrix $H$ with entries $\pm 1$ such that $HH^T = vI$. $H$ is a \textit{regular Hadamard matrix}\index{Hadamard matrix!regular} if $HJ = JH = cJ$ for some constant $c$; one easily verifies that $c = \sqrt{v}$. Two Hadamard matrices $H_1$ and $H_2$ are \textit{unbiased}\index{Hadamard matrix!mutually unbiased} if $\frac{1}{\sqrt{v}}H_1H_2^T$ is itself a Hadamard matrix. Finally, a set of Hadamard matrix matrices are mutually unbiased if each pair is unbiased. Using these definitions, consider the following:
\begin{thm}
	\label{LSSDtoHad}
	Let $\Gamma$ be an optimistic $LSSD(v,k,\lambda;w)$. If $\vert v-2k\vert = 2\sqrt{k-\lambda}$, then there exists a set of $w-1$ real mutually unbiased regular Hadamard matrices of order $v$.
\end{thm}
\begin{proof}
	Let $(X,\mathcal{R})$ be the association scheme arising from $\Gamma$ with Bose-Mesner algebra $\mathbb{A}$. Let $\left\{E_0,E_1,E_2,E_3\right\}$ be the set of idempotents under the natural $Q$-polynomial ordering. From Lemma \ref{twoside}, $G = w(E_0+E_1)$ is the Gram matrix of a set of $w$ MUBs in $\mathbb{R}^v$. Let the $w$ MUBs be given by the columns of the $w$ orthogonal matrices $\left\{M_1,\dots,M_w\right\}$ and without loss of generality assume $M_1 = I$. Then any column from another $M_i$ $(i\neq 1)$ must have entries $\pm \frac{1}{\sqrt{v}}$. For $1<i\leq w$, let $H_i = \sqrt{v}M_i$. First note that $H_iH_i^T = vM_iM_i^T = vI$, therefore for each $1<i\leq w$, $H_i$ is a Hadamard matrix. Now consider that the first $v$ rows of $G$ will have the block form
	$\left[\begin{array}{ccccc}
	I & M_2 & M_3 & \dots & M_{w}\\
	\end{array}\right].$
	However from \eqref{MubGram}, we have that the positive (resp., negative) entries of $M_i$ represent adjacency between vertices in the first and $i^\text{th}$ fibers of the $\mu$-heavy (resp., $\nu$-heavy) LSSD. Since each vertex in the first fiber must be adjacent to $k$ vertices in the $i^\text{th}$ fiber in the $\mu$-heavy LSSD, each row of $M_i$ must have $k$ positive entries and $v-k$ negative entries. Similarly, each vertex in the $i^\text{th}$ fiber is adjacent to $k$ vertices in the first fiber in the $\mu$-heavy LSSD, thus each column of $M_i$ has $k$ positive entries and $v-k$ negative entries. Therefore each $H_i$ must be regular. Now define $\binom{w}{2}$ matrices $M_{i,j}$ where $M_{i,j}$ is the orthogonal matrix representing basis $j$ when basis $i$ is taken to be the standard basis (so $M_{1,j} = M_j$). Then we repeat all previous arguments to show that $G$ has the block form:
	\[G = \left[\begin{array}{ccccccc}
	I & M_{1,2} & \dots & M_{1,w}\\
	M_{2,1} & I & \dots & M_{2,w}\\
	\vdots & \vdots  & \ddots & \vdots\\
	M_{w,1} & M_{w,2} & \dots & I\\
	\end{array}\right]\]
	where $\sqrt{v}M_{i,j}$ is a Hadamard matrix for all $1\leq i\neq j\leq w$. Now consider a second association scheme $(X',\cR')$ arising from the subgraph of $\Gamma$ induced on three distinct fibers $X_i$, $X_j$, and $X_k$. The matrix $G^\prime = w(E_0^\prime + E_1^\prime)$ will have the form:
	\[G^\prime = \left[\begin{array}{ccccccc}
	I & M_{i,j} & M_{i,k}\\
	M_{j,i} & I  & M_{j,k}\\
	M_{k,i} & M_{k,j} & I\\
	\end{array}\right].\]
	Noting that $G^2 = wG$, the block in the $(1,2)$ block of $G^2$ gives us that
	\[2M_{i,j} + M_{i,k}M_{k,j} = 3M_{i,j},\qquad \text{ or equivalently } \qquad M_{i,k}M_{k,j} = M_{i,j}.\]
	Therefore, if we return to the original $LSSD$ and define $H_{i,j} = \sqrt{v}M_{i,j}$, we find that $\frac{1}{\sqrt{v}}H_i^TH_j = H_{i,j}$. It follows that the set $\left\{H_2,\dots,H_w\right\}$ is a set of $w-1$ regular mutually unbiased Hadamard matrices.
\end{proof}
We now show the converse:
\begin{thm}
	Assume $w>2$. Let $\left\{H_2,\dots,H_w\right\}$ be $w-1$ regular unbiased Hadamard matrices of order $v$. Then there exists an optimistic $LSSD(v,k,\lambda;w)$.
\end{thm}
\begin{proof}
	Assume without loss of generality that the row sum of each of our Hadamard matrices is positive. Define vectors $x_{i,j}$ for $2\leq i\leq w$ and $1\leq j\leq v$ such that $x_{i,j}$ is the $j^\text{th}$ column of $H_i-\frac{1}{\sqrt{v}}J$. Let $x_{1,j}$ be the $j^\text{th}$ column of $\sqrt{v}I-\frac{1}{\sqrt{v}}{J}$. Note that for all $1\leq i\leq w$, $\vnorm{x_{i,j}} = v-1$. Then, for all $i,j$, let $\hat{x}_{i,j} = \frac{x_{i,j}}{\sqrt{v-1}}$. Letting $X_i = \left\{\hat{x}_{i,1},\dots,\hat{x}_{i,v}\right\}$, we claim that $\left\{X_1,\dots, X_w\right\}$ is a set of linked simplices. To show this, fix $j\neq k$, $i\neq i'$ and consider the following four inner products:
	\begin{align}
	\label{had:1}\left<\hat{x}_{1,j},\hat{x}_{1,k}\right> &= \frac{1}{v-1}\left[vI-J\right]_{j,k} = -\frac{1}{v-1},\\
	\label{had:2}\left<\hat{x}_{i,j},\hat{x}_{i,k}\right> &=\frac{1}{v-1}\left[H_iH_i^T -J\right]_{j,k} = -\frac{1}{v-1},\\
	\label{had:3}\left<\hat{x}_{1,j},\hat{x}_{i,k}\right> &=\frac{\sqrt{v}}{v-1}\left[H_i^T-\frac{1}{\sqrt{v}}J\right]_{j,k},\\
	\label{had:4}\left<\hat{x}_{i,j},\hat{x}_{i',k}\right> &=\frac{\sqrt{v}}{v-1}\left[\frac{1}{\sqrt{v}}H_iH_{i^\prime}^T -\frac{1}{\sqrt{v}}J\right]_{j,k}.
	\end{align}
	\eqref{had:1} and \eqref{had:2} give us the inner products within each $X_i$. Since $\frac{1}{\sqrt{v}}H_iH_{i^\prime}^T$ is a Hadamard matrix, \eqref{had:3} and \eqref{had:4} tell us that inner products between sets $X_i$ and $X_{i^\prime}$ take values of $\frac{\pm\sqrt{v}-1}{v-1}$. Finally note that the all ones vector is orthogonal to all $\hat{x}_{i,j}$, implying $\left\{X_1,\dots,X_w\right\}$ is a set of $w$ simplices in $\mathbb{R}^{v-1}$ such that inner products between simplices can take only two possible values. Finally consider that the possible inner products are $\frac{\sqrt{v}}{v-1}\left(\pm 1 - \frac{1}{\sqrt{v}}\right)$. This tells us that $\vert \gamma\vert <\vert \delta\vert$ where $\gamma$ is the positive inner product and $\delta$ is the negative. Since the centroid of any simplex is the origin, we must have more positive inner products between simplices than negative, telling us our $LSSD$ is optimistic.
\end{proof}
This provides us with the following theorem
\lssdhadequiv\vspace{-9mm}\qed
\subsection{Constructing LSSDs from real MUBs}
Using the results from the last section and the close relation between MUBs and Hadamard matrices, we wish to build new LSSDs. From Theorem \ref{equiv} and Theorem \ref{MubLssdrestrictions}, we are only going to find optimistic LSSDs with Menon design parameters. The Cameron-Seidel scheme is a construction for $w = 2u^2$ whenever $u$ is a power of 2 (see Section \ref{kerdock}). We skip this case and instead look for constructions where $u$ (and equivalently $v$) is not necessarily a power of 2.
\subsubsection{Wocjan and Beth construction}
Wocjan and Beth in \cite{Wocjan2005} detail a way to create MUBs from MOLS. They take a set of $t$ MOLS with side length $d$ and create $t+2$ MUBs in dimension $d^2$. The process is to convert the MOLS into an orthogonal array with $d^2$ rows. They then expand the array by replacing each column with $d$ columns given by the characteristic vector of each symbol in that column. Finally, they extend this matrix by replacing each 1 in the array with a row from a Hadamard matrix and each 0 by an appropriate length vector of 0s. The result is that the $d$ columns arising from each original column are orthogonal to each other. We will focus on the case where the resultant MUBs produce regular Hadamard matrices as scalar multiples of the mixed Gram matrices between any pair of fibers.\par
For our purposes, an orthogonal array of size $(n^2\times N)$ has entries from the set $\left\{1,\dots, n\right\}$ and any two columns contain each ordered pair exactly once. Let $O$ be an orthogonal array of size $n^2\times N$, let $C^i$ denote the $i^{th}$ column of $O$ with entries $C^i_h$ ($1\leq h\leq n^2$). We may uniquely express
\[C^i = \sum_{j=1}^n jB^{i,j} \]
where each $B^{i,j}$ is a 01-vector of length $n^2$. As each symbol $j$ appears in each column $C^i$ exactly $n$ times, $B^{i,j}$ will have $n$ 1s and $n^2-n$ 0s. Let $H$ be a Hadamard matrix matrix of order $n$. For $1\leq l\leq n$ define a matrix $M^{i,j,l}$ as follows: for $1\leq h\leq n$, we replace the $h^\text{th}$ 1, counting from the top, in $B^{i,j}$ with the $H_{h,l}$. This produces $n^2N$ columns each with $n^2$ entries $M_h^{i,j,l}\in\left\{0,1,-1\right\}$.\par
The fact that $H^TH = nI$ together with $B^{i,j}\circ B^{i,j^\prime} = 0$ for $j\neq j^\prime$ give us that $\mathcal{B}_i = \left\{M^{i,1,1},\dots,M^{i,n,n}\right\}$ is an orthogonal basis for each $i=1,\dots,N$. Each vector in these bases has squared norm $n$. For $i\neq i'$, $C^i$ and $C^{i'}$ denote distinct columns in our orthogonal array implying, for any $j$ and $j'$ (not necessarily distinct), $B^{i,j}\circ B^{i',j'}$ has one non-zero entry. Then $M^{i,j,l}\circ M^{i',j',l'}$ also has one non-zero entry giving
\begin{equation}\left<M^{i,j,l},M^{i',j',l'}\right>=M^{i,j,l}_hM^{i',j',l'}_h=\pm 1\\
\end{equation}
where $O_{h,i} = j$ and $O_{h,i'} = j'$. This tells us the bases $\mathcal{B}_1,\dots,\mathcal{B}_s$ produced by Wocjan and Beth are pairwise unbiased.\par
We now show that if $H$ is regular, then the resulting unbiased Hadamard matrices are regular (see proof of Theorem \ref{LSSDtoHad} for the construction of these Hadamard matrices). For each Hadamard matrix, the row sum is the sum of inner products between a column $M^{i,j,l}$ of one basis with the set of $n^2$ columns $M^{i',j',l'}$ ($i\neq i'$, $1\leq j',l'\leq n)$ of the second basis used in its construction. We first sum $\left<M^{i,j,l},M^{i',j',l'}\right>$ over $l'$ to get
\begin{equation*}
\sum_{l'} \left<M^{i,j,l},M^{i',j',l'}\right>=M^{i,j,l}_h\left(\sum_{l'}M^{i',j',l'}_h\right)=pM^{i,j,l}_h
\end{equation*}
where $p$ is the row sum of our Hadamard matrix $H$. We then sum this result over $j'$, noting that $h$ was chosen so that $O_{h,i} = j$ and $O_{h,i'} = j'$, meaning it depends on $j'$. As this sum will include every non-zero entry in $M^{i,j,l}$ exactly once, we know
\begin{equation}
\label{rowsum}
\sum_{j'}\sum_{l'}\left<M^{i,j,l},M^{i',j',l'}\right> =\sum_{j'}pM^{i,j,l}_{h}=p\sum_{j'}H_{j',l}=p^2.
\end{equation}
Then the sum of any row of the Hadamard matrix built from $M^{i,j,l}$ and $M^{i',j',l'}$ ($i\neq i'$) will be $p^2$. Further, we showed in Theorem \ref{LSSDtoHad} that these Hadamard matrices are unbiased. Noting that $n = 4t^2$ for some $t$, Theorem \ref{equiv} tells us that the resultant LSSD will be an optimistic $LSSD(16t^4,k,\lambda;N)$. This leads to our final theorem.
\buildinghad
The following theorem is due to Beth \cite{Beth1983} and will be used to show the number of fibers is unbounded as $v$ increases.
\begin{thm}\cite{Beth1983}\label{beth}
	Let $N(n)$ be the maximum size of a set of mutually orthogonal latin squares of side $n$. Then, for $n$ sufficiently large, $N(n)\geq n^\frac{1}{14.8}$.
\end{thm}
\asymplssd
\begin{proof}
	Using Theorem \ref{beth}, we know that for sufficiently large $n$, we may find a set of mutually orthogonal latin squares of side $n$ with at least $n^\frac{1}{14.8}$ squares. However, given $t$ mutually orthogonal latin squares of side $n$, we may build an orthogonal array with $n+2$ columns with $n$ symbols. This is achieved by indexing the rows of our orthogonal array by the $n^2$ positions in a latin square. The first and second columns of the orthgonal array denote the row and column (resp.) of the position within the latin square. Each remaining column corresponds to one of the mutually orthogonal latin squares where the symbols in each position generate the column. Thus, as long as we may find a regular Hadamard matrix of order $n$, we may build a $LSSD(n^2,k,\lambda;w)$ with $w\geq n^\frac{1}{14.8}$.
\end{proof}
\oddlssd
\begin{proof}
	Considering a symmetric Bush-type Hadamard matrix as a specific case of regular Hadamard matrices, Theorem 3.5 of \cite{Xiang2006} tell us that there exists regular Hadamard matrices of order $4t^4$ for any odd $t$. Let $t\geq 1$ and $H_t$ be one such regular Hadamard matrix of order $4t^4$. Using Corollary \ref{asymptotics}, we can choose $t$ large enough to guarantee the existence of a $LSSD(16t^8,k,\lambda;w)$. Now consider the Hadamard matrix
	\[H = \left[\begin{array}{rrrr}
	-1 & 1 & 1 & 1\\
	1 & -1 & 1 & 1\\
	1 & 1 & -1 & 1\\
	1 & 1 & 1 & -1\\
	\end{array}\right].\]
	Using this matrix, we can now build the regular Hadamard matrix $H_{n,t} = H_t\otimes^{n-1} H$ which is regular of order $4^nt^4$. This matrix, again paired with Corollary \ref{asymptotics}, now guarantees the existence of a $LSSD(16^nt^8,k,\lambda;w)$ for any choice of $n$.
\end{proof}
\lssdthreesix
\begin{proof}
	Using the MacNeish construction \cite{Macneish1922},\cite[Thm~1.1.2]{Bommel2015}, there exists an orthogonal array $O_n$ of size $36^{2n}\times(4^n+1)$. Consider the regular Hadamard matrix of order 36 found by Seberry \cite{SeberryHad}:
	\[H = \setlength{\arraycolsep}{0.6pt}{\fontsize{5}{6}\selectfont\left[\begin{array}{cccccccccccccccccccccccccccccccccccc}
		-&-&-&-&+&-&-&+&+&+&+&-&+&+&-&+&-&+&+&+&+&+&-&-&-&+&+&+&+&-&+&+&+&-&+&-\\
		+&-&-&-&-&+&-&-&+&+&-&+&+&-&+&-&+&+&+&+&+&-&-&-&+&+&+&+&-&+&+&+&-&+&-&+\\
		+&+&-&-&-&-&+&-&-&-&+&+&-&+&-&+&+&+&+&+&-&-&-&+&+&+&+&-&+&+&+&-&+&-&+&+\\
		-&+&+&-&-&-&-&+&-&+&+&-&+&-&+&+&+&-&+&-&-&-&+&+&+&+&+&+&+&+&-&+&-&+&+&-\\
		-&-&+&+&-&-&-&-&+&+&-&+&-&+&+&+&-&+&-&-&-&+&+&+&+&+&+&+&+&-&+&-&+&+&-&+\\
		+&-&-&+&+&-&-&-&-&-&+&-&+&+&+&-&+&+&-&-&+&+&+&+&+&+&-&+&-&+&-&+&+&-&+&+\\
		-&+&-&-&+&+&-&-&-&+&-&+&+&+&-&+&+&-&-&+&+&+&+&+&+&-&-&-&+&-&+&+&-&+&+&+\\
		-&-&+&-&-&+&+&-&-&-&+&+&+&-&+&+&-&+&+&+&+&+&+&+&-&-&-&+&-&+&+&-&+&+&+&-\\
		-&-&-&+&-&-&+&+&-&+&+&+&-&+&+&-&+&-&+&+&+&+&+&-&-&-&+&-&+&+&-&+&+&+&-&+\\
		+&+&-&+&+&-&+&-&+&+&+&+&+&-&+&+&-&-&+&-&+&-&+&+&+&-&+&-&-&-&+&+&+&-&-&-\\
		+&-&+&+&-&+&-&+&+&-&+&+&+&+&-&+&+&-&-&+&-&+&+&+&-&+&+&-&-&+&+&+&-&-&-&-\\
		-&+&+&-&+&-&+&+&+&-&-&+&+&+&+&-&+&+&+&-&+&+&+&-&+&+&-&-&+&+&+&-&-&-&-&-\\
		+&+&-&+&-&+&+&+&-&+&-&-&+&+&+&+&-&+&-&+&+&+&-&+&+&-&+&+&+&+&-&-&-&-&-&-\\
		+&-&+&-&+&+&+&-&+&+&+&-&-&+&+&+&+&-&+&+&+&-&+&+&-&+&-&+&+&-&-&-&-&-&-&+\\
		-&+&-&+&+&+&-&+&+&-&+&+&-&-&+&+&+&+&+&+&-&+&+&-&+&-&+&+&-&-&-&-&-&-&+&+\\
		+&-&+&+&+&-&+&+&-&+&-&+&+&-&-&+&+&+&+&-&+&+&-&+&-&+&+&-&-&-&-&-&-&+&+&+\\
		-&+&+&+&-&+&+&-&+&+&+&-&+&+&-&-&+&+&-&+&+&-&+&-&+&+&+&-&-&-&-&-&+&+&+&-\\
		+&+&+&-&+&+&-&+&-&+&+&+&-&+&+&-&-&+&+&+&-&+&-&+&+&+&-&-&-&-&-&+&+&+&-&-\\
		+&+&+&+&-&-&-&+&+&-&+&-&+&-&-&-&+&-&+&+&+&+&-&+&+&-&-&+&+&-&+&-&+&+&-&+\\
		+&+&+&-&-&-&+&+&+&+&-&+&-&-&-&+&-&-&-&+&+&+&+&-&+&+&-&+&-&+&-&+&+&-&+&+\\
		+&+&-&-&-&+&+&+&+&-&+&-&-&-&+&-&-&+&-&-&+&+&+&+&-&+&+&-&+&-&+&+&-&+&+&+\\
		+&-&-&-&+&+&+&+&+&+&-&-&-&+&-&-&+&-&+&-&-&+&+&+&+&-&+&+&-&+&+&-&+&+&+&-\\
		-&-&-&+&+&+&+&+&+&-&-&-&+&-&-&+&-&+&+&+&-&-&+&+&+&+&-&-&+&+&-&+&+&+&-&+\\
		-&-&+&+&+&+&+&+&-&-&-&+&-&-&+&-&+&-&-&+&+&-&-&+&+&+&+&+&+&-&+&+&+&-&+&-\\
		-&+&+&+&+&+&+&-&-&-&+&-&-&+&-&+&-&-&+&-&+&+&-&-&+&+&+&+&-&+&+&+&-&+&-&+\\
		+&+&+&+&+&+&-&-&-&+&-&-&+&-&+&-&-&-&+&+&-&+&+&-&-&+&+&-&+&+&+&-&+&-&+&+\\
		+&+&+&+&+&-&-&-&+&-&-&+&-&+&-&-&-&+&+&+&+&-&+&+&-&-&+&+&+&+&-&+&-&+&+&-\\
		+&+&-&+&+&+&-&+&-&+&+&+&-&-&-&+&+&+&-&-&+&-&+&-&-&+&-&+&+&+&+&-&+&+&-&-\\
		+&-&+&+&+&-&+&-&+&+&+&-&-&-&+&+&+&+&-&+&-&+&-&-&+&-&-&-&+&+&+&+&-&+&+&-\\
		-&+&+&+&-&+&-&+&+&+&-&-&-&+&+&+&+&+&+&-&+&-&-&+&-&-&-&-&-&+&+&+&+&-&+&+\\
		+&+&+&-&+&-&+&+&-&-&-&-&+&+&+&+&+&+&-&+&-&-&+&-&-&-&+&+&-&-&+&+&+&+&-&+\\
		+&+&-&+&-&+&+&-&+&-&-&+&+&+&+&+&+&-&+&-&-&+&-&-&-&+&-&+&+&-&-&+&+&+&+&-\\
		+&-&+&-&+&+&-&+&+&-&+&+&+&+&+&+&-&-&-&-&+&-&-&-&+&-&+&-&+&+&-&-&+&+&+&+\\
		-&+&-&+&+&-&+&+&+&+&+&+&+&+&+&-&-&-&-&+&-&-&-&+&-&+&-&+&-&+&+&-&-&+&+&+\\
		+&-&+&+&-&+&+&+&-&+&+&+&+&+&-&-&-&+&+&-&-&-&+&-&+&-&-&+&+&-&+&+&-&-&+&+\\
		-&+&+&-&+&+&+&-&+&+&+&+&+&-&-&-&+&+&-&-&-&+&-&+&-&-&+&+&+&+&-&+&+&-&-&+
		\end{array}\right]}.\]
	Alternatively, one may use the Menon difference set \[\begin{aligned}\left\{(0\right.&10),(011),(012),(020),(021),(022),(100),(110),\\&\left.(120),(200),(211),(222),(300),(312),(321)\right\}\end{aligned}\] within $\mathbb{Z}_4\times \mathbb{Z}_3^2$ to generate a regular Hadamard matrix $H$. In either case, since $H$ is regular, $H_n = H^{\otimes n}$ is a regular Hadamard of order $36^n$. Then $O_n$ and $H_n$, along with Theorem \ref{newLSSD}, give us the desired LSSD.
\end{proof}
The same construction gives, for example, $LSSD(100^{2n},k,\lambda;4^n+1)$ for all $n\geq 1$. Finally, we note that if we can build a regular Hadamard matrix of order $4t^2$ for $1\leq t\leq 50$, the table of largest known orthogonal arrays for small $n$ in \cite{Colbourn2006} gives us LSSDs with the following number of fibers.\par
\begin{table}[ht]
	\scalebox{.9}{\begin{tabular}{c|*{17}{c}}
			$t$ & 1 & 2 & 3 & 4 & 5 & 6 & 7 & 8 & 9 & 10& 11 & 12 & 13 & 14 & 15 & 16 & 17 \\\hline
			$w$&5&17&9&65&10&12&8&257&10&17&17&10&10&10&29&1025&10\\\\
			$t$ & 18 & 19 & 20& 21 & 22 & 23 & 24 & 25 & 26 & 27 & 28 & 29 & 30& 31 & 32 & 33 \\\hline
			$w$&26&11&26&11&17&11&32&10&17&10&50&30&30&12&4097&32\\\\
			$t$ & 34 & 35 & 36 & 37 & 38 & 39 & 40 & 41 & 42 & 43 & 44 & 45 & 46 & 47 & 48 & 49 &\\\hline
			$w$&18&32&65&32&18&32&26&13&20&32&65&17&32&32&30&17
	\end{tabular}}\caption[Number of fibers for each $t$ with $v = 4t^2$.]{For each $t$, $w$ represents the largest known number of fibers for a LSSD on $v=16t^4$ vertices. In each case, we will also require the existence of a regular Hadamard matrix of order $4t^2$ in order to apply Theorem $\ref{newLSSD}$; Remark 1.25 in the handbook of combinatorial designs \cite{Colbourn2006} guarantees such regular Hadamard matrices up to $t=46$.}
\end{table}
To give an example of the construction for Theorem $\ref{newLSSD}$ we build a $\text{LSSD}(16,10,6; 3)$. In all matrices that follow, ``$+$" denotes a positive 1, ``$-$" denotes a $-1$, and an empty space denotes $0$. We begin by using the orthogonal array $O$ and the Hadamard matrix $H$:
\[O^T = \left[\begin{array}{cccccccccccccccc}
1 & 1 & 1 & 1 & 2 & 2 & 2 & 2 & 3 & 3 & 3 & 3 & 4 & 4 & 4 & 4\\
1 & 2 & 3 & 4 & 1 & 2 & 3 & 4 & 1 & 2 & 3 & 4 & 1 & 2 & 3 & 4\\
1 & 2 & 3 & 4 & 2 & 3 & 4 & 1 & 3 & 4 & 1 & 2 & 4 & 1 & 2 & 3\\
\end{array}\right],\qquad H = \left[\begin{array}{rrrr}
- & + & + & +\\
+ & - & + & +\\
+ & + & - & +\\
+ & + & + & -\\
\end{array}\right].\]
Using this OA, we have
\[B^{:,:} = \setlength{\arraycolsep}{0.8pt}{\fontsize{5}{6}\selectfont\left[\begin{array}{cccc|cccc|cccc}
	+&&& &+&&& &+&&&\\
	+&&& &&+&& &&+&&\\
	+&&& &&&+& &&&+&\\
	+&&& &&&&+ &&&&+\\
	&+&& &+&&& &&+&&\\
	&+&& &&+&& &&&+&\\
	&+&& &&&+& &&&&+\\
	&+&& &&&&+ &+&&&\\
	&&+& &+&&& &&&+&\\
	&&+& &&+&& &&&&+\\
	&&+& &&&+& &+&&&\\
	&&+& &&&&+ &&+&&\\
	&&&+ &+&&& &&&&+\\
	&&&+ &&+&& &+&&&\\
	&&&+ &&&+& &&+&&\\
	&&&+ &&&&+ &&&+&\\
	\end{array}\right]}.\]
Below we display the three arrays $M^{1,:,:}$, $M^{2,:,:}$, and $M^{3,:,:}$ respectively\footnote{Note that we need not use the same Hadamard matrix $H$ in each of these substitutions: any twelve regular Hadamard matrices of order $4$ will do. This allows us to construct (potentially) non-isomorphic 3-class association schemes with the same parameters.}
\[\begin{aligned}
\setlength{\arraycolsep}{0.4pt}{\fontsize{5}{6}\selectfont\left[\begin{array}{rrrr|rrrr|rrrr|rrrr}
	- & + & + & +&&&&&&&&&&&&\\
	+ & - & + & +&&&&&&&&&&&&\\
	+ & + & - & +&&&&&&&&&&&&\\
	+ & + & + & -&&&&&&&&&&&&\\
	&&&&- & + & + & +&&&&&&&&\\
	&&&&+ & - & + & +&&&&&&&&\\
	&&&&+ & + & - & +&&&&&&&&\\
	&&&&+ & + & + & -&&&&&&&&\\
	&&&&&&&&- & + & + & +&&&&\\
	&&&&&&&&+ & - & + & +&&&&\\
	&&&&&&&&+ & + & - & +&&&&\\
	&&&&&&&&+ & + & + & -&&&&\\
	&&&&&&&&&&&&- & + & + & +\\
	&&&&&&&&&&&&+ & - & + & +\\
	&&&&&&&&&&&&+ & + & - & +\\
	&&&&&&&&&&&&+ & + & + & -\\
	\end{array}\right]},\quad&\setlength{\arraycolsep}{0.4pt}{\fontsize{5}{6}\selectfont\left[\begin{array}{rrrr|rrrr|rrrr|rrrr}
	- & + & + & +&&&&&&&&&&&&\\
	&&&&- & + & + & +&&&&&&&&\\
	&&&&&&&&- & + & + & +&&&&\\
	&&&&&&&&&&&&- & + & + & +\\
	+ & - & + & +&&&&&&&&&&&&\\
	&&&&+ & - & + & +&&&&&&&&\\
	&&&&&&&&+ & - & + & +&&&&\\
	&&&&&&&&&&&&+ & - & + & +\\
	+ & + & - & +&&&&&&&&&&&&\\
	&&&&+ & + & - & +&&&&&&&&\\
	&&&&&&&&+ & + & - & +&&&&\\
	&&&&&&&&&&&&+ & + & - & +\\
	+ & + & + & -&&&&&&&&&&&&\\
	&&&&+ & + & + & -&&&&&&&&\\
	&&&&&&&&+ & + & + & -&&&&\\
	&&&&&&&&&&&&+ & + & + & -\\
	\end{array}\right]},\quad&\setlength{\arraycolsep}{0.4pt}{\fontsize{5}{6}\selectfont\left[\begin{array}{rrrr|rrrr|rrrr|rrrr}
	- & + & + & +&&&&&&&&&&&&\\
	&&&&- & + & + & +&&&&&&&&\\
	&&&&&&&&- & + & + & +&&&&\\
	&&&&&&&&&&&&- & + & + & +\\
	&&&&+ & - & + & +&&&&&&&&\\
	&&&&&&&&+ & - & + & +&&&&\\
	&&&&&&&&&&&&+ & - & + & +\\
	+ & - & + & +&&&&&&&&&&&&\\
	&&&&&&&&+ & + & - & +&&&&\\
	&&&&&&&&&&&&+ & + & - & +\\
	+ & + & - & +&&&&&&&&&&&&\\
	&&&&+ & + & - & +&&&&&&&&\\
	&&&&&&&&&&&&+ & + & + & -\\
	+ & + & + & -&&&&&&&&&&&&\\
	&&&&+ & + & + & -&&&&&&&&\\
	&&&&&&&&+ & + & + & -&&&&\\
	\end{array}\right]}.
\end{aligned}\]
Finding the inner products of each pair of bases we find the three Hadamard matrices $H_{1,2}$, $H_{1,3}$, and $H_{2,3}$ respectively.

\[\begin{aligned}
\setlength{\arraycolsep}{.8pt}{\fontsize{5}{6}\selectfont\left[\begin{array}{rrrrrrrrrrrrrrrr}
	+&-&-&-&-&+&+&+&-&+&+&+&-&+&+&+\\
	-&+&+&+&+&-&-&-&-&+&+&+&-&+&+&+\\
	-&+&+&+&-&+&+&+&+&-&-&-&-&+&+&+\\
	-&+&+&+&-&+&+&+&-&+&+&+&+&-&-&-\\
	-&+&-&-&+&-&+&+&+&-&+&+&+&-&+&+\\
	+&-&+&+&-&+&-&-&+&-&+&+&+&-&+&+\\
	+&-&+&+&+&-&+&+&-&+&-&-&+&-&+&+\\
	+&-&+&+&+&-&+&+&+&-&+&+&-&+&-&-\\
	-&-&+&-&+&+&-&+&+&+&-&+&+&+&-&+\\
	+&+&-&+&-&-&+&-&+&+&-&+&+&+&-&+\\
	+&+&-&+&+&+&-&+&-&-&+&-&+&+&-&+\\
	+&+&-&+&+&+&-&+&+&+&-&+&-&-&+&-\\
	-&-&-&+&+&+&+&-&+&+&+&-&+&+&+&-\\
	+&+&+&-&-&-&-&+&+&+&+&-&+&+&+&-\\
	+&+&+&-&+&+&+&-&-&-&-&+&+&+&+&-\\
	+&+&+&-&+&+&+&-&+&+&+&-&-&-&-&+\\
	\end{array}\right]},\quad
&\setlength{\arraycolsep}{.8pt}{\fontsize{5}{6}\selectfont\left[\begin{array}{rrrrrrrrrrrrrrrr}
	+&-&-&-&-&+&+&+&-&+&+&+&-&+&+&+\\
	-&+&+&+&+&-&-&-&-&+&+&+&-&+&+&+\\
	-&+&+&+&-&+&+&+&+&-&-&-&-&+&+&+\\
	-&+&+&+&-&+&+&+&-&+&+&+&+&-&-&-\\
	+&-&+&+&-&+&-&-&+&-&+&+&+&-&+&+\\
	+&-&+&+&+&-&+&+&-&+&-&-&+&-&+&+\\
	+&-&+&+&+&-&+&+&+&-&+&+&-&+&-&-\\
	-&+&-&-&+&-&+&+&+&-&+&+&+&-&+&+\\
	+&+&-&+&+&+&-&+&-&-&+&-&+&+&-&+\\
	+&+&-&+&+&+&-&+&+&+&-&+&-&-&+&-\\
	-&-&+&-&+&+&-&+&+&+&-&+&+&+&-&+\\
	+&+&-&+&-&-&+&-&+&+&-&+&+&+&-&+\\
	+&+&+&-&+&+&+&-&+&+&+&-&-&-&-&+\\
	-&-&-&+&+&+&+&-&+&+&+&-&+&+&+&-\\
	+&+&+&-&-&-&-&+&+&+&+&-&+&+&+&-\\
	+&+&+&-&+&+&+&-&-&-&-&+&+&+&+&-\\
	\end{array}\right]},&\setlength{\arraycolsep}{.8pt}{\fontsize{5}{6}\selectfont\left[\begin{array}{rrrrrrrrrrrrrrrr}
	+&-&-&-&+&-&+&+&+&+&-&+&+&+&+&-\\
	-&+&+&+&-&+&-&-&+&+&-&+&+&+&+&-\\
	-&+&+&+&+&-&+&+&-&-&+&-&+&+&+&-\\
	-&+&+&+&+&-&+&+&+&+&-&+&-&-&-&+\\
	+&+&+&-&+&-&-&-&+&-&+&+&+&+&-&+\\
	+&+&+&-&-&+&+&+&-&+&-&-&+&+&-&+\\
	+&+&+&-&-&+&+&+&+&-&+&+&-&-&+&-\\
	-&-&-&+&-&+&+&+&+&-&+&+&+&+&-&+\\
	+&+&-&+&+&+&+&-&+&-&-&-&+&-&+&+\\
	+&+&-&+&+&+&+&-&-&+&+&+&-&+&-&-\\
	-&-&+&-&+&+&+&-&-&+&+&+&+&-&+&+\\
	+&+&-&+&-&-&-&+&-&+&+&+&+&-&+&+\\
	+&-&+&+&+&+&-&+&+&+&+&-&+&-&-&-\\
	-&+&-&-&+&+&-&+&+&+&+&-&-&+&+&+\\
	+&-&+&+&-&-&+&-&+&+&+&-&-&+&+&+\\
	+&-&+&+&+&+&-&+&-&-&-&+&-&+&+&+\\
	\end{array}\right]}.\\
\end{aligned}\]
This gives us a rank $v$ idempotent
\[M = \frac{1}{12}\left[\begin{array}{ccc}
4I & H_{12}& H_{13}\\
H_{12}^T & 4I & H_{23}\\
H_{13}^T& H_{23}^T & 4I
\end{array}\right].\]
Replacing the positive entries of the off-diagonal blocks of this matrix with ones and all other entries with zeros gives us the adjacency matrix of a $\mu$-heavy $\text{LSSD}(16,10,6;3)$.
\chapter{Orthogonal projective double covers}
\label{4classbip}
Throughout this thesis we have examined many different spherical $t$-distance sets and their close relation to association schemes. In some cases, we may find an equivalence between such $t$-distance sets and certain classes of association schemes; thus is the case discussed in Chapter \ref{3class} where we found 3-class $Q$-antipodal association schemes were equivalent to the 3-distance sets called linked simplices. In other cases we find that equivalence occurs only when we impose additional restrictions on the $t$-distance set, for instance only certain sets of equiangular lines occur within 3-class $Q$-bipartite association schemes. In this chapter, we introduce a line system known as a \emph{projective double} of a graph, denoted $PD_m(\Gamma;\beta,\delta)$. This object is a set of lines in $\mathbb{R}^m$ with two possible angles, $\beta$ and $\delta$, where $\Gamma$ is the incidence graph on $\beta$. We note that projective double covers of complete graphs $(\Gamma = K_n)$ give sets of equiangular lines, and thus are well-studied in the literature. Here, as with sets of equiangular lines, $\beta$ and $\delta$ represent the cosine of the angles between lines; despite this we will abuse notation and refer to them as the ``angles" of our projective double. We find that these systems are closely related to $4$- and $5$-class $Q$-bipartite association schemes. In this chapter, we will focus on the $4$-class case and define a related line system called an \emph{orthogonal projective double} of a graph $\Gamma$, denoted $OPD_m(\Gamma;\beta)$. These are projective doubles for which $\delta = 0$. We begin by considering the general objects but will add restrictions along the way with the motivation of finding an equivalence between these objects and our association schemes. This material in this chapter contains joint work with W.\ J.\ Martin. Below is a list of the main results found in this chapter.
\begin{restatable*}{prop}{projsrg}\label{projnec}
	An $OPD_m(\Gamma;\beta)$ for simple graph $\Gamma$ induces an association scheme only if $\Gamma$ is strongly regular.
\end{restatable*}
\begin{restatable*}{cor}{cocliquealpha}\label{snegsquare}
	Let $\Gamma$ be a connected strongly regular graph with $v$ vertices and eigenvalues $k>r>s$ which contains a Delsarte coclique $C$ of size $m$. Then an $OPD_{m}(\Gamma;\beta)$ exists only if $\beta=\frac{1}{\sqrt{-s}}$. Further, either $\Gamma$ is complete bipartite or $\beta^{-1}\in \mathbb{Z}$.
\end{restatable*}
\begin{restatable*}{cor}{dimiffassoc}\label{dimiffassoc}
	Let $\Gamma$ be a connected strongly regular graph with $v$ vertices and eigenvalues $k>r>s$. An $OPD_{m}(\Gamma;\beta)$ with $m<v$ induces an association scheme if and only if $m = v\left(1+k\beta^2\right)^{-1}$.
\end{restatable*}

\begin{restatable*}{thm}{fourclasssixzero}\label{thm60}
	Suppose we have a feasible parameter set for a $4$-class association scheme which is $Q$-bipartite but not $Q$-antipodal. Let $k=P_{01}$, $r=P_{21}$, and $s=P_{41}$ where $P$ is the first eigenmatrix using the natural ordering. Then the scheme is realizable only if $s=-n^2$ for some integer $n>1$ and
	\[15n^4(2n^2-3)r^2 + (n^6-45kn^2+76k)n^2r+k(16k+n^6)(n^2-2)\geq 0.\]
\end{restatable*}

\section{Orthogonal projective doubles of regular graphs}
As we saw with other examples, it becomes useful to represent a line system by a set of unit vectors, defining the ``angle" between two lines to be the absolute value of the inner product between representative unit vectors. As we have two choices of a representative for each line, we have many equivalent such representations of any line set. In this chapter, we will instead include both directions; that is, we will represent our line system with an antipodal set of unit vectors with twice as many vectors as there are lines. Further, we will describe a projective double by not only the parameters $m$, $\beta$, and $\delta$ but also the graph which is induced on the set of lines where two lines are adjacent if they intersect in angle $\beta$.

Let $\Gamma = \Gamma(V,E)$ be an undirected graph on $v$ vertices. A \emph{projective double}\index{projective double} ($PD_m(\beta,\delta)$) of $\Gamma$ is an antipodal spherical code, say $L = \left\{\ell_1,\dots,\ell_{2v}\right\}\subset \mathbb{R}^m$ with inner products $A = \left\{\pm 1,\pm \beta,\pm\delta\right\}$ such that there exists a surjective mapping $\phi:L\rightarrow V$ with the properties
\begin{enumerate}[label=$(\roman*)$]
	\item $\phi(\ell_i) = \phi(\ell_j)$ if and only if $\left\vert\left<\ell_i,\ell_j\right>\right\vert = 1$,
	\item $\phi(\ell_i) \sim \phi(\ell_j)$ if and only if $\left\vert\left<\ell_i,\ell_j\right>\right\vert = \beta$	
\end{enumerate}
for $1\leq i,j\leq 2v$. Note that, since $L$ is an antipodal set, any subset of size $v$ in which the inner product $-1$ does not occur will completely determine our entire set. This also tells us that we may order our vectors, say $\ell_1,\dots,\ell_v,-\ell_1,\dots,-\ell_v$, such that the Gram matrix corresponding to this ordering is 
\[\left[\begin{array}{cc}
G & -G\\
-G & G
\end{array}\right]\]
where $G$ is the Gram matrix of the set $\left\{\ell_1,\dots,\ell_v\right\}$. Further, this principal submatrix $G$ gives us the adjacency matrix of $\Gamma$ when we zero out the main diagonal and replace all $\pm\beta$ entries with $1$ and $\pm\delta$ entries with $0$.

We note that $PD_m(\beta,\delta)$ of $\Gamma$ is immediately a $PD_m(\delta,\beta)$ of $\overline{\Gamma}$ --- for this reason, we adopt the convention that $\beta>\delta$ for all of our projective doubles. So long as $\delta>0$, a $PD_m(\Gamma;\beta,\delta)$ will be a spherical $5$-distance set and thus the Gram matrix of such a projective double cannot belong to the Bose-Mesner algebra of a 4-class association scheme. Thus, our first restriction is that $\delta = 0$ and we define an \emph{orthogonal projective double},\index{orthogonal projective double} $OPD_m(\Gamma;\beta)$, as a projective double of $\Gamma$ for which $\delta = 0$. It is this class of objects which we will focus on for the entirety of this chapter.

Our primary goal is to determine which $OPD$s induce association schemes; that is, when is the Schur closure of the Gram matrix a Bose-Mesner algebra. We begin by asking the simpler question --- for which graphs does there exist an OPD. We find quickly that without any restrictions on the dimension, we may always find an $OPD$ for a given graph $\Gamma$ --- consider the following two propositions.
\begin{prop}\label{dirnaive}
	For any non-empty simple graph $\Gamma(V,E)$, there exists an $OPD_m(\Gamma;\frac{1}{d})$ for some $m\leq \vert V\vert$ with $d$ the maximum degree of $\Gamma$.
\end{prop}
\begin{proof}
	Let $\Gamma = \Gamma(V,E)$ be given. Now orient every edge of $\Gamma$ and define $e_i^+,e_i^-\in V$ so that $e_i = (e_i^-,e_i^+)$, thus $e_i$ points from vertex $e_i^-$ to $e_i^+$. Let $M$ be the matrix with rows indexed by vertices and columns indexed by edges such that
	\[\left[M\right]_{ij} =\begin{cases} 1 &\text{ if }v_i = e_j^+\\
	-1 &\text{ if }v_i = e_j^-\\
	0 &\text{ otherwise. }
	\end{cases}\]
	Then we find 
	\[\left[MM^T\right]_{ij} = \begin{cases}
	k_i & \text{ if }i=j,\\
	1 & \text{ if } i\sim j,\\
	0 & \text{ otherwise}
	\end{cases}\]
	where $k_i$ is the degree of vertex $i$. Thus two distinct rows are orthogonal if and only if their corresponding vertices are non-adjacent. However unless $k_i$ is constant independent of $i$ (i.e.\ unless the graph is regular), the rows do not have the same norm. To fix this, let $d = \max_{i}(k_i)$ (the maximum degree of $\Gamma$) and define the diagonal matrix $D$ whose $i^\text{th}$ diagonal entry is $\sqrt{d-k_i}$. Then the matrix $N = \left[\begin{array}{c|c}
	M & D
	\end{array}\right]$ has the property that
	\[\left[NN^T\right]_{ij} = \left[MM^T\right]_{ij} + (d-k_i)\delta_{ij} = \begin{cases}
	d & \text{ if }i=j,\\
	1 & \text{ if } i\sim j,\\
	0 & \text{ otherwise.}
	\end{cases}\]
	Thus the rows of $\frac{1}{\sqrt{d}}N$, along with their negatives, result in an orthogonal projective double of $\Gamma$ with angle $\frac{1}{d}$. Further, the rank of $N$ is no larger than $\min\left\{\vert V\vert, \vert V\vert+\vert E\vert\right\} = \vert V\vert$ and thus $m\leq \vert V\vert$.
\end{proof}
\begin{cor}\label{regnaive}
	Let $\Gamma$ be a regular graph with valency $k>0$. There exists an $OPD_m(\Gamma;\frac{1}{k})$ for some $m\leq \vert T\vert$ where $T$ a spanning forest.
\end{cor}
\begin{proof}
	We follow the same approach as in the proof of Proposition \ref{dirnaive}, however we note that since our graph is regular, the rows of $M$ all have the same norm. Thus the normalized rows (with their negatives) suffice as our orthogonal projective double. Now, consider any cycle $C$ in $\Gamma$ and assume without loss of generality that $C = \left\{e_1,\dots,e_s\right\}$. Then, replacing a column with its negative if necessary, $\displaystyle{M_{e_s} = \sum_{i=1}^{s-1}M_{e_{i}}}$, where $M_e$ denote the column of $M$ indexed by $e$. We then reorder the columns of $M$ so that the first $T$ edges correspond to the edges of a spanning forest and note that every remaining column lies in the span of these.
\end{proof}
Proposition \ref{dirnaive} and Corollary \ref{regnaive} provide upper bounds on the dimension necessary for an $OPD$ to exist for a given graph. The following observation gives us a lower bound using the size of a largest coclique in $\Gamma$. We denote the cardinality of any such maximum coclique as the independence number, $\alpha\left(\Gamma\right)$, of the graph $\Gamma$.
\begin{prop}\label{cocliquebnd}
	Let $\Gamma$ be a simple graph for which an $OPD_m(\Gamma;\beta)$ exists. Then $m\geq \alpha\left(\Gamma\right)$ where $\alpha\left(\Gamma\right)$ is the independence number of $\Gamma$.
\end{prop}
\begin{proof}
	Assume $\left\{\pm\ell_1,\dots,\pm\ell_{\vert V\vert}\right\}$ is an $OPD_m(\Gamma;\beta)$ with $\phi(\pm\ell_i) = v_i$ for $1\leq i\leq\vert V\vert$. Let $\alpha = \alpha\left(\Gamma\right)$ and, without loss of generality, let $S = \left\{v_1,\dots,v_\alpha\right\}$ be an independent set. Then $\left\{\ell_1,\dots,\ell_\alpha\right\}$ is an orthonormal set, forcing $m\geq\alpha$.
\end{proof}
\begin{example}\label{doublecoverc4}
	Consider the graph $C_4$. This is a regular graph with 3 edges in any spanning tree, thus Corollary \ref{regnaive} tells us there exists an orthogonal projective double in $\mathbb{R}^3$ with angle $\frac{1}{2}$. In fact, the columns of $U_1$ serve as one such $OPD$, where 
	\[U_1 = \left[\begin{array}{rrrrrrrr}
	1 & -1 & \frac{1}{2} & -\frac{1}{2} & 0 & 0 &\frac{1}{2} &-\frac{1}{2}\\
	0 & 0 & \frac{\sqrt{3}}{2} & -\frac{\sqrt{3}}{2} & \frac{1}{\sqrt{3}} & -\frac{1}{\sqrt{3}} & -\frac{1}{\sqrt{12}}& \frac{1}{\sqrt{12}}\\
	0 & 0 & 0 & 0 & \sqrt{\frac{2}{3}} & -\sqrt{\frac{2}{3}} & \sqrt{\frac{2}{3}} & -\sqrt{\frac{2}{3}} \\			
	\end{array}\right].\]
	A largest coclique in this graph has size $2$ and thus Proposition \ref{cocliquebnd} allows for the possibility of an $OPD_2(C_4;\beta)$. While it is not hard to show that we cannot find an $OPD_2(C_4;\beta)$ with $\beta=\frac{1}{2}$, the following is an example in dimension $2$ with angle $\frac{1}{\sqrt{2}}$.
	\[U_2 = \left[\begin{array}{rrrrrrrr}
	1 & -1 & 0 & 0 & \frac{\sqrt{2}}{2} & -\frac{\sqrt{2}}{2} & \frac{\sqrt{2}}{2} & -\frac{\sqrt{2}}{2}\\
	0 & 0 & 1 & -1 & \frac{\sqrt{2}}{2} & -\frac{\sqrt{2}}{2} & -\frac{\sqrt{2}}{2} & \frac{\sqrt{2}}{2}\\				
	\end{array}\right]\]
	To illustrate the difference between the two $OPD$'s above, we define for any $OPD_m(\Gamma;\beta)$ the graph $\Gamma_\beta$ on the $OPD$ with two vectors adjacent if and only if their inner product is $\beta$. Using the columns of $U_1$ as our $OPD$, we find that $\Gamma_{\frac{1}{2}}$ is given below where $v_i$ represents the $i^\text{th}$ column of $U_1$.
	\[\begin{tikzpicture}[shorten >=1pt,auto,node distance=2cm,
	thin,main node/.style = {circle,draw, inner sep = 0pt, minimum size = 15pt}]
	
	\node[main node,fill=white] (1) {$v_1$};
	\node[main node,fill=white] [right of = 1](2) {$v_3$};
	\node[main node,fill=white] [above of = 1](3) {$v_7$};
	\node[main node,fill=white] [above of =2](4) {$v_5$};
	\node[main node,fill=white] [right of =2](5) {$v_2$};
	\node[main node,fill=white] [right of = 5](6) {$v_4$};
	\node[main node,fill=white] [above of = 5](7) {$v_8$};
	\node[main node,fill=white] [above of =6](8) {$v_6$};
	\node at (3.1,2.5) (9) {$\Gamma_{\frac{1}{2}}$};
	
	\path[-]
	(1) edge node {} (2)
	edge node {} (3)
	(4) edge node {} (3)
	edge node {} (2)
	(5) edge node {} (6)
	edge node {} (7)
	(8) edge node {} (7)
	edge node {} (6);
	\end{tikzpicture}\]
	Alternatively, if we consider $U_2$, $\Gamma_{\frac{1}{\sqrt{2}}}$ is given below.
	\[\begin{tikzpicture}[shorten >=1pt,auto,node distance=2cm,
	thin,main node/.style = {circle,draw, inner sep = 0pt, minimum size = 15pt}]
	\node[main node,fill=white] (1) {$v_1$};
	\node[main node,fill=white] [below right of = 1](2) {$v_5$};
	\node[main node,fill=white] [above of = 1](3) {$v_7$};
	\node[main node,fill=white] [above right of =3](4) {$v_4$};
	
	\node[main node,fill=white] [right of =2](5) {$v_3$};
	\node[main node,fill=white] [above right of = 5](6) {$v_8$};
	\node[main node,fill=white] [right of = 4](7) {$v_6$};
	\node[main node,fill=white] [above of =6](8) {$v_2$};
	\node at (2.5,4) (9) {$\Gamma_{\frac{1}{\sqrt{2}}}$};
	
	\path[-]
	(1) edge node {} (2)
	edge node {} (3)
	(4) edge node {} (3)
	edge node {} (7)
	(5) edge node {} (6)
	edge node {} (2)
	(8) edge node {} (7)
	edge node {} (6);
	\draw[-,dashed] (1) --(8) -- (6) --(3) --(4) -- (5) -- (2) -- (7);
	\end{tikzpicture}\]
	Thus $U_1$ and $U_2$ are clearly different as their corresponding graphs are non-isomorphic double covers of $C_4$. We similarly define the five relations $R_0,\dots,R_4$ on any $OPD_m(\Gamma;\beta)$ via the inner products $1$, $\beta$, $0$, $-\beta$, and $-1$ respectively so that, for instance, $(x,y)\in R_2$ if and only if the corresponding vectors are orthogonal. These five relations satisfy conditions \emph{(i) -- (iii)} of an association scheme, leaving condition \emph{(iv)} as our final test; that is, we must determine if the intersection numbers are well-defined. In the case of the first $OPD$, we find that even though both $(v_1,v_5)$ and $(v_1,v_6)$ are in $R_2$,
	\[\left\vert\left\{v_x : (v_x,v_1),(v_5,v_x)\in R_{\beta}\right\}\right\vert\neq\left\vert\left\{v_x : (v_x,v_1),(v_6,v_x)\in R_{\beta}\right\}\right\vert.\]
	On the other hand, the second $OPD$ has well-defined intersection numbers and we find that the columns of $U_2$, along with the relations $R_0,\dots,R_4$, give the association scheme $C_8$.
\end{example}
As one might expect, there are many graphs for which we cannot find an orthogonal projective double in the dimension given by the independence number. In other words, Proposition \ref{cocliquebnd} is often not tight. To see an example, consider the following proposition

\begin{prop}\label{completemulti}
	Let $\Gamma$ be a complete multipartite graph with $w$ parts of size $v$. Let $U$ be the matrix with columns corresponding to an $OPD_{\alpha(\Gamma)}(\Gamma;\beta)$. Then a subset of the columns of $U$ form a set of $w$ mutually unbiased bases in $\mathbb{R}^v$.
\end{prop}
\begin{proof}
	Let $\Gamma = \overline{wK_v}$ be the complete multipartite graph with $w$ parts of size $v$. Label the vertices as $v_{i,j}$ for $1\leq i\leq w$, $1\leq j\leq v$ so that the sets $S_i = \left\{v_{i,j}\right\}_{1\leq j\leq v}$ for fixed $1\leq i\leq w$ give the maximal independent sets. Assume $\left\{\pm\ell_{i,j}\right\}_{1\leq i\leq w;1\leq j\leq v}$ is an $OPD_m(\Gamma;\beta)$ with $\phi(\pm\ell_{i,j}) = v_{i,j}$ for $1\leq i\leq w$, $1\leq j\leq v$. Then each of the sets $B_i = \left\{\ell_{i,j}\right\}_{1\leq j\leq v}$ form an orthonormal basis for $\mathbb{R}^v$. Further, the inner products between vectors in distinct bases will be one of $\pm \beta$. Thus $\left\{B_1,\dots,B_w\right\}$ is a set of $w$ mutually unbiased bases in $\mathbb{R}^v$.
\end{proof}
\begin{cor}\label{nonexist}
	Let $\Gamma = K_{t,t}$ for $t\neq 0\mod 4$. There does not exist an $OPD_{\alpha(\Gamma)}(\Gamma;\beta)$.
\end{cor}
\begin{proof}
	If $K_{t,t}$ had an orthogonal projective double in $\mathbb{R}^t$, Proposition \ref{completemulti} would guarantee the existence of two mutually unbiased bases in $\mathbb{R}^t$. This is only possible when $t$ is a multiple of $4$.
\end{proof}

\section{Schemes induced by projective doubles}
The OPD given by the columns of $U_2$ in Example \ref{doublecoverc4} provides an orthogonal projective double which gives an association scheme on the vectors by relating vectors based on their inner product. In this section, we consider which graphs produce such an association scheme and some properties of the association scheme which arise. First, let $L = \left\{\ell_i\right\}$ be an orthogonal projective double of some graph $\Gamma$ and let $G$ be the Gram matrix of $L$; that is, the matrix whose entry in row $i$ and column $j$ is $\left<\ell_i,\ell_j\right>$. We denote by $\left<G\right>_\circ$ the smallest vector space of matrices  containing $G$ and closed under entrywise products; we say this vector space is the \emph{Schur closure} of $G$. Note that the adjacency matrices of each graph $\Gamma_{\theta}$ for $\theta\in\left\{\pm 1,\pm \beta,0\right\}$ are all contained in $\left<G\right>_{\circ}$. Thus $\left<G\right>_\circ$ is a Bose-Mesner algebra if and only if it is closed under standard matrix multiplication. If this occurs, we say $L$ induces the corresponding 4-class association scheme. Thus, from Example \ref{doublecoverc4}, the association scheme $C_8$ is induced by the columns of $U_2$. We similarly define $\left<A\right>_*$ for any matrix $A$ and note that this algebra is a Bose-Mesner algebra if and only if it is closed under Schur products.

Recall from Section \ref{BMA} that strongly regular graphs are 2-class association schemes. While many strongly regular graphs are nice examples and thus have been seen throughout this thesis, in this chapter we will need strongly regular graphs more explicitly and thus we define them here. A \emph{strongly regular graph}\index{strongly regular graph} (see \cite{Brouwer1989}) with parameters $(v,k,\lambda,\mu)$ is a $k$-regular graph with $v$ points where every pair of adjacent vertices have exactly $\lambda$ neighbors in common while distinct non-adjacent vertices have $\mu$ neighbors in common. Using the terminology of association schemes, a strongly regular graph is a 2-class association scheme with parameters $k = p^0_{11}$, $\lambda = p^1_{11}$ and $\mu = p^2_{11}$.

\projsrg
\begin{proof}
	We will prove our result by showing that $\overline{\Gamma}$, the complement of $\Gamma$, is strongly regular. Let $V$ be the vertex set of $\Gamma$ and define $v = \vert V\vert$. Now, let $L = \left\{\ell_1,\dots,\ell_{2v}\right\}$ be our $OPD$ with the projective mapping $\phi:L\rightarrow V$. First let $R_0,\dots,R_4$ be the relations of the association scheme induced by $L$ where $R_2$ is given by orthogonality, $R_0$ is the identity relation, and the remaining relations are given by the inner products $\beta,-\beta,$ and $-1$ respectively. By definition of our mapping, $\phi(\ell)=\phi(\ell^\prime)$ for $\ell\neq \ell^\prime$ if and only if $\ell=-\ell^\prime$. Thus for distinct vertices $u,w\in V$, $u\not\sim w$ if and only if $\phi^{-1}(w)\subset \phi^{-1}(u)^\perp$. Then the number of vertices not adjacent to $u$ is half the number of vectors orthogonal to either vector in $\phi^{-1}(u)$; this value is $\frac{1}{2}p^0_{22}$. Similarly, assuming $u\not\sim w$, the number of vertices adjacent to neither $u$ nor $w$ must be half the number of vectors orthogonal to any pair of vectors, one from $\phi^{-1}(v)$ and the other from $\phi^{-1}(w)$; that is, $\frac{1}{2}p^2_{22}$. Similarly, assume $v\sim w$ and we find the number of vertices adjacent to neither $v$ nor $w$ must be $\frac{1}{2}p^1_{22} = \frac{1}{2}p^3_{22}$. Thus $\overline{\Gamma}$ is strongly regular with parameters $(v,\frac{1}{2}p^0_{22},\frac{1}{2}p^{2}_{22},\frac{1}{2}p^1_{22})$.
\end{proof}

This proposition tells us that we must only consider strongly regular graphs if we wish to find orthogonal projective doubles which induce association schemes. Note that the converse of Proposition \ref{projnec} is certainly not true; the first projective double of Example \ref{doublecoverc4} does not result in an association scheme even though $C_4$ is strongly regular. Thus we will look for further necessary or sufficient conditions for a $OPD$ to induce an association scheme. Before we continue, we review a few details of strongly regular graphs which will be useful for us.

Let $\Gamma$ be a strongly regular graph with parameters $(v,k,\lambda,\mu)$. Let $R_1$ be the relation given by adjacency in $\Gamma$ and define parameters $r,s,f,g$ so that the spectrum of $\Gamma$ is $k^1,r^f,s^g$. Using the first and second orthogonality relations (Lemma \ref{orthorels}) we find the first and second eigenmatrices of the association scheme are:
\begin{equation}\label{PQsrg}P = \left[\begin{array}{ccr}
1 & k & v-k-1\\
1 & r & -(r+1)\\
1 & s & -(s+1)
\end{array}\right],\qquad Q = \left[\begin{array}{ccc}
1 & f & g\\
1 & \frac{fr}{k} & \frac{gs}{k}\\
1 & \frac{f(1+r)}{k+1-v} & \frac{g(1+s)}{k+1-v}
\end{array}\right].\end{equation}
The following lemma from \cite{Brouwer1989}, shows us that the parameters $k,r,$ and $s$ are sufficient to define all other parameters as long as $k+rs\neq 0$. In the case of $k+rs=0$, $\Gamma$ is a union of cliques and $v$ is not uniquely determined by the spectrum --- we will ignore this case in our discussion.
\begin{lem}\cite[Theorem.~1.3.1.(iii,vi)]{Brouwer1989}\label{srgparams} Whenever $k+rs>0$, the parameters of a strongly regular graph may be expressed in terms of $r$, $s$, and $k$ with $g = v-f-1$:
	\[\mu = k+rs, \qquad v = \frac{(k-r)(k-s)}{\mu},\qquad \lambda = \mu+r+s,\qquad f = \frac{(s+1)k(k-s)}{\mu(s-r)}.\]
\end{lem}
The association scheme structure allows us to improve on the naive upper bound given in Corollary \ref{regnaive} by using the techniques discussed in Section \ref{equianglines}.
\begin{thm}
	Let $\Gamma$ be a strongly regular graph with $v$ vertices and spectrum $k^1,$ $r^f$, $s^g$ ($r>s$). There exists an $OPD_{f+1}(\Gamma;\beta)$ where $\beta = \bigslant{k+r(v-1)}{v+r-k}$.
\end{thm}
\begin{proof}
	Let $A_0=I$. Let $A_1$ and $A_2$ be the adjacency matrices of $\Gamma$ and $\overline{\Gamma}$ respectively. From Equation \eqref{PQsrg}, $E_1 = \frac{1}{v}\left(fA_0 + \frac{fr}{k}A_1 + \frac{f(1+r)}{k+1-v}A_2\right)$. Then the matrix
	\[G = \frac{(1+r)}{v-k-1}E_0 + \frac{1}{f}E_1 = \left(\frac{v+r-k}{v(v-k-1)}\right)A_0 +\left(\frac{k+r(v-1)}{v(v-k-1)}\right)A_1\]
	is a $v\times v$ positive semidefinite matrix with rank $f+1$ with the off diagonal entries we seek. We may then find an $(f+1)\times v$ matrix $U$ such that $\left(\frac{v(v-k-1)}{v+r-k}\right)G = U^TU$; that is, the columns of $U$ are unit vectors in $\mathbb{R}^{f+1}$ such that $u_i\perp u_j$ if the corresponding points in $X$ are related by $R_2$ and $\left<u_i,u_j\right>$ is otherwise constant for $i\neq j$. Then $L = \left\{\pm u_1,\dots,\pm u_v\right\}$ is an $OPD_{f+1}\left(\Gamma;\frac{k+r(v-1)}{v+r-k}\right)$ where $u_i$ is the $i^\text{th}$ column of $U$.
\end{proof}
Note that this construction does not induce a 4-class association scheme. We see this by noting that all off diagonal entries in $G$ are non-negative. Thus we may split our $OPD$ into two sets $L^+$ and $L^-$ where $L^+$ contains all the columns of $U$ and $L^-$ contains their negatives. Then for vectors $u\perp v$, the number of vectors $w$ such that $\left<v,w\right> = \left<u,w\right>=\beta$ could be either $0$ (if $v\in L^+$ and $u\in L^-$) or $\lambda$ (if $v,w\in L^+$). Thus this value is not solely dependent on the inner product $\left<u,v\right>$ and $p^2_{11}$ is not well defined. While this does not solve our question of which $OPD$'s induce association schemes, it does provide us with a better upper bound on the dimension needed for strongly regular graphs. For example, this gives an $OPD$ of the Petersen graph in dimension $5$ while Corollary \ref{regnaive} produces one in dimension $9$. Now consider the following theorem of Delsarte.

\begin{thm}\cite[Theorem 3.8]{Delsarte1973} \label{delsarte}
	Let $\Gamma=\Gamma(V,E)$ be a strongly regular graph with $v$ vertices and eigenvalues $k>r>s$. Then
	\[\alpha\left(\Gamma\right)\leq v\left(1-\frac{k}{s}\right)^{-1}. \]
	A coclique $C\subset V$ achieves equality in this bound if and only if every vertex $x\notin C$ has the same number of neighbors (namely $-s$) in $C$.\qed
\end{thm}
We will refer to a \emph{Delsarte coclique}\index{Delsarte coclique} as a coclique for which this bound is tight. This theorem, along with Proposition \ref{cocliquebnd}, gives a lower bound on the dimension of any $OPD$ in terms of the spectrum whenever $\Gamma$ contains a Delsarte coclique. Further, we may use the second half of Theorem \ref{delsarte} to learn more information about any $OPD$ achieving this bound.
\cocliquealpha
\begin{proof}
	Let $L$ be the $OPD_{m}(\beta)$ of $\Gamma$. Further, let $\ell_1,\dots,\ell_{m}$ be vectors in $L$ such that the set $\left\{\phi(\ell_1),\dots,\phi(\ell_{m})\right\}$ is a Delsarte coclique. Then $\left\{\ell_1,\dots,\ell_{m}\right\}$ forms an orthonormal basis for $\mathbb{R}^{m} = \text{span}(L)$. Let $a\in L$ be given with $\phi(a)\notin C$. By Theorem \ref{delsarte}, $\phi(a)$ must be adjacent to exactly $-s$ points in $C$ and thus, reordering the vectors and replacing $\ell_i$ with $-\ell_i$ as needed, we may assume $\left<a,\ell_i\right> = \beta$ for $1\leq i\leq -s$. Therefore $a = \sum_{i=1}^{-s} \beta\ell_i$ implying that $-s\beta^2 = 1$ and thus $s = -\beta^{-2}$. Note $\left<a,\ell_i\right>=0$ for $-s<i\leq m$. Now, as long as $\Gamma$ is not complete bipartite (i.e., provided $s\neq k$), there must be another vector $b\in L$ for which $\phi(b)\notin C$ and $\phi(b) \sim\phi(a)$; assume $\left<b,a\right> = \beta$ taking $-b$ if needed. We again find that $\phi(b)$ is adjacent to exactly $-s$ vertices in $C$; let $h$ be the number of vertices adjacent to both $a$ and $b$. Without loss of generality $b =  \sum_{i=1}^{h} \beta_i\ell_i + \sum_{i=-s+1}^{-2s-h}\beta\ell_i$ where $\beta_i = \pm\beta$. Thus $\left<a,b\right> = (p-q)\beta^2$ where $p$ is the number of vectors in $\left\{\ell_1,\dots,\ell_h\right\}$ with $\left<b,\ell_i\right> = \left<a,\ell_i\right>$ and $q = h-p$. However, since $a$ and $b$ have inner product $\beta$, this implies $\beta^{-1} = p-q$.
\end{proof}
While this theorem only provides information about $OPD$s of strongly regular graphs with Delsarte cocliques, there are many common examples which contain these cocliques for which we may apply our theorem. For instance, consider the following result.
\begin{cor}
	There do not exist $OPD$s for either the Petersen graph in $\mathbb{R}^4$ or the Paley graph on $\bbF_9$ in $\mathbb{R}^3$. 
\end{cor}
\begin{proof}
	Recall that the Petersen graph is an srg$(10,3,0,1)$ with $s=-2$. Thus a Delsarte coclique has size $\bigslant{10}{\left(1+\frac{3}{2}\right)} = 4$; we may verify quickly that such a coclique exists. Thus Corollary \ref{snegsquare} tells us a projective double of the Petersen graph in $\mathbb{R}^4$ would require that $\sqrt{-s}$ is an integer, which is false. Similarly the Paley graph on $\bbF_9$ is an srg$(9,4,1,2)$ with $s=-2$. Using the same reasoning, noting that here a Delsarte coclique has size 3, we have our result.
\end{proof}
Corollary \ref{snegsquare} hints that OPDs in the smallest possible dimension for a given graph have some extra structure imposed on them. The Lemma \ref{mindim} and Theorem \ref{dimtoassociation} will detail much of this extra structure in the general case. Before those results, we collect several useful facts about Gram matrices and spherical designs from \cite{Delsarte1977}.
\begin{lem}\cite[Thm.\ 5.5 and Ex.\ 5.7]{Delsarte1977}\label{DGS}
	Let $X$ be a spherical $s$-distance set in $\mathbb{R}^m$ with inner products $A = \left\{\alpha_1,\dots,\alpha_s\right\}$.
	\begin{itemize}
		\item[(i)] Let $A' = A\cup\left\{1\right\}$ and denote by $Q^m_k(x)$ the degree $k$ Gegenbauer polynomial. Let $d_\alpha$ denote the sum of the elements of the distance matrix $D_\alpha$ for $\alpha\in A$. Then
		\[\sum_{\alpha\in A'}d_\alpha Q^{(m)}_k(\alpha)\geq 0,\]
		and equality holds for $k=1,2,\dots,t$ if and only if $X$ is a $t$-design.
		\item[(ii)] If $X$ is an antipodal set, $X$ is a 3-design if and only if $G_x$ has two eigenvalues, $\frac{\vert X\vert}{m}$ and $0$.
	\end{itemize} 
\end{lem}
\begin{lem}\label{mindim}
	Let $\Gamma$ be a connected strongly regular graph with $v$ vertices and eigenvalues $k>r>s$. Let $G$ be the Gram matrix of an $OPD_m(\Gamma;\beta)$. Then $m\geq v\left(1+k\beta^2\right)^{-1}$ with equality if and only if $\frac{m}{2v}G$ is idempotent.
\end{lem}
\begin{proof}
	This is an immediate result of the previous theorem of Delsarte, Goethals, and Seidel. Lemma \ref{DGS} tells us $\sum_{i,j}Q^m_\ell\left(G_{ij}\right)\geq 0$ for all $\ell\geq 0$ with equality for $\ell=1,2$ if and only if $G$ is the Gram matrix of a spherical $2$-design. The antipodal nature of our OPD makes it clear that $\sum_{i,j}Q_1^m\left(G_{ij}\right) = \sum_{i,j}G_{ij} = 0$. Using the second degree Gegenbauer polynomial (see Equation \eqref{gegdef}) we find
	\[\sum_{i,j}Q_2^m\left(G_{ij}\right) = 2v\left(\frac{m\left(2+2k\beta^2\right)-2v}{m-1}\right).\]
	Thus we must have $m\left(2+2k\beta^2\right)\geq2v$ with equality if and only if our OPD admits a spherical $2$-design. The latter half of Lemma \ref{DGS} tells us this occurs exactly when $G$ has two distinct eigenvalues: $0$ and $\frac{2v}{m}$; that is, $\frac{m}{2v}G$ is idempotent.
\end{proof}
We note that, while $Q_2^m$ denoted the second degree Gegenbauer polynomial in this lemma, for the remainder of the chapter we will exclusively use $Q$ to denote the second eigenmatrix of an association scheme. The next few results give us a close connection between whether an OPD gives rise to an association scheme and the dimension which the OPD is in. These results are summarized in Corollary \ref{dimiffassoc}.

\begin{thm}\label{dimtoassociation}
	Let $\Gamma$ be a connected strongly regular graph with $v$ vertices and eigenvalues $k>r>s$. Let $G$ be the Gram matrix of an $OPD_m(\Gamma;\beta)$ with $m = v\left(1+k\beta^2\right)^{-1}$. Then $\left<G\right>_\circ$ is the Bose-Mesner algebra of a 4-class association scheme.
\end{thm}
\begin{proof}
	We begin by ordering the vectors in our orthogonal projective double so that $\ell_1,\dots,\ell_v$ are representatives from distinct lines and $\ell_{i+v} = -\ell_{i}$ for $1\leq i\leq v$. Likewise we order the vertices of $\Gamma$ so that $\ell_i$ and $\ell_{v+i}$ are mapped to vertex $i$. Let $G$ be the Gram matrix of our $OPD_m(\Gamma;\beta)$ with rows and columns ordered in this fashion; that is, $G_{ij} = \left<\ell_i,\ell_j\right>$ for $1\leq i,j\leq 2v$. This ordering implies there exists a matrix $\tilde{G}$\footnote{Here, we use $\sim$ only to emphasize that $\tilde{G}$ is the Gram matrix induced on a subset of the vertices. This matrix does not belong to the Bose-Mesner algebra of the quotient scheme.} such that
	\[G =\def\arraystretch{1.4}\left[\begin{array}{r:r}
	\tilde{G} & -\tilde{G}\\\hdashline[2pt/2pt]
	-\tilde{G} & \tilde{G}\\[2pt]
	\end{array}\right].\]
	Now let $\tilde{A}_1$ and $\tilde{A}_2$ be the adjacency matrices of $\Gamma$ and $\overline{\Gamma}$ respectively. Let $\tilde{E}_0$, $\tilde{E}_1$, and $\tilde{E}_2$ be the minimal idempotents corresponding to the eigenvalues $k$, $r$, and $s$ respectively; that is, $\tilde{A}_1\tilde{E}_0 = k\tilde{E}_0$, $\tilde{A}_1\tilde{E}_1 = r\tilde{E}_1$, and $\tilde{A}_1\tilde{E}_2 = s\tilde{E}_2$. For each matrix, assume the rows and columns are ordered via the vertex ordering defined above. Recall that the second eigenmatrix of this association scheme is
	\[\tilde{Q} = \left[\begin{array}{ccc}
	1 & f & g\\
	1 & \frac{fr}{k} & \frac{gs}{k}\\
	1 & \frac{f(r+1)}{k+1-v} & \frac{g(s+1)}{k+1-v}
	\end{array}\right].\]
	
	We now define five matrices $E_0,\dots,E_4$, which we will show are orthogonal idempotents. First, define $E_0 = \frac{1}{2v}J$ and $E_1 = \frac{m}{2v}G$. We then define $E_2$ and $E_4$ using the idempotents of our quotient scheme via
	\[E_2 = \frac{1}{2}\def\arraystretch{1.4}\left[\begin{array}{c:c}
	\tilde{E}_1 & \tilde{E}_1\\\hdashline[2pt/2pt]
	\tilde{E}_1 & \tilde{E}_1\\[2pt]
	\end{array}\right],\qquad E_4 = \frac{1}{2}\def\arraystretch{1.4}\left[\begin{array}{c:c}
	\tilde{E}_2 & \tilde{E}_2\\\hdashline[2pt/2pt]
	\tilde{E}_2 & \tilde{E}_2\\[2pt]
	\end{array}\right].\]
	We note that the definition of $E_0$ implies an analogous structure using $\tilde{E}_0$. It follows that $E_0$, $E_2$, and $E_4$ are pairwise orthogonal idempotents. Now, Lemma \ref{mindim} tells us that $E_1$ is also idempotent since $m = v\left(1+k\beta^2\right)^{-1}$. Additionally, the block structure of the matrices $E_0$, $E_1$, $E_2$, and $E_4$ imply that $E_1E_0=E_1E_2=E_1E_4 = 0$ telling us that $E_0$, $E_1$, $E_2$, $E_4$ are pairwise orthogonal idempotents. In view of Lemma \ref{AElem} \emph{($\mathit{iii'}$)}, we define $E_3 = I-E_0-E_1-E_2-E_4$ and may immediately compute.
	\[E_3^2 = \left(I-E_0-E_1-E_2-E_4\right)^2 = I-E_0-E_1-E_2-E_4\]
	and $E_3E_i = 0$ for $i\neq 3$. Therefore the vector space $\left<E_0,E_1,E_2,E_3,E_4\right>$ is symmetric, closed under matrix multiplication, and contains both the identity matrix and the all ones matrix. In order to show this vector space is a Bose-Mesner algebra, we must also show it is closed under entrywise products. First note that since $G$ contains exactly five distinct entries, we have $\left<G\right>_\circ = \left<A_1,A_{\beta},A_0,A_{-\beta},A_{-1}\right>$ where, for each $\theta\in\left\{\pm 1, \pm\beta, 0\right\}$,
	\[A_{\theta} = \begin{cases}
	1 & G_{ij} = \theta,\\
	0 & o.w.\\
	\end{cases}\]
	Now, the ordering of rows and columns of our quotient matrices imply the following\vspace{-1cm}
	\begin{multicols}{2}
		\[2v\left[E_2\right]_{ij} = \begin{cases}
		f & \text{ if }G_{ij} = \pm1,\\
		\frac{fr}{k} & \text{ if }G_{ij} = \pm\beta,\\
		\frac{f(r+1)}{k+1-v} & \text{ if }G_{ij} = 0,\\
		\end{cases}\]\[ 2v\left[E_4\right]_{ij} = \begin{cases}
		g & \text{ if }G_{ij} = \pm1,\\
		\frac{gs}{k} & \text{ if }G_{ij} = \pm\beta,\\
		\frac{g(s+1)}{k+1-v} & \text{ if }G_{ij} = 0.\\
		\end{cases}\]
		\null\vfill
		\[2v\left[E_3\right]_{ij} = \begin{cases}
		v-m & \text{ if }G_{ij} = 1,\\
		-m\beta & \text{ if }G_{ij} = \beta,\\
		0 & \text{ if }G_{ij} = 0,\\
		m\beta & \text{ if }G_{ij} = -\beta,\\
		m-v & \text{ if }G_{ij} = -1,\\
		\end{cases}
		\]\vfill\null
	\end{multicols}
	Thus the entries of our idempotents $E_0,\dots,E_4$ depend solely on the corresponding entries of $G$. It follows that each idempotent is contained in $\left<A_1,A_{\beta},A_0,A_{-\beta},A_{-1}\right>$ forcing $\left<E_0,E_1,E_2,E_3,E_4\right>\subset \left<G\right>_\circ$. Finally, since $\left<G\right>_\circ$ has dimension five, we must have equality. Therefore $\left<E_0,E_1,E_2,E_3,E_4\right>$ is Schur-closed and is the Bose-Mesner algebra of a 4-class association scheme. We complete our proof by listing the Krein parameters, referring to the parameter definitions in Theorem \ref{srgparams} for the strongly regular graph parameters.
	\[L_0^* =  \left[ \begin {array}{ccccc} 1&0&0&0&0\\ \noalign{\medskip}0&1&0&0&0
	\\ \noalign{\medskip}0&0&1&0&0\\ \noalign{\medskip}0&0&0&1&0
	\\ \noalign{\medskip}0&0&0&0&1\end {array} \right],\qquad L_1^* = \left[\begin{array}{ccccc}
	0 & m & 0 & 0 & 0\\
	1 & 0 & \frac{f(1+r\beta^2)}{1+k\beta^2} & 0 & \frac{g(1+s\beta^2)}{1+k\beta^2}\\
	0 & \frac{m(1+r\beta^2)}{1+k\beta^2} & 0 & \frac{\beta^2(k-r)m}{1+k\beta^2} & 0\\
	0 & 0 & \frac{f(k-r)}{k(1+k\beta^2)} & 0 & \frac{g(k-s)}{k(1+k\beta^2)}\\
	0 & \frac{m(1+s\beta^2)}{1+k\beta^2} & 0 & \frac{\beta^2(k-s)m}{1+k\beta^2} & 0
	\end{array}\right],\]
	\[L_2^* =   \left[ \begin {array}{ccccc} 0&0&f&0&0\\
	\noalign{\medskip}0&{\frac {\left( {\beta}^{2}r+1 \right)f }{  \left( {\beta}^{2}k+1 \right) }}&0&{\frac {
			\left( k-r \right) {\beta}^
			{2}f}{ \left( {\beta}^{2}k+1
			\right) }}&0\\
	\noalign{\medskip}1&0&f-1+{\frac { \left( k-r
			\right)gs}{ \left( r-s \right)k }}&0&{-\frac { \left( k-r
			\right)gs}{ \left( r-s \right)k }}\\
	\noalign{\medskip}0&{\frac { f\left( k-r \right) }{ k\left( {\beta}^{2}k+1 \right) }}&0&{\frac { \left( {\beta}^{2}{k}^{2}+r \right)f }{
			k  \left( {\beta}^{2}k+1
			\right) }}&0\\ \noalign{\medskip}0&0&{-\frac {\left( k-r \right) sf}{ \left( r-s \right)k }}&0&{\frac { \left( k-s \right) r f }{ \left( r-s \right) k }}
	\end {array} \right],\]
	\[L_3^* = \left[ \begin {array}{ccccc} 0&0&0&m{\beta}^{2}k&0\\
	\noalign{\medskip}0&0&{\frac {f{\beta}^{2} \left( k-r
			\right) }{{\beta}^{2}k+1}}&0&{\frac {g{\beta}^{2} \left( k-s
			\right) }{{\beta}^{2}k+1}}\\
	\noalign{\medskip}0&{\frac {m{\beta}^{
				2} \left( k-r \right) }{ {\beta}^{2}k+1 }}&0&{
		\frac {m{\beta}^{2} \left( {\beta}^{2}{k}^{2}+r \right) }{ {
				\beta}^{2}k+1}}&0\\ \noalign{\medskip}1&0&{\frac {f
			\left( {\beta}^{2}{k}^{2}+r \right) }{k \left( {\beta}^{2}k+1
			\right) }}&0&{\frac {g \left( {\beta}^{2}{k}^{2}+s \right) }{k
			\left( {\beta}^{2}k+1 \right) }}\\ \noalign{\medskip}0&{\frac {m{
				\beta}^{2} \left( k-s \right) }{  {\beta}^{2}k+1}
	}&0&{\frac {{\beta}^{2} \left( {\beta}^{2}{k}^{2}+s \right) }{
			{\beta}^{2}k+1}}&0\end {array} \right],\]
	\[ L_4^* = \left[ \begin {array}{ccccc} 0&0&0&0&g
	\\ \noalign{\medskip}0&{\frac {g \left( {\beta}^{2}s+1 \right) }{\left( {\beta}^{2}k+1 \right) }}&0&{\frac {
			\left( k-s \right) {\beta}^{2}g }{\left( {\beta}^{
				2}k+1 \right) }}&0\\
	\noalign{\medskip}0&0&-{\frac {\left( k-r \right)gs}{ \left( r-s \right)k}}&0&{\frac { \left( k-s \right) r g }{ \left( r-s \right) k }}
	\\ \noalign{\medskip}0&{\frac { \left( k-s \right) g }{k
			\left( {\beta}^{2}k+1 \right) }}&0&{\frac { g \left( {\beta}^{2}{k}^{2}+s \right) }{
			k\left( {\beta}^{2}k+1
			\right) }}&0\\
	\noalign{\medskip}1&0&{\frac { \left( k-s \right)rf }{ \left( r-s \right) k  }
	}&0&g-1-\frac{(k-s)rf}{(r-s)k}\end {array} \right].\qedhere\]
\end{proof}
\begin{cor}
	\label{delcocliquetoqbip}
	Let $\Gamma$ be a connected strongly regular graph with $v$ vertices and eigenvalues $k>r>s$ which contains a Delsarte coclique. Let $G$ be the Gram matrix of an $OPD_m(\Gamma;\beta)$ with $m = v\left(1-\frac{k}{s}\right)^{-1}$. Then $\left<G\right>_\circ$ is the Bose-Mesner algebra of a 4-class $Q$-bipartite association scheme.
\end{cor}
\begin{proof}
	Corollary \ref{snegsquare} tells us that $\beta = \frac{1}{\sqrt{-s}}$ and therefore $m = v\left(1+k\beta^2\right)$. Then Theorem \ref{dimtoassociation} gives that $\left<G\right>_\circ$ is the Bose-Mesner algebra of a 4-class association scheme. The $Q$-bipartite property follows as $(1+s\beta^2) = 0$, implying $L_1^*$ is tridiagonal.
\end{proof}

Theorem \ref{dimtoassociation} tells us that $OPD$s of strongly regular graphs induce association schemes whenever the dimension is tight with respect to Lemma \ref{mindim}. It turns out this is also a necessary condition as long as the dimension is not too far away from optimal; that is, $m<v$.
\begin{thm}\label{genqbip}
	Let $\Gamma$ be a strongly regular graph with $v$ vertices, valency $k$, and smallest eigenvalue $s$. Let $L$ be an $OPD_m(\Gamma;\beta)$ with $m<v$. $L$ induces an association scheme only if $m = v\left(1+k\beta^2\right)^{-1}$. Further, either $\text{rank}\left(G\circ G\right)=v$ or the induced scheme is $Q$-bipartite and $s=-\beta^{-2}$.
\end{thm}
\begin{proof}
	We prove this by building the $Q$ matrix of the resultant scheme. First let $\BMA = \left<G\right>_\circ$ and $\BMB = \left<A_\Gamma\right>_*$ where $A_\Gamma$ is the adjacency matrix of $\Gamma$. Since $G$ has five distinct values, $\BMA$ must be a 4-class association scheme with basis matrices $A_0,A_1,A_2,A_3,$ and $A_4$ corresponding to the values $1,\beta,0,-\beta,$ and $-1$. By definition of an $OPD$, we find that $R_0\cup R_4$ gives a system of imprimitivity where $\BMB$ is the quotient algebra of $\BMA$. Since $\cI = \left\{0,4\right\}$, the matrix $A_0+A_4$ must be one basis matrix; the other two matrices are $A_1+A_3$ and $A_2$. Further, there exist three basis idempotents of $\BMA$, call them $E_0$, $E_2$, and $E_4$, which span the same subalgebra as $A_0+A_4$, $A_1+A_3$, and $A_2$. By Lemma \ref{repeatedcols}, we must have $\tilde{Q}_{\tilde{k}j} = Q_{kj}$ for $j\in\left\{0,2,4\right\}$ and $0\leq k\leq 4$ where $\tilde{Q}$ is the second eigenmatrix of the strongly regular graph. Equation \eqref{PQsrg} tells us this matrix is
	\[\tilde{Q} = \left[\begin{array}{ccc}
	1 & f & g\\
	1 & \frac{fr}{k} & \frac{gs}{k}\\
	1 & \frac{f(1+r)}{k+1-v} & \frac{g(1+s)}{k+1-v}
	\end{array}\right]\]
	and thus the second eigenmatrix of $\BMA$ must be ($*$ denotes an unknown value)
	\[Q = \left[\begin{array}{crcrc}
	1 & * & f & * & g\\
	1 & * & \frac{fr}{k}  & * & \frac{gs}{k}\\
	1 & * & \frac{f(r+1)}{k+1-v}  & *& \frac{g(1+s)}{k+1-v}\\
	1 & * & \frac{fr}{k} & * & \frac{gs}{k}\\
	1 & * & f & * & g\\
	\end{array}\right].\]
	Let $n_1$ and $n_3$ be the remaining two multiplicities corresponding to $E_1$ and $E_3$ respectively. Since $1+f+g=v$ and $\vert X\vert = 2v$, we must have $n_1+n_3 = v$. Now, by construction, $G = A_0 + \beta A_1 -\beta A_3 -A_4$ and therefore any diagonal entry of $GE_2$ is
	\[\left[GE_2\right]_{ii} = \frac{1}{\vert X\vert}\left(f+k\left(\frac{fr}{k}\right)\beta - k\left(\frac{fr}{k}\right)\beta -f\right) = 0.\]
	Similar calculations for $GE_4$ and $GE_0$ show that $\text{tr}(GE_0) = \text{tr}(GE_2) = \text{tr}(GE_4)=0$. Since $G$ is contained within this commutative algebra, we find $GE_i = E_iG$ for each idempotent $E_i$ and therefore the matrices $GE_4$, $GE_2$, and $GE_0$ are all symmetric, forcing $GE_0=GE_2=GE_4=0$. Then $G = c_1E_1+c_3E_3$ for some $c_1,c_2\in\bbR$. Since $m<v$, only one of these constants may be non-zero. Without loss of generality assume $c_1\neq0$ giving $G = \frac{2v}{m}E_1$. Lemma \ref{mindim} then provides $m = v\left(1+k\beta^2\right)^{-1}$.
	
	Now we may return to our $Q$ matrix and fill in the entries of the first column. Further, the orthogonality relations (Lemma \ref{orthorels}) tell us that $\sum_{j}Q_{ij} = \delta_{0j}\vert X\vert$. Using the same fact for $\tilde{Q}$, we may find the final column as well. 
	\[Q = \left[\begin{array}{crccc}
	1 & m & f & v-m & g\\
	1 & m\beta & \frac{fr}{k}  & -m\beta & \frac{gs}{k}\\
	1 & 0 & \frac{f(r+1)}{k+1-v}  & 0& \frac{g(1+s)}{k+1-v}\\
	1 & -m\beta & \frac{fr}{k} & m\beta & \frac{gs}{k}\\
	1 & -m & f & m-v & g\\
	\end{array}\right]\]
	Since we now have the entire $Q$ matrix, we may use Lemma $\ref{kitchensink}$ \emph{($\mathit{xiii^\prime}$)} to find the Krein parameters of our scheme. In particular we find that $q^3_{11} = q^4_{12} = 0$ as well as
	\[\begin{aligned}q^2_{11} &= \frac{1}{2v f}\sum_{h=0}^d\left(k_hQ_{h1}Q_{h1}Q_{h2}\right) = \frac{m^2\left(1+\beta^2r\right)}{v},\\
	q^3_{12} &= \frac{1}{2v (v-m)}\sum_{h=0}^d\left(k_hQ_{h1}Q_{h2}Q_{h3}\right) = \frac{mf\left(v-m(1+\beta^2r)\right)}{v(v-m)},\\
	q^4_{13} &= \frac{1}{2v g}\sum_{h=0}^d\left(k_hQ_{h1}Q_{h3}Q_{h4}\right) = \frac{m\left(v-m(1+\beta^2s)\right)}{v}.\\	
	\end{aligned}\]
	Recall that $m = v(1+k\beta^2)^{-1}$ and therefore $v-m(1+\beta^2k)=0$. Thus we have both $v-m(1+\beta^2r)>0$ and $v-m(1+\beta^2s)>0$ as long as $r<k$ (i.e.\ $\Gamma$ is connected), forcing all three of the above Krein parameters to be non-zero. Thus $\BMA$ is $Q$-polynomial if and only if $q^{4}_{11}=0$. Calculating this similarly, we find
	\[q^4_{11} = \frac{1}{2v g}\sum_{h=0}^d\left(k_hQ_{h1}Q_{h1}Q_{h4}\right) = \frac{m^2\left(1+\beta^2s\right)}{v}.\]
	We therefore find $q^4_{11}=0$ if and only if $s = -\beta^{-2}$. Finally, noting that $q^1_{11}=0$ (calculated similarly), we must have $\text{rank}(G\circ G) = 1+f+g=v$ if $q^4_{11}>0$ and $\text{rank}(G\circ G) = 1+f<v$ otherwise.
\end{proof}
\dimiffassoc
\begin{proof}
	The result follows immediately from Theorems \ref{genqbip} and \ref{dimtoassociation}.
\end{proof}

From these results we are very close to the statement ``The association scheme induced by an $OPD_m(\Gamma;\beta)$ is $Q$-bipartite if and only if $\beta = \frac{1}{\sqrt{-s}}$", however this statement is ultimately false. Consider the Gram matrix of any $OPD_m(\Gamma;\frac{1}{\sqrt{-s}})$ with $m = v\left(1-\frac{k}{s}\right)^{-1}$, following the proof of Theorem \ref{genqbip} we find that $\left<G\right>_\circ$ generates a
\chapter{Connectivity of basis relations}\label{connectivity}
Brouwer and Mesner \cite{Brouwer1985} showed in 1985 that the vertex connectivity of a strongly regular graph 
$\Gamma$ is equal to its valency and that the only disconnecting sets of minimum size are the neighborhoods $\Gamma(a)$  of its vertices.  (Brouwer \cite{Brouwer1996} mentions that the corresponding result for edge connectivity
was established by Plesn\'{\i}k in 1975.) This result on vertex connectivity was extended by Brouwer and Koolen 
\cite{Brouwer2009} in 2009 to show that a distance-regular graph of valency at least three has vertex connectivity equal to its valency  and that the only disconnecting sets of minimum size are again the neighborhoods $\Gamma(a)$.  Meanwhile a conjecture of Brouwer on the size and nature of the ``second smallest'' disconnecting sets in a strongly regular graph has inspired both new results and interesting examples by Cioab\u{a}, et al.\ \cite{Cioaba2010,Cioaba2012,Cioaba2013,Cioaba2014,Cioaba2017}.

Godsil \cite{Godsil1981} conjectured in 1981 that the edge connectivity of a connected basis relation in any symmetric association scheme is equal to the valency of that graph. Brouwer \cite{Brouwer1996} claimed in 1996 that the same should hold for the vertex connectivity. In \cite{Godsil1981}, Godsil proved that if $\Gamma=\Gamma(X,R_1)$ is 
regular of valency $k_1$, then the edge connectivity of $\Gamma$ is at least $\frac{k_1}{2} \frac{|X|}{|X|-1}$.
In 2006, Evdokimov and Ponomarenko proved Brouwer's conjecture for $\Gamma=\Gamma(X,R_1)$
in the case when $(X,\cR)$ is equal to the projection onto $X$  of the $k_1$-fold
tensor product $\bigotimes_{h=1}^{k_1} (X,\cR)$. See \cite{Evdokimov2004} for definitions and details.

Much more is known about the connectivity of  vertex- and edge-transitive graphs. 
(See  \cite[Sec.~3.3-4]{Godsil2001}.)  Mader  \cite{Mader1970}  and Watkins \cite{Watkins1970}
independently obtained the following two results in 1970. The vertex connectivity of an edge-transitive graph
is equal to the smallest valency. A vertex-transitive graph of valency $k$ has vertex connectivity 
at least $\frac23(k+1)$. Further, in 1971, Mader \cite{Mader1971} proved that any vertex-transitive graph 
has edge connectivity equal to its valency. 

We begin this chapter with some preliminary results which will work towards proving the main theorems in this chapter. With reference to a fixed undirected graph $\Gamma = \Gamma(V,E)$, we say that $a$ and $b$ are \emph{twins}\index{twins} if $a\neq b$ yet $\Gamma(a) = \Gamma(b)$. A graph $\Gamma$ is \emph{complete multipartite}\index{complete multipartite} if any two non-adjacent vertices are twins; i.e., the complement of $\Gamma$ is a union of complete graphs. We explore which association schemes contain twins, particularly in the polynomial case. We then move to examine the structure of a fixed relation given a basepoint $a$. First, we define a homomorphism mapping $\Gamma_i$, the graph of some basis relation, to the unweighted distribution diagram $H_i$ of that relation. This allows us to compare the structure of $\Gamma_i$ and $H_i$, particularly by projecting and lifting paths from one to the other. Write\footnote{Note that some authors assign another meaning to $\bot$; here, we follow \cite[p.\ 440]{Brouwer1989}.} $a^\bot = \left\{a\right\}\cup \Gamma(a)$ and define the subgraph $\Gamma_a = \Gamma\backslash a^\bot$. We decompose the vertices $\Gamma_a$ into three (possibly empty) sets based on each vertices relation with $a$ and compare these vertex sets to sets of vertices in $H_i^\prime=H'\backslash\left\{0,1\right\}$. By employing these techniques using multiple basepoints, we show that one of the sets in the decomposition of $\Gamma_a$ must be empty. Using this result we prove Theorem \ref{Tmain} below, along with the corollaries that follow. We finish this chapter by applying Theorem $\ref{Tmain}$ along with a spectral lemma in the case of small valency.

This chapter is based on joint work with W.\ J.\ Martin, published in \cite{Kodalen2017}. The main goal of this chapter is to prove the following theorem related to connectivity in symmetric association schemes:

\begin{thm} \label{Tmain}
	Let $(X,\cR)$ be a symmetric association scheme. Assume the graph $\Gamma=\Gamma(X,R_i)$ is connected 
	and not complete multipartite. Let
	$H=H_i$ be the corresponding unweighted distribution diagram on $\{0,1,\ldots, d\}$. The following are equivalent:
	\begin{itemize}
		\item[(1)] there exists $a\in X$ for which the subgraph $\Gamma \setminus a^\bot$ is connected;
		\item[(2)] for all $a\in X$, the subgraph $\Gamma \setminus a^\bot$ is connected;
		\item[(3)]  the subgraph $H \setminus \{0,i\}$ is connected;
		\item[(4)] $\Gamma$ contains no twins.
	\end{itemize} 
\end{thm}

Recall that a \emph{commutative association scheme} (see Section \ref{association}) is a more general combinatorial object in which we replace property \emph{(iii)} with the condition: For each $0\leq i\leq d$, there exists some index $i^\prime$ such that $R_i^T = R_{i^\prime}$; that is, $(x,y)\in R_i$ if and only if $(y,x)\in R_{i^\prime}$. We also require that for all $0\leq i,j,k\leq d$, $p^k_{ij}=p^k_{ji}$ --- thus we preserve the property that our Bose-Mesner algebra is commutative. Given a commutative association scheme $(X,\mathcal{R})$ we define the \emph{symmetrization}\index{symmetrization!of relation} of the relation $R_i$ as the relation $R_i\cup R_{i'}$ noting that this is exactly $R_i$ if $i=i'$ (i.e.\ $R_i$ is already symmetric). In this way, we define the \emph{symmetrization}\index{symmetrization!of association scheme} of $(X,\mathcal{R})$ to be $(X,\mathcal{R}')$ where
\[\mathcal{R} = \left\{R\cup R^T\mid R\in\mathcal{R}\right\}.\]
The edge sets of all the graphs considered in the following corollaries are given by relations in the symmetrization, thus Theorem \ref{Tmain} extends immediately to give these corollaries. Due to this, we will only consider the symmetric case when proving the main theorem. The following are the remaining main results of this chapter.

\begin{cor}  \label{C2}
	Let $(X,\cR)$ be a commutative association scheme. Assume the undirected graph 
	$\Gamma=\Gamma(X,R_i  \cup R_{i'})$ is connected and $a\in X$. Then $\Gamma \setminus 
	\Gamma(a)$ contains at most one non-singleton component.
\end{cor}

\begin{cor} \label{C1}
	Let $(X,\cR)$ be a commutative association scheme.  Assume the undirected graph 
	$\Gamma=\Gamma(X,R_i  \cup R_{i'})$ is connected and $a\in X$.  Then, for any  $T \subseteq a^\bot$ with $\Gamma(a) \not\subseteq T$, the graph  $\Gamma \setminus T$ is connected.
\end{cor}

\begin{cor} \label{C3}
	Let $(X,\cR)$ be a commutative association scheme. Assume the undirected graph 
	$\Gamma=\Gamma(X,R_i  \cup R_{i'})$ is connected and $C\subseteq  X$ is the vertex set of a clique in $\Gamma$. Then $\Gamma \setminus C$ is connected.
\end{cor}

\begin{restatable*}{thm}{cutsizetwo}
	\label{Tcutsize2}
	Let $(X,\cR)$ be a symmetric association scheme and let $\Gamma=\Gamma(X,R_1)$  be the graph 
	associated to a connected basis relation. If $\Gamma$ admits a  disconnecting set of size two, 
	then $\Gamma$ is isomorphic to a polygon.
\end{restatable*}

\begin{restatable*}{thm}{cutsizethree}
	\label{Tcutsize3}
	Let $(X,\cR)$ be a symmetric association scheme and let $\Gamma=\Gamma(X,R_1)$  be the graph 
	associated to a connected basis relation. If $\Gamma$ has diameter two, then either
	$\Gamma$ has vertex connectivity at least four or
	$\Gamma$ is isomorphic to one of the following graphs: the  4-cycle, the 5-cycle, $K_{3,3}$, 
	the Petersen graph.
\end{restatable*}

\begin{restatable*}{lem}{spectrallem}
	\label{Lspec-cut}
	Let $(X,\cR)$ be a symmetric association scheme and let $\Gamma=\Gamma(X,R_1)$ be the graph associated to a con\-nec\-ted basis relation.
	Assume that $\Gamma$ contains no induced subgraph isomorphic to $K_{2,1,1}$. 
	If  $T\subseteq X$ is a disconnecting set for $\Gamma$, then $|T|> p_{11}^1$.
\end{restatable*}

We should remark that Corollaries \ref{C1} and \ref{C3} extend naturally to the case where $\Gamma$ is not connected in that the deletion of vertices does not increase the number of components. One verifies this by applying the
respective corollary to the subscheme induced by vertices in a particular component of $\Gamma$.


\section{Twins}
Let $(X,\mathcal{R})$ be a symmetric association scheme. Let  $\Gamma = (X,R)$ be the graph of a basis relation in $(X,\cR)$.  Write $R(a)=\Gamma(a)$. Examples
where twins arise (i.e., $R(a)=R(b)$ for $a\neq b$) include not only complete multipartite graphs  but   
antipodal distance-regular graphs such as the $n$-cube in which case $R$ is the distance-$\frac{n}{2}$ 
relation of the association scheme.

\begin{lem} \label{Ltwins}
	Let $(X,\cR)$ be a symmetric association scheme and let $\Gamma=\Gamma(X,R_i)$ for some $i \neq 0$. If $a$ and $b$ 
	are twins, then $(X,\cR)$ is imprimitive and some  system of imprimitivity exists in which $a$ and $b$ belong
	to the same fiber.
\end{lem}

\begin{proof}
	One easily checks that the following relation $\sim$ on $X$ is an equivalence relation: 
	$a\sim b$ if either $a=b$ or $a$ and $b$ are twins in $\Gamma$. To see that this is a system of imprimitivity, we verify that $\sim$
	is the union of basis relations $R_j$ for which $p_{ii}^j = p_{ii}^0$. Since $p_{ii}^i < p_{ii}^0$ and we are assuming
	at least one pair of twins exists, the equivalence 
	relation is non-trivial and  $(X,\cR)$ is imprimitive.
\end{proof}

\begin{remark} We now discuss twins in polynomial association schemes.  We use the well-known fact that
	an association scheme is imprimitive if and only if some idempotent $E_j$ ($1\le j\le d$) has repeated columns (see, e.g., \cite[Theorem~2.1]{Martin2007}).
	Denote by $u_j(a)$ the column of $E_j$ indexed by $a\in X$. If $R_i(a)=R_i(b)$, then for each $0 \le j\le d$
	$$ P_{ji} u_j(a) = A_i u_j(a) = \sum_{(x, a)\in R_i} u_j(x) =  \sum_{(x,b)\in R_i} u_j(x)   = A_i u_j(b) = P_{ji} u_j(b)$$
	so that either $P_{ji}=0$ or $u_j(a)=u_j(b)$.

	\begin{enumerate}
		\item Assume $(X,\cR)$ is the association scheme coming from a distance-regular graph $\Gamma=\Gamma(X,R_1)$ with  distance-$k$ relation $R_k$ for $0 \le k\le d$
		and assume $R_i(a)=R_i(b)$ for distinct vertices $a$ and $b$. Suppose $a$ and $b$ do not belong to a common antipodal fiber in an antipodal system of imprimitivity. Then $\Gamma$ must be bipartite, in which case columns $a$ and $b$ of $E_j$ can be identical only for $j\in \{0,d\}$ (where  $E_0,\ldots,E_d$ are ordered so
		that $P_{01}>P_{11}>\cdots >P_{d1}=-P_{01}$ \cite[Prop.~4.4.7]{Brouwer1989}). But then, except for $d=2$, there is 
		some $j \neq 0,d$ for which $P_{ji}\neq 0$; thus $a=b$ for $d>2$. So bipartite systems of imprimitivity only 
		arise for $d=2$. Viewing complete bipartite graphs as having the antipodal property, we then have that any distinct $a$ and $b$ with  $R_i(a)=R_i(b)$ must belong to the same antipodal fiber, $d$ is even, and $i=d/2$.
		\item Assume $(X,\cR)$ is a $Q$-polynomial (or ``cometric'') association scheme  \cite[Section~2.7]{Brouwer1989}, 
		not a polygon, and $a\neq b$ yet 
		$R_i(a)=R_i(b)$. Then, by Theorem \ref{suzukiimprim}, $(X,\cR)$ is either 
		$Q$-bipartite or $Q$-antipodal. Let $E_0,\ldots,E_d$ be a $Q$-polynomial ordering of the primitive idempotents
		and order relations such that $Q_{01}>Q_{11}>\cdots > Q_{d1}$. If $a$ and $b$ belong to the same fiber of a 
		$Q$-bipartite imprimitivity system, then $d$ must be even and $i=\frac{d}{2}$ by Corollary 4.2 in \cite{Martin2007}. Otherwise, $a$ and $b$ must belong to the same $Q$-antipodal fiber and $u_j(a)=u_j(b)$ only for 
		$j \in \{0,d \}$.  So $P_{ji}=0$ for $1\le j< d$, forcing $(X,R_i)$ to be an imprimitive strongly regular graph 
		(as it is regular with three  eigenvalues). Since the scheme is cometric with an imprimitive strongly regular 
		graph as a basis relation, we must have $d=2$ and $a$ and $b$ are non-adjacent vertices in a complete multipartite graph.
	\end{enumerate}
\end{remark}

\section{The graph homomorphism \texorpdfstring{$\varphi_a$}{phia}}

For $0< i \le d$, let $\Gamma_i=(X,R_i)$ and let $H_i$ denote the \emph{unweighted distribution diagram}\index{unweighted distribution diagram} corresponding to symmetric relation $R_i$ --- that is, $H_i = \Gamma(V,E)$ where $V = \left\{0,\dots,d\right\}$ and $(j,k)\in E$ if and only if $p^k_{ij}>0$. Note $H_i$ need not be simple, in fact, there will be a loop at vertex $j$ whenever $p^j_{ij}>0$. 

\begin{prop}  \label{Pphi}
	For any $a\in X$, the map $\varphi_{a,i} : \Gamma_i \rightarrow H_i$ sending $b\in X$ to $j$ where $(a,b)\in R_j$
	is a graph homomorphism. Under this map, every walk in $\Gamma_i$ projects to a walk in $H_i$ of the same length. As a partial converse, for any $b\in X$ with $(a,b)\in R_{j_0}$ and any walk 
	$$ w = (j_0,j_1,\ldots, j_\ell) $$
	in $H_i$, there is at least one walk $(b=b_0,b_1,\ldots, b_\ell)$  of length $\ell$ in $\Gamma_i$ such 
	that $\varphi_{a,i}(b_s)=j_s$ for each $0\le s \le \ell$.  $\Box$
\end{prop}

We will call $\varphi_{a,i}$ the \emph{projection map}\index{projection map} and will omit the second subscript when 
it is clear from the context.

For vertices $x$ and $y$ in  an undirected graph $\Gamma$, we use $d_\Gamma(x,y)$ 
to denote the path-length distance from $x$ to $y$ in  $\Gamma$, setting $d_\Gamma(x,y)=\infty$ when no path from
$x$ to $y$ exists in $\Gamma$.

\begin{lem} \label{LdiamGamma}
	Let $(X,\cR)$ be a symmetric association scheme and, for some $0<i\le d$, let $\Gamma=\Gamma(X,R_i)$ with corresponding unweighted distribution diagram $H$. If $\Gamma$ is connected, then for $(a,b)\in R_j$, $d_\Gamma(a,b)=d_H(0,j)$.
\end{lem}

\begin{proof}
	A shortest path in $H$ from $j$ to $0$ lifts via $\varphi_{a,i}^{-1}$ to a walk in $\Gamma$ from $b$ to
	a vertex in $R_0(a)$ --- i.e., lifts to a walk from $b$ to $a$ --- of length $d_H(j,0)$. 
	Conversely, each path from $b$ to $a$ in $\Gamma$ projects to a walk of the same length from 
	$j$ to $0$ in $H$.
\end{proof}

\section{The decomposition \texorpdfstring{$\{ I_a, U_a, W_a\}$}{{Ia, Ua, Wa}}}
\label{Sec:IUW}

For simplicity, we henceforth take $\Gamma=\Gamma(X,R_1)$ with unweighted distribution diagram $H=H_1$ in the symmetric association scheme  $(X,\cR)$. We assume throughout the remainder of the chapter
that $\Gamma$ itself is a connected graph.
We will compare the graphs $\Gamma_a := \Gamma \setminus a^\bot$ and $H' :=H \setminus \{0,1\}$ and show that, with known exceptions, one is connected if and only if the other is connected. One direction is straightforward.

\begin{prop}
	\label{PH'disconn}
	If $H'$ is not a connected graph, then for any $a\in X$, $\Gamma_a$ is also disconnected. If 
	$i$ and $j$ are in distinct components of $H'$, then   $\Gamma_a$  contains no path from $R_i(a)$ to $R_j(a)$.
\end{prop}

\begin{proof}
	Let $x\in R_i(a)$ and $y\in R_j(a)$ and suppose 
	$ x=x_0,x_1,\ldots,x_\ell=y $
	is a path in    $\Gamma_a$. Then
	$ i = \varphi_a(x_0),  \varphi_a(x_1), \ldots,  \varphi_a(x_\ell) = j$
	is a walk from $i$ to $j$ in $H$. Since $H'$ is disconnected, $\varphi_a(x_t)\le 1$ for some $t$ which forces
	$x_t \in a^\bot$, a contradiction.
\end{proof}

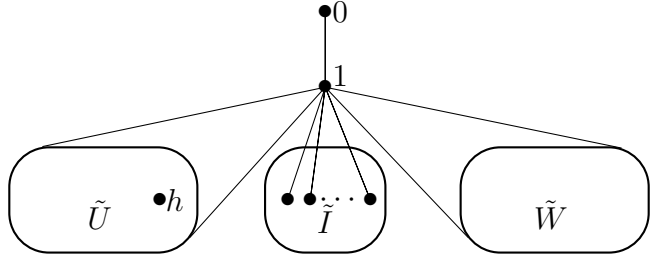
\begin{figure}[!ht]
	\begin{center}
		\begin{tikzpicture}
		\draw[thick,black]
		(-1.2,-0.7) {[rounded corners=15pt] --
			++(2.5,0)  -- 
			++(0,1.4) --
			++(-2.5,0) --
			cycle};
		\draw[thick,black]
		(2.2,-0.7) {[rounded corners=15pt] --
			++(1.6,0)  -- 
			++(0,1.4) --
			++(-1.6,0) --
			cycle};
		\draw[thick,black]
		(4.8,-0.7) {[rounded corners=15pt] --
			++(2.5,0)  -- 
			++(0,1.4) --
			++(-2.5,0) --
			cycle};
		\node (uu) at (0,-0.2) {$\tilde{U}$};
		\node (ii) at (3,-0.25) {$\tilde{I}$};
		\node (ww) at (6,-0.2) {$\tilde{W}$};
		\node (br0) at (3,2.5) {$\bullet$};
		\node (br1) at (3,1.5) {$\bullet$};
		\node (brh) at (0.8,0) {$\bullet$};
		\node (bri) at (2.2,0) {};
		\node (r0) at (3.2,2.5) {$0$};
		\node (r1) at (3.2,1.65) {$1$};
		\node (rh) at (1,0) {$h$};
		\node (bri1) at (2.5,0) {$\bullet$};
		\node (bri2) at (2.8,0) {$\bullet$};
		\node (brid) at (3.2,0) {$\cdots$};
		\node (bri3) at (3.6,0) {$\bullet$};
		\draw (-0.76,0.71) -- (3,1.5) -- (3,2.5) -- (3,1.5) -- (6.94,0.69);
		\draw (1.2,-0.5) -- (3,1.5);
		\draw (4.9,-0.5) -- (3,1.5);
		\draw (2.5,0) -- (3,1.5) --  (2.8,0) -- (3,1.5) --  (3.6,0) -- (3,1.5);
		\end{tikzpicture} 
	\end{center}
	\caption{Graph $H$. Upon deletion of $0$ and $1$, the isolated vertices in $\tI$ are those subconstituents which contain all twins of the basepoint while $\tU$ is the vertex set of a component outside $\tI$ which minimizes $\sum_{i\in \tU } k_i$.
	}\label{FigH}
\end{figure}

\begin{prop}
	\label{Pxya}
	If $x$ and $y$ lie in distinct components of  $\Gamma_a$, then $\Gamma(x)\cap \Gamma(y) \subseteq \Gamma(a)$. $\Box$
\end{prop}

For $\tU \subseteq \{0,1,\ldots, d\}$, note that $| \varphi_a^{-1}(\tU ) | = \sum_{i\in \tU } k_i$.
We now assume that $H'$ is disconnected and we define a decomposition of its vertex set. Let
$$ \tI = \{ i > 0 \mid  p_{11}^i = p_{11}^0 \} . $$
Now the set $\{2,\ldots,d\} \setminus \tI$ decomposes naturally into the vertex sets of the connected components of $H'$, excluding the isolated vertices in $\tI$. Let $\tU$ be the vertex set of some
component of $H' \setminus \tI$ such that $| \varphi_a^{-1}(\tU ) |$ is minimized. Let $\tW = \{2,\ldots, d\} \setminus \left( \tI \cup \tU \right)$ as depicted in Figure \ref{FigH}. For 
$x\in X$, set 
$$I_x = \varphi_x^{-1}(\tI ),  \qquad U_x = \varphi_x^{-1}(\tU ),   \qquad W_x = \varphi_x^{-1}(\tW ) $$
and note that $| I_x |$, $|U_x|$, and $|W_x |$ are independent of the choice of $x\in X$. Observe that $x$ and $y$ are twins if and only if $y \in I_x$.  While our basepoint will vary in what follows, our choice of $\tU$, $\tW$ and 
$\tI$ will remain fixed for this connected graph $\Gamma$.

\begin{lem}  \label{LUdist2}
	If $\tW \neq \emptyset$, then for every $u\in U_x$, $d_\Gamma(x,u)=2$.
\end{lem}

\begin{proof}
	By way of contradiction, assume $u \in U_x$ with $\Gamma(x) \cap \Gamma(u) = \emptyset$. For any
	$w\in W_x$, we note that $\Gamma$ contains an $xw$-path which does not pass through $u^\bot$.
	So $x$ and $w$ lie in the same connected component of $\Gamma_u$. But if $(x,u)\in R_h$ then
	$h\in \tU$ so $x\in U_u$ by symmetry. It follows that $W_x \cup \{x\} \subseteq U_u$. But this contradicts
	$| \varphi_u^{-1}(\tU ) | \le | \varphi_x^{-1}(\tW ) |$.
\end{proof}

\subsection{Comparing the view from multiple basepoints}
In this section we will select $a,b\in X$ with $b\in U_a$ and compare components of $\Gamma_a$ against those of $\Gamma_b$, defining vertex sets $V\Delta$, $Y_a$, and $Z_a$ relative to the pair $a,b$.

\begin{prop}  \label{PW_b}
	For any $a\in X$ and any $b\in U_a$, we have $W_a \cap I_b = \emptyset$.
\end{prop}

\begin{proof}
	If $x$ and $b$ are twins, then $x$ cannot be a twin of $a$ since $b$ is not a twin of $a$. So 
	$\Gamma(x) = \Gamma(b) \subseteq U_a \cup \Gamma(a)$ gives $\Gamma(x) \cap U_a \neq \emptyset$.
	So $x\not\in W_a$.
\end{proof}

Now fix $a\in X$ and choose $b\in U_a$. 
Consider the component $\Delta$ of $\Gamma_b$ containing $a$.
Denote by $V\Delta$ the vertex set of $\Delta$.
Since $b$ and $a$ are not twins, some element of $\Gamma(a)$ is a vertex of $\Delta$ and hence $\Delta$ contains vertices in $W_a$ unless $\tW = \emptyset$. Let $Z_a = V\Delta \cap W_a$ and let $Y_a= W_a \setminus Z_a$. This vertex decomposition is depicted in Figure \ref{FigGamma}. 
Since $b\in U_a$, we have $a\in U_b$ and, since $\Delta$ is connected, $Z_a \subseteq U_b$.

In  
the next two results, we proceed under the assumption that $H'$ is disconnected and that $\tilde{W}$ is non-empty. Further we assume that vertices  $a$ and $b \in U_a$ have been chosen and the 
sets $Y_a$ and $Z_a$ are defined as above relative to this pair of vertices.

\begin{lem} \label{LpathYa}
	Let $w=(u_0,u_1,\ldots,u_\ell)$ be a walk in $\Gamma$ with $u_0\in Y_a$ and $u_\ell$ lying some other component
	of $\Gamma_a$. Let $t \in \{1,\ldots,\ell\}$ be the smallest subscript with $u_t \not\in Y_a$. Then $u_t \in \Gamma(a)$. $\Box$
\end{lem}

\begin{lem} \label{LYZsplit}
	For $0\le i \le d$, $R_i(a) \cap Y_a\neq \emptyset$ implies $R_i(a) \cap Z_a \neq \emptyset$. So no subconstituent of $\Gamma$ with respect to $a$ is entirely contained in $Y_a$.
\end{lem}

\begin{proof}
	Let $y \in Y_a$ and consider a shortest $ya$-path $\pi$ in $\Gamma$, of length $\ell$ say, and label its vertices as follows: $\pi=(y=v_\ell,v_{\ell-1},\ldots,v_1,v_{0} = a)$.
	Then, by Lemma \ref{LpathYa}, $v_s \in Y_a$ for  $1 < s \le \ell$.  Consider $j_s=\varphi_a(v_s)$, $0\le s \le \ell$, and assume $j_\ell=i$. Then we have $p_{1,j_{s+1}}^{j_s} >0$ for $0\le s < \ell$.  Note $j_0=0$ and $j_1=1$. Now we lift the walk $(j_0,\ldots,j_\ell)$ in $H$ to a different walk in $\Gamma$. Since $a$ and $b$ are not twins,
	we may choose $v'_1 \in \Gamma(a) \setminus \Gamma(b)$.
	Since $p_{1j_2}^{1}>0$, there exists $v'_2\in R_{j_2}(a)$ with $v'_2$ adjacent to  $v'_1$ in $\Gamma$. Continuing in this manner, we
	may construct a walk $\pi'=(a=v'_0,v'_1,\ldots, v'_\ell)$ in $\Gamma$ with $\varphi_a(v'_s)=j_s$. 
	Since $\Gamma(b) \subseteq \Gamma(a) \cup U_a$, none of the vertices $v'_s$ lie in 
	$\Gamma(b)$, so the entire walk $\pi'$ is contained in one component of $\Gamma_b$. By definition of $Z_a$, we then have $v'_\ell \in Z_a \cap R_i(a)$.
\end{proof}

\begin{figure}[!ht]
	\begin{center}
		\begin{tikzpicture}
		\draw[thick,black]
		(0,0) {[rounded corners=15pt] --  ++(2.25,0)  --   ++(0,5) --  ++(-2.25,0) --  cycle};
		\draw[thick,black]
		(2.5,0) {[rounded corners=15pt] --  ++(1.5,0)  --   ++(0,5) --  ++(-1.5,0) --  cycle};
		\draw[thick,black]
		(4.25,0) {[rounded corners=15pt] --  ++(2.25,0)  --   ++(0,5) --  ++(-2.25,0) --  cycle};
		\draw[thick,black]
		(0,5.25) {[rounded corners=15pt] --  ++(6.5,0)  --   ++(0,1) --  ++(-6.5,0) --  cycle};  
		\draw[thick,black]  (4.25,4) -- (6.5,2);
		\draw[thick,black]  (0.9,0) -- (0.9,5);
		\draw[thick,black]  (0.9,5.25) -- (0.9,6.25);
		\node (bra) at (3.25,6.6) {$\bullet$};  \node (ra) at (3.5,6.6) {$a$};
		\node (brb) at (0.25,2.5) {$\bullet$}; \node (rb) at (0.5,2.5) {$b$};
		\node (li) at (3.2,5.75) {$\Gamma(a)$};
		\node (lu) at (1.35,0.75) {$U_a$};
		\node (li) at (3.2,0.75) {$I_a$};
		\node (lw) at (5.45,0.75) {$W_a$};
		\node (ly) at (5,2.75) {$Y_a$};
		\node (lz) at (5.5,3.5) {$Z_a$};
		\node (lbb1) at (0.55,3.5) {$b^\bot$};
		\node (lbb2) at (0.55,5.75) {$b^\bot$};
		\fill[gray,opacity=.5]   
		(0.9,5) {[rounded corners=15pt] --   (2.25,5)  -- (2.25,3.5) -- (0.9,3.5)} -- cycle; 
		\fill[gray,opacity=.5]  
		(4.25,4) {[rounded corners=15pt] --  (4.25,5)  -- (6.5,5)}  -- (6.5,2) --  cycle;
		\fill[gray,opacity=.5]   
		(0.9,5.25) {[rounded corners=15pt] --   (6.5,5.25)  --  (6.5,6.25)} -- (0.9,6.25) -- cycle;
		\fill[gray,opacity=.5]   
		(2.5,0) {[rounded corners=15pt] --   ++(1.5,0)  --   ++(0,5) --  ++(-1.5,0) --  cycle};
		\end{tikzpicture} 
	\end{center}
	\caption{This diagram depicts $\Gamma$ as decomposed relative to basepoint $a$. In $\Gamma_b$,
		vertex $a$ belongs to component $\Delta$, whose vertex set is indicated by the shaded region. 
		The set $W_a$ splits into $Z_a$ and $Y_a$ according to membership in $V\Delta$.
		\label{FigGamma}}
\end{figure}
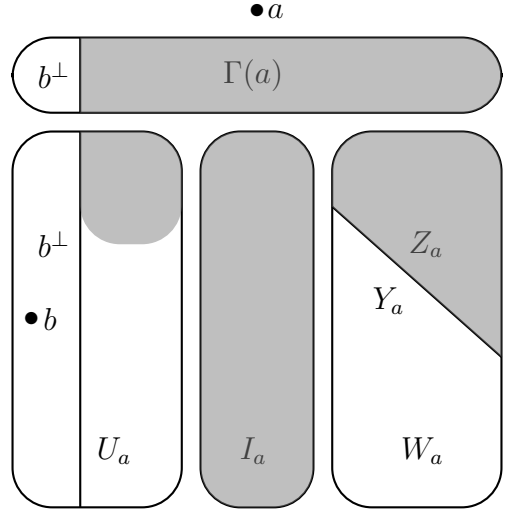

\begin{lem} \label{LWa-empty}
	If $\tW \neq \emptyset$, then $\Gamma$ has diameter two; i.e., $p_{11}^i > 0$ for all $i>1$.
\end{lem}

\begin{proof}
	Let $a,x \in X$ with $x \not\in a^\bot$. Let $V\Gamma_a$ be the vertex set of the graph $\Gamma_a$. Choose $b\in U_a$ as above and consider, in turn, each part of
	the decomposition 
	$$ V\Gamma_a \ = \ I_a \ \dot{\cup} \ U_a  \ \dot{\cup} \  Z_a  \ \dot{\cup}  \  Y_a $$
	relative to $a$ and $b$.
	If $x\in I_a$, $\Gamma(x) = \Gamma(a)$; if $x \in U_a$, then $d_\Gamma(a,x)=2$ by Lemma \ref{LUdist2}.
	Next consider $x\in Z_a$. Then $d(x,b)=2$ but $\Gamma(x)\cap \Gamma(b) \subseteq \Gamma(a)$ since 
	$x$ and $b$ lie in distinct components of $\Gamma_a$ (Proposition \ref{Pxya}). Finally, consider $x\in Y_a$ with $(a,x)\in R_i$.
	By Lemma \ref{LYZsplit}, there exists $x'\in Z_a \cap R_i(a)$. Since $x'$ has a neighbor in $\Gamma(a)$,
	$p_{11}^i > 0$ which then implies that some neighbor of $x$ lies in $\Gamma(a)$ as well.
\end{proof}

\begin{thm}  \label{TWempty}
	Let $(X,\cR)$ be any symmetric association scheme and let $\Gamma=\Gamma(X,R_1)$ be
	any connected basis relation. With reference to the above definitions, $\tW = \emptyset$.
\end{thm}

\begin{proof}
	By way of contradiction, assume $\tW \neq \emptyset$ and define
	$$ \mu = \min \{ p_{11}^i \mid i \in \tU \},  \qquad \omega = \min \{ p_{11}^i \mid i \in \tW \}$$
	and select $k\in \tU$ and $\ell \in \tW$ with $p_{11}^k = \mu$ and $p_{11}^\ell = \omega$. Note that
	$\mu,\omega > 0$ by Lemma \ref{LWa-empty}.
	Now choose $a\in X$, and select $x$ in $R_k(a)$.  Since $x$ is not a twin of $a$, we may choose
	$a'\in \Gamma(a) \setminus \Gamma(x)$ and since $p_{1\ell}^1>0$, we may choose $y$ in $R_\ell(a)$ which is a neighbor of 
	$a'$.  Since $\Gamma_x$ contains a path from $a$ to $y$ and $a\in U_x$, we have $y\in U_x$. So
	$| \Gamma(x) \cap \Gamma(y) | \ge \mu$.  By Proposition \ref{Pxya}, $\Gamma(x)\cap \Gamma(y) \subseteq \Gamma(a)$. (See Figure \ref{FigWempty}.) But $a' \in \Gamma(y)\cap \Gamma(a)$. So 
	$$\omega \ge 1 + | \Gamma(x) \cap \Gamma(y) | > \mu.$$
	
	Now we simply reverse the roles of $x$ and $y$; more precisely, we swap $\ell$ and $k$.
	
	Select $x$ in $R_\ell(a)$ and,  choosing
	$a'\in \Gamma(a) \setminus \Gamma(x)$, we may find a vertex $y$ in $R_k(a)$ which is a neighbor of 
	$a'$.  Since $\Gamma_x$ contains a path from $a$ to $y$ and $a\in W_x$, we have $y\in W_x$. So
	$| \Gamma(x) \cap \Gamma(y) | \ge \omega$.  By Proposition \ref{Pxya}, $\Gamma(x)\cap \Gamma(y) \subseteq \Gamma(a)$. But $a' \in \Gamma(y)\cap \Gamma(a)$. So 
	$$\mu \ge 1 + | \Gamma(x) \cap \Gamma(y) | > \omega.$$
	We have $\omega > \mu$ and $\mu > \omega$, producing the desired contradiction.
\end{proof}


\begin{figure}[!ht]
	\begin{center}\vspace{-1.5em}
		\begin{tikzpicture}
		\def\y{.2}
		\def\x{.5}
		\def\z{.2} 
		\draw[thick,black] (1.8,\y) ellipse (0.75in and 0.3in);
		\draw[thick,black] (4.2,\y) ellipse (0.75in and 0.3in);
		\draw[thick,black]
		(-.1,-2.9+\x+\y+\z) {[rounded corners=15pt] --
			++(3.,0)  -- 
			++(0,1.7-\x) --
			++(-3.,0) --
			cycle};
		\draw[thick,black]
		(3.1,-2.9+\x+\y+\z) {[rounded corners=15pt] --
			++(3,0)  -- 
			++(0,1.7-\x) --
			++(-3,0) --
			cycle};
		\node (uu) at (2.2,-2.4+\x+\y+\z) {$U_a$};
		\node (ww) at (3.8,-2.4+\x+\y+\z) {$W_a$};
		\node (ba) at (3,1.7) {$\bullet$};   \node (la) at (3.3,1.8) {a};
		\node (bx) at (.8,-1.7+\y+\z) {$\bullet$};   \node (lx) at (1.1,-1.8+\y+\z) {$x$};
		\node (by) at (5.2,-1.7+\y+\z) {$\bullet$};   \node (ly) at (5.5,-1.7+\y+\z) {$y$};
		\node (bap) at (4.8,.35+\y) {$\bullet$};   \node (lap) at (5.1,.35+\y) {$a'$};
		\node (lgamax) at (1.2,-.1+\y) {$\Gamma(a)\cap \Gamma(x)$};
		\node (lgamay) at (4.9,-.1+\y) {$\Gamma(a)\cap \Gamma(y)$};
		\draw (-0.08,-0.2+.25) -- (.8,-1.7+\y+\z) -- (3.2,-0.7+.4);
		\draw (6.08,-0.18+.25) -- (5.2,-1.7+\y+\z) -- (2.8,-0.7+.4);
		\draw (0.6,0.78) -- (3,1.7) -- (5.4,0.8);
		\end{tikzpicture} 
	\end{center}
	\vspace{-1em}
	\caption{Since $\Gamma$ has diameter two, all common neighbors of $x$ and $y$ are 
		in  $\Gamma(a)$. }\label{FigWempty}
\end{figure}
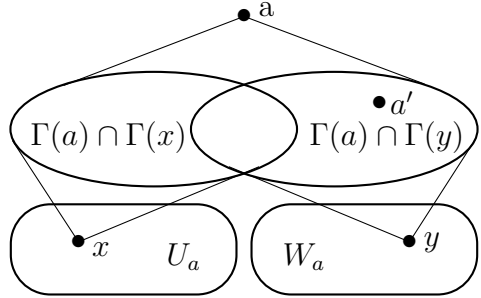


\section{Proofs of the main theorem and its corollaries}
\label{Sec:proofs}

We are now ready to present the proof of our main theorem. We continue with the notation established previously in this chapter. Recall that, for $a\in X$,  $\Gamma_a := \Gamma \setminus a^\bot$ is the subgraph of $\Gamma = \Gamma(X,R_i)$ 
obtained by deleting basepoint $a$ and all its neighbors.

\bigskip

\noindent {\sl Proof of Theorem \ref{Tmain}.} \ As before, we assume $i=1$ for notational convenience.

We begin by showing $(3) \Leftrightarrow (4)$. If $a$ and $b$ are twins in $\Gamma$ with $(a,b)\in R_j$, then
$j > 1$ and $p_{11}^j = k_1$ so that $j\in \tI$ and $\{ j\}$ is the entire vertex set of some component of 
$H' = H   \setminus \{0,1\}$. So either $H'$ is not connected or $d=j=2$ and $\Gamma$, being imprimitive,
is a complete multipartite graph.  Conversely, by Theorem \ref{TWempty}, $\tW=\emptyset$ so if $\tI=\emptyset$
we have that  $H'$ is connected.

The assertion $(2) \Rightarrow (1)$ is trivial. Proposition \ref{PH'disconn} gives us $(1) \Rightarrow (3)$. So
we need only check that $(3)$ implies $(2)$.

Assume now that $H'$ is connected and yet there is some $a\in X$ with $\Gamma_a$ not connected. 
By Proposition \ref{Pphi}, any $x$ in $\Gamma_a$ is joined by a walk in $\Gamma_a$ to some vertex in 
$R_j(a)$ for every $j>1$. (Simply lift a walk in $H'$ from $\ell$ to $j$ where $(a,x)\in R_\ell$.)
So for every $j>1$ every connected component of $\Gamma_a$ intersects 
subconstituent $R_j(a)$  non-trivially.  
Select $j>1$ so as to maximize $D:=d_H(0,j)$ and choose $x,y \in R_j(a)$ such that $x$ and $y$ lie in
distinct components of $\Gamma_a$.  Then every $xy$-path in $\Gamma$ must include a vertex in $\Gamma(a)$,
so $d_\Gamma(x,y) \ge 2(D-1)$. Since $d_\Gamma(x,y) \le D$ by Lemma~\ref{LdiamGamma}, this forces $D\le 2$. In particular, $p_{11}^\ell >0$ for every $\ell > 1$.

Select $\ell >1$
so as to minimize $p_{11}^\ell$ and select $x,y \in R_\ell(a)$ from distinct components of $\Gamma_a$.
Then $(x,y) \in R_j$ for some $j>1$ and so $|\Gamma(x)\cap \Gamma(y)| \ge p_{11}^\ell$. But 
since these two vertices lie in distinct components, Proposition \ref{Pxya} gives us
$$ \Gamma(x) \cap  \Gamma(y) \subseteq \Gamma(a) \cap \Gamma(y)$$
so $p_{11}^j = p_{11}^\ell$ and $ \Gamma(x) \cap  \Gamma(y) =\Gamma(a) \cap \Gamma(y)$.
If $a' \in \Gamma(a)$, then $a'$ has $p_{1\ell}^1 > 0$ neighbors in $R_\ell(a)$. For any such neighbor $z$, we
must have either $\Gamma(z) \cap \Gamma(a) = \Gamma(x) \cap \Gamma(a)$ or $\Gamma(z) \cap \Gamma(a) = \Gamma(y) \cap \Gamma(a)$, both of which force $a' \in
\Gamma(x)$. So vertices $a$ and $x$ must be twins. The only possibility that remains is that $\Gamma$ is a 
complete multipartite graph.
$\Box$

The proofs of Corollaries \ref{C1}, \ref{C2} and \ref{C3} are now rather immediate. Since each is a statement about the  symmetrization of some commutative scheme, Theorem \ref{Tmain} applies.

\bigskip

\noindent {\sl Proof of Corollary \ref{C2}.}  This is essentially Theorem \ref{TWempty}. $\Box$

\bigskip

\noindent {\sl Proof of Corollary \ref{C1}.} We apply Theorem \ref{Tmain} to prove this. 
First, if we have no twins then
$\Gamma_a$ is connected. Any 
$a' \in \Gamma(a)$ has at least one neighbor in $V\Gamma_a$. If $a\not\in T$, then some $a' \in \Gamma(a)$
is also not included in $T$. So the graph $\Gamma\setminus T$ is 
connected as long as $T \neq \Gamma(a)$. 

If $b$ is a twin of $a$ in $\Gamma$, then $b$ is adjacent to every $x\in \Gamma(a)$. 
Since $\Gamma(a) \not\subseteq  T$,    some $a' \in \Gamma(a)$ is a vertex of $\Gamma \setminus T$.
By Corollary \ref{C2}, $\Gamma \setminus a^\bot$ has at most one non-singleton component. Let 
$\Xi$ be the component of $\Gamma \setminus T$ containing this component as a connected subgraph.
(If $\Gamma \setminus a^\bot$ consists only of singletons, choose $\Xi$ to be any component
of $\Gamma \setminus T$ containing some twin of $a$.) Since $a'$ has at least one neighbor in $V\Gamma_a$, 
the component $\Xi$ contains $a'$ and every twin $b$ of $a$ since each of these is a neighbor of $a'$.
Likewise, if $a\not\in T$, then $a$ belongs to $\Xi$ since it is adjacent to $a'$. So in this case as well,
$\Gamma \setminus T$ is connected.
$\Box$

\bigskip

\noindent {\sl Proof of Corollary \ref{C3}.} Let $a\in C$ and take $T=C$. Then apply Corollary \ref{C1}. $\Box$

We finish this section with a simple generalization arising from the proof above. 

\begin{thm}
	Assume $(X,\cR)$, $\Gamma$ and $H$ are defined as in Theorem \ref{Tmain}.
	Let $B_{H,t}(0) = \{ i \mid 0\le i\le d, \ d_H(0,i)\le t\}$ and $B_{\Gamma,t}(a) = \cup_{B_{H,t}(0)} R_i(a)$.
	\begin{itemize}
		\item[(a)] 
		If $\Gamma' := \Gamma \setminus B_{\Gamma,t}(a)$ is disconnected and $b\in X$ with $d_\Gamma(a,b)=D$ (the diameter of $\Gamma$), then for any $x\not\in B_{\Gamma,t}(a)$ not in the same component of $\Gamma'$ as 
		$b$, we have $d_\Gamma(a,x) \le 2t$. 
		\item[(b)]
		If $H \setminus B_{H,t}(0)$ is connected and yet $\Gamma'$ is disconnected,
		then $D \le 2t$.  
	\end{itemize}
\end{thm} 

\begin{proof}
	(a) Since $x$ and $b$ are in distinct components of $\Gamma'$, there must exist some $y\in X$ such that
	$  d_\Gamma(a,y) \le t$ and $ d_\Gamma(x,b) = d_\Gamma(x,y) + d_\Gamma(y,b)$. This gives
	$ d_\Gamma(y,b) \ge D-t$ which then implies  $ d_\Gamma(x,a) \le  d_\Gamma(x,y) + d_\Gamma(y,a) \le 2t$.
	
	(b) Since $H \setminus B_{H,t}(0)$ is connected, for every $j\not\in B_{H,t}(0)$, $R_j(a)$ has non-trivial intersection
	with every component of $\Gamma'$. So we may select $x,b$ in distinct components of $\Gamma'$ both 
	satisfying  $ d_\Gamma(a,x) = d_\Gamma(a,b)  = D$ and then apply part (a).
\end{proof}

\section{Further results on connectivity}
\label{Sec:conn}

In this section, we develop some machinery for the study of small disconnecting sets which are not
localized. We then apply these tools to show that, with the exception of polygons, a basis relation in a  symmetric association scheme has vertex connectivity at least three. We can say a bit more in the case where $\Gamma$
has diameter two. For the remainder of this chapter, we assume without loss of generality that $\Gamma=\Gamma(X,R_1)$
in order to simplify notation.

Elementary graph theoretic techniques allow us to handle the case where $\Gamma$ is in some sense locally connected. For example, if $\Gamma(y)$ induces a connected subgraph for every $y\in T$ and $d_\Gamma(y,y')\ge 3$ for any pair of distinct elements $y,y'\in T$, then $\Gamma \setminus T$ is connected. 
The proof of this claim is essentially the same as the proof of the following proposition, which applies more generally
to any connected simple graph $\Gamma$.

\begin{prop}
	\label{Psptree}
	Let $\Gamma=\Gamma(X,R_1)$ be the graph associated to a con\-nec\-ted basis relation in  a 
	symmetric association scheme  $(X,\cR)$.
	Suppose any two vertices at distance two in  $\Gamma$ lie in some common cycle of length at most  $g$ and $T \subseteq V\Gamma$ satisfies $d_\Gamma(y,y') 
	\ge g+1$ for all pairs $y,y'$ of distinct vertices from $T$. Then $\Gamma \setminus T$ is connected.
\end{prop}

\begin{proof}
	Set $\delta = \lfloor g/2 \rfloor$ and, for $y\in T$ set 
	$B_\delta(y) =  \{ x\in X \mid d_\Gamma(x,y) \le \delta \}$. The 
	induced subgraph $\Gamma[B]$ of $\Gamma$ determined by $B=B_\delta(y)$ is connected so 
	admits a spanning tree. Moreover, since $y$ is not a cut vertex of $\Gamma[B]$,
	there exists a spanning tree $\cT_y$ for $\Gamma[B]$ in which $y$ is a leaf vertex. For $y\in T$, 
	let $E_y$ denote the edge set of $\cT_y$ with the sole edge incident to $y$ removed.
	
	Now consider the minor $\Delta$
	of $\Gamma$ obtained by contracting $B_\delta(y)$ to a single vertex for every $y\in T$. Since $\Delta$ is
	again a connected graph, it admits a spanning tree $\cT$. Lift the edge set $E_\cT$ of $\cT$ back to $E\Gamma$
	and note that $E_\cT$ contains no edge from any of the induced subgraphs $\Gamma[B_\delta(y)]$, $y\in T$.
	So $E_\cT \cup \left( \cup_{y\in T} E_y \right)$ is the edge set of a spanning tree in $\Gamma \setminus T$,
	which demonstrates that   $\Gamma \setminus T$ is connected.
\end{proof}

\subsection{A spectral lemma}
\label{Sec:spec}

Eigenvalue techniques such as applications of eigenvalue interlacing play an important role in \cite{Brouwer1985} and \cite{Brouwer2009}. The following lemma is inspired by those ideas. This can be used, in conjunction with
Lemma \ref{PK211}, to show that a graph with a small disconnecting set $T$ whose elements are not too close
together must be locally a disjoint union of cliques of size at most $|T|$. 

\spectrallem

\begin{proof}
	The result obviously holds when $\Gamma$ is complete multipartite, so assume $\Gamma$ is not a com\-plete 
	multipartite graph.  By \cite[Cor.~3.5.4(ii)]{Brouwer1989}, we then know that the second largest eigenvalue $\theta$ of $\Gamma$ is positive. Order the eigenspaces of the scheme so that $A_1 E_1 = \theta E_1$ and abbreviate 
	$E=E_1$. For $K,L\subseteq X$, denote by $E_{K,L}$ the submatrix of $E$ obtained by restricting
	to rows in $K$ and columns in $L$. Let $C$ be any
	clique in $\Gamma$. Then, because $k_1 > \theta > 0$, the matrix $E_{C,C} =  \frac{m_1}{|X|}I
	+ \frac{\theta m_1}{k_1|X|}(J-I)$ is invertible. 
	
	Assume now that some disconnecting set $T \subseteq X$ has $|T| \le p_{11}^1$.
	Let $\Xi$ and $\Xi'$ be two connected components of $\Gamma \setminus T$ with 
	vertex sets $B$ and $ B'$, respectively,  and let
	$\rho$ and $\rho'$ denote the spectral radii of these two graphs.  Assume, without loss, that $\rho\le \rho'$. By eigenvalue interlacing, $\rho \le \theta$.  (see, e.g.,\cite[Theorem~3.3.1]{Brouwer1989}.) We now show $\rho=\theta$.

	Since $\Gamma$ does not contain $K_{2,1,1}$ as an induced subgraph, it is locally a disjoint union of cliques
	and every edge of $\Gamma$ lies in a clique $C$ of size $p_{11}^1 + 2$.  If $\Xi$ is edgeless, then $T$
	contains all neighbors of some vertex, which is impossible since $|T| \le p_{11}^1 < k_1$. So $\Xi$
	contains at least one edge and  $B\cup T$ contains some clique $C$ 
	of size at least $p_{11}^1 + 2$. It follows that the submatrix  $E_{X,B\cup T}$
	has rank at least $p_{11}^1 + 2$.  But $|T|\le p_{11}^1$. So the row space of $E_{X,B\cup T}$ 
	contains at least two linearly  independent vectors which are zero in every entry indexed by an 
	element of $T$. Restricting these two vectors to coordinates in $B$ only, we obtain two linearly
	independent eigenvectors for graph $\Xi$ belonging to eigenvalue $\theta$. It follows that
	$\rho = \theta$ and $\rho$, the spectral radius of $\Xi$, is not a simple eigenvalue. This contradicts
	the Perron-Frobenius Theorem  (see, e.g., \cite[Theorem~3.1.1]{Brouwer1989}) since $\Xi$ was chosen to
	be a connected graph. 
\end{proof}

\begin{remark}
	The hypotheses of the above lemma may clearly be weakened. The proof simply requires that both $B\cup T$ and 
	$B' \cup T$ contain cliques of size $|T|+2$ or larger and that the entries $E_{xy}$ of idempotent $E$ 
	are the same for all adjacent $x$ and $y$ in $V\Gamma$.
\end{remark}

\subsection{Intervals and metric properties of \texorpdfstring{$\Gamma$}{Gamma}}
\label{Sec:intervals}

Let $(X,\cR)$ be a symmetric association scheme and $\Gamma=\Gamma(X,R_1)$ with unweighted distribution diagram $H$.
For $a,b\in X$, if  $(a,b)\in R_i$, Lemma \ref{LdiamGamma} tells us that 
the path-length distance  $d_\Gamma(a,b)$ between $a$ and $b$ in 
graph $\Gamma$  is equal to the path-length distance $d_H(0,i)$ between $0$ and $i$ in $H$. It follows
that the diameter, $D$ say, of $\Gamma$ is equal to $\max_i d_H(0,i)$, which happens to be the diameter 
of $H$. We thus partition the index set  $\{0,1,\ldots,d\}$ according to distance from $0$ in $H$.
For each $0\le h\le D$, define $I_h = \{ i : d_H(0,i) = h\}$. 
For $0 \le i \le d$ with $i \in I_h$, define   
$$c(i) = \sum_{ j \in I_{h-1} } p_{1j}^i ~.$$

\begin{prop}
	\label{Lc(i)}
	With $c(i)$ defined as above
	\begin{itemize}
		\item[(a)]
		For any geodesic $0=\ell_0, 1=\ell_1,\ell_2,\ldots, \ell_h$ in $H$, 
		$$1= c(\ell_1) \le c(\ell_2) \le \cdots \le c(\ell_h). $$
		\item[(b)]
		If $c(i)=1$, then for any $\ell \in \{1,\ldots, d\}$ which lies along a geodesic from $0$ to $i$ in $H$, $c(\ell)=1$ as well.
		\item[(c)]
		If $c(i)=1$, then there is a unique shortest path in $H$ from $0$ to $i$ and, for $(a,b)\in R_i$, there is a unique
		shortest path in $\Gamma$ from $a$ to $b$.
	\end{itemize}
\end{prop}

\begin{proof}
	For part (a), observe that for $(a,b)\in R_{\ell_h}$ there exists $a' \in R_{\ell_{h-1}}(b)$ adjacent  to $a$ since
	$p_{1,\ell_{h-1}}^{\ell_h} > 0$ so that 
	$$ \{ x\in R_1(b) \mid \ d_\Gamma(x,a') = d_\Gamma(b,a')-1 \} \! \subseteq  \! 
	\{ x\in R_1(b) \mid\ d_\Gamma(x,a) = d_\Gamma(b,a)-1 \}. $$ 
	Parts (b) and (c) follow immediately.
\end{proof}

\medskip

For $a,b\in X$, we define the \emph{interval} $[a,b]$ to be the union of the vertex sets of all geodesics in 
$\Gamma$ from $a$ to $b$:
$$ [a,b] = \left\{ x\in X \mid d_\Gamma(a,x) + d_\Gamma(x,b) = d_\Gamma(a,b) \right\}. $$

For the purpose of the present discussion, we introduce a piece of terminology. For $x\in X$ and 
$y\in T\subseteq X$,  we say that $x$ is \emph{proximal}\index{proximal} to $y$ (relative to $T$) if 
$d_\Gamma(x,y) \le d_\Gamma(x,y')$ for all $y'\in T$. Vertex $x$ is then \emph{proximal only} to $y\in T$ if 
$d_\Gamma(x,y) < d_\Gamma(x,y')$ for all $y'\in T$ distinct from $y$.

\begin{prop}
	\label{Pci=1}
	Let $T$ be a disconnecting set for $\Gamma$  and let $x$ and $z$ be vertices lying in different components of  $\Gamma \setminus T$  with $(x,z)\in R_i$. Suppose there is some $y\in T$ such that $x$  is proximal only to $y$ and $z$ is proximal to $y$ with $(x,y)\in R_s$ and  $(z,y)\in R_t$.  If $c(s)=1$ or $c(t)=1$, then $c(i)=c(s)=c(t)=1$.
\end{prop}

\begin{proof}
	Every shortest path joining $x$ to $z$ in $\Gamma$ must pass through $y$. Apply Proposition \ref{Lc(i)}(b).
\end{proof}

\subsection{Small disconnecting sets}
\label{Sec:smallcut}

We continue under the assumption that $\Gamma=\Gamma(X,R_1)$ is the graph of some connected basis relation in 
the symmetric association scheme $(X,\cR)$. We begin by examining a simple condition which 
guarantees that $\Gamma$ is locally a disjoint union of  cliques.

\begin{lem}
	\label{PK211}
	Let $T$ be a minimal disconnecting set for $\Gamma$, $y\in T$. Suppose $d_\Gamma(y,y')\ge 4$ for all $y'\in T$ with $y'\neq y$. Then $c(j)=1$ for all indices $j$ in $I_2$.
\end{lem}

\begin{proof}
	Let $j\in I_2$ and let $x \in R_j(y)$. Let $z \sim y$ be some vertex lying in a different component of 
	$\Gamma \setminus T$ from that containing $x$. For $(z,x)\in R_i$, we find $c(i)=1$ by Proposition
	\ref{Pci=1}.  So $c(j)=1$ by Lemma \ref{Lc(i)}.
\end{proof}

\begin{lem}
	\label{Pfarawayx}
	Let $T$ be a disconnecting set for $\Gamma$, $y\in T$.
	\begin{itemize}
		\item[(a)] Let $x$ and $z$ be vertices lying in different components of $\Gamma \setminus T$. If
		$d_\Gamma(x,y')+d_\Gamma(y',z) > D$ for every $y'\in T$ except $y$, then $z$ has a 
		unique neighbor lying closer to $x$ and $z$ has a unique neighbor lying closer to $y$.
		\item[(b)] Suppose $x\in X \setminus T$ satisfies $d_\Gamma(x,y')=D$ for every $y'\in T$ except $y$. If $z\in X$ lies in a component of $\Gamma \setminus T$ distinct from that containing $x$, then $z$ has a unique neighbor 
		lying closer to $x$ and $z$ has a unique neighbor lying closer to $y$. 
	\end{itemize}
	In both cases, for $(x,z)\in R_i$, and $(y,z)\in R_j$,  we have $c(i)=c(j)=1$.
\end{lem}

\begin{proof}
	Clearly (b) follows from (a). So first verify (a) for the case $z\sim y$. Next, 
	observe that any geodesic joining $x$ to $z$ passes through $y$. So $ [x,z] = [x,y] \cup [y,z]$.
	Let $z' \in \Gamma(y) \cap [y,z]$. Since $[x,y] 
	\subseteq [x,z]$ and $[x,z'] = [x,y] \cup \{z'\}$, 
	we find $\Gamma(x) \cap [x,z] = \Gamma(x) \cap [x,z']$, a set of size one. 
	By the same token, $[y,z] \subseteq [x,z]$ and so $\Gamma(z) \cap [y,z] \subseteq 
	\Gamma(z) \cap [x,z]$ gives $|\Gamma(z) \cap [y,z] |=1$. 
\end{proof}

\begin{lem}
	\label{Lxp111}
	Let $T$ be a minimal disconnecting set for $\Gamma$, $y\in T$,  and suppose $x\in X$ satisfies $d_\Gamma(x,y')=D$ for every $y'\in T$ except $y$. Then
	\begin{itemize}
		\item[(a)]
		for $(x,y)\in R_i$ where $i\in I_h$, we have
		$ \sum_{\ell \in I_h} p_{1\ell}^i = p_{11}^1$.
		\item[(b)]
		for $z \in X\setminus T$ which is separated from $x$ by deletion of $T$, if 
		$\Gamma(z) \cap T \subseteq \{y\}$, then
		$ \sum_{\ell \in I_k} p_{1\ell}^j = p_{11}^1$ where $(y,z)\in R_j$ with $j\in I_k$.
	\end{itemize}
\end{lem}

\begin{proof}
	Let $z$ be a neighbor of $y$ which is separated from $x$ by deletion of $T$.
	Since $d_\Gamma(x,z) \le D$, we see that $x$ is proximal only to $y$ and $[x,z]=[x,y] \cup \{z\}$.  
	The set  $\Gamma(y)\cap \Gamma(z)$ has size $p_{11}^1$ and every $z' \in \Gamma(y)\cap \Gamma(z)$ 
	lies at distance $h+1$  from $x$ in $\Gamma$. Since every other neighbor of $z$, with the exception 
	of $y$, is further away from  $x$, we have $\sum_{ \ell \in I_{h+1} } p_{1 \ell}^j = p_{11}^1$ where 
	$(x,z)\in R_j$. Reversing roles, we
	see that $x$ then has exactly $p_{11}^1$ neighbors which lie at distance $h+1$ from $z$. But, for
	$x' \sim x$, $d_\Gamma(x',y)=d_\Gamma(x',z)-1$. This gives (a). To obtain (b), observe that 
	every neighbor $x'$ of $x$ with $d_\Gamma(x',z)=d_\Gamma(x,z)$ must have  $d_\Gamma(x',y)=
	d_\Gamma(x,y)$. By part (a), there are exactly $p_{11}^1$ such vertices. So, for $(x,z)\in R_s$, 
	$ \sum_{\ell \in I_{h+k}} p_{1\ell}^s = p_{11}^1$. Reversing roles, we see that exactly $p_{11}^1$
	neighbors of $z$ lie at distance $h+k$ from $x$. But this is precisely the set of vertices adjacent 
	to $z$ which lie at distance $k$ from $y$.
\end{proof}

\cutsizetwo
\begin{proof}
	Let $T=\{y,y'\}$ be a disconnecting set of size two. Let  $D=\diam \Gamma$ and let $B$  be the vertex set of some connected component  of $\Gamma \setminus T$. First consider the 
	case where $y'$ is the unique vertex at distance $D$ from $y$ in $\Gamma$. Then every vertex
	is at distance $D$ from exactly one other vertex. On the other hand,  if $x\in B \cap \Gamma(y)$, then 
	any neighbor of $y'$ not lying in $B$ must be at distance $D$ from $x$ by the triangle inequality. It follows that 
	$y$ has exactly one neighbor not in $B$ and, symmetrically, exactly one neighbor in $B$. So the
	graph has valency two in this special case. For the remainder of the proof, assume $d_\Gamma(y,y') < D$.
	
	By Corollary \ref{C1}, we have $d_\Gamma(y,y')\ge 3$.   Let $x$ (resp., $x'$) 
	denote some vertex  at distance $D$ from $y'$ (resp., $y$). (Note $x\neq y$, $x' \neq y'$.) Let $B$ and $B'$ be the vertex sets of  two connected components  $\Xi$ and $\Xi'$, respectively, of $\Gamma \setminus T$ and assume
	$x\in B$. By Lemma  \ref{Pfarawayx}(b), any $z\in B'$ has a unique neighbor lying closer to $y$.  (Choosing $j\in I_2$ and $z\in R_j(y)$, we see that this implies $\Gamma$ is $K_{2,1,1}$-free.) By 
	Lemma \ref{Lxp111}(a), any $z \in B' \setminus \Gamma(y')$ has exactly $p_{11}^1$ neighbors $z'$ satisfying
	$d_\Gamma(z',y)=d_\Gamma(z,y)$.  Since $d_\Gamma(x,y)+d_\Gamma(y,x') > D$ and $d_\Gamma(x,y')
	+d_\Gamma(y',x') > D$, we must have $x' \in B$ also. So we can swap the roles of $x$ and $x'$, $y'$ and $y$,
	to find that any $z  \in B' \setminus \Gamma(y)$  has a unique neighbor closer to $y'$ and exactly $p_{11}^1$
	neighbors $z'$ with  $d_\Gamma(z',y')=d_\Gamma(z,y')$. Now select $z\in B'$ so as to maximize 
	$d_\Gamma(z,y) + d_\Gamma(z,y')$.   Since $d_\Gamma(y,y') \ge 3$, $z$ is non-adjacent to at least one of 
	$y,y'$; assume $z$ is not adjacent to $y'$.  Then $z$  has exactly $p_{11}^1$ neighbors
	$z'$ satisfying $d_\Gamma(z',y)=d_\Gamma(z,y)$.  Since $z$ maximizes  $d_\Gamma(z,y) + d_\Gamma(z,y')$,
	any neighbor of $z$ which lies farther away from $y$ must lie closer to $y'$. But there is exactly one such 
	vertex. In all, we have $| \Gamma(z) | = 1 + p_{11}^1 + 1$.  But $\Gamma$ is $K_{2,1,1}$-free so the
	neighborhood of any vertex is partitioned into cliques of size $p_{11}^1+1$. We find that $p_{11}^1+1$ 
	divides $p_{11}^1+2$. This can only happen if $p_{11}^1=0$; i.e., $\Gamma$ is triangle-free. But then 
	$z$ has degree two and $\Gamma$ must be a polygon.
\end{proof}

Our final two results deal with the special case where graph $\Gamma$ has diameter two.

\begin{thm}
	\label{Tdiam2}
	Let $(X,\cR)$ be a symmetric association scheme and let $\Gamma=\Gamma(X,R_1)$  be the graph 
	associated to a connected basis relation. If $\Gamma$ has diameter two and $|X| > k_1(t-1)+2$, 
	then  $\Gamma$ has vertex connectivity at least $t+1$ unless $t=k_1$.  
\end{thm}

\begin{proof}
	Let $T$ be a minimal disconnecting set of size at most $t$. For each $y\in T$, we 
	use the fact that any two vertices have at least one common neighbor to obtain
	$$ \left| \bigcup_{y \neq y' \in T} \Gamma(y') \right|  \le (k_1-1)(t-1)+1 $$
	so that there is some $x\in X \setminus T$ not adjacent to any element of $T$ except possibly $y$.
	Let $B$ be the component of 
	$\Gamma \setminus T$ containing $x$.  Since $\Gamma$ has diameter two, $x \sim y$ and 
	every $z \in X \setminus (B\cup T)$ must also be adjacent to $y$.  Swapping roles of the vertices in 
	$T$, we find that, for every $y$ in $T$, there is some vertex $x$ (necessarily in $B$) 
	with $\Gamma(x) \cap T = \{y\}$. But this implies that every 
	$z \in X \setminus (B\cup T)$ is adjacent to every vertex in $T$,  so $T = \Gamma(z)$
	for every $z \not\in B \cup T$.
\end{proof}

\begin{remark}
	With reference to Theorem \ref{Tdiam2}, we expect the case of $t=k_1$ to be very rare. If $t=k_1$, then we find that 
	$X \setminus (B\cup T)=\{z\}$ is a singleton and all but at most $k_1-2$ elements of 
	$B$ have exactly one neighbor in $T=\Gamma(z)$.  With $|X| \ge k_1^2 - k_1 + 3$  so
	close to the Moore bound, does this condition force $\Gamma$ to be a Moore graph?
\end{remark}

\cutsizethree
\begin{proof} The case where $\Gamma$ admits a disconnecting set of size two is handled by Theorem \ref{Tcutsize2}. Let $T=\{y_1,y_2,y_3\}$ be a minimal disconnecting set of size three.
	
	\noindent \underline{\sl Case (i):} $T \subseteq a^\bot$ for some $a\in X$.
	
	By Corollary \ref{C1}, we have $T=\Gamma(a)$ and $\Gamma$ has valency three; i.e., 
	$\Gamma$ is isomorphic to either $K_{3,3}$ or the Petersen graph.
	
	\medskip
	
	\noindent \underline{\sl Case (ii):} Assume $T$ is not contained in $a^\bot$ for any vertex $a$.
	
	In view of Theorem \ref{Tdiam2}, we may assume $|X| \le 2k_1  + 2$. (There is no cubic graph on nine vertices.)
	Let $B$ and $B'$ denote the vertex sets of two distinct connected components of $\Gamma \setminus T$
	and assume, without loss of generality, that $|B| \le |B'|$. Then we have $|B| \le \frac{ |X| - 3}{2}$. So
	$|B| - 1 \le k_1 - 2$. In view of Case (i), 
	we may assume each $x\in B$ is adjacent to exactly two members of $T$ and every pair of distinct
	vertices in $B$ is adjacent. This forces $|B|= k_1-1$. 
	Looking at $x \sim x'$ in $B$, we find that $p_{11}^1 \ge |B|-2 + 1$
	since $x$ and $x'$ must share a common neighbor in $T$.  Now compare this to some $y \in T$. Since 
	we are not in Case (i), some $y\in T$ is not adjacent to any other element of $T$. For this $y$, choose  some 
	neighbor $z$ of $y$ where $z\in B$ if $|\Gamma(y) \cap B| \le \frac{k_1}{2}$ and $z \in B'$ if  
	$|\Gamma(y) \cap B| > \frac{k_1}{2}$. The number of common neighbors of $y$ and $z$ is then at
	most $\frac{k_1}{2} - 1$. The inequalities $k_1 -2 \le p_{11}^1 \le \frac{k_1}{2} - 1$ then imply that 
	$\Gamma$ is a polygon, which is impossible as $T$ was chosen to be minimal.
\end{proof}

	\appendix
\label{appendix}
\chapter{Feasible parameter sets for LSSDs}
\label{families}
The Handbook of Combinatorial Designs gives us a list of 21 distinct families of symmetric designs. We now examine each family to determine which parameter sets could be the incident symmetric design between fibers in a LSSDs with three or more fibers. The two main conditions we will employ are that $s=\sqrt{k-\lambda}$ and $\nu = \frac{k(k\pm s)}{v}$ are integers, though recall from Lemma \ref{gcd} (and the remark that followed) that for any feasible parameter set, the following must hold in order for the set to be realizable:
\begin{enumerate}[label=(\roman*)]
	\item (Corollary \ref{comp}) $v$ must be composite;
	\item (Lemma \ref{gcd}\emph{(ii)}) $\gcd(v,k)>1$;
	\item (Lemma \ref{gcd}\emph{(iii)}) $\gcd(v,s)>1$;
	\item (Lemma \ref{gcd}\emph{(i)}) At most one of $\frac{k(k\pm s)}{v}$ is integral.
\end{enumerate}
Our results show that Families 6, 7, 9, 12, 13, and 14 are feasible. Further, Families 15-19 are feasible in specific cases ($m=1$) but will not be feasible in general. It should be noted that this does not mean that we can find LSSDs in each of these families with $w>2$, instead this means that we cannot disprove the existence of such LSSDs using only our integrality conditions. In fact, two of the families (McFarland/Wallis and Spence) were ruled out by Jedwab et al.\ (\cite{Jedwab2017}) in the case where the symmetric designs come from certain known constructions of difference sets. It is still open whether these families can produce LSSDs which do not arise from linking systems of difference sets.
\subsection*{Family 1 (Point-hyperplane Designs)}
\[\begin{aligned}
v &= q^m+\dots+1, \quad
k = q^{m-1}+\dots+1,\quad
\lambda = q^{m-2}+\dots+1,\quad n= q^{m-1},\quad s=q^{\frac{m-1}{2}}.\end{aligned}\]
Since $s$ is a power of $q$, we know that $\gcd(s,v) = 1$. Therefore via \emph{(iii)}, any LSSD with these design parameters will have $w=2$.
\subsection*{Family 2 (Hadamard matrix designs)}
\[\begin{aligned}
v &= 4n-1, \qquad k= 2n-1, \qquad \lambda = n-1,\qquad s = \sqrt{n}.\\
\end{aligned}\]
Since $s$ divides $v+1$, we know that $\gcd(s,v) = 1$. Therefore via \emph{(iii)}, any LSSD with these design parameters will have $w=2$.

\subsection*{Family 3 (Chowla)}
\[\begin{aligned}
v&=4t^2+1, \qquad k=t^2, \qquad \lambda = \frac{1}{4}(t^2-1),
\qquad s=\frac{1}{2}\sqrt{3t^2+1}.
\end{aligned}\]
Chowla designs require that $v$ is prime, therefore any LSSD with these design parameters will have $w=2$ due to \emph{(i)}.

\subsection*{Family 4 (Lehmer)}
\begin{enumerate}[label=(\arabic*)]
	\item
	\[\begin{aligned}
	v&= 4t^2+9,\qquad k=t^2+3,\qquad \lambda = \frac{1}{3}(t^2+3),\qquad  n=\frac{3}{4}k.
	\end{aligned}\]
	\item
	\[\begin{aligned}
	v&=8t^2+1 = 64u^2+9, \qquad k=t^2,\qquad \lambda=u^2,\qquad  n=t^2-u^2.\\
	\end{aligned}\]
	
	\item
	\[\begin{aligned}
	v&=8t^2+49 = 64u^2+441,\qquad k=t^2+6,\qquad \lambda = u^2+7,\qquad  n=t^2-u^2-1.\\
	\end{aligned}\]
	All three of the Lehmer designs require $v$ to be prime, therefore any LSSD with these design parameters will have $w=2$ due to \emph{(i)}.
\end{enumerate}

\subsection*{Family 5 (Whiteman)}
\[\begin{aligned}
v&=pq, \qquad k=\frac{1}{4}(pq-1), \qquad \lambda = \frac{1}{16}(pq-5),\qquad s=\frac{1}{4}(3p+1).
\end{aligned}\]
where $p$ and $q = 3p+2$ are both prime. Since $\gcd(s,v) >1$ we must have $s = p$ or $s=q$. However $s=q$ implies that $p$ is negative while $s=p$ implies that $p=1$ and $q=5$. As this case gives the design parameters $(5,1,0)$, only the degenerate case is possible. Therefore any non-degenerate LSSD using Whiteman design parameters will require $w=2$.

\subsection*{Family 6 (Menon)}
\[\begin{aligned}
v &= 4t^2, \qquad k= 2t^2-t, \qquad \lambda = t^2-t,\\
n&= t^2, \qquad s = t,\\
\nu &= \frac{(2t^2-t)(2t^2-t\pm t)}{4t^2}=\frac{1}{2}(2t-1)\left(t-\frac{1\mp1}{2}\right).
\end{aligned}\]
Since $2t-1$ will always be odd, we must have that $\left(t-\frac{1\mp1}{2}\right)$ is even. This means that for odd $t$, we must choose the $+$ so that we have $\nu = (2t-1)\frac{t-1}{2}$. If instead $t$ is even then we must choose the $-$ so that $\nu = (2t-1)\frac{t}{2}$. This means that Menon design parameters are feasible for all $t>0$, though our choice of $\mu$-heavy or $\nu$-heavy depends on the parity of $t$.

\subsection*{Family 7 (Wallis; McFarland)}
\[\begin{aligned}
v &= q^{m+1}(q^m+\dots+q+2), \quad k= q^m(q^m+\dots+q+1), \quad \lambda = q^m(q^{m-1}+\dots+q+1),\\ s &= q^m,\quad
\nu = \frac{q^m(q^m+\dots+q+1)(q^m(q^m+\dots+q+1)\pm q^m)}{q^{m+1}(q^m+\dots+q+2)}.\\
\end{aligned}\]
Consider first the case of $\nu$-heavy parameters,
\[\begin{aligned}
\nu&=\frac{q^m(q^m+\dots+q+1)(q^m(q^m+\dots+q+2))}{q^{m+1}(q^m+\dots+q+2)}=q^{m-1}(q^m+\dots+q+1).
\end{aligned}\]
As this is always an integer, we note using \emph{(iv)} that $\mu$-heavy parameters will never be feasible.

\subsection*{Family 8 (Wilson; Shrikhande and Singhi)}
\[\begin{aligned}
v&= m^3+m+1, \qquad k=m^2+1,\qquad \lambda = m,\qquad n=m^2-m+1.\\
\end{aligned}\]
Note that $v = mk+1$. Therefore $\gcd(k,v) = 1$ and, from \emph{(ii)}, any LSSD using these design parameters will have $w=2$.

\subsection*{Family 9 (Spence)}
\[\begin{aligned}
v&= 3^m\left(\frac{3^m-1}{2}\right),\qquad k=3^{m-1}\left(\frac{3^m+1}{2}\right),\qquad \lambda = 3^{m-1}\left(\frac{3^{m-1}+1}{2}\right),\qquad 
s=3^{m-1},\\
\nu&=\frac{\frac{1}{2}3^{m-1}(3^m+1)(\frac{1}{2}3^{m-1}(3^m+1)\pm3^{m-1})}{ \frac{1}{2}3^m(3^m-1)}.\end{aligned}\]

First consider $\mu$-heavy parameters,
\[\begin{aligned}
\nu&=\frac{(3^m+1)\left(\frac{1}{2}3^{m-1}(3^m-1)\right)}{3(3^m-1)}=3^{m-2}\left(\frac{3^m+1}{2}\right).
\end{aligned}\]
As this is always an integer, we note using \emph{(iv)} that $\nu$-heavy parameters will never be feasible.
\subsection*{Family 10 (Rajkundlia and Mitchell; Ionin)}
\[\begin{aligned}
v&=1+qr\left(\frac{r^m-1}{r-1}\right),\qquad k=r^m,\qquad \lambda = r^{m-1}\left(\frac{r-1}{q}\right),\qquad r=\frac{q^d-1}{q-1}.
\end{aligned}\]
Since $r$ divides $v-1$ and $k$ is a power of $r$, we know that $\gcd(v,k) = 1$. Therefore, by \emph{(ii)}, any LSSD using these design parameters will require $w=2$.
\subsection*{Family 11 (Wilson; Brouwer)}
\[\begin{aligned}
v&= 2(q^m+\dots + q)+1,\qquad k=q^m,\qquad \lambda=\frac{1}{2}q^{m-1}(q-1),\qquad n=\frac{1}{2}q^{m-1}(q+1).\\
\end{aligned}\]
Since $q$ divides $v-1$ and $k$ is a power of $q$, we must have that $\gcd(k,v) = 1$. Therefore, by \emph{(ii)}, any LSSD using these design parameters will require $w=2$.
\subsection*{Family 12 (Spence, Jungnickel and Pott, Ionin)}
\[\begin{aligned}
v&=q^{d+1}\left(\frac{r^{2m}-1}{r-1}\right),\qquad k=r^{2m-1}q^d,\qquad\lambda = (r-1)r^{2m-2}q^{d-1},\qquad s=r^{m-1}q^d, \\r&=\frac{q^{d+1}-1}{q-1},\qquad
\nu=\frac{r^{2m-1}q^d(r^{2m-1}q^d\pm r^{m-1}q^d)}{q^{d+1}\left(\frac{r^{2m}-1}{r-1}\right)} = \frac{q^{d-1}r^{3m-2}(r^{m}\pm 1)}{r^{2m-1}+\dots+1}.
\end{aligned}\]
First consider when $m=1$,
\[\begin{aligned}
v&=q^{d+1}\left(q^d+\dots+q+2\right),\qquad k=q^d(q^d+\dots+1),\qquad\lambda=q^{d}\left(q^{d-1}+\dots+q+1\right),\\ s&=q^d.
\end{aligned}\]
giving us the same design parameters as McFarland (Family 7). While these constructions may be distinct, our conditions only depend on the design parameters and thus these will work for $\nu$-heavy designs when $m=1$. If $m> 1$ however, $r^{3m-2}$ is relatively prime with the denominator, so we must have $(r^{2m-1}+\dots+1)\vert q^{d-1}\left(r^m\pm 1\right)$. Since $r = \frac{q^{d+1}-1}{q-1} = q^{d}+\dots+1$, we have that $q^{d-1}<r$. Therefore $q^{d-1}\left(r^m\pm 1\right)<r^{m+1}\pm r<r^{2m-1}\dots+1$ and thus any LSSD using these design parameters with $m> 1$ will require $w=2$.

\subsection*{Family 13 (Davis and Jedwab)}
\[\begin{aligned}
v&=2^{2d+4}\left(\frac{2^{2d+2}-1}{3}\right),\quad k=2^{2d+1}\left(\frac{2^{2d+3}+1}{3}\right),\quad \lambda = 2^{2d+1}\left(\frac{2^{2d+1}+1}{3}\right),\quad s=2^{2d+1},\\
\nu&=\frac{2^{2d+1}\left(\frac{2^{2d+3}+1}{3}\right)\left(2^{2d+1}\left(\frac{2^{2d+3}+1}{3}\right)\pm 2^{2d+1}\right)}{2^{2d+4}\left(\frac{2^{2d+2}-1}{3}\right)}=\frac{\left(2^{2d+3}+1\right)\left(\left(2^{2d+3}+1\right)\pm 3\right)2^{2d-2}}{3\left(2^{2d+2}-1\right)}.\\
\end{aligned}\]
First consider $\mu$-heavy parameters,
\[\begin{aligned}
\nu&=\frac{\left(2^{2d+3}+1\right)\left(2^{2d+3}-2\right)2^{2d-2}}{3\left(2^{2d+2}-1\right)}=\frac{\left(2^{2d+3}+1\right)2^{2d-1}}{3}.
\end{aligned}\]
As $2^n + 1$ is divisible by 3 any time $n$ is odd, this will always be an integer. Therefore, using \emph{(iv)}, $\nu$-heavy parameters will never be feasible.

\subsection*{Family 14 (Chen)}
\[\begin{aligned}
v&=4q^{2d}\left(\frac{q^{2d}-1}{q^2-1}\right),\quad k=q^{2d-1}\left(1+2\left(\frac{q^{2d}-1}{q+1}\right)\right),\quad\lambda = q^{2d-1}(q-1)\left(\frac{q^{2d-1}+1}{q+1}\right),\\ s&=q^{2d-1},\quad
\nu=\frac{q^{2d-1}\left(1+2\left(\frac{q^{2d}-1}{q+1}\right)\right)\left(q^{2d-1}\left(1+2\left(\frac{q^{2d}-1}{q+1}\right)\right)\pm q^{2d-1}\right)}{4q^{2d}\left(\frac{q^{2d}-1}{q^2-1}\right)}.\\
\end{aligned}\]
First consider $\mu$-heavy parameters,
\[\begin{aligned}
\nu&=\frac{\left(1+2\left(\frac{q^{2d}-1}{q+1}\right)\right)\left(2q^{2d-1}\left(\frac{q^{2d}-1}{q+1}\right)\right)}{4q\left(\frac{q^{2d}-1}{q^2-1}\right)}=\frac{q^{2d-2}(q-1)\left(1+2\left(\frac{q^{2d}-1}{q+1}\right)\right)}{2}\\
\end{aligned}\]
Since $2$ will always divide either $q^{2d-2}$ or $q-1$, we have that $\nu$ is integral under $\mu$-heavy parameters. Then from \emph{(iv)}, $\nu$-heavy parameters will never be feasible.

\subsection*{Family 15 (Ionin)}
\[\begin{aligned}
v&=q^d\left(\frac{r^{2m}-1}{(q-1)(q^d+1)}\right),\quad k= q^dr^{2m-1},\quad \lambda = q^d(q^d+1)(q-1)r^{2m-2},\quad s=q^dr^{m-1},\\
r&=q^{d+1}+q-1,\\
\nu&=\frac{q^dr^{2m-1}\left(q^dr^{2m-1}\pm q^dr^{m-1}\right)}{q^d\left(\frac{r^{2m}-1}{(q-1)(q^d+1)}\right)}=\frac{(q-1)(q^d+1)q^dr^{3m-2}}{(r^m\mp1)}.\\
\end{aligned}\]
First assume that $m=1$. Then,
\[\begin{aligned}
\nu&=\frac{(q-1)(q^d+1)q^dr}{(r\mp1)}.
\end{aligned}\]
First considering $\mu$-heavy parameters,
\[\begin{aligned}
\nu&=\frac{(q-1)(q^d+1)q^dr}{(r+1)}=\frac{(q-1)(q^d+1)q^dr}{q(q^d+1)}=(q-1)q^{d-1}r
\end{aligned}\]
Therefore these design parameters are feasible using $\mu$-heavy parameters when $m=1$ (and via \emph{(iv)}, $\nu$-heavy parameters are infeasible). Now consider when $m> 2$. In this case, note that $r^{3m-2}$ is relatively prime to $r^m\mp 1$. Therefore if $\nu$ is integral, then $r^m\mp 1$ must divide $q^d(q-1)(q^d+1)$. However, since $q\geq 2$ we know that $r = q(q^d+1)-1>q^d+1$ and $r = q^{d+1}+q-1>q^{d+1}-q^d$. Therefore $r^m\geq r^2 > q^d(q-1)(q^d+1)$ meaning that it is not possible for $r^m$ to divide the latter. Therefore $\nu$ will never be integral when $m>1$.

\subsection*{Family 16 (Ionin)}
\[\begin{aligned}
v&=2\cdot3^d\left(\frac{q^{2m}-1}{3^d+1}\right),\qquad k= 3^dq^{2m-1},\qquad\lambda = \frac{1}{2}3^d(3^d+1)q^{2m-2},\qquad s = 3^dq^{m-1},\\ q&=\frac{1}{2}(3^{d+1}+1),\qquad
\nu=\frac{3^dq^{2m-1}(3^dq^{2m-1}\pm3^dq^{m-1})}{2\cdot3^d\left(\frac{q^{2m}-1}{3^d+1}\right)}=\frac{3^d(3^d+1)q^{3m-2}}{2(q^m\mp1)}.\\
\end{aligned}\]
We again must consider the case when $m=1$ separately. If $m=1$, then
\[\begin{aligned}\nu&=\frac{3^d(3^d+1)q}{2(q\mp1)}\\
\end{aligned}\]
We first consider $\mu$ heavy parameters,
\[\begin{aligned}
\nu&=\frac{3^d(3^d+1)q}{3^{d+1}+3}=3^{d-1}q\\
\end{aligned}\]
Therefore when $m=1$, these design parameters are feasible with $\mu$-heavy parameters (and via \emph{(iv)}, $\nu$-heavy parameters are infeasible). Using the same arguments as before, we can quickly find that $\nu$ will not be an integer for $m>1$ noting that $q$ is relatively prime to $q^m\mp 1$ and $q^m-1> 3^d(3^d+1)$.

\subsection*{Family 17 (Ionin)}
\[\begin{aligned}
v&=3^d\left(\frac{q^{2m}-1}{2(3^d-1)}\right),\qquad k= 3^dq^{2m-1},\qquad\lambda = 2\left(3^d\right)(3^d-1)q^{2m-2},\qquad s=3^dq^{m-1},\\ q&=3^{d+1}-2,\\
\nu&=\frac{3^dq^{2m-1}\left(3^dq^{2m-1}\pm 3^dq^{m-1}\right)}{3^d\left(\frac{q^{2m}-1}{2(3^d-1)}\right)}=\frac{3^dq^{3m-2}\left(q^m\pm 1\right)\left(2(3^d-1)\right)}{\left(q^{2m}-1\right)}=\frac{2q^{3m-2}3^d(3^d-1)}{\left(q^{m}\mp1\right)}.\\
\end{aligned}\]
As before, we first consider the case when $m=1$ using $\nu$-heavy parameters,
\[\begin{aligned}
\nu&=\frac{2q^{3m-2}3^d(3^d-1)}{\left(q-1\right)}=\frac{2q^{3m-2}3^d(3^d-1)}{\left(3^{d+1}-3\right)}=2q^{3m-2}3^{d-1}.\\
\end{aligned}\]
Therefore when $m=1$, these design parameters are feasible with $\nu$-heavy parameters (and via \emph{(iv)}, $\mu$-heavy parameters are infeasible). We again find that $m>1$ will not permit $\nu$ to be an integer as $q^{3m-2}$ is relatively prime to $q^m\pm 1$ and $q^m-1>2\cdot3^d(3^d-1)$.

\subsection*{Family 18 (Ionin)}
\[\begin{aligned}
v&=2^{2d+3}\left(\frac{q^{2m}-1}{q+1}\right),\quad k= 2^{2d+1}q^{2m-1},\quad \lambda = 2^{2d-1}(q+1)q^{2m-2},\quad s=2^{2d+1}q^{m-1},\\ q&=\frac{1}{3}\left(2^{2d+3}+1\right),\quad
\nu=\frac{2^{2d+1}q^{2m-1}\left(2^{2d+1}q^{2m-1}\pm 2^{2d+1}q^{m-1}\right)}{2^{2d+3}\left(\frac{q^{2m}-1}{q+1}\right)}=\frac{(q+1)2^{2d-1}q^{3m-2}}{\left(q^{m}\mp1\right)}.\\
\end{aligned}\]
First consider $m=1$ using $\mu$-heavy parameters, then $\nu = 2^{2d-1}q^{3m-2}$. Due to \emph{(iv)} we see that $\nu$-heavy parameters will not be feasible. Further, $\nu$ is non integral when $m>1$ noting that $q^{3m-2}$ is relatively prime to $q^m\mp 1$ and $q^m-1>(q+1)2^{2d-1}$.

\subsection*{Family 19 (Ionin)}
\[\begin{aligned}
v&=2^{2d+3}\left(\frac{q^{2m}-1}{3q-3}\right),\quad k = 2^{2d+1}q^{2m-1}, \quad \lambda = 3\left(2^{2d-1}\right)(q-1)q^{2m-2},\quad s = 2^{2d+1}q^{m-1},\\ q&=2^{2d+3}-3,\quad
\nu=\frac{2^{2d+1}q^{2m-1}\left(2^{2d+1}q^{2m-1}\pm 2^{2d+1}q^{m-1}\right)}{2^{2d+3}\left(\frac{q^{2m}-1}{3q-3}\right)}=\frac{2^{2d-1}q^{3m-2}3(q-1)}{\left(q^m\mp1\right)}.\\
\end{aligned}\]
If $m=1$ and we take $\nu$-heavy parameters then $\nu = 2^{2d-1}3q$. Due to \emph{(iv)} we see that $\nu$-heavy parameters will not be feasible. Further, $\nu$ is non integral when $m>1$ noting that $q^{3m-2}$ is relatively prime to $q^m\mp 1$ and $q^m\mp1>3\cdot2^{2d-1}(q-1)$.

\subsection*{Family 20 (Ionin)}
For this family we use the only known realization where $p=2$ and $q=2^d-1$ is a Mersenne prime.
\[\begin{aligned}
v&=1+2^{d+1}\frac{2^{dm}-1}{2^d+1},\qquad k=2^{2dm} ,\qquad \lambda = 2^{2dm-d-1}(2^{d}+1),\qquad n=2^{2dm-d-1}(2^d-1).\\
\end{aligned}\]
Our first restriction tells us that $n$ must be a square. However since $2$ does not divide $2^d-1$, we know that $2^d-1$ must be a square in order for $n$ to be a square giving us a contradiction. Thus any LSSD with these design parameters will require $w=2$.
\subsection*{Family 21 (Kharaghani and Ionin)}
\[\begin{aligned}
v&=4t^2\left(\frac{q^{m+1}-1}{q-1}\right),\quad k=(2t^2-t)q^m,\quad \lambda = (t^2-t)q^m,\quad s=tq^{\frac{m}{2}},\quad q=(2t-1)^2,\\
\nu &=\frac{(2t^2-t)q^m\left((2t^2-t)q^m\pm tq^{\frac{m}{2}}\right)}{4t^2\left(\frac{q^{m+1}-1}{q-1}\right)}
=\frac{(2t-1)^{3m+1}(q-1)}{4\left((2t-1)^{m+1}\mp1\right)}.\\
\end{aligned}\]
First, since $(2t-1)$ is odd, we have that $(2t-1)^{3m+1}$ is relatively prime to $4((2t-1)^{m+1}\mp 1)$. However since $m\geq1$, $4((2t-1)^{m+1}\mp1) \geq 4(q-1)$ and thus $\nu$ is never integral. Thus any LSSD with these design parameters will require $w=2$.
\subsection*{Summary}
We have shown here that only Families 6, 7, 9, 13, and 14 will always satisfy our integrality conditions. Further, Families 12, 15, 16, 17, 18, and 19 satisfy our integrality conditions whenever $m=1$. Finally all remaining families will not allow for any LSSDs with $w>2$. 
	\addcontentsline{toc}{chapter}{Bibliography}
	\bibliographystyle{abbrv}
	\bibliography{Dissertation}
	\addcontentsline{toc}{chapter}{Index}
	\printindex
\end{document}